\documentclass[11pt,reqno]{amsart}
\usepackage{html}
\usepackage{url}
\usepackage[dcucite]{harvard}
\usepackage{graphicx,amsmath,amssymb,epsfig,harvard}
\usepackage{epstopdf}
\usepackage{rotating,setspace}
\usepackage{multirow}
\usepackage{harvard}

\long\def\comment#1{}
\newcommand{\En}{\mathbb{E}_n}
\usepackage{amsfonts}
\newtheorem{theorem}{Theorem}
\theoremstyle{plain}

\newtheorem{corollary}{Corollary}

\newtheorem{definition}{Definition}

\newtheorem{lemma}{Lemma}

\newtheorem{ConditionC}{Condition C.\!}

\newtheorem*{ConditionS}{Condition S}
\newtheorem*{ConditionNS}{Condition NS}
\newtheorem*{ConditionNK}{Condition NK}

\newtheorem*{ConditionP}{Condition P}

\newtheorem*{ConditionV}{Condition V}
\newtheorem{algorithm}{Algorithm}

\theoremstyle{definition}
\newtheorem{remark}{Remark}
\newtheorem{example}{Example}

\newcommand{\ba}{\begin{eqnarray*}}
\newcommand{\ea}{\end{eqnarray*}}
\newcommand{\ban}{\begin{eqnarray*}}
\newcommand{\ean}{\end{eqnarray*}}

\newcommand{\bs}{\begin{align}\begin{split}\nonumber}
\newcommand{\bsnumber}{\begin{align}\begin{split}}
\newcommand{\es}{\end{split}\end{align}}

\newcommand{\A}{A}
\newcommand{\B}{B}

\newcommand{\V}{\mathcal{V}}
\newcommand{\Vn}{V_n}

\newcommand{\barVn}{\overline{V}_n}
\newcommand{\KK} {\frac{1}{h_n^{d/2}}\mathbf{K}\( \frac{z - Z_i}{h_n}\)}
\newcommand{\KU}{ \KK U_{ij}}

\def\norm#1{\left\| #1 \right\|_{\Pn,2}}
\def\enorm#1{\left\| #1 \right\|_{\EPn,2}}
\newcommand{\normKU}{\norm{\KU}}
\newcommand{\enormKU}{\enorm{\KU}}

\newcommand{\Z}{\mathcal{Z}}

\newcommand{\D}{\mathcal{D}}

\renewcommand{\Pr}{\mathrm{P}}
\renewcommand{\P}{\mathrm{P}}
\newcommand{\Pn}{\mathrm{P}_n}
\newcommand{\Gn}{\mathbb{G}_n}
\newcommand{\Bn}{\mathbb{B}_n}
\newcommand{\barBn}{\bar{\mathbb{B}}_n}

\newcommand{\EPn}{\mathbb{P}_n}

\renewcommand{\qed}{\hfill \ensuremath{\blacksquare}}

\newcommand{\di}{\displaystyle}

\renewcommand{\(}{\left(}
\renewcommand{\)}{\right)}
\renewcommand{\cite}{\citeasnoun}
\DeclareMathOperator*{\arginf}{arg\,inf}

\numberwithin{equation}{section}

\begin{document}
\title[ ]{Intersection Bounds: Estimation and Inference}
\thanks{\tiny We are especially grateful to D. Chetverikov, K. Kato, Y. Luo, A. Santos, five anonymous referees, and a co-editor for making several extremely useful suggestions that have led to substantial improvements.  We thank T. Armstrong, R. Blundell, A. Chesher, F. Molinari, W. Newey, C. Redmond, N. Roys, S. Stouli, and J. Stoye for detailed discussion and suggestions, and participants at numerous seminars and conferences
for their comments. This paper is a revised version of ``Inference on Intersection Bounds'' initially presented and circulated at the University of Virginia and the Harvard/MIT econometrics seminars in December 2007, and presented at the March 2008 CEMMAP/Northwestern conference on ``Inference in Partially Identified Models with Applications.''
We gratefully acknowledge financial support from the National Science Foundation, the Economic and Social Research Council (RES-589-28-0001, RES-000-22-2761) and the European Research Council (ERC-2009-StG-240910-ROMETA)}
\author{Victor Chernozhukov}\thanks{\tiny Victor Chernozhukov: Department of Economics, Massachusetts Institute of Technology, vchern@mit.edu.}
\author{Sokbae Lee}\thanks{\tiny Sokbae Lee: Department of Economics, Seoul National University and CeMMAP, sokbae@gmail.com.}
\author{Adam M. Rosen}\thanks{\tiny Adam Rosen: Department of Economics, University College London and CeMMAP, adam.rosen@ucl.ac.uk.}
\date{March 2013. First version: December 2007.}

\begin{abstract}
\footnotesize{We develop a practical and novel method for inference on intersection
bounds, namely bounds defined by either the infimum or supremum of a
parametric or nonparametric function, or equivalently, the value of a linear
programming problem with a potentially infinite constraint set. We show that many bounds characterizations in econometrics, for instance bounds on parameters under conditional moment inequalities, can be formulated as intersection bounds. Our
approach is especially convenient for models comprised of a continuum of
inequalities that are separable in parameters, and also applies to models
with inequalities that are nonseparable in parameters.  Since analog
estimators for intersection bounds can be severely biased in finite
samples, routinely underestimating the size of the identified set, we also
offer a median-bias-corrected estimator of such
bounds as a by-product of our inferential procedures.
We develop theory for large sample inference based on the strong approximation of a sequence of series or kernel-based empirical processes by a sequence of ``penultimate'' Gaussian processes.  These penultimate processes are generally not weakly convergent, and thus non-Donsker.  Our theoretical results establish that we can nonetheless perform asymptotically valid inference based on these processes. Our construction
also provides new adaptive inequality/moment selection methods. We provide conditions for the use of nonparametric kernel and series
estimators, including a novel result that establishes strong approximation
for any general series estimator admitting linearization, which may be of independent interest.\\
}

{\tiny \noindent \textsc{Key words.} Bound analysis, conditional moments,  partial identification, strong approximation, infinite-dimensional constraints, linear programming, concentration inequalities, anti-concentration inequalities, non-Donsker empirical process methods, moderate
deviations, adaptive moment selection. \\

}

\end{abstract}

\maketitle

\newpage

\section{Introduction}

This paper develops a practical and novel method for estimation
and inference on intersection bounds.
Such bounds arise in settings where the parameter of interest, denoted $\theta
^{\ast } $, is known to lie within the bounds $\left[ \theta
^{l}\left( v\right) ,\theta ^{u}\left( v\right) \right] $ for
each $v$ in some set $\mathcal{V} \subseteq \mathbb{R}^d$, which may be uncountably infinite. \ The identification region for $\theta ^{\ast}$ is then
\begin{equation}\label{eq:Ibound}
\Theta _{I}=\cap _{v\in \mathcal{V}}\left[ \theta ^{l}\left(
v\right)
,\theta ^{u}\left( v\right) \right] =\left[ \sup\nolimits_{v\in \mathcal{V}%
}\theta ^{l}\left( v\right) ,\inf\nolimits_{v\in
\mathcal{V}}\theta ^{u}\left( v\right) \right] \text{.}
\end{equation}
Intersection bounds stem naturally from exclusion restrictions (\cite{Manski:03}) and appear in numerous applied and theoretical examples.\footnote{Examples include average treatment effect bounds from instrumental variable restrictions (\cite{Manski:90}), bounds on the distribution of treatment effects in a randomized experiment (\cite{Heckman/Smith/Clements:97}), treatment effect bounds from nonparametric selection equations with exclusion restrictions (\cite{Heckman/Vytlacil:99}), monotone instrumental variables and the returns to schooling (\cite{Manski/Pepper:00}), English auctions (\cite{Haile/Tamer:03}), the returns to language skills (\cite{Gonzalez:05}), changes in the distribution of wages (\cite{Blundell/Gosling/Ichimura/Meghir:07}), the study of disability and employment
(\cite{Kreider/Pepper:07}), unemployment compensation reform
(\cite{Lee/Wilke:05}), set identification with Tobin regressors (\cite{Chernozhukov/Rigobon/Stoker:07}), endogeneity with discrete outcomes (\cite{Chesher:07}), estimation of income poverty measures
(\cite{Nicoletti/Foliano/Peracchi:07}), bounds on average structural functions and treatment effects in triangular systems (\cite{Shaikh/Vytlacil:11}), and set identification with imperfect instruments (\cite{Nevo/Rosen:08}).}
\ A leading case is that where the bounding functions are conditional expectations with continuous conditioning variables, yielding conditional moment inequalities. More generally, the
methods of this paper apply to any estimator for the value of a linear
programming problem with an infinite dimensional
constraint set.

This paper covers both parametric and non-parametric estimators of bounding
functions $\theta ^{l}\left(\cdot\right)$ and $\theta ^{u}\left(\cdot\right)$.  We provide formal justification for parametric, series, and kernel-type estimators via asymptotic theory based on the strong approximation of a
sequence of empirical processes by a sequence of
Gaussian processes. This includes an important
new result on strong approximation for series estimators that
applies to any estimator that admits a linear approximation,
essentially providing a functional central limit theorem for
series estimators for the first time in the literature. In addition, we generalize existing results on the strong approximation of kernel-type estimators to regression models with multivariate outcomes, and we provide a novel multiplier method to approximate the distribution of such estimators. For each of these estimation methods, the paper provides
 \begin{itemize}
 \item[(i)] confidence regions that achieve a desired asymptotic level,
  \item[(ii)] novel adaptive inequality selection (AIS) needed to construct sharp
   critical values, which in some cases result in confidence regions with exact asymptotic size,\footnote{
 The previous literature, e.g. \cite{Chernozhukov/Hong/Tamer:07} and contemporaneous papers, such as \cite{Andrews/Shi:08}, use ``non-adaptive" cutoffs such as $C\sqrt{\log n}$. Ideally $C$ should depend on the problem at hand and so careful calibration might be required in practice.
 Our
 new AIS procedure provides data-driven, adaptive cutoffs, which do not require calibration.
 Note that our AIS procedure could be iterated via stepdown, for example, as in \cite{Chetverikov:12}. We omit the details for brevity.}
 \item[(iii)] convergence rates for the boundary points of these regions,
  \item[(iv)] a characterization of local alternatives against which the associated tests have non-trivial power,
  \item[(v)] half-median-unbiased estimators of the intersection bounds.
  \end{itemize}
Moreover, our paper also extends inferential theory based on empirical processes in Donsker settings to non-Donsker cases.  The empirical processes arising in our problems do not converge weakly to a Gaussian process, but can be strongly approximated by a sequence of ``penultimate" Gaussian processes, which we use directly for inference without resorting to further approximations, such as extreme value approximations as in \cite{Bickel/Rosenblatt:73}. These new methods may be of independent interest for a variety of other problems.

Our results also apply to settings where a parameter of interest, say $\mu$, is characterized by intersection bounds of the form (\ref{eq:Ibound}) on an auxiliary function $\theta \( \mu \)$.  Then the bounding functions have the representation
\begin{equation}\label{eq:parametricBGF}
\theta ^{l}\left(v\right):=\theta ^{l}\left(v;\mu\right)\text{ and }\theta ^{u}\left(v\right):=\theta ^{u}\left(v;\mu\right)\text{,}
\end{equation}
and thus inference statements for $\theta^{*} := \theta(\mu)$ bounded by $\theta ^{l}\left(\cdot\right)$ and $\theta ^{u}\left(\cdot\right)$ can be translated to inference statements for the parameter $\mu$. \ This includes cases where the bounding functions are a collection of conditional moment functions indexed by $\mu$. When the auxiliary function is additively separable in $\mu$, the relation between the two is simply a location shift.  When the auxiliary function is nonseparable in $\mu$, inference statements on $\theta^{*}$ still translate to inference statements on $\mu$, though the functional relation between the two is more complex.

This paper overcomes significant complications for estimation of and inference on intersection bounds.
\ First, because the bound estimates are suprema and infima of
parametric or nonparametric estimators, closed-form characterization of their asymptotic
distributions is typically unavailable or difficult to establish. As a consequence, researchers have often
used the canonical bootstrap for inference, yet the recent literature indicates that the canonical bootstrap is
not generally consistent in such settings, see e.g. \cite{Andrews/Han:08}, \cite{Bugni:07}, and \cite{Canay:07}.\footnote{The recent papers \cite{Andrews/Shi:08} and \cite{Kim:08} provide justification for subsampling procedures for the statistics they employ for inference with conditional moment inequalities.  We discuss these papers further in our literature review below.} Second, since sample analogs of the bounds
of $\Theta _{I}$ are the suprema and infima of estimated bounding functions, they have substantial finite sample bias, and estimated bounds tend to be
much tighter than the population bounds. \ This has been noted by Manski and Pepper (2000, 2009),
and some heuristic bias adjustments have been proposed by \cite{Haile/Tamer:03} and \cite{Kreider/Pepper:07}.

We solve the problem of estimation and inference for
intersection bounds by proposing bias-corrected estimators of the upper and lower bounds, as well
as confidence intervals. Specifically, our approach employs a
precision-correction to the estimated bounding
functions $ v \mapsto \widehat{\theta}^{l}\left( v\right) $ and
$v \mapsto \widehat{\theta}^{u}\left( v\right)$  before
applying the supremum and infimum operators. \ We
adjust the estimated bounding functions for their
precision by adding to each of them an appropriate critical
value times their pointwise standard error. \ Then, depending
on the choice of the critical value, the intersection of these
precision-adjusted bounds provides (i) confidence sets for either the identified
set $\Theta_{I}$ or the true parameter value $\theta
^{\ast }$, or (ii) bias-corrected estimators for the lower and upper bounds. Our bias-corrected estimators are half-median-unbiased in the sense that the upper bound estimator $\widehat{\theta}^{u}$ exceeds $\theta^u$ and the lower bound estimator $\widehat{\theta}^{l}$ falls below $\theta^l$ each with probability at least one half asymptotically. Due to the presence of the $\inf$ and $\sup$ operators in the definitions of $\theta^u$ and $\theta^l$, achieving unbiasedness is impossible in general, as shown by \cite{Hirano/Porter:09}, and this motivates our half-unbiasedness property.  Bound estimators with this property are also proposed by \cite{Andrews/Shi:08}, henceforth AS.  An attractive feature of our approach is that
the only difference in the construction of our estimators and
confidence intervals is the choice of a critical value. Thus,
practitioners need not implement two entirely different methods
to construct estimators and confidence bands with desirable
properties.

This paper contributes to a growing literature on inference on
set-identified parameters bounded by inequality restrictions.
The prior literature has focused
primarily on models with a finite number of unconditional
inequality restrictions.  Some examples include
\cite{Andrews/Jia:08},
\cite{Andrews/Guggenberger:09},
\cite{Andrews/Soares:10},
\cite{Beresteanu/Molinari:08},
\cite{Bugni:07},
\cite{Canay:07},
\cite{Chernozhukov/Hong/Tamer:07},
\cite{Galichon/Henry:06a},
\cite{Romano/Shaikh:06b}, \cite{Romano/Shaikh:06a},
and \cite{Rosen:05}, among others.
We contribute to this literature by considering inference with a continuum of inequalities. \ Contemporaneous and independently written research on conditional moment inequalities includes AS,
\cite{Kim:08}, and \cite{Menzel:08}. \ Our approach differs from all of these. Whereas we treat the problem with fundamentally nonparametric methods, AS provide inferential statistics that transform the model's conditional restrictions to unconditional ones through the use of instrument functions.\footnote{Thus, the two approaches also require different assumptions. We rely on the strong approximation of a studentized version of parametric or nonparametric bounding function estimators (e.g. conditional moment functions in the context of conditional moment inequalities), while AS require that a functional central limit theorem hold for the \emph{transformed} unconditional moment functions, which involve instrument functions not present in this paper.} In this sense our approach is similar in spirit to that of \cite{Haerdle/Mammen:93} (although they use the $L^2$ norm and we use a sup test), while the approach of AS parallels that of \cite{Bierens:82} for testing a parametric specification against a nonparametric alternative. As such, these approaches are complementary, each with their relative merits, as we describe in further detail below. \cite{Kim:08} proposes an inferential method related to that of AS, but where data dependent indicator functions play the role of instrument functions. \cite{Menzel:08} considers problems where the number of moment inequalities defining the identified set is large relative to the sample size.  He provides results on the use of a subset of such restrictions in any finite sample, where the number of restrictions employed grows with the sample size, and examines the sensitivity of estimation and inference methods to the rate with which the number of moments used grows with the sample size.

The classes of models to which our approach and others in the recent literature apply have considerable overlap, most notably in models comprised of conditional moment inequalities, equivalently models whose bounding functions are conditional moment functions. Relative to other approaches, our approach is especially convenient for inference in parametric and nonparametric models with a continuum of inequalities that are separable in parameters, i.e. those admitting representations of the form $$\sup\nolimits_{v\in \mathcal{V}}\theta ^{l}\left( v\right) \leq \theta^{\ast} \leq \inf\nolimits_{v\in \mathcal{V}}\theta ^{u}\left( v\right)\text{,}$$ as in \eqref{eq:Ibound}. 
Our explicit use of nonparametric estimation of bounding functions renders our method applicable in settings where the bounding functions depend on exogenous covariates in addition to the variable $V$, i.e. where the function $\theta(x)$ \emph{at a point} $x$ is the object of interest, with
$$\sup\nolimits_{v\in \mathcal{V}
}\theta ^{l}\left(x,v\right) \leq \theta(x) \leq \inf\nolimits_{v\in
\mathcal{V}}\theta ^{u}\left(x,v\right) \text{.}$$
When the functions $\theta ^{l}\left(x,v\right)$ and $\theta ^{u}\left(x,v\right)$ are nonparametrically specified, these can be estimated by either the series or kernel-type estimators we study in Section \ref{sec:leading-cases}.  At present most other approaches do not appear to immediately apply when we are interested in $\theta(x)$ at a point $x$, when covariates $X$ are continuously distributed, with the exception of the recent work by \cite{Fan/Park:11} in the context instrumental variable (IV) and monotone instrumental variable (MIV) bounds,  and that of \cite{Andrews/Shi:11}, which extends methods developed in AS to this case.\footnote{The complication is that inclusion of additional covariates in a nonparametric framework requires a method for localization of the bounding function around the point $x$.  With some non-trivial work and under appropriate conditions, the other approaches can likely be adapted to this context.}

To better understand the comparison between our point and interval estimators and those of AS when both are applicable, consider as a simple example the case where $\theta^* \leq E[Y|V]$ almost surely, with $E[Y|V=v]$ continuous in $v$. \ Then the upper bound on $\theta^*$ is $\theta_0  = \inf_{v \in \V} E[Y|V=v]$ over some region $\V$. \ $\theta_0$ is a nonparametric functional and can in general only be estimated at a nonparametric rate. That is, one can not construct point or interval estimators that converge to $\theta_0$ at superefficient rates, i.e. rates that exceed the optimal nonparametric rate for estimating $\theta(v) := \theta^{u}(v) = E[Y|V=v]$.\footnote{Suppose for example that $V_0= \arginf_{v \in \V}\theta(v)$ is singleton, with $\theta_0 = \theta(v)$ for some $v \in \V$. Then $\theta_0$ is a nonparametric function evaluated at a single point, which cannot be estimated faster than the optimal nonparametric rate. Lower bounds on the rates of convergence in nonparametric models are characterized e.g. by \cite{Stone:82} and \cite{Tsybakov:09}.  Having a uniformly super-efficient procedure would contradict these lower bounds.}  Our procedure delivers point and interval estimators that can converge to $\theta_0$ at this rate, up to an undersmoothing factor. However, there exist point and interval estimators that can achieve faster (superefficient) convergence rates at \textit{some} values of $\theta(\cdot)$. In particular, if the bounding function $\theta(\cdot)$ happens to be flat on the argmin set $V_0=\{ v \in \V: \theta(v)= \theta_0\}$, meaning that $V_0$ is a set of positive Lebesgue measure, then the point and interval estimator of AS can achieve the convergence rate of $n^{-1/2}$. As a consequence, their procedure for testing $\theta_{na} \leq \theta_0$ against $\theta_{na} > \theta_0$, where $\theta_{na} = \theta_0 + C/\sqrt{n}$ for $C>0$, has non-trivial asymptotic power, while our procedure does not. If, however, $\theta(\cdot)$ is not flat on $V_0$, then the testing procedure of AS no longer has power against the aforementioned $n^{-1/2}$ alternatives, and results in point and interval estimators that converge to $\theta_0$ at a sub-optimal rate.\footnote{With regard to confidence intervals/interval estimators, we  mean here that the upper bound of the confidence interval does not converge at this rate.} In contrast, our procedure delivers point and interval estimators that can converge at nearly the optimal rate, and hence can provide better power in these cases.   In applications both flat and non-flat cases are important, and we therefore believe that both testing procedures are useful.\footnote{Note also that non-flat cases can be justified as generic if e.g. one
takes $\theta(\cdot)$ as a random draw from the Sobolev ball equipped with the Gaussian (Wiener) measure.} For further comparisons, we refer the reader to our Monte-Carlo Section and to Supplemental Appendices \ref{sec:local-asy-power} and \ref{sec:mc-supplement}, which confirm these points both analytically and numerically. \footnote{See
Supplemental Appendix \ref{sec:local-asy-power} for specific examples, and see \cite{Armstrong:11}
for a comprehensive analysis of the power properties of the procedure of \cite{Andrews/Shi:08}.  We also note that this qualitative comparison of local asymptotic power properties conforms with related results regarding tests of parametric models versus nonparametric (PvNP) alternatives, which involve moment \emph{equalities}. \ Recall that our test relies on
nonparametric estimation of bound-generating functions, which often take the form of conditional moment inequalities, and is similar in spirit to the approach of e.g. \cite{Haerdle/Mammen:93} in the PvNP testing literature. On the other hand, the statistics employed by AS rely on a transformation of conditional restrictions to unconditional ones in similar spirit to \cite{Bierens:82}. \ Tests of the latter type have been found to have power against some $n^{-1/2}$ alternatives, while the former do not. However, tests of the first type typically have non-trivial power against a larger class of alternatives, and so achieve higher power against some classes of alternatives. \ For further details see for example \cite{Horowitz/Spokoiny:01} and the references therein. }

There have also been some more recent additions to the literature on conditional moment inequalities. \cite{Lee/Song/Whang:11} develop a test for conditional moment inequalities using a one-sided version of $L^p$-type functionals of kernel estimators. Their approach is based on a least favorable configuration that permits valid but possibly conservative inference using standard normal critical values. \cite{Armstrong:11} and \cite{Chetverikov:11} both propose interesting and important approaches to estimation and inference based on conditional moment inequalities, which can be seen as introducing full studentization in the procedure of AS, fundamentally changing its behavior.  The resulting procedures use a collection of fully-studentized nonparametric estimators for inference, bringing them much closer to the approach of this paper.   Their implicit nonparametric estimators are locally constant with an adaptively chosen bandwidth. Our approach is specifically geared to smooth cases, where $\theta^u(\cdot)$ and $\theta^l(\cdot)$ are continuously differentiable of order $s \geq 2$ resulting in more precise estimates and hence higher power in these cases.  On the other hand, in less smooth cases, the procedures of \cite{Armstrong:11} and \cite{Chetverikov:11} automatically adapt to deliver optimal estimation and testing procedures, and so can perform somewhat better than our approach in these cases.\footnote{Note that to harness power gains higher order kernels
or series estimators should be used; our analysis allows for either.} \cite{Armstrong:11b} derives the convergence rate and asymptotic distribution for a test statistic related to that in \cite{Armstrong:11} when evaluated at parameter values on the boundary of the identified set, drawing a connection to the literature on nonstandard M-estimation. \cite{Ponomareva:10} studies bootstrap procedures for inference using kernel-based estimators, including one that can achieve asymptotically exact inference when the bounding function is uniquely maximized at a single point and locally quadratic. Our simulation-based approach does not rely on these conditions for its validity, but will automatically achieve asymptotic exactness with appropriately chosen smoothing parameters in a sufficiently regular subset of such cases.

\subsection*{Plan of the Paper} We organize the paper as follows. In Section \ref{sec:examples and overview}, we motivate the analysis with examples and provide an informal overview of our
results. \ In Section \ref{sec:est-inf-general-conditions} we provide a formal treatment of our
method under high level conditions.  In Section \ref{sec:leading-cases} we provide conditions and theorems for validity for parametric and nonparametric series and kernel-type estimators. We provide several examples that demonstrate the use of primitive conditions to verify the conditions of Section \ref{sec:est-inf-general-conditions}. This includes sufficient conditions for the application of these estimators to models comprised of conditional moment inequalities. 
In Section \ref{sec:strong:series:main-text} we provide a theorem that establishes strong approximation for series estimators admitting an asymptotic linear representation, and which covers the examples of Section \ref{sec:leading-cases}. Likewise, we provide a theorem that establishes strong approximation for kernel-type estimators in Section \ref{sec:proofs-strong-kernel} of an on-line supplement. \ In Section \ref{sec:implementation-details} we provide step-by-step implementation guides for parametric and non-parametric series and kernel-type estimators. \ In Section \ref{sec:monte-carlo} we illustrate the performance of our method using both series and kernel-type estimators in Monte Carlo experiments, which we compare to that of AS in terms of coverage frequency and power. \ Our method performs well in these experiments, and we find that our approach and that of AS perform favorably in different models, depending on the shape of the bounding function. \ Section \ref{sec:conclusion} concludes. \ In Appendices \ref{sec:definitions and notation} - \ref{sec:proofs-strong-series} we recall the definition of strong approximation and provide proofs, including the proof of the strong approximation result for series estimators. \ The on-line supplement contains further appendices. The first of these, Appendix \ref{sec:omitted-proof-appendix} provides proofs omitted from the main text.\footnote{Specifically, Appendix \ref{sec:omitted-proof-appendix} contains the proofs of Lemmas \ref{lemma: estimation of V} and \ref{lemma:parametric}.} Appendices \ref{sec:est:local:CMIexample}-\ref{sec:est:local:proofs} concern kernel-type estimators, providing primitive conditions for their application to conditional moment inequalities, strong approximation results and the multiplier method enabling inference via simulation, and proofs. Appendix \ref{sec:prac:roymodel} provides additional details on the use of primitive conditions to verify an asymptotic linear expansion needed for strong approximation of series estimators and Appendix \ref{sec:prac:lle} gives some detailed arguments for local polynomial estimation of conditional moment inequalities. Appendix \ref{sec:local-asy-power} provides local asymptotic power analysis that supports the findings of our Monte Carlo experiments. Appendix \ref{sec:mc-supplement} provides further Monte Carlo evidence.

\subsection*{Notation}  For any two reals $a$ and $b$, $a\vee b = \max\{a,b\}$ and $a\wedge b = \min\{a,b\}$. $Q_p(X)$ denotes the $p$-th quantile of random variable $X$. We use wp$\rightarrow 1$ as shorthand for ``with probability approaching one as $n\rightarrow\infty$.''  We write $\mathcal{N}_k =_d N(0,I_k)$ to denote that the k-variate random vector $\mathcal{N}_k$ is distributed multivariate normal with mean zero and variance the $k\times k$ identity matrix. To denote probability statements conditional on observed data, we write statements conditional on $\mathcal{D}_n$.
$\mathbb{E}_n$ and $\mathbb{P}_n$ denote the sample mean and empirical measure, respectively. That is, given i.i.d. random vectors $X_1,\ldots,X_n$, we have $\mathbb{E}_n f = \int f d\mathbb{P}_n = n^{-1} \sum_{i=1}^n f(X_i)$.
In addition, let $\mathbb{G}_n f = \sqrt{n}(\mathbb{E}_n - E) f = n^{-1/2} \sum_{i=1}^n [ f(X_i) - E f(X) ]$.
The notation $a_n \lesssim b_n$ means that $a_n \leq C b_n$ for all $n$;  $X_n \lesssim_{\P_n} c_n$ abbreviates  $X_n=O_{\P_n}(c_n)$. $X_n\rightarrow_{\P_n}\infty$ means that for any constant $C>0$, $\P_n(X_n<C)\rightarrow 0$. We use $\textsf{V}$ to denote a generic compact subset of $\mathcal{V}$, and we write $\text{diam}(\textsf{V})$ to denote the diameter of $\textsf{V}$ in the Euclidean metric. $\|\cdot\|$ denotes the Euclidean norm, and for any two sets $A,B$ in Euclidean space, $d_H(A,B)$ denotes the Hausdorff pseudo-distance between $A$ and $B$ with respect to the Euclidean norm.  $C$ stands for a generic positive constant, which may be different in different
places, unless stated otherwise.  For a set $\textsf{V}$ and an element $v$ in Euclidean space, let $d(v,\textsf{V}) := \inf_{v' \in \textsf{V}} \| v - v' \|$.
For a function $p(v)$, let $\text{lip}(p)$ denote the Lipschitz coefficient, that is $\text{lip}(p) := L$ such that $\| p(v_1) - p(v_2) \| \leq L \| v_1 - v_2 \|$ for all $v_1$ and $v_2$ in the domain of $p(v)$.

\section{Motivating Examples and Informal Overview of Results}\label{sec:examples and overview}

In this section we briefly describe three examples of intersection
bounds from the literature and provide an informal overview of our results.

\subsection*{Example A: Treatment Effects and Instrumental Variables}

In the analysis of treatment response, the ability to uniquely
identify the distribution of potential outcomes is
typically lacking without either experimental data or strong
assumptions. \ This owes to the fact that for each individual unit
of observation, only the outcome from the received treatment is
observed; the counterfactual outcome that would have occurred
given a different treatment is not known. \ Although we focus here
on treatment effects, similar issues are present in other areas of
economics. \ In the analysis of markets, for example, observed
equilibrium outcomes reveal quantity demanded at the observed
price, but do not reveal what demand would have been at other
prices.

Suppose only that the support of the outcome
space is known, $Y\in[0,1]\text{,}$ but no other assumptions are made regarding the
distribution of counterfactual outcomes. \ \nocite{Manski:89}\nocite{Manski:90} Manski (1989, 1990) provides worst-case bounds on mean treatment
outcomes for any treatment $t$ conditional on observables $(X,V) = (x,v)$,
$$\theta ^{l}\left( x,v\right) \leq E\left[ Y\left( t\right) |X=x,V=v\right] \leq \theta ^{u}\left( x,v\right)\text{,}$$
where the bounds are
\begin{eqnarray*}
\theta ^{l}\left( x,v\right) &:=& E [ Y \cdot 1\{Z=t\}|X=x,V=v] \text{,} \\
\theta ^{u}\left( x,v\right) &:=& E [ Y \cdot 1\{Z=t\} + 1\{Z\neq t\}|X=x,V=v] \text{,}
\end{eqnarray*}
where $Z$ is the observed treatment. \ If $V$ is an
instrument satisfying $E\left[ Y\left( t\right) |X,V\right] =E%
\left[ Y\left( t\right) |X\right] $, then for any fixed $x$, bounds on
$\theta^\ast := \theta^\ast(x) :=  E\left[ Y\left( t\right) |X=x\right] $ are given by
$$\sup\nolimits_{v\in \mathcal{V}}\theta ^{l}\left(x,v\right) \leq \theta^\ast(x) \leq
\inf\nolimits_{v\in\mathcal{V}}\theta ^{u}\left(x,v\right),$$
for any $\mathcal{V} \subseteq \text{support}(V|X=x)$, where the subset
$\mathcal{V}$ will be taken as known for estimation purposes. Similarly,
bounds implied by restrictions such as monotone treatment
response, monotone treatment selection, and monotone
instrumental variables, as in \cite{Manski:97} and
\cite{Manski/Pepper:00}, also take the form of intersection
bounds.

\subsection*{Example B: Bounding Distributions to Account for Selection}

Similar analysis applies to inference on
distributions whose observations are
censored due to selection. \ This approach is used by \cite{Blundell/Gosling/Ichimura/Meghir:07} to study changes in male and
female wages. \ The starting
point of their analysis is that the cumulative distribution $F(w|x,v)$ of
wages $W$ at any point $w$, conditional on
observables $(X,V)=(x,v)$ must satisfy the worst case bounds
\begin{equation}
\theta ^{l}\left( x,v\right) \leq
F\left( w|x,v\right)
\leq \theta ^{u}\left( x,v\right) \text{,}
\label{Roy IV Bounds}
\end{equation}
where $D$ is an indicator of employment, and hence observability of $W$, so that
\begin{eqnarray*}
\theta ^{l}\left( x,v\right) &:=& E [ D \cdot 1\{W\leq w\}|X=x,V=v] \text{,} \\
\theta ^{u}\left( x,v\right) &:=& E [ D \cdot 1\{W\leq w\} + \(1-D\)|X=x,V=v] \text{.}
\end{eqnarray*}
This relation is used
to bound quantiles of conditional wage distributions.
Additional restrictions motivated by economic theory are then used to tighten the bounds.

One such restriction  is an exclusion restriction
of the continuous variable out-of-work income, $V$. \ They consider the use of $V$ as either an excluded or monotone instrument. \
 The former restriction implies bounds on the parameter $\theta^* := F\left( w|x\right)$,
\begin{align}\label{ex2-eq1}
\sup\nolimits_{v\in \mathcal{V}}\theta ^{l}\left(x,v\right) \leq  F\left( w|x\right)
\leq \inf\nolimits_{v\in\mathcal{V}}\theta ^{u}\left(x,v\right)
\text{,}
\end{align}
for any $\mathcal{V} \subseteq \text{support}(V|X=x)$, while the weaker monotonicity restriction, namely that
$F(w|x,v)$ is weakly increasing in $v$, implies the following bounds on $\theta^* := F\left( w|x,v_{0}\right)$ for any $v_{0}$ in $\text{support}(V|X=x)$,
\begin{equation}\label{ex2-eq2}
\sup\nolimits_{v \in \mathcal{V}_l}\theta ^{l}\left(x,v\right) \leq  F\left( w|x,v_{0}\right) \leq \inf\nolimits_{v \in \mathcal{V}_u}\theta ^{u}\left(x,v\right)\text{,}
\end{equation}
where $\mathcal{V}_l = \{v \in \mathcal{V}: v \leq v_0 \}$ and $\mathcal{V}_u = \{v \in \mathcal{V}: v \geq v_0 \}$.

\subsection*{Example C: (Conditional) Conditional Moment Inequalities}

Our inferential method can also be used for pointwise
inference on parameters restricted by (possibly conditional) conditional moment
inequalities. \ Such restrictions arise
naturally in empirical work in industrial organization, see for example
\cite{Pakes/Porter/Ho/Ishii:05} and \cite{Berry/Tamer:07}.

To illustrate, consider the restriction
\begin{equation}
E\left[ m_j\left( X,\mu _{0}\right) | Z=z\right] \geq 0\text{ for all } j=1,...,J \text { and }  z \in  \mathcal{Z}_j.
\label{cond-moment-restriction}
\end{equation}
where each $m_j\left( \cdot ,\cdot \right), j=1,...,J$, is a real-valued function,
$\left( X,Z\right) $ are observables, and$\ \mu _{0}$ is the parameter of interest.
Note that this parameter can depend on a particular covariate value. Suppose for instance that $Z=(Z_1,Z_2)$ and interest lies in the subgroup of the population with $Z_1 = z_1$, so that the researcher wishes to condition on $Z_1 = z_1$. In this case $\mu_0 = \mu_0(z_1)$ depends on $z_1$. Conditioning on this value, we have from \eqref{cond-moment-restriction} that $$E\left[ m_j\left( X,\mu _{0}\right) | Z_1=z_1, Z_2=z_2\right] \geq 0\text{ for all } j=1,...,J \text { and }  z_2 \in \text{ Supp}(Z_2|Z_1=z_1),$$ which is equivalent to \eqref{cond-moment-restriction} with $\mathcal{Z}_j = \text{ Supp} (Z|Z_1 = z_1)$. \ Note also that regions $\mathcal{Z}_j$ can depend on the inequality
$j$ as in \eqref{ex2-eq2} of the previous example, and that the previous two examples can in fact be cast as special cases of this one.

Suppose that we would like to test (\ref{cond-moment-restriction}) at level $\alpha $ for the
conjectured parameter value $\mu _{0}=\mu $ against an
unrestricted alternative.  \ To
see how this can be done, define
$$v = (z,j), \ \ \mathcal{V}:= \{ (z,j): z \in \mathcal{Z}_j, j \in \{1,...,J\}\} \text{
 and } \theta \left( \mu ,v\right) := E\left[ m_{j}\left( X,\mu \right)
|Z=z\right]$$ and $\widehat{\theta}\left( \mu ,v\right) $ a
consistent estimator. \ Under some continuity conditions
this is equivalent to a test of $\theta _{0}\left( \mu \right) := \inf_{v\in
\mathcal{V}}\theta \left( \mu ,v\right) \geq 0\text{ against
} \inf_{v\in \mathcal{V}}\theta \left( \mu ,v\right)
<0$ . \ Our method for
inference delivers a statistic
\[
\widehat{\theta}_{\alpha }( \mu ) =\inf_{v\in
\mathcal{V}}\left[ \widehat{\theta}\left( \mu ,v\right) + \widehat{k}\cdot
s\left( \mu ,v\right) \right]
\]%
such that $\lim\sup_{n\rightarrow \infty }P( \theta _{0}\left(
\mu \right) \geq \widehat{\theta}_{\alpha }(\mu)) \leq
\alpha $ under the null hypothesis. \
Here, $s\left( \mu ,v\right)$ is the standard error of $\widehat{\theta}\left( \mu ,v\right)$ and $\widehat{k}$ is an estimated critical value, as we describe below.
If $\widehat{\theta}_{\alpha }( \mu )  < 0$,
we reject the null hypothesis, while if
$\widehat{\theta}_{\alpha }( \mu ) \geq 0$, we do not.

\subsection*{Informal Overview of Results}

We now provide an informal description of our method for
estimation and inference. \ Consider an upper bound $\theta_0$ on
$\theta^\ast$ of the form
\begin{align}\label{upper-bound}
\theta^* \leq \theta_0 := \inf_{v \in \mathcal{V}} \theta(v),
\end{align}
where $v \mapsto \theta(v)$ is a bounding function, and
$\mathcal{V}$ is the set over which the infimum is taken. \ We
focus on describing our method for the upper bound \eqref{upper-bound}, as the lower bound is entirely symmetric. \ In fact, any combination of upper and lower bounds can be combined into upper bounds on an auxiliary function of $\theta^{\ast}$ of the form \eqref{upper-bound}, and this can be used for inference on $\theta^{\ast}$, as we describe in Section \ref{sec:implementation-details}.\footnote{Alternatively, one can combine one-sided intervals for lower and upper bounds for inference on the identified set $\Theta_I$ using Bonferroni's inequality, or for inference on $\theta^{\ast}$ using the method described in \cite{Chernozhukov/Lee/Rosen:09} Section 3.7, which is a slight generalization of methods previously developed by \cite{Imbens/Manski:04} and \cite{Stoye:07}.}

What are good \textit{estimators} and \textit{confidence
regions} for the bound $\theta_0$? A natural idea is to base estimation and inference on the sample
analog: $\inf_{v \in \mathcal{V}} \widehat \theta(v)$. However,
this estimator does not perform well in practice. First, the
analog estimator tends to be downward
biased in finite samples. As discussed in the introduction, this will typically result in bound estimates that are much narrower than those in the population, see e.g. \cite{Manski/Pepper:00} and \cite{Manski/Pepper:08} for more on this point. Second, inference must appropriately take account of
sampling error of the estimator $\widehat \theta(v)$
across all values of $v$. Indeed,
different levels of precision of $\widehat \theta(v)$ at
different points can severely distort the perception of the minimum
of the bounding function $\theta(v)$. Figure \ref{fig1}
illustrates these problems geometrically. The solid curve is
the true bounding function $v \mapsto \theta(v)$, and
the dash-dotted thick curve is its estimate $v \mapsto \widehat
\theta(v)$. The remaining dashed curves represent eight
additional potential realizations of the estimator,
illustrating its precision. In particular, we
see that the precision of the estimator is much lower on the
right side than on the left.  A na\"{i}ve sample analog
estimate for $\theta_0$ is provided by the minimum of the
dash-dotted curve, but this estimate can  in fact be quite far
away from $\theta_0$. This large deviation from the true value
arises from both the lower precision of the estimated curve  on
the right side of the figure and from the downward bias created
by taking the minimum of the estimated curve.

To overcome these problems, we propose a
\emph{precision-corrected} estimate of $\theta_0$:
\begin{align}\label{estimator}
\widehat \theta_0(p):= \inf_{v \in \V} [\widehat \theta(v) + k(p) \cdot s(v)],
\end{align}
where $s(v)$ is the standard error of $\widehat \theta(v)$,
 and $k(p)$ is a critical value, the selection of which is described below.
That is, our estimator $\widehat \theta_0 (p)$ minimizes the
\emph{precision-corrected curve} given by $\widehat \theta(v)$
plus critical value $k(p)$ times the pointwise standard error
$s(v)$. Figure \ref{fig2} shows a precision-corrected curve as
a dashed curve with a particular choice of critical value $k$.
In this figure, we see that the minimizer of the
precision-corrected curve can indeed be much closer to
$\theta_0$ than the sample analog $\inf_{v \in
\mathcal{V}}\widehat \theta(v)$.

These issues are important both in theory and in practice, as can be seen in the application of the working paper \cite{Chernozhukov/Lee/Rosen:09}. There we used the data from the National Longitudinal Survey of Youth of
1979 (NLSY79), as in \cite{Carneiro/Lee:07}, to estimate bounds on expected log wages $Y_i$ as a function of years of schooling $t$. We used Armed Forces Qualifying Test score (AFQT) normalized to have mean zero as a monotone instrumental variable, and estimated the MIV-MTR (monotone instrument variable - monotone treatment response) bounds of \cite{Manski/Pepper:00}. Figures \ref{fig2-real-b} and \ref{fig2-real} highlight the same issues as the schematic figures \ref{fig1} and \ref{fig2} with the NLSY data and the MIV-MTR upper bound for the parameter $\theta^{\ast} = \P[Y_i(t) > y|V_i=v]$, at $y = \log(24)$ ($\sim 90^{th}$ percentile of hourly wages) and $v=0$ for college graduates $(t=16)$.\footnote{The parameter $\theta^{\ast}$ used for this illustration differs from the conditional expectations bounded in \cite{Chernozhukov/Lee/Rosen:09}. For further details regarding the application and the data we refer to that version, available at \underline{http://cemmap.ifs.org.uk/wps/cwp1909.pdf}.}

Figure \ref{fig2-real-b} shows the nonparametric series estimate of the bounding function using B-splines as described in Section \ref{sec:monte-carlo-series} (solid curve) and 20 bootstrap estimates (dashed curves). The precision of the estimate is worst when the AFQT is near 2, as demonstrated by the bootstrap estimates. At the same time, the bounding function has a decreasing shape with the  minimum at $\text{AFQT} = 2$. Figure \ref{fig2-real} shows a precision-corrected curve (solid curve) that adjusts the bound estimate $\widehat{\theta}(v)$ (dashed curve) by an amount proportional to its point-wise standard error, and the horizontal dashed line shows the end point of a 95\% one-sided confidence interval. As in Figure \ref{fig2}, the minimizer of the precision-corrected curve is quite far from that of the uncorrected estimate of the bounding function.

The degree of precision correction, both in these figures and in general, is driven by the critical value $k(p)$. The main input in the selection of $k(p)$ for the estimator $\widehat \theta_0(p)$ in (\ref{estimator}) is the standardized process
$$
Z_n(v) = \frac{\theta(v) - \widehat \theta(v)}{\sigma(v)},
$$
where $\sigma(v)/s(v) \to 1$ in probability uniformly in $v$. Generally, the finite sample distribution of the process $Z_n$ is unknown, but we can
approximate it uniformly by a sequence of Gaussian processes $Z_n^{\ast}$ such that for an appropriate sequence of constants $\bar a_n$
\begin{equation}
\bar a_n \sup_{v \in \mathcal{V}} | Z_n(v) -Z_n^{\ast}(v)| = o_p(1)\text{.}
\end{equation}
For any compact set $\textsf{V}$, used throughout to denote a generic compact subset of $\mathcal{V}$, we then approximate the quantiles of $\sup_{v\in \textsf{V}}Z_n^{\ast}(v)$ either by analytical methods based on asymptotic approximations, or by simulation.  We then use the $p$-quantile of this statistic, $k_{n,\textsf{V}}(p)$, in place of $k(p)$ in (\ref{estimator}). We show that in general simulated critical values provide sharper inference, and therefore advocate their use.

For the estimator in (\ref{estimator}) to exceed $\theta_0$ with probability no less than $p$ asymptotically, we require that wp$\rightarrow 1$ the set $\textsf{V}$ contains the argmin set $$V_0:= \arginf_{v\in\mathcal{V}}\theta(v)\text{.}$$
A simple way to achieve this is to use $\textsf{V}=\mathcal{V}$, which leads to asymptotically valid but conservative inference. \ For construction of the critical value $k_{n,\textsf{V}}(p)$ above we thus propose the use of a preliminary set estimator $\widehat V_n$ for $V_0$ using a novel adaptive inequality selection procedure. \ Because the critical value $k_{n,\textsf{V}}(p)$ is non-decreasing in $\textsf{V}$ for $n$ large enough, this yields an asymptotic critical value no larger than those based on $\textsf{V}=\mathcal{V}$. \ The set estimator $\widehat V_n$ is shown to be sandwiched between two non-stochastic sequences of sets, a lower envelope $V_n$ and an upper envelope $\barVn$ with probability going to one. We show in Lemma \ref{lemma: concentrate on balls} that our inferential procedure using $\widehat V_n$ concentrates on the lower envelope $\Vn$, which is a neighborhood of the argmin set $V_0$. This validates our use of the set estimator $\widehat V_n$. The upper envelope $\barVn$, a larger - but nonetheless shrinking - neighborhood of the argmin set $V_0$, plays an important role in the derivation of estimation rates and local power properties of our procedure. Specifically, because this set contains $V_0$ wp $\to 1$, the tail behavior of $\sup_{v \in \barVn} Z^*_n(v)$ can be used to bound the estimation error of $\widehat \theta_{0}(p)$ relative to $\theta_{0}$.

Moreover, we show that in some cases inference based on simulated critical values using $\widehat V_n$ in fact ``concentrates'' on $V_0$ rather than just $V_n$. \ These cases require that the scaled penultimate process $\bar a_n Z_n^*$ behaves sufficiently well (i.e. to be stochastically equicontinuous) within $r_n$ neighborhoods of $V_0$, where $r_n$ denotes the rate of convergence of the set estimator $\widehat V_n$ to the argmin set $V_0$. \ When this holds the tail behavior of $\sup_{v \in V_0} Z^*_n(v)$ rather than $\sup_{v \in \barVn} Z^*_n(v)$ bounds the estimation error of our estimator. This typically leads to small improvements in the convergence rate of our estimator and the local power properties of our approach. The conditions for this to occur include the important special case where $V_0$ is singleton, and where the bounding function is locally quadratic, though can hold more generally. The formal conditions are given in Section \ref{sec:v-est}, where we provide conditions for consistency and rates of convergence of $\widehat V_n$ for $V_0$, and in Section \ref{sec:auto-sharpness} where we provide the aforementioned equicontinuity condition and a formal statement of the result regarding when inference concentrates on $V_0$.

At an abstract level our method does not distinguish parametric
estimators of $\theta(v)$ from nonparametric estimators;
however, details of the analysis and regularity conditions are
quite distinct. Our theory for nonparametric estimation relies on undersmoothing, although for locally constant or sign-preserving estimation of bounding functions, this does not appear essential since the approximation bias is conservatively signed. In such cases, our inference algorithm still applies to nonparametric estimates of bounding functions without undersmoothing, although our theory would require some minor modifications to handle this case. We do not formally pursue this here, but we provide some simulation results for kernel estimation without undersmoothing as part of the additional Monte Carlo experiments reported in supplementary appendix \ref{sec:mc-supplement}.

For all estimators, parametric and nonparametric, we employ strong approximation analysis to approximate
the quantiles of $\sup_{v\in \textsf{V}}Z_n(v)$, and we verify our conditions separately for each case. The formal definition of strong approximation is provided in Appendix \ref{sec:definitions and notation}.

\section{Estimation and Inference Theory under General Conditions}\label{sec:est-inf-general-conditions}

\subsection{Basic Framework}\label{sec:basic-framework}

In this and subsequent sections we allow the model and the probability measure to depend on $n$.
Formally, we work with a probability space $(A, \mathcal{A}, \P_n)$ throughout.  This
approach is conventionally used in asymptotic statistics to ensure robustness of statistical conclusions with respect to perturbations in $\P_n$. It guarantees the validity of our inference procedure under
any sequence of probability laws $\P_n$ that obey our conditions, including the case with fixed $\P$. We thus generalize our notation in this section to allow model parameters to depend on $n$.

The basic setting is as follows:
\begin{ConditionC}[Setting]
There is a non-empty compact set $\mathcal{V} \subset \mathcal{K} \subset \Bbb{R}^d$, where $\mathcal{V}$
can depend on $n$, and $\mathcal{K}$ is a bounded fixed set, independent of $n$. There is a continuous
real valued function $v \mapsto \theta_n(v)$.  There is an estimator $v \mapsto \widehat \theta_n(v)$ of this function, which is an a.s. continuous stochastic process.  There is  a continuous function $v \mapsto \sigma_n(v)$  representing non-stochastic normalizing factors bounded by $ \bar \sigma_n:= \sup_{v \in \V} \sigma_n(v)$, and there is an estimator $v \mapsto s_n(v)$ of these factors, which is an a.s. continuous stochastic process, bounded above by $  \bar s_n:= \sup_{v \in \V} s_n(v)$.
\end{ConditionC}
We are interested in constructing point estimators and one-sided interval estimators for
$$
\theta_{n0} = \inf _{v \in \V} \theta_n(v).
$$
The main input in this construction is the standardized process
$$
Z_n(v) = \frac{\theta_n(v) - \widehat \theta_n(v)}{\sigma_n(v)}.
$$
In the following we require that this process can be approximated by a standardized Gaussian process in the metric space
$\ell^\infty(\V)$  of bounded functions mapping $\V$ to $\Bbb{R}$, which can be simulated for inference.

\begin{ConditionC}[Strong Approximation] (a)
$Z_n$ is strongly approximated by a sequence of penultimate Gaussian processes $Z_n^*$ having zero mean and a.s. continuous sample paths:
\begin{eqnarray*}\label{C.1a}
 &  & \sup_{v \in \mathcal{V}}|Z_n (v) - Z_n^*(v) | =  o_{\Pn}\left(\delta_n\right),
\end{eqnarray*}
where $E_{\Pn}[(Z_n^*(v))^2]=1$ for each $v \in \mathcal{V}$, and $\delta_n = o(\bar a_n^{-1})$ for the sequence of constants $\bar a_n$ defined in Condition C.3 below. (b) Moreover, for simulation purposes, there is a process $Z_n^\star$, whose distribution is zero-mean Gaussian conditional on the data $\D_n$ and such that $E_{\Pn}[(Z_n^\star(v))^2 \mid \mathcal{D}_n] =1$ for each $v \in \mathcal{V}$, that can approximate an identical copy $\bar Z_n^*$ of $Z^*_n$, where $\bar Z_n^*$ is independent of $\D_n$,  namely
there is an $o(\delta_n)$ term such that
$$
\P_n\left[\sup_{v \in \V}|\bar Z_n^*(v) - Z_n^\star(v)|> o(\delta_n) \mid \D_n\right] = o_{\Pn}\left(1/\ell_n\right)
$$ for some $\ell_n \to \infty$ chosen below.

\end{ConditionC}
For convenience we refer to Appendix \ref{sec:definitions and notation}, where the definition
of strong approximation is recalled.  The penultimate process  $Z_n^*$ is often called  a coupling, and we construct such couplings for parametric and nonparametric estimators under both high-level and primitive conditions. It is convenient to work with $Z_n^*$, since we can rely on the
fine properties of Gaussian processes.  Note that $Z_n^*$ depends on $n$ and generally does not converge weakly
to a fixed Gaussian process, and therefore is not asymptotically Donsker.  Nonetheless we can
perform either analytical or simulation-based inference based on these processes.

Our next condition captures the so-called \textit{concentration} properties of Gaussian processes:

\begin{ConditionC}[Concentration]  For all $n$ sufficiently large and for any compact, non-empty $\textsf{V} \subseteq \V$, there is a normalizing factor $a_n(\textsf{V})$ satisfying
\begin{equation*}
1 \leq a_n(\textsf{V}) \leq a_n(\V)=:\bar a_n, \ \  a_n(\textsf{V}) \text{ is weakly increasing in } \textsf{V},
\end{equation*}
such that
 $$\mathcal{E}_n(\textsf{V}):= a_n(\textsf{V}) \left( \sup_{v \in \textsf{V}} Z^*_n(v)- {a}_n(\textsf{V})\right)$$
obeys
\begin{equation}
\Pn[\mathcal{E}_n(\textsf{V}) \geq x] \leq \Pr[\mathcal{E} \geq x],
\end{equation}
where $\mathcal{E}$ is a random variable with continuous distribution function such that for some $\eta>0$, $$\P(\mathcal{E} > x) \leq \exp(-x/\eta)\text{.}$$
\end{ConditionC}
The \textit{concentration} condition will be verified in our applications by appealing to the Talagrand-Samorodnitsky inequality
for the concentration of the suprema of Gaussian processes, which is sharper than the classical concentration inequalities.\footnote{For details see Lemma \ref{lemma: key concentration} in Appendix \ref{sec:proofs-max-strong-auxiliary}.} These concentration properties play a key role in our analysis, as they determine the uniform speed of convergence
$\bar a_n \bar \sigma_n$ of the estimator $\widehat \theta_{n0} (p)$
to $\theta_{n0}$, where the estimator is defined later.
In particular this property implies that for any compact $V_n \subseteq \V$,
$
E_{\P_n}[\sup_{v \in V_n} Z^*_n(v)] \lesssim \bar a_n$.
As there is concentration, there is an opposite force, called \emph{anti-concentration}, which implies that
under C.2(a) and C.3 for any $\delta_n = o(1/\bar a_n)$ we have
\begin{equation}\label{eq: anti}
\sup_{x \in \Bbb{R}} \Pn \Big ( |\sup_{v \in V_n} Z^*_n(v) -x|  \leq  \delta_n \Big ) \to 0.
\end{equation}
This follows from a generic anti-concentration inequality derived in \cite{Chernozhukov/Kato:11}, quoted in Appendix
\ref{sec:main-section-proofs} for convenience.  Anti-concentration simplifies the construction of our confidence intervals.
Finally, the exponential tail property of $\mathcal{E}$ plays an important
role in the construction of our adaptive inequality selector, introduced below, since
it allows us to bound moderate deviations of the one-sided estimation noise of $\sup_{v \in \textsf{V}} Z^*_n(v)$.

Our next assumption requires uniform consistency as well as suitable estimates of  $\sigma_n$:

\begin{ConditionC}[Uniform Consistency] \text{} We have that
$$
\di \text{(a) } \bar {a}_n \bar \sigma_n  = o\(1\) \ \ \text{ and }  \ \ \text{(b)} \ \sup_{v \in \V}  \left |\frac{s_n(v)}{\sigma_n(v)} - 1 \right|
 = o_{\Pn}\(\frac{\delta_n}{\bar a_n + \ell\ell_n}\), $$
where $\ell\ell_n \nearrow \infty$ is a sequence of constants defined below.

\end{ConditionC}

\noindent In what follows we let
$$
\ell_n := \log n, \text{ and } \ell\ell_n := \log \ell_n,
$$
but it should be noted that $\ell_n$ can be replaced by other slowly increasing sequences.

\subsection{The Inference and Estimation Strategy}\label{sec:inf-est-strategy}

For any compact subset $\textsf{V} \subseteq \mathcal{V}$ and $\gamma \in (0,1)$, define:
$$
\kappa _{n, \textsf{V}}(\gamma) :=  Q_{\gamma}\(\sup_{ v \in \textsf{V}} Z^*_n(v)\).
$$

The following result is useful for establishing validity of our inference procedure.

\begin{lemma}[\textbf{Inference Concentrates on a Neighborhood $\Vn$ of $V_0$}]\label{lemma: concentrate on balls} Under C.1-C.4
$$
\di \Pn \( \sup_{v \in \mathcal{V}} \frac{\theta_{n0} - \widehat \theta_n(v)}{s_n(v)} \leq x \)  \geq
 \Pn\(\sup_{v \in \Vn} {Z^*_n(v) }\leq  x\) - o(1),
$$
uniformly in  $x \in [0,\infty)$, where
\begin{eqnarray}\label{define:V-epsilon n}
&& \di \Vn := \left \{ v \in \mathcal{V}: \theta_n(v) \leq \theta_{n0} + \kappa_n \sigma_n(v) \right \},  \text{ for } \kappa_n:= \kappa_{n, \mathcal{V}}(\gamma_n'),
\end{eqnarray}
where $\gamma'_n$ is any sequence such that $\gamma_n' \nearrow1$ with
$\kappa_n\leq\bar a_n + \eta(\ell\ell_n+C^\prime)/\bar a_n$ for some constant $C^\prime>0$.
\end{lemma}

Thus, with probability converging to one, the inferential process ``concentrates'' on a neighborhood of $V_0$ given by $\Vn$. The ``size" of the neighborhood is determined by $\kappa_n$, a
high quantile of $\sup_{v \in \V} Z^*_n(v)$, which summarizes the maximal one-sided estimation error over $\V$. We use this to construct half-median-unbiased estimators for $\theta_{n0}$  as well as  one-sided
interval estimators for $\theta_{n0}$ with correct asymptotic level, based on analytical and simulation methods for obtaining critical values proposed below.  

\begin{remark}[\textbf{Sharp Concentration of Inference}] In general, it is not possible for the inferential processes to concentrate on smaller subsets than $V_n$.  However, as shown, in Section \ref{sec:auto-sharpness}, in some special cases, e.g. when $V_0$ is a well-identified singleton, the inference process will in fact concentrate on $V_0$. In this case our simulation-based construction will automatically adapt to deliver median-unbiased estimators for $\theta_{n0}$ as well as one-sided interval estimators for $\theta_{n0}$ with exact asymptotic size. Indeed, in the special but extremely important case
of $V_0$ being singleton we can achieve
$$
\Pn\(\sup_{v \in \Vn} {Z^*_n(v) }  > x\) = \mathrm{Pr}\Big( N(0,1) >x \Big) - o(1),
$$
under some regularity conditions. In this case, our simulation-based procedure will automatically produce a critical value that approaches the $p$-th quantile of the standard normal, delivering
asymptotically exact inference.
\qed \end{remark}

Our construction relies first on an \emph{auxiliary critical value} $k_{n, \mathcal{V}}(\gamma_n)$, chosen so that wp $\to$ 1, \begin{equation} \label{auxiliary cv inequality}
 k_{n, \mathcal{V}}(\gamma_n) \geq  \kappa_{n, \V}(\gamma_n'),
 \end{equation}
where we set $\gamma_n:=1 - .1/\ell_n \nearrow 1$ and $\gamma_n \geq \gamma_n' = \gamma_n - o(1) \nearrow 1$. This critical value is used to obtain a preliminary set estimator
\begin{equation} \label{define: Vn}
\widehat V_n = \left\{ v \in \V: \widehat \theta_n(v) \leq \inf_{\tilde{v} \in \mathcal{V}}\( \widehat \theta_n(\tilde{v}) +  k_{n, \mathcal{V}}(\gamma_n) s_n(\tilde{v})\) + 2   k_{n, \mathcal{V}}(\gamma_n) s_n(v) \right \},
 \end{equation}
 The set estimator $\widehat V_n$ is then used in the construction of the \emph{principal critical value $k_{n,\widehat V_n}(p)$}, $p \geq 1/2,$ where we require that wp $\to$ 1,
\begin{equation} \label{principal cv inequality}
k_{n,\widehat V_n}(p) \geq \kappa_{n, \Vn}(p-o(1)).
 \end{equation}
The principal critical value is fundamental to our construction of confidence regions and estimators, which we now define.

\begin{definition}[Generic Interval and Point Estimators]\label{main definition} Let $p \geq 1/2$, then our interval estimator takes the form:
 \begin{equation} \label{define: thetap}
   \widehat \theta_{n0}(p) = \inf_{v \in \V}\left [\widehat \theta_n(v) +  k_{n,\widehat V_n}(p) s_n(v) \right],
\end{equation}
where the half-median unbiased estimator corresponds to $p=1/2$. 
\end{definition}

The principal and auxiliary critical values are constructed below so as to satisfy \eqref{auxiliary cv inequality} and \eqref{principal cv inequality} using either analytical or simulation methods. As a consequence, we show in Theorems \ref{theorem: inference analytical} and \ref{theorem: inference1} that
\begin{equation}\label{eq: key fact}
\di\Pn\left \{ \theta_{n0} \leq \widehat \theta_{n0}(p) \right \} \geq p - o(1)\text{,}
 \end{equation}
for any fixed $1/2 \leq p  < 1$.  The construction relies on the new set estimator $\widehat V_n$, which we
call an \textit{adaptive inequality selector} (AIS), since it uses the problem-dependent cutoff $k_{n, \mathcal{V}}(\gamma_n)$,
which is a bound on a high quantile of $\sup_{v \in \V} Z_n^*(v)$.  The analysis therefore must take into account the  moderate deviations (tail behavior) of the latter.

Before proceeding to the details of its construction, we note that the argument for establishing the coverage results and analyzing power properties of the procedure depends crucially on the following
result proven in Lemma \ref{lemma: estimation of V} below:
$$\Pn \Big \{ \Vn \subseteq \widehat V_n \subseteq \barVn \Big \}  \to 1, $$
where
\begin{equation} \label{define: Veps-bar-n}
\barVn := \left\{ v \in \V: \theta_n(v) \leq \theta_{n0} +  \bar\kappa_n  \bar\sigma_n \right \}, \text{  for  } \bar \kappa_n :=  7( \bar a_n  + \eta \ell\ell_n/\bar a_n),
 \end{equation}
where $\eta > 0 $ is defined by Condition C.3.
Thus, the preliminary set estimator $\widehat V_n$ is sandwiched between two deterministic sequences of sets, facilitating the analysis of its impact on the convergence of $\widehat \theta_{n0}(p)$ to $\theta_{n0}$.

\subsection{Analytical Method and Its Theory}\label{sec:analytical}

Our first construction is quite simple and demonstrates the main -- though not the finest -- points.
This construction uses the majorizing variable $\mathcal{E}$ appearing in C.3.

\begin{definition}[Analytical Method for Critical Values]\label{analytical definition} For any compact set $\textsf{V}$  and  any $p \in (0,1)$, we set
\begin{eqnarray} \label{define: anal 1}
k_{n, \textsf{V}}(p) = a_n(\textsf{V}) + c(p)/ a_n(\textsf{V}),
 \end{eqnarray}
 where $c(p) = Q_p(\mathcal{E})$ is the p-th quantile of the majorizing variable $\mathcal{E}$ defined in C.3.
 where for any fixed $p \in (0,1)$, we require that $\textsf{V} \mapsto k_{n, \textsf{V}}(p)$ is weakly monotone increasing in $\textsf{V}$ for sufficiently large $n$.

\end{definition}

\noindent The first main result is as follows.

\begin{theorem}[\textbf{Analytical Inference, Estimation, Power under C.1-C.4}]\label{theorem: inference analytical} Suppose C.1-C.4 hold. Consider  the interval estimator given in Definition \ref{main definition} with critical value function given in
Definition \ref{analytical definition}. Then,  for a given $p \in [1/2,1),$ \\
1. The interval estimator has asymptotic level $p$:
$$
\di\Pn\left \{ \theta_{n0} \leq \widehat \theta_{n0}(p) \right \} \geq p - o(1).
$$
2. The estimation risk is bounded by, wp $\to 1$ under $\Pn$,
\begin{equation*}
\di \left | \widehat \theta_{n0}(p)-\theta_{n0}\right |  \leq  4 \bar\sigma_n
 \left( a_{n}(\barVn)+  \frac{O_{\Pn}(1)}{{a}_{n}(\barVn)}\right) \lesssim_{\Pn} \bar\sigma_n \bar a_n.
\end{equation*}%
3. Hence, any
alternative $\theta_{na}> \theta_{n0}$ such that $$
\di \theta _{na} \geq \theta_{n0} +  4 \bar\sigma_n \left( a_{n}(\barVn)+  \frac{\mu_n }{{a}_{n}(\barVn)} \), \   \mu_n \to_{\Pn} \infty, $$
is rejected with probability converging to 1 under $\Pn$.
\end{theorem}

Thus, $(-\infty,\widehat \theta_{n0}(p)]$ is a valid one-sided interval estimator for $\theta_{n0}$.  Moreover, $\widehat \theta_{n0}(1/2)$  is a half-median-unbiased estimator for
$\theta_{n0}$ in the sense that
$$
\liminf_{n \to \infty} \P_n\left[\theta_{n0} \leq \widehat \theta_{n0}(1/2)\right] \geq 1/2.
$$
The rate of convergence of $\widehat \theta_{n0}(p)$ to $\theta_{n0}$ is bounded above by the uniform rate $\bar\sigma_n \bar a_n$ for estimation of
the bounding function $v \mapsto \theta_n(v)$. This implies that the test of $\mbox{H}_0: \theta_{n0}= \theta_{na}$ that rejects if $\theta_{na} > \widehat \theta_{n0}(p)\text{ }$
asymptotically rejects all local alternatives that are more
distant\footnote{Here and below we ignore various constants appearing
in front of terms like $\bar\sigma_n\bar a_n$.} than $\bar\sigma_n \bar a_n$, including fixed alternatives as a special case.
In Section \ref{sec:leading-cases} below we show that in parametric cases this results in power against $n^{-1/2}$ local alternatives. \ For series estimators $\bar a_n\bar \sigma_n$ is proportional to $(\log n)^c \sqrt{K/n}$ where $c$ is some positive constant, and $K\rightarrow\infty$ is the number of series terms. \ For kernel-type estimators of bounding functions the rate $\bar a_n\bar \sigma_n$ is proportional to $(\log n)^c/\sqrt{nh^d}$ where $c$ is some positive constant and $h$ is the bandwidth, assuming some undersmoothing is done. For example, if the bounding function is $s$-times differentiable, $\bar \sigma_n$ can
be made close to $(\log n/n)^{s/(2s+d)} $ apart from some undersmoothing factor by considering a local polynomial estimator, see \cite{Stone:82}. For both series and kernel-type estimators we show below that $\bar a_n$
can be bounded by $\sqrt{\log n}$.

\subsection{Simulation-Based Construction and Its Theory}\label{sec:simulation}

Our main and preferred approach is based on the simple idea of simulating
quantiles of relevant statistics.

\begin{definition}[Simulation Method for Critical Values]\label{def:simulation-cv}
For any compact set $\textsf{V} \subseteq \V$, we set
\begin{equation} \label{define: kn}
k_{n,\textsf{V}}(p) = Q_{p}\(\sup_{ v \in \textsf{V}} Z^\star_n(v)\mid \D_n\).
\end{equation}
\end{definition}

\noindent We have the following result for simulation inference, analogous to that obtained for analytical inference.
\begin{theorem}[\textbf{Simulation Inference, Estimation, Power under C.1-C.4}]\label{theorem: inference1}   Suppose
C.1-C.4 hold. Consider  the interval estimator given in Definition \ref{main definition} with the critical value function
specified in
Definition \ref{def:simulation-cv}. Then,   for a given $p \in [1/2,1),$\\
1. The interval estimator has asymptotic level $p$:
$$
\di\Pn\left \{ \theta_{n0} \leq \widehat \theta_{n0}(p) \right \} \geq p - o(1).
$$
2. The estimation risk is bounded by, wp $\to 1$ under $\Pn$,
\begin{equation*}
\di \left | \widehat \theta_{n0}(p)-\theta _{n0}\right |  \leq  4 \bar\sigma_n
 \left( a_{n}(\barVn)+  \frac{O_{\Pn}(1)}{{a}_{n}(\barVn)}\right) \lesssim_{\Pn} \bar\sigma_n \bar a_n.
\end{equation*}%
3. Any
alternative $\theta_{na}> \theta_{n0}$ such that $$
\di \theta _{na} \geq \theta_{n0} +  4 \bar\sigma_n \left( a_{n}(\barVn)+  \frac{\mu_n }{{a}_{n}(\barVn)} \), \   \mu_n \to_{\Pn} \infty, $$
is rejected with probability converging to 1 under $\Pn$.
\end{theorem}

\subsection{Properties of the Set Estimator $\widehat V_{n}$}\label{sec:v-est}

In this section we establish some containment properties for the estimator $\widehat V_n$. Moreover, these containment properties imply  a useful rate result under the following condition:

\begin{ConditionV}[Degree of Identifiability for $V_0$] There exist constants $\rho_n>0$ and $c_n>0$, possibly
dependent on $n$, and a positive constant $\delta$, independent of $n$, such that
\begin{equation}\label{identifiability of V}
\theta_n(v ) - \theta_{n0} \geq (c_n d (v, V_0))^{\rho_n} \wedge \delta, \ \  \forall v \in \V\text{.}
\end{equation}
\end{ConditionV}
We say $(c_n, 1/\rho_n)$ characterize the degree of identifiability of $V_0$, as these parameters determine the rate at which $V_0$ can be consistently estimated. Note that if $V_0 = \mathcal{V}$, then this condition holds with $c_n = \infty$ and $\rho_n =1$, where we adopt the convention that $0 \cdot \infty = 0$.

We have the following result, whose first part we use in the proof of Theorems \ref{theorem: inference analytical} and \ref{theorem: inference1} above, and whose second part we use below in the proof of Theorem \ref{theorem: sharp inference}.

\begin{lemma}[\textbf{Estimation of $\Vn$} and $V_0$]\label{lemma: estimation of V} Suppose C.1-C.4 hold.  \\ 1. (Containment).  Then wp $\to$ 1,
for either analytical or simulation methods,
$$
\Vn \subseteq \widehat V_n \subseteq \barVn,
$$
for $\Vn$ defined in (\ref{define:V-epsilon n}) with $\gamma_n' = \gamma_n - o(1)$, and $\barVn$ defined in (\ref{define: Veps-bar-n}).

\noindent 2. (Rate) If also Condition V holds and $\bar \kappa_n \bar \sigma_n \to 0$, then wp $\to$ 1
\begin{align*}
 d_H(\widehat V_n, V_{0}) &\leq d_H(\widehat V_n, \Vn) +  d_H(\Vn, V_{0})
   \\
 & \leq d_H( \barVn, \Vn) + d_H( \Vn, V_{0})   \leq
r_n:=  2  (\bar \kappa_n \bar \sigma_n)^{1/\rho_n}/c_n.
 \end{align*}

\end{lemma}

\subsection{Automatic Sharpness of Simulation Construction}\label{sec:auto-sharpness}

When the penultimate process $Z^*_n$ does not lose equicontinuity too fast,
and $V_0$ is sufficiently well-identified, our simulation-based inference
procedure becomes sharp in the sense of not only achieving
the right level but in fact automatically achieving the right size.  In such
cases we typically have some small improvements in the rates of convergence
of the estimators.  The most
important case covered is that where $V_0$ is singleton
(or a finite collection of points) and  $\theta_n$ is locally quadratic, i.e. $\rho_n \geq 2$
and $c_n \geq c >0$ for all $n$. These sharp situations occur when the inferential process concentrates
on $V_0$ and not just on the neighborhood $\Vn$, in the sense described below. For this to happen we impose the following condition.

\begin{ConditionS}[\textbf{Equicontinuity radii are not smaller than $r_n$}] Under Condition V holding, the scaled penultimate process $\bar a_n Z_n^*$
has an equicontinuity radius $\varphi_n$ that is no smaller than $r_n :=  2  (\bar \kappa_n \bar \sigma_n)^{1/\rho_n}/c_n$, namely
$$
\sup_{\|v-v'\|\leq \varphi_n} \bar a_n |Z^*_n(v) - Z^*_n(v')| =  o_{\Pn}(1), \ \ r_n \leq \varphi_n.
$$
\end{ConditionS}

When $Z_n^*$ is Donsker, i.e. asymptotically equicontinuous,  this condition holds automatically, since in this case $\bar a_n \propto1$, and  for any $o(1)$ term, equicontinuity radii obey $\varphi_n = o(1)$, so that consistency $r_n = o(1)$ is sufficient. When $Z_n^*$ is not Donsker, its finite-sample equicontinuity
properties decay as $n \to \infty$, with radii $\varphi_n$ characterizing the decay.   However, as long as $\varphi_n$ is not smaller than $r_n$, we have just enough finite-sample equicontinuity left to achieve the following result.

\begin{lemma}[\textbf{Inference Sometimes Concentrates on $V_{0}$}]\label{lemma: concentrate on it} Suppose C.1-C.4, S, and V hold.
Then,
$$
\di \Pn \( \sup_{v \in \mathcal{V}} \frac{\theta_{n0} - \widehat \theta_n(v)}{s_n(v)} \leq x \)  =
 \Pn\(\sup_{v \in V_{0}} {Z^*_n(v) }\leq  x\) + o(1).
$$
\end{lemma}

\noindent Under the stated conditions,  our inference and estimation procedures \textit{automatically} become sharp
in terms of size and rates.

\begin{theorem}[\textbf{Sharpness of Simulation Inference}]\label{theorem: sharp inference}  Suppose
 C.1-C.4, S, and V hold. Consider the interval estimator given in Definition \ref{main definition} with the critical value function
  specified in Definition \ref{def:simulation-cv}. Then,   for a given $p \in [1/2,1),$\\
1. The interval estimator has asymptotic size $p$:
$$
\di\Pn\left \{ \theta_{n0} \leq \widehat \theta_{n0}(p) \right \} = p + o(1).
$$
2. Its estimation risk is bounded by, wp $\to 1$ under $\Pn$,
\begin{equation*}
\di \left | \widehat \theta_{n0}(p)-\theta _{n0}\right |  \leq  4  \bar\sigma_n
 \left( a_{n}(V_{0})+  \frac{O_{\Pn}(1)}{{a}_{n}(V_{0})}\right) \lesssim_{\Pn} \bar\sigma_n a_n(V_0).
\end{equation*}%
3. Any
alternative $\theta_{na}> \theta_{n0}$ such that $$
\di \theta _{na} \geq \theta_{n0} + 4  \bar\sigma_n \left( a_{n}(V_{0})
+  \frac{\mu_n }{a_{n}(V_{0})} \), \   \mu_n \to_{\Pn} \infty, $$
is rejected with probability converging to 1 under $\Pn$.
\end{theorem}

\section{Inference on Intersection Bounds in Leading Cases} \label{sec:leading-cases}

\subsection{Parametric estimation of bounding function.}\label{Section:ConditionP}

We now show that the above conditions apply to various parametric estimation methods for $v \mapsto \theta_n(v)$.
This is an important practical, and indeed tractable, case. The required conditions cover standard parametric estimators of bounding functions such as least squares, quantile regression, and other estimators.

\begin{ConditionP}[\textbf{Finite-Dimensional Bounding Function}] We have that (i) $\theta_n(v) := \theta_n(v, \beta_{n})$, where $\V \times \mathcal{B} \mapsto \theta_n(v,\beta)$ is a known function parameterized by finite-dimensional vector $\beta \in \mathcal{B}$, where $\mathcal{V}$ is a compact subset of $\Bbb{R}^d$
and $\mathcal{B}$ is a subset of $\Bbb{R}^k$, where the sets do not depend on $n$.
(ii) The function  $(v, \beta) \mapsto p_n(v, \beta):= \partial
\theta_n(v, \beta)/\partial \beta$ is uniformly Lipschitz with Lipschitz coefficient $L_n \leq L$,
where $L$ is a finite constant that does not depend on $n$. (iii) An estimator $\widehat \beta_n$ is available such that
$$\Omega_n^{-1/2}\sqrt{n} (\widehat \beta_n - \beta_{n}) = \mathcal{N}_k + o_{\Pn}(1),
 \ \ \mathcal{N}_k =_d N(0, I_k),
$$
i.e. $\mathcal{N}_k$ is a random $k$-vector with the multivariate standard normal distribution. (iv) $\|p_n(v, \beta_n)\|$ is bounded away from zero, uniformly in $v$ and $n$. The eigenvalues of $\Omega_n$ are bounded from above
and away from zero, uniformly in $n$. (v) There is also a consistent estimator $\widehat \Omega_n$
such that $\|\widehat \Omega_n - \Omega_n\| = O_{\P_n}(n^{-b})$ for some constant $b>0$, independent of $n$.
 \end{ConditionP}

\begin{example}[\textbf{A Saturated Model}]As a simple, but relevant example we consider the following model. Suppose that $v$ takes on a finite set of values, denoted $1,...,k$, so that $\theta_n(v,\beta) = \sum_{j=1}^k \beta_j 1(v = j)\text{.}$
Suppose first that $\P_n= \P$ is fixed, so that $\beta_n = \beta_0$, a fixed value. Condition (ii) and the boundedness requirement of (iv) follow from  $\partial
\theta_n(v, \beta)/\partial \beta_j = 1(v = j)$ for each $j =1,\ldots,k$.   Condition (v) applies to many estimators. Then if the estimator $\widehat \beta$
satisfies $\Omega^{-1/2}\sqrt{n} (\widehat \beta - \beta_0) \to_d N(0, I_k)$ where $\Omega$ is positive definite,
the strong approximation in condition (iii) follows from Skorohod's theorem and Lemma \ref{lemma: DP}.\footnote{See Theorem 1.10.3 of \cite{VanDerVaart/Wellner:96} on page 58 and the subsequent historical discussion attributing the earliest such results to \cite{Skorohod:56}, later generalized by Wichura and Dudley.}
Suppose next that $\Pn$ and the true value $\beta_n = (\beta_{n1},\ldots,\beta_{nk})'$ change with $n$. Then if
$$
\Omega_n^{-1/2} \sqrt{n} (\hat \beta_n - \beta_n) = \frac{1}{\sqrt{n}} \sum_{i=1}^n u_{i,n} + o_{\Pn}(1),
$$
with $\{u_{i,n}\}$ i.i.d. vectors with mean zero and variance matrix $I_k$ for each $n$, and $E\|u_{i,n}\|^{2 + \delta}$ bounded uniformly in $n$ for some $\delta > 0$,  then
$\Omega_n^{-1/2} \sqrt{n} (\hat \beta_n - \beta_n) \to_d N(0,I_k)$, and again condition (iii)
follows from Skorohod's theorem and Lemma \ref{lemma: DP}.
 \qed \end{example}

\begin{lemma}[\textbf{P and V imply C.1-C.4, S}]\label{lemma:parametric}
Condition P implies Conditions C.1-C.4, where, for $p_n(v,\beta) :=  \frac{\partial \theta_n(v, \beta)}{\partial \beta}$,
\begin{eqnarray*}
& & Z_n(v) = \frac{ \theta_n(v) - \widehat \theta_n(v)}{\sigma_n(v)}, \ Z_n^*(v) =  \frac{ p_n(v,\beta_n)' \Omega_n^{1/2}}{\| p_n(v,\beta_n)' \Omega_n^{1/2}\|}\mathcal{N}_k, \ Z_n^{\star}(v) =  \frac{p_n(v,\hat \beta_n)' \widehat \Omega_n^{1/2}}{\|p_n(v,\hat\beta_n)' \widehat \Omega_n^{1/2}\|}\mathcal{N}_k, \\
& & \sigma_n(v) = \| n^{-1/2}p_n(v,\beta_n)' \Omega_n^{1/2} \|,
\ \ s_n(v) = \| n^{-1/2}p_n(v,\hat \beta_n)'\widehat \Omega_n^{1/2} \|,\   \delta_n = o(1),\\
& &  \bar a_n \lesssim 1, \ \ \bar \sigma_n \lesssim \sqrt{1/n},  \ \ a_n(\textsf{V}) = \(2 \sqrt{ \log \{ C ( 1+ C ' L_n \text{diam}(\textsf{V}))^d \} }\) \vee (1 + \sqrt{d}),
\end{eqnarray*}
for some positive constants $C$ and $C'$, and $ \ P[\mathcal{E} > x] = \exp( - x/2 )$. Furthermore, if also Condition V holds and $c_n^{-1}(\ell\ell_n/\sqrt{n})^{1/\rho_n} = o(1),$ then
Condition S holds.
\end{lemma}

\noindent The following is an immediate consequence of Lemma \ref{lemma:parametric} and Theorems \ref{theorem: inference analytical}, \ref{theorem: inference1}, and \ref{theorem: sharp inference}.

\begin{theorem}[\textbf{Estimation and Inference with Parametrically Estimated Bounding Functions}]\label{Theorem:Parametric}
Suppose Condition P holds and  consider  the interval estimator $\widehat \theta_{n0}(p)$ given in Definition \ref{main definition} with
 simulation-based critical values specified in
Definition \ref{def:simulation-cv} for the simulation process $Z_n^\star$ specified above. (1) Then $ (i) \ \Pn[\theta_{n0} \leq \widehat \theta_{n0} (p)] \geq  p - o(1),  \ (ii) \ |\theta_{n0} - \widehat \theta_{n0} (p)| = O_{\Pn}(\sqrt{1/n}),  \ (iii) \ \Pn( \theta_{n0} +   \mu_n \sqrt{1/n} \geq \widehat \theta_{n0}(p) ) \to 1$
 for any $\mu_n \to_{\Pn} \infty$.  (2) If Condition V holds with $c_n\geq c>0$ and $\rho_n \leq \rho< \infty$, then
  $ \Pn[\theta_{n0} \leq \widehat \theta_{n0} (p)] =  p + o(1)$.
\end{theorem}

We next provide two examples that generalize the simple, but well-used, saturated example of Example 1 to more substantive cases. Aside from being practically relevant due to the common use of parametric restriction in applications, these examples offer a natural means of transition to the next section, which deals with series estimation, and which can be viewed as parametric estimation with parameters of increasing dimension and vanishing approximation errors.

\begin{example}[\textbf{Linear Bounding Function}]
Suppose that $ \theta_n(v,\beta_n) = p_n(v)'\beta_n,$
where $  p_n(v)'\beta: \V\times\mathcal{B} \mapsto \Bbb{R}$.
Suppose that  (a) $v \mapsto p_n(v)$ is Lipschitz with Lipschitz coefficient $L_n \leq L$, for all $n$, with the first
component equal to 1,
(b) there is an estimator available that is asymptotically linear
$$
\Omega_n^{-1/2} \sqrt{n} (\hat \beta_n - \beta_n) = \frac{1}{\sqrt{n}} \sum_{i=1}^n u_{i,n} + o_{\Pn}(1),
$$
with $\{u_{i,n}\}$ i.i.d. vectors with mean zero and variance matrix $I_k$ for each $n$, and $E\|u_{i,n}\|^{2 + \delta}$ bounded uniformly in $n$ for some $\delta > 0$, and (c)
$\Omega_n$ has eigenvalues bounded away from zero and from above. These conditions
imply Condition P(i)-(iv).  Indeed, (i),(ii), and (iv) hold immediately, while
(iii) follows from the Lindeberg-Feller CLT, which implies that under $\Pn$
$$\Omega_n^{-1/2} \sqrt{n} (\hat \beta_n - \beta_n) \to_d N(0,I_k),$$
and the strong approximation follows by the Skorohod representation
and Lemma \ref{lemma: DP} by suitably enriching the probability space if needed.   Note that if $\theta_n(v,\beta_n)$ is the conditional expectation of $Y_i$ given $V_i=v$, then $\widehat \beta_n$
can be obtained by the mean regression of $Y_i$ on $p_n(V_i)$, $i=1,...,n$; if $\theta_n(v,\beta_n)$ is the conditional u-quantile of $Y_i$ given $V_i=v$, then $\widehat \beta_n$ can be obtained by the $u$-quantile regression of $Y_i$ on $p_n(V_i)$, $i=1,...,n$. Regularity
conditions that  imply the ones stated above can be found in e.g. \cite{White:84} and \cite{Koenker:05}.  Finally estimators
of $\Omega_n$ depend on the estimator of $\beta_n$; for mean regression the standard estimator is the Eicker-Huber-White estimator,
and for quantile regression the standard estimator is Powell's estimator, see \cite{Powell:84}. For brevity
we do not restate sufficient conditions for Condition P(v), but these are readily available for common
estimators.  \qed
\end{example}

\begin{example}[\textbf{Conditional Moment Inequalities}]

This is a generalization of the previous example where now the bounding function is the
minimum of $J$ conditional mean functions. Referring to the conditional moment
inequality setting specified in Section \ref{sec:examples and overview},  suppose we have an i.i.d. sample
of $(X_i, Z_i), i=1,..., n$, with $\text{support}(Z_i) = \mathcal{Z} \subseteq [0,1]^d$.  Let $v=(z,j)$, where $j$ denotes the enumeration index for the conditional moment
inequality, $j \in \{1,...,J\}$, and suppose $\V \subseteq \mathcal{Z} \times \{1,...,J\}$.
  The parameters $J$ and $d$ do not depend on $n$.
Hence
$$\theta_{n0} = \min_{v \in \V} \theta_{n}(v) =\min_{(z,j) \in \mathcal{V}}\theta_n(z,j). $$
Suppose that $\theta_n(v)=E_{\Pn}[ m(X,\mu,j)|Z=z] = b(z)'\chi_n(j)$, for $b:
\mathcal{Z} \mapsto \Bbb{R}^m$, denoting some transformation of $z$, with $m$ independent of $n$, and where
$\chi_n(j)$ are the population regression coefficients in the regression of
$Y(j):= m(X,\mu,j)$ on  $b(Z), j=1,...,J$, respectively, under $\P_n$.
Suppose that the first $J_0/2$ pairs correspond to moment inequalities generated from
moment equalities so that $\theta_n(j) = - \theta_n(j-1),  \ \ j=2,4,..., J_0,$
and so these functions are replicas of each other up to sign; also note that
$
\chi_n(j) = - \chi_n(j-1), \ \ j=2,4,..., J_0.$
Then we can rewrite
 \ba
 & & \theta_n(v)=E_{\Pn}[ m(X,\mu,j)|Z=z] = b(z)'\chi_n(j) := p_n(v)'\beta_n, \\
 & & \beta_n = (\chi_n(j)', j \in \mathcal{J}),' \ \  \mathcal{J}:=\{2,4,...,J_0, J_0+1, J_0+2,...,J\}',
 \ea
where $\beta_n$ is a $K$-vector of regression coefficients, and $p_n(v)$
is a $K$-vector such that $p_n(z,j)= [0_m',..., 0_m', (-1)^{j+1}b'_m(z), 0'_m,...,0_m']'$ with $b'_m(z)$ appearing in the  $\lceil j/2\rceil$th block for $1 \leq j \leq J_0$; $p_n(z,j)= [0_m',..., 0_m', b'_m(z), 0'_m,...,0_m']'$
with $b(z)$ appearing in the  $j$-th block for $J_0+ 1 \leq j \leq J$, where $0_m$ is an m-dimensional vector of zeroes.\footnote{Note the absence of $\chi_n(j)$ for odd $j$ up to $J_0$ in the definition of the coefficient vector $\beta_n$. This is required to enable non-singularity of $E_{\Pn}[ \epsilon_i \epsilon_i' \mid Z_i=z]$. Imposing non-singularity simplifies the proofs, and is not needed for practical implementation.}

We impose the following conditions:

\begin{quote}(a)  $b(z)$ includes constant $1$,
(b) $z \mapsto b(z)$ has Lipschitz coefficient bounded above by $L$,
(c)  for $Y_i = (Y_i(j), j \in \mathcal{J})'$ and for $\epsilon_i := Y_i - E_{\Pn}[Y_i|Z_i]$,
 the eigenvalues of $E_{\Pn}[ \epsilon_i \epsilon_i' \mid Z_i=z]$ are bounded away from zero and from above,
 uniformly in $z \in \mathcal{Z}$ and $n$; (d)
  $Q = E_{\Pn}[b(Z_i) b(Z_i)']$ has eigenvalues bounded away from zero and from above, uniformly
  in $n$, and  (e) $E_{\Pn}\|b(Z_i)\|^4$ and $E_{\Pn}\|\epsilon_i\|^4$ are bounded from above uniformly in $n$.
\end{quote}

 Then it follows from e.g. by \cite{White:84} that for $\widehat \chi_n(j)$ denoting the ordinary least square estimator obtained
by regressing $Y_i(j), i=1,...,n,$ on $b(Z_i), i=1,...,n$,
$$
\sqrt{n}(\widehat \chi_n(j) - \chi_n(j)) =   Q^{-1}  \frac{1}{\sqrt{n}} \sum_{i=1}^n b(Z_i) \epsilon_i(j) + o_{\Pn}(1), \ \  j \in \mathcal{J},
$$
so that
$$
\sqrt{n}(\widehat \beta_n - \beta_n) = (I_{| \mathcal{J}|} \otimes Q )^{-1} \frac{1}{\sqrt{n}} \sum_{i=1}^n  \underbrace{(I_{| \mathcal{J}|} \otimes b(Z_i)) \epsilon_i}_{u_i}   + o_{
\Pn}(1).
$$
By conditions (c) and (d) $E_{\Pn}[u_i u_i']$ and $Q$ have eigenvalues bounded away from zero and from above, so the same is true
of $\Omega_n =  (I_{| \mathcal{J}|} \otimes Q )^{-1} E_{\Pn}[u_i u_i']  (I_{| \mathcal{J}|} \otimes Q )^{-1}.$
These conditions verify condition P(i),(ii),(iv). Application of the Lindeberg-Feller CLT, Skorohod's theorem,
and Lemma \ref{lemma: DP} verifies Condition P(iii).
By the argument given in Chapter VI of \cite{White:84}, Condition P(v) holds for the standard analog estimator for $\Omega_n$:
$$
\hat \Omega_n =  (I_{| \mathcal{J}|} \otimes \hat Q )^{-1} \mathbb{E}_n[\hat u_i \hat u_i']  (I_{| \mathcal{J}|} \otimes \hat Q )^{-1},
$$
where $\hat Q = \mathbb{E}_n[b(Z_i) b(Z_i)']$ and $\hat u_i = (I_{| \mathcal{J}|} \otimes b(Z_i)) \hat \epsilon_i$, with $\hat \epsilon_i(j) = Y_i(j) - b(Z_i)'\hat \chi_n(j)$, and $\hat \epsilon_i = (\hat \epsilon_i(j), j \in \mathcal{J})'$.
 \qed
 \end{example}

\subsection{Nonparametric Estimation of $\theta_n(v)$ via Series}\label{sec:NP-est-series}

Series estimation is effectively like parametric estimation, but the dimension of the estimated parameter tends
to infinity and bias arises due to approximation based on a
finite number of basis functions. If we select the number of
terms in the series expansion so that the estimation error is
of larger magnitude than the approximation error, i.e. if we undersmooth, then the
analysis closely mimics the parametric case.

\begin{ConditionNS}
The function $v\mapsto\theta_n(v)$ is continuous in $v$. The series estimator $\widehat \theta_n(v)$ has the form $\widehat{\theta}(v) = p_{n}(v)'\widehat{\beta}_{n},$
where $p_{n}(v):= (p_{n,1}(v),\ldots,p_{n,K_n}(v))'$ is a collection of
$K_n $ continuous series functions mapping $\mathcal{V} \subset \mathcal{K} \subset \Bbb{R}^d$ to $\Bbb{R}^{{K_n}}$,  and $\widehat {\beta}_{n}$ is a $K_n$-vector of coefficient estimates, and $\mathcal{K}$ is a fixed compact set. Furthermore,

\noindent\textbf{NS.1} (a) The estimator satisfies the following linearization and strong
approximation condition:
$$
\frac{\widehat \theta_n(v) - \theta_n(v)}{\|  p_{n}(v)' \Omega_n^{1/2}\|/\sqrt{n}} =
\frac{ p_{n}(v)' \Omega_n^{1/2}}{\| p_{n}(v)' \Omega_n^{1/2}\|}\mathcal{N}_n + R_n(v),
$$
where
\begin{eqnarray*}
 \ \mathcal{N}_n =_d N(0,I_{K_n}), \ \ \
\sup_{v \in \mathcal{V}} |R_n(v)|=
o_{\P_n}( 1/\log n).
\end{eqnarray*}
(b) The matrices $\Omega_n$ are  positive definite, with eigenvalues bounded from above and away from
zero, uniformly in $n$. Moreover, there are sequences of constants $\zeta_n$ and $\zeta_n'$ such that $1 \leq \zeta'_n \lesssim \|p_{n}(v)\|  \leq \zeta_n$  uniformly for all $v \in \mathcal{V}$ and $\sqrt{ \zeta^2_n \log n /n } \to 0$, and $\|p_{n}(v)- p_{n}(v')\|/\zeta'_n
\leq L_n \|v-v'\|$ for all $v,v' \in \V$, where $\log L_n \lesssim \log n$, uniformly in $n$.

\noindent\textbf{NS.2} There exists $\widehat \Omega_n$ such that $ \| \widehat \Omega_n - \Omega_n\|= O_{\P_n}(n^{-b})$, where $b>0$ is
 a constant.
\end{ConditionNS}
Condition NS is not primitive, but reflects the function-wise large sample normality of series estimators.  It requires that the studentized nonparametric process is approximated by a sequence of Gaussian processes, which take a very simple intuitive form, rather than by a fixed single Gaussian process. Indeed, the latter would be impossible in non-parametric settings, since the sequence of Gaussian processes is not asymptotically tight. Note also that the condition implicitly requires that some undersmoothing takes place so that the approximation error is negligible relative to
the sampling error.  We provide primitive conditions that imply condition NS.1 in three examples presented below. In particular, we show that the asymptotic linearization for $\widehat \beta_n - \beta_n$, which is available from the literature on series regression, e.g. from \cite{andrews:series} and \cite{Newey:97}, and the use of Yurinskii's coupling \cite{Yurinskii:77} imply  condition NS.1.  This result could be of independent interest, although we only provide sufficient conditions for the strong approximation to hold.

Note that under condition NS, the uniform rate of convergence of $\widehat \theta_n(v)$ to $\theta_n(v)$ is given by
$
\sqrt{ \zeta_n^2/n} \sqrt{ \log n} \to 0,
$
where $\zeta_n \propto \sqrt{{K_n}}$ for standard series terms such as B-splines or trigonometric series.

\begin{lemma}[\textbf{NS implies C.1-C.4}]\label{lemma:series} Condition NS implies Conditions C.1-C.4 with
\begin{eqnarray*}
& & Z_n(v) = \frac{ \theta_n(v) - \widehat \theta_n(v)}{\sigma_n(v)}, \ Z_n^*(v) =  \frac{ p_{n}(v)' \Omega_n^{1/2}}{\| p_{n}(v)' \Omega_n^{1/2}\|}\mathcal{N}_n, \ Z_n^{\star}(v) =  \frac{ p_{n}(v)' \widehat \Omega_n^{1/2}}{\| p_{n}(v)' \widehat \Omega_n^{1/2}\|}\mathcal{N}_n, \\
& & \sigma_n(v) = \| n^{-1/2}p_{n}(v)' \Omega_n^{1/2}\|, \ \ s_n(v) = \|n^{-1/2}p_{n}(v)'\widehat \Omega_n^{1/2}\|,\   \delta_n = 1/\log n,\\
& &  \bar a_n \lesssim \sqrt{ \log n}, \ \ \bar \sigma_n \lesssim \sqrt{\zeta_n^2/n},  \ \
a_n(\textsf{V}) = \(2 \sqrt{ \log \{ C ( 1+ C ' L_n \text{diam}(\textsf{V}))^d \} }\) \vee (1 + \sqrt{d}),
\end{eqnarray*}
for some constants $C$ and $C'$, where $\text{ diam}(\textsf{V})$ denotes the diameter of the set $\textsf{V}$, and
$
 \ P[\mathcal{E} > x] = \exp( - x/2 ).
$
\end{lemma}

\begin{remark}\label{remark: feasible an series case} Lemma \ref{lemma:series} verifies the main conditions C.1-C.4.
These conditions enable construction of simulated or analytical
critical values. For the latter, the $p$-th quantile of $\mathcal{E}$ is given by
$c(p) = -2\log (1-p),$
so we can set
\begin{equation} \label{series-exponential}
k_{n,\textsf{V}}(p)= a_n(\textsf{V})  -2\log (1-p)/a_n(\textsf{V}),
\end{equation}
where
\begin{equation} \label{eq: anV series case}
a_n(\textsf{V}) =
\(2\sqrt{\log  \{ \ell_n \( 1 + \ell_n L_n \text{diam}(\textsf{V})\)^d \}}  \),
\end{equation}
is a feasible scaling factor which bounds the scaling factor in the statement of Lemma \ref{lemma:series},
at least for all large $n$. Here, all unknown constants have been replaced by slowly growing numbers $\ell_n$ such that $\ell_n > C \vee C'$
for all large $n$. Note also that $\textsf{V} \mapsto k_{n,\textsf{V}}(p)$ is monotone in $\textsf{V}$ for all sufficiently large $n$, as required in
the analytical construction given in Definition \ref{analytical definition}. A sharper analytical approach can be based on Hotelling's tube method; for details we refer to \cite{Chernozhukov/Lee/Rosen:09}. That approach is
tractable for the case of $d=1$ but does not immediately extend to $d>1$. Note that the simulation-based approach is effectively
a numeric version of the exact version of the tube formula, and is less conservative than using
simplified tube formulas.
\qed \end{remark}

\begin{lemma}[\textbf{Condition NS implies S in some cases}]\label{lemma: series-v-est} Suppose Condition NS holds. Then,(1) The radius $\varphi_n$
of equicontinuity of $Z^*_n$ obeys:
$$
\varphi_n \leq o(1) \cdot \(\frac{1}{L_n \sqrt{\log n} }\),
$$
for any $o(1)$ term.  (2) If Condition V holds and
\begin{equation}
\label{eq: condition S.2 for series}
\( \sqrt{\frac{\zeta_n^2}{ n}\log n} \)^{1/\rho_n}c_n^{-1} = o\( \frac{1}{L_n \sqrt{\log n} } \),
\end{equation}
then Condition S holds. (3) If $V_0$ is  singleton and \eqref{eq: condition S.2 for series} holds, $\rho_n \leq 2$, and $c_n \geq c>0$, for all $n$,
then $a_n(V_0) \propto 1$ and \eqref{eq: condition S.2 for series} reduces to
$$L_n^4 K_n \log^{3}n /n  \to 0.$$
\end{lemma}

\noindent The following is an immediate consequence of Lemmas \ref{lemma:series} and \ref{lemma: series-v-est} and Theorems \ref{theorem: inference analytical}, \ref{theorem: inference1}, and \ref{theorem: sharp inference}.

\begin{theorem}[\textbf{Estimation and Inference with Series-Estimated Bounding Functions}]\label{Theorem:Series}
Suppose Condition NS holds and  consider  the interval estimator $\widehat \theta_{n0}(p)$ given in Definition \ref{main definition} with either
analytical critical value $c(p)=-2\log(1-p)$, or simulation-based critical values from
Definition \ref{def:simulation-cv} for the simulation process $Z_n^\star$ above. (1) Then $ (i) \ \Pn[\theta_{n0} \leq \widehat \theta_{n0} (p)] \geq  p - o(1)$,  $(ii)$ $ |\theta_{n0} - \widehat \theta_{n0} (p)| = O_{\Pn}(
\sqrt{ \log n} \sqrt{\zeta^2_n/n}),$  $(iii)$ $ \Pn( \theta_{n0} +   \mu_n \sqrt{ \log n} \sqrt{\zeta^2_n/n} \geq \widehat \theta_{n0}(p) ) \to 1$
 for any $\mu_n \to_{\Pn} \infty$.  (2) Moreover, for the simulation-based
 critical values, if Condition V and relation (\ref{eq: condition S.2 for series}) hold, then
  $ (i) \ \Pn[\theta_{n0} \leq \widehat \theta_{n0} (p)] =  p - o(1)$,  $(ii) \ |\theta_{n0} - \widehat \theta_{n0} (p)| = O_{\Pn}(
\sqrt{\zeta^2_n/n})$,  $(iii) \ \Pn( \theta_{n0} +   \mu_n  \sqrt{\zeta^2_n/n} \geq \widehat \theta_{n0}(p) ) \to 1$
 for any $\mu_n \to_{\Pn} \infty$.
\end{theorem}

\noindent We next present some examples with primitive conditions that imply Condition NS.

\begin{example}[\textbf{Bounding Function is Conditional Quantile}] Suppose that $\theta_n(v) := Q_{Y_i|V_i}[\tau | v]$
is the $\tau$-th conditional quantile of $Y_i$ given $V_i$ under $\Pn$, assumed to be a continuous function in $v$. Suppose
we estimate $\theta_n(v)$ with a series estimator.
There is an i.i.d. sample $(Y_i, V_i), i=1,...,n$, with  $\text{support}(V_i) \subseteq [0,1]^d$
for each $n$, defined on a probability space equipped
with probability measure $\P_n$.   Suppose that the intersection region
of interest is $\mathcal{V} \subseteq \text{support}(V_i)$.
Here the index $d$ does not depend on $n$, but all
other parameters, unless stated otherwise, can depend on $n$.   Then $\theta_n(v) = p_{n}(v)'\beta_{n} + A_{n}(v)$,  where
 $p_{n}: [0,1]^d \mapsto \Bbb{R}^{K_n}$ are  the series functions, $\beta_{n}$ is the quantile regression coefficient in the population, $A_{n}(v)$ is the approximation error, and $K_n$ is the number of series terms that
 depends on $n$.  Let $C$ be a positive constant.

 We impose the following technical conditions to verify NS.1 and NS.2:
 \begin{quote}
Uniformly in $n$, (i)  $p_{n}$ are either B-splines of a fixed order or trigonometric series terms or any
other terms $p_n = (p_{n1},\ldots,p_{nK_n})'$ such that $\|p_n(v)\| \lesssim \zeta_n = \sqrt{K_n}$ for all $v \in \text{support}(V_i)$,
$\|p_n(v)\| \gtrsim   \zeta_n' \geq 1$ for all $v \in \mathcal{V}$, and $\log \textrm{lip}(p_{n}) \lesssim \log K_n$, (ii) the mapping $v \mapsto \theta_n(v)$
  is sufficiently smooth, namely $\sup_{v \in \mathcal{V}}|A_{n}(v)| \lesssim K_n^{-s}$, for some $s>0$,
  (iii) $\lim_{n \to \infty}(\log n)^c K_n^{-s+1} = 0$ and $\lim_{n \to \infty} (\log n)^c \sqrt{n} K_n^{-s}/\zeta_n' = 0$, for each $c>0$, (iv)
eigenvalues of  $\Lambda_{n}=E_{\Pn}[p_{n}(V_i) p_{n}(V_i)']$ are bounded away from zero and from above,(v)
$f_{Y_i|V_i}(\theta_n(v)|v)$ is bounded uniformly over $v\in \V$ away from zero
  and from above, (vi) $ \lim_{n \to \infty} K_n^5 (\log n)^c/n = 0$ for each $c>0$,
  and (vii) the restriction on the bandwidth sequence
in Powell's estimator $\hat Q_{n}$ of $Q_{n} = E_{\P_n} [f_{Y_i|V_i}(\theta_n(V_i)|V_i) p_{n}(V_i) p_{n}(V_i)']$
specified in \cite{Belloni/Chernozhukov/Fernandez-Val:10} holds.
 \end{quote}

Suppose that we use the standard quantile regression estimator
$$\widehat \beta_{n} = \arg \min_{b \in \Bbb{R}^{K_n}} \En[ \rho_\tau( Y_i - p_{n}(V_i)'b)],$$
so that $\widehat \theta_n(v) = p_{n}(v)'\widehat \beta$ for $\rho_\tau(u) = (\tau- 1(u<0))u.$ Then by  \cite{Belloni/Chernozhukov/Fernandez-Val:10}, under conditions (i)-(vi), the following asymptotically linear representation holds:
$$
\sqrt{n}(\widehat \beta_{n}  - \beta_{n}) =  Q_{n}^{-1} \frac{1}{\sqrt{n}} \sum_{i=1}^n  \underbrace{p_{n}(V_i) \epsilon_i}_{u_i} + o_{\Pn}\(\frac{1}{\log n}\),
$$
for $\epsilon_i = (\tau - 1(w_i \leq \tau))$, where $(w_i, i=1,...,n)$ are i.i.d. uniform, independent of $(V_i, i=1,...,n)$. Note that by conditions (iv) and (v) $S_n:= E_{\Pn}[u_i u_i']= \tau(1-\tau) \Lambda_{n}$, and $Q_{n}$ have eigenvalues bounded away from zero and from above uniformly in $n$, and so  the same is also true of $\Omega_n = Q_{n}^{-1} S_{n} Q_{n}^{-1}$.  Given other restrictions imposed in condition (i), Condition NS.1(b) is verified. Next using condition (iv) and others the strong approximation required in NS.1(a) follows by invoking Theorem \ref{theorem:strong series}, Corollary \ref{corollary:strong} in Section \ref{sec:strong:series:main-text}, which is based on Yurinskii's coupling.
To verify Condition NS.2, consider the plug-in estimator $\widehat \Omega_n =  \hat Q_{n}^{-1} \hat S_{n} \hat Q_{n}^{-1}$,
where $\widehat Q_n$ is Powell's estimator for $Q_n$, and
$\hat S_{n} = \tau (1- \tau) \cdot \En[p_n(V_i) p_n(V_i)]$.  Then under condition (vii) it follows from the proof of Theorem 7 in \cite{Belloni/Chernozhukov/Fernandez-Val:10} that $\|\widehat \Omega_n - \Omega_n\|= O_{\P_n}(n^{-b})$ for some $b>0$. \qed
\end{example}

\begin{example}[\textbf{Bounding Function is Conditional Mean}] Now suppose that $\theta_n(v)$ $=$ $E_{\Pn}[Y_i | V_i =v]$,
assumed to be a continuous function with respect to $v \in \text{support}(V_i)$, and the intersection region is
$\mathcal{V} \subseteq \text{support}(V_i)$. Suppose we are using the series approach to approximating and estimating $\theta_n(v)$.  There is an i.i.d. sample $(Y_i, V_i), i=1,...,n$, with $\text{support}(V_i) \subseteq [0,1]^d$
for each $n$. Here $d$ does not depend on $n$, but all
other parameters, unless stated otherwise, can depend on $n$.  Then
we  have  $\theta_n(v) = p_{n}(v)'\beta_{n} + A_{n}(v)$, for $p_{n}: [0,1]^d\mapsto \Bbb{R}^{K_n}$ representing
the series functions;  $\beta_{n}$
is the coefficient of the best least squares approximation to $\theta_n(v)$ in the population,
and $A_{n}(v)$ is the approximation error.  The number of series terms $K_n$ depends on $n$.

We impose the following technical conditions:
\begin{quote}
Uniformly in $n$, (i)  $p_{n}$ are either B-splines of a fixed order or trigonometric series terms or any
other terms $p_n = (p_{n1},\ldots,p_{nK_n})'$ such that $\|p_n(v)\| \lesssim \zeta_n = \sqrt{K_n}$ for all $v \in \text{support}(V_i)$,
$\|p_n(v)\| \gtrsim   \zeta_n' \geq 1$ for all $v \in \mathcal{V}$, and $\log \textrm{lip}(p_{n}) \lesssim \log K_n$, (ii) the mapping $v \mapsto \theta_n(v)$
  is sufficiently smooth, namely $\sup_{v \in \mathcal{V}}|A_{n}(v)| \lesssim K_n^{-s}$, for some $s>0$,  (iii) $\lim_{n \to \infty }(\log n)^c\sqrt{n}K_n^{-s} = 0$ for each $c>0$,\footnote{This condition,
which is based on \cite{Newey:97} can be relaxed to $(\log n)^c K_n^{-s + 1} \to 0$ and $(\log n)^c \sqrt{n} K^{-s}_n/\zeta_n' \to 0$, using the
recent results of \cite{Belloni/Chen/Chernozhukov:11} for least squares series estimators.} (iv) for  $\epsilon_i = Y_i - E_{\Pn}[Y_i|V_i]$,
  $E_{\Pn}[\epsilon_i^2|V_i=v]$ is bounded away from zero  uniformly in $v \in \text{support}(V_i)$, and (v)
  eigenvalues of $Q_{n} = E_{\Pn}[p_{n}(V_i) p_{n}(V_i)']$ are bounded away from zero and from above, and (vi) $E_{\Pn}[|\epsilon_i|^4 |V_i=v]$ is bounded from above uniformly in $v \in \text{support}(V_i)$,  (vii) $\lim_{n \to \infty}(\log n)^c K_n^5/n =0$ for each $c>0$.  \end{quote}

We use the standard least squares estimator
$$\widehat \beta_{n} = \En[ p_{n}(V_i) p_{n}(V_i)']^{-1} \En[p_{n}(V_i) Y_i],$$
so that $\widehat \theta_n(v) = p_{n}(v)'\beta_{n}.$
Then by  \cite{Newey:97}, under conditions (i)-(vii), we have the following asymptotically linear representation:
$$
\sqrt{n}(\widehat \beta_n  - \beta_n ) =  Q_{n}^{-1} \frac{1}{\sqrt{n}} \sum_{i=1}^n  \underbrace{p_{n}(V_i) \epsilon_i}_{u_i} + o_{\Pn}(1/\log n).
$$
For details, see Supplementary Appendix \ref{sec:prac:roymodel}. Note that  $E_{\Pn}(u_i u_i')$ and $Q_{n}$ have eigenvalues bounded away from zero and from above uniformly in $n$, and so  the same is also true of $\Omega_n =  Q_{n}^{-1} E_{\Pn}(u_i u_i') Q_{n}^{-1}$.  Thus, under condition (i),
Condition NS.1(a) is verified.  The strong approximation condition NS.1(a) now follows from invoking Theorem \ref{theorem:strong series}
   in Section \ref{sec:strong:series:main-text}.  Finally, \cite{Newey:97} verifies that NS.2 holds for the standard analog estimator
   $\hat \Omega_n = \hat Q_{n}^{-1} \En (\hat u_i \hat u_i') \hat Q_{n}^{-1}$ for $\hat u_i = p_n(V_i)(Y_i - \hat \theta_n(V_i))$ and $\hat Q_{n} = \En[p_n(V_i) p_n(V_i)']$
   under conditions that are implied by those above.

   Finally, note that if we had $\epsilon_i \sim N(0, \sigma^2(V_i))$, conditional on $V_i$, we could establish Condition NS.1   with a much weaker growth restriction than (vii).  Thus, while  our use of Yurinskii's coupling provides concrete
  sufficient conditions
  for strong approximation, the function-wise large sample normality is likely to hold
  under weaker conditions
 in many situations.    \qed
\end{example}

\begin{example}[\textbf{Bounding Function from Conditional Moment Inequalities}] Consider now Example C of Section \ref{sec:examples and overview}, where now the bounding function is the
minimum of $J$ conditional mean functions. Suppose we have an i.i.d. sample
of $(X_i, Z_i), i=1,..., n$, with $\text{support}(Z_i) = \mathcal{Z} \subseteq [0,1]^d$,
defined on a probability space equipped with probability measure
$\Pn$.  Let $v=(z,j)$, where $j$ denotes the enumeration index for the conditional moment
inequality, $j \in \{1,...,J\}$, and $\V \subseteq \mathcal{Z} \times \{1,...,J\}$.  The parameters $J$ and $d$ do not depend on $n$.
Hence
$$\theta_{n0} = \min_{v \in \V} \theta_{n}(v), $$
for $\theta_n(v) = E_{\Pn}[ m(X_i,\mu,j)|Z_i=z]$, assumed to be a continuous
function with respect to $z \in \mathcal{Z}$.  Suppose we use the series approach
to approximate and estimate $\theta_n(z,j)$ for each $j$.  Then  $E_{\Pn}[ m(X,\mu,j)|z] = b_n(z)'\chi_{n}(j) + A_n(z,j)$, for $b_n:
[0,1]^d \mapsto \Bbb{R}^{m_n}$ denoting an $m_n$-vector of series functions; $\chi_{n}(j)$ is the coefficient
of the best least squares approximation to  $E_{\Pn}[ m(x,\mu,j)|z]$ in the population,  and $
A_n(z,j)$ is the approximation error.  Let $\mathcal{J}$ be a subset of $\{1,...,J\}$
as defined as in the parametric Example 3 (to handle inequalities associated with equalities).

We impose the following conditions: \begin{quote}
Uniformly in $n$, (i) $b_n(z)$ are either B-splines of a fixed order or trigonometric series terms or any
other terms $b_n(z) = (b_{n1}(z),\ldots,b_{nm_n}(z))'$ such that
$\|b_n(z)\| \lesssim \zeta_n = \sqrt{m_n}$ for all $z \in \mathcal{Z}$,
$E_{\Pn}[\| b_n(Z_i)\|^3] \lesssim m_n^{3/2}$, $\|b_n(z)
\| \gtrsim   \zeta_n' \geq 1$ for all $z \in \mathcal{Z}$, and $ \log \textrm{lip}(b_n(z)) \lesssim \log m_n$;
(ii) the mapping $z \mapsto \theta_n(z, j)$ is sufficiently smooth, namely $\sup_{z \in \mathcal{Z}}|A_n(z,j)| \lesssim m_n^{-s}$, for some $s>0$, for all $j \in \mathcal{J}$;
 (iii)  $\lim_{n \to \infty }(\log n)^c\sqrt{n} m_n^{-s} = 0$ for each $c>0$;\footnote{See the previous footnote on a possible relaxation
 of this condition.} (iv)
 for $Y(j):=m(X,\mu,j)$ and  $Y_i :=( Y_i(j), j \in \mathcal{J})'$ and $\epsilon_i := Y_i - E_{\Pn}[Y_i|Z_i]$,
the eigenvalues of $E_{\Pn}[ \epsilon_i \epsilon_i' \mid Z_i=z]$ are bounded away from zero,
 uniformly in $z \in \mathcal{Z}$; (v) eigenvalues of
  $Q_n= E_{\Pn}[b_n(Z_i) b_n(Z_i)']$  are bounded away from zero and from above;
  (vi) $E_{\Pn}[\|\epsilon_i\|^4\mid Z_i=z]$ is bounded above, uniformly in $z \in \mathcal{Z}$;
  and  (vii)  $ \lim_{n \to \infty} m_n^5 (\log n)^c/n =0$ for each $c>0$.
\end{quote}

 The above construction implies  $\theta_n(v) =  b_n(z)'\chi_n(j)  + A_n(z,j)
 =: p_{n}(v)'\beta_n + A_{n}(v),$ for $\beta_{n} =  (\chi_n'(j), j \in \mathcal{J})'$, where $p_n(v)$ and $\beta_{n}$ are vectors of dimension $K_n := m_n \times |\mathcal{J}|$, defined as in parametric Example 3.  Consider the standard least squares estimator $\widehat \beta_{n} = (\widehat \chi_n'(j), j \in \mathcal{J})'$ consisting of $|\mathcal{J}|$ least square estimators, where $\widehat \chi_n(j) = \En[ b_n(Z_i) b_n(Z_i)']^{-1} \En[b_n(Z_i) Y_i(j)]$.  Then it follows
from \cite{Newey:97} that for $Q_n= E_{\Pn} [ b_n(Z_i) b_n(Z_i)' ]^{-1}$
$$
\sqrt{n}(\widehat \chi_n(j) - \chi_n(j)) =  \frac{1}{\sqrt{n}} \sum_{i=1}^n Q_n^{-1} b_n(Z_i) \epsilon_i(j) + o_{\Pn}(1/\log n), \ \  j \in \mathcal{J},
$$
so that
$$
\sqrt{n}(\widehat \beta_n - \beta_n) = ( I_{|\mathcal{J}|}\otimes Q_n)^{-1} \frac{1}{\sqrt{n}} \sum_{i=1}^n  \underbrace{(I_\mathcal{|J|} \otimes b_n(Z_i)) \epsilon_i}_{u_i}   + o_{
\Pn}(1/\log n).
$$
By conditions (iv), (v), and (vi) $E_{\Pn}[u_i u_i']$ and $Q_n$ have eigenvalues bounded away from zero and from above, so the same is true
of $\Omega_n =  (I_{|\mathcal{J}|} \otimes Q_n)^{-1} E_{\Pn}[u_i u_i']  (I_{|\mathcal{J}|} \otimes Q_n)^{-1}.$
This and condition (i) imply that Condition NS.1(b) holds. Application of Theorem \ref{theorem:strong series}, based on Yurinskii's coupling, verifies  Condition NS.1(a). Finally,
Condition NS.2 holds for the standard plug-in estimator for $\Omega_n$, by the same argument as given in the proof of Theorem 2 of \cite{Newey:97}. \qed
\end{example}

\subsection{Nonparametric Estimation of $\theta_n(v)$ via Kernel Methods}\label{sec:est:local}

In this section we provide conditions under which kernel-type estimators satisfy Conditions C.1-C.4.  These conditions cover both standard kernel estimators as well as local polynomial estimators.

\begin{ConditionNK}
Let $v = (z,j)$ and $\mathcal{V} \subseteq \mathcal{Z} \times \{1,...,J\}$, where $\mathcal{Z}$ is a compact
convex set that does not depend on $n$. The estimator $v \mapsto \widehat \theta_n(v)$ and the function $v \mapsto \theta_n(v)$ are continuous in $v$.  In what follows, let $e_j$ denote the $J$- vector with $j$th element one and all other elements zero. Suppose
 that $(U, Z)$ is a $(J+d)$-dimensional random vector, where
 $U$ is a generalized residual
 such that $E[U|Z]= 0$ a.s. and $Z$ is a covariate;
 the density $f_n$ of $Z$ is continuous and bounded away from zero and from above on $\mathcal{Z}$, uniformly in $n$;
  and the support of $U$ is bounded uniformly in $n$. $\mathbf{K}$ is a twice continuously differentiable, possibly higher-order, product kernel function with support on $[-1,1]^{d}$, $\int \mathbf{K}(u)du=1$;  and $h_{n}$ is a sequence of bandwidths such that $h_n \rightarrow 0$ and $nh_n^d \to \infty$ at a polynomial rate in $n$.

\noindent \textbf{NK.1}  We have that uniformly in $ v \in \mathcal{V}$,
$$
(nh_n^d)^{1/2}(\widehat \theta_n(v) - \theta_n(v)) =  \mathbb{B}_n (g_v)  + o_{\mathrm{P}_n} (\delta_n),  \ \ g_v(U, Z):=
\frac{e_j' U }{ (h_n^d )^{1/2} f_n(z) } \mathbf{K} \left( \frac{z - Z}{h_n} \right),
$$
where $\mathbb{B}_n$ is a $\P_n$-Brownian bridge such that $v \mapsto \mathbb{B}_n (g_v)$
has continuous sample paths over $\V$. Moreover, the latter process can be approximated via the Gaussian multiplier method, namely
there exist sequences $o(\delta_n)$ and $o (1/ \ell_n)$ such that
$$
\Pn\(  \sup_{v \in \V} \left| \Gn^o( g_v) - \barBn(g_v) \right| > o( \delta_n) \Big| \D_n\)=o_{\Pn}(1/\ell_n),
$$
for some independent (from data) copy $v \mapsto \bar{\mathbb{B}}_n (g_v)$ of the process $v \mapsto \mathbb{B}_n (g_v)$.  Here,
$\Gn^o (g_v) =  \frac{1}{\sqrt{ n}} \sum_{i=1}^n \eta_i g_v (U_i,Z_i),$ where
$ \eta_i$ are i.i.d.  $N(0,1)$, independent of the data $\mathcal{D}_n$ and of $\{(U_i,Z_i)\}_{i=1}^n $, which
are i.i.d. copies of $(U,Z).$  Covariates $\{Z_i\}_{i=1}^n$ are part of the data, and $\{ U_i \}_{i=1}^n$ are a measurable transformation
of data.

\noindent \textbf{NK.2}   There exists  an estimator $z \mapsto \hat f_n(z)$, having continuous sample paths, such
that $\sup_{z \in \Z}| \hat f_n(z) - f_n(z) | = O_{\P_n}(n^{-b})$, and there are estimators
$\widehat U_i$ of generalized residuals such that $\max_{1 \leq i \leq n} \| \hat U_i - U_i \| = O_{\P_n}(n^{-\tilde b})$
for some constants $b > 0$ and $\tilde b > 0$.
\end{ConditionNK}

Condition NK.1 is a high-level condition that captures the large sample Gaussianity
of the entire estimated function where estimation is done via a kernel or local method. Under some mild regularity conditions,
specifically those stated in Appendix \ref{sec:proofs-strong-local},  NK.1 follows from the Rio-Massart coupling
and from the Bahadur expansion holding uniformly in $ v \in \mathcal{V}$:
\begin{align*}
(nh_n^d)^{1/2}(\widehat \theta_n(v) - \theta_n(v)) =  \mathbb{G}_n (g_v)  + o_{\mathrm{P}_n} (\delta_n).
\end{align*}
Uniform Bahadur expansions have been established for a variety of local estimators, see e.g. \cite{Masry:06} and \cite{Kong/Linton/Xia:10}, including higher-order kernel and local polynomial estimators.  It is possible to use more primitive sufficient conditions
stated in Appendix \ref{sec:proofs-strong-local} based on the Rio-Massart coupling (\cite{Rio:94} and \cite{Massart:89}), but these conditions are merely sufficient
and other primitive conditions may also be adequate. Our general argument, however, relies only on validity of Condition NK.1.

For simulation purposes, we define
\begin{eqnarray*}
& & \Gn^o (\hat{g}_v) =  \frac{1}{\sqrt{ n}} \sum_{i=1}^n \eta_i \hat g_v (U_i,Z_i), \  \  \eta_i \text{ i.i.d.  $N(0,1)$, independent of the data $\mathcal{D}_n$}, \\
& & \hat g_v (U_i,Z_i) = \frac{ e_j'\hat U_{i} }{ (h_n^d )^{1/2} \hat f_n(z) } \mathbf{K} \left( \frac{z - Z_i}{h_n} \right).
\end{eqnarray*}

\begin{lemma}[\textbf{Condition NK implies C.1-C.4}]\label{lemma:kernels}
Condition NK implies C.1-C.4 with  $v=(z,j) \in \V \subseteq \mathcal{Z}\times
\{1,...,J\}$,
\begin{eqnarray*}
& & Z_n(v) = \frac{ \theta_n(v) - \widehat \theta_n(v)}{\sigma_n(v)}, \ Z_n^*(v) =  \frac{\Bn (g_v)}{\sqrt{E_{\Pn}[g^2_v]}}, \ Z_n^{\star}(v) =   \frac{\Gn^o (\hat{g}_v)}{\sqrt{\En[\hat g^2_v]}}, \\
& & \sigma^2_n(v) = E_{\Pn}[g^2_v]/(n h_n^d), \ \ s_n^2(v) = \En[\hat g^2_v]/(n h_n^d), \   \delta_n = 1/\log n, \\
& &  \bar a_n \lesssim \sqrt{ \log n}, \ \ \bar \sigma_n \lesssim \sqrt{1/(n h^d)},  \ \ \text{and} \\
& & a_n(\textsf{V}) = \(2 \sqrt{ \log \{ C ( 1+ C ' (1+ h^{-1}_n) \text{diam}(\textsf{V}))^d \} }\) \vee (1 + \sqrt{d}),
\end{eqnarray*}
for some constants $C$ and $C'$, where $\text{ diam}(\textsf{V})$ denotes the diameter of the set $\textsf{V}$.
 Moreover,
$
 \P[\mathcal{E} > x] = \exp( - x/2 ).
$
\end{lemma}

\begin{remark}\label{remark: feasible an kernel case} Lemma \ref{lemma:kernels} verifies the main conditions C.1-C.4.
These conditions enable construction of either simulated or analytical
critical values. For the latter, the $p$-th quantile of $\mathcal{E}$ is given by
$c(p) = -2\log (1-p),$
so we can set
\begin{equation} \label{exponential ker}
k_{n,\textsf{V}}(p)= a_n(\textsf{V})  -2\log (1-p)/a_n(\textsf{V}),
\end{equation}
where
\begin{equation} \label{eq: anV kernel case}
a_n(\textsf{V}) =
\(2\sqrt{\log  \{\ell_n \( 1 + \ell_n (1+ h_n^{-1}) \text{diam}(\textsf{V})\)^d \}  }\),
\end{equation}
is a feasible version of the scaling factor, in which unknown constants have been replaced by the slowly growing sequence $\ell_n$. Note that $\textsf{V} \mapsto k_{n,\textsf{V}}(p)$ is monotone in $\textsf{V}$ for large $n$, as required in
the analytical construction given in Definition \ref{analytical definition}.   A sharper analytical approach can be based on Hotelling's tube method or on the use of extreme value theory. For details of the extreme value approach, we refer the reader to \cite{Chernozhukov/Lee/Rosen:09}. Note that the simulation-based approach is effectively
a numeric version of the exact version of the tube formula, and is less conservative than using
simplified tube formulas. In \cite{Chernozhukov/Lee/Rosen:09} we established that inference based on extreme value theory is valid, but the asymptotic approximation is accurate only when sets $\textsf{V}$ are ``large", and does not seem to provide an accurate approximation when $\textsf{V}$ is small. Moreover, it often requires
a very large sample size for accuracy even when $\textsf{V}$ is large.
\qed \end{remark}

\begin{lemma}[\textbf{Condition NK implies S in some cases}]\label{lemma: kernels-v-est} Suppose Condition NK holds. Then (1) The radius $\varphi_n$
of equicontinuity of $Z^*_n$ obeys:
$$
\varphi_n \leq o(1) \cdot \(\frac{h_n}{\sqrt{\log n} }\),
$$
for any $o(1)$ term.  (2) If Condition V holds and
\begin{equation}
\label{eq: condition S.2 for kernel}
\( \sqrt{\frac{\log n}{ nh^d}\log n} \)^{1/\rho_n}c_n^{-1} = o\( \frac{h_n}{ \sqrt{\log n} } \),
\end{equation}
then Condition S holds. 
\end{lemma}

\noindent The following is an immediate consequence of Lemmas \ref{lemma:kernels} and \ref{lemma: kernels-v-est} and Theorems \ref{theorem: inference analytical}, \ref{theorem: inference1}, and \ref{theorem: sharp inference}.

\begin{theorem}[\textbf{Estimation and Inference for Bounding Functions Using Local Methods}]\label{theorem:kernels}
Suppose Condition NK holds and  consider  the interval estimator $\widehat \theta_{n0}(p)$ given in Definition \ref{main definition} with either
analytical critical values
specified in Remark \ref{remark: feasible an kernel case} or simulation-based critical values given in
Definition \ref{def:simulation-cv} for the simulation process $Z_n^\star$ specified above. (1) Then $ (i) \ \Pn[\theta_{n0} \leq \widehat \theta_{n0} (p)] \geq  p - o(1),  \ (ii) \ |\theta_{n0} - \widehat \theta_{n0} (p)| = O_{\Pn}\(
\sqrt{ \log n/ (nh_n^d)} \),$ (iii) $ \Pn( \theta_{n0} +   \mu_n \sqrt{ \log n/ (nh_n^d)} \geq \widehat \theta_{n0}(p) ) \to 1$
 for any $\mu_n \to_{\Pn} \infty$.  (2) Moreover, for simulation-based
 critical values, if condition V and (\ref{eq: condition S.2 for kernel}) hold, then
  $ (i) \ \Pn[\theta_{n0} \leq \widehat \theta_{n0} (p)] =  p - o(1)$,  $(ii) \ |\theta_{n0} - \widehat \theta_{n0} (p)| = O_{\Pn}(\sqrt{ 1/ (nh_n^d)})$,  $(iii) \ \Pn( \theta_{n0} +   \mu_n  \sqrt{ 1/ (nh_n^d)}\geq \widehat \theta_{n0}(p) ) \to 1$
 for any $\mu_n \to_{\Pn} \infty$.
\end{theorem}

\noindent In Supplementary Appendix \ref{sec:est:local:CMIexample} we provide an example where the bounding function is obtained from conditional moment inequalities, and where Condition NK holds under primitive conditions. We provide only one example for brevity, but more examples can be covered as for series estimation in Section \ref{sec:NP-est-series}. In Supplementary Appendix \ref{sec:proofs-strong-local} we provide conditions under which the required strong approximation in Condition NK.1 holds.

\section{Strong Approximation for Asymptotically Linear Series Estimators}\label{sec:strong:series:main-text}

In the following theorem we establish strong approximation for series estimators appearing in the previous section as part of Condition NK.1.  In Appendix \ref{sec:prac:roymodel} of the on-line supplement we demonstrate as a leading example how the required asymptotically linear representation can be achieved from primitive conditions for the case of estimation of a conditional mean function.

\begin{theorem}[Strong Approximation For Asymptotically Linear Series Estimators]\label{theorem:strong series}
Let $(A, \mathcal{A}, \P_n)$ be the probability space for each $n$, and let $n \to \infty$. Let $\delta_n \to 0$ be a sequence of constants converging to $0$ at no faster than a polynomial rate in $n$. \ Assume (a) the series estimator has the form $\widehat \theta_n(v) = p_{n}(v)'\widehat \beta_n,$
where $p_{n}(v):= (p_{n,1}(v),\ldots,p_{n,K_n}(v))'$ is a collection of
${K_n}$-dimensional approximating functions such that  ${K_n} \to \infty$ and $\widehat {\beta}_n$ is a ${K_n}$-vector of estimates; (b) The estimator $\widehat {\beta}_n$ satisfies an  asymptotically linear representation around some ${K_n}$-dimensional vector $ \beta_n$
 \begin{align}\label{series-bahadur}
&  \Omega_n^{-1/2} \sqrt{n} (\widehat \beta_n - \beta_n) = n^{-1/2}  \sum_{i=1}^n u_{i,n} + r_n, \ \ \|r_n\| = o_{\P_n}(\delta_n),  \\
& u_{i, n}, i=1,...,n \text{ are independent with }  E_{\Pn}[ u_{i,n}] = 0,  E_{\Pn} [u_{i,n} u_{i,n}'] = I_{K_n},  \text{ and} \\
& \Delta_n = \sum_{i=1}^n E\|u_{i,n}\|^3/n^{3/2} \text{ such that }  {K_n} \Delta_n/\delta_n^3 \to 0,
\end{align}
where  $\Omega_n$ is a sequence of $K_n \times K_n$ invertible matrices. (c) The function $\theta_n(v)$ admits the approximation $
\theta_n(v) = p_{n}(v)'\beta_n + A_n(v),$
where the approximation error  $A_n(v)$ satisfies $ \sup_{v \in \mathcal{V}} \sqrt{n} |A_n(v)|/ \|g_n(v)\| = o(\delta_n)$, for $g_n(v) := p_{n}(v)' \Omega_n^{1/2}$.  Then we can find a random normal vector $\mathcal{N}_n =_d \mathcal{N}(0, I_{K_n})$ such that
$
\| \Omega_n^{-1/2}\sqrt{n} (\widehat \beta_n - \beta_n) - \mathcal{N}_n  \| = o_{\P_n}(\delta_n) $
and
$$
\sup_{v \in \mathcal{V}} \left | \frac{\sqrt{n} (\widehat \theta_n(v) - \theta_n(v) ) }{\|g_n(v)\|} - \frac{g_n(v)}{\|g_n(v)\|} \mathcal{N}_n \right | = o_{\P_n}(\delta_n).
$$
\end{theorem}

The following corollary covers the cases considered in the examples of the previous section.

\begin{corollary}[\textbf{A Leading Case of Influence Function}]\label{corollary:strong} Suppose the conditions
of Theorem \ref{theorem:strong series} hold with
  $u_{i,n}:= \Omega_n^{-1/2} Q_n^{-1} p_{n}(V_i) \epsilon_i$, where $(V_i,\epsilon_i)$  are i.i.d. with $E_{\Pn}[\epsilon_i p_{n}(V_i)] = 0$,  $S_n := E_{\Pn}[ \epsilon_i^2 p_{n}(V_i) p_{n}(V_i)']$ , and \ $ \Omega_n :=  Q_n^{-1} S_n (Q_n^{-1})',$
where  $Q_n^{-1}$ is a non-random invertible matrix, and $\| \Omega_n^{-1/2} Q_n^{-1} \| \leq \tau_n$;  $E_{\Pn}[|\epsilon_i|^3|V_i=v]$ is bounded above uniformly in $v \in \textrm{support}(V_i)$,
and $E_{\Pn}[\|p_n(V_i)\|^3] \leq C_n K_n^{3/2}$. Then, the key growth restriction
on the number of series terms ${K_n} \Delta_n / \delta_n^3 \to 0$ holds if
$\tau_n^6 C_n^2 {K_n}^5/(n\delta_n^6) \to 0.$
\end{corollary}

\begin{remark}[Applicability] In this paper $\delta_n = 1/\log n$.  Sufficient conditions for linear approximation (b) follow from results
in the literature on series estimation, e.g. \cite{andrews:series},  \cite{Newey:95}, and
\cite{Newey:97}, and \cite{Belloni/Chernozhukov/Fernandez-Val:10}. See also \cite{Chen:07} and references therein for a general overview of sieve estimation and recent developments. The main text provides several examples, including mean and quantile regression, with primitive conditions that provide sufficient conditions for the linear approximation. \qed \end{remark}

\section{Implementation}\label{sec:implementation-details}

In Section \ref{param-and-series-est-simulation-cv} we lay out steps for implementation of parametric and series estimation of bounding functions, while in Section \ref{kernel-est-simulation-cv} we provide implementation steps for kernel-type estimation. \ The end goal in each case is to obtain estimators $\widehat{\theta}_{n0}(p)$ that provide bias-corrected estimates or the endpoints of confidence intervals depending on the chosen value of $p$, e.g. $p=1/2$ or $p=1-\alpha$.
As before, we focus here on the upper bound. If instead
$\widehat{\theta}_{n0}(p)$ were the lower bound for $\theta^{\ast}$, given by
the supremum of a bounding function, the same algorithm could be applied to perform inference on $-\theta^{\ast}$, bounded above by the infimum of the negative of the original bounding function, and then any inference statements for $-\theta^{\ast}$ could trivially be transformed to inference statements for $\theta^{\ast}$. Indeed, any set of lower and upper bounds can be similarly transformed to a collection of upper bounds, and the above algorithm applied to perform inference on $\theta^*$, e.g. according to the methods laid out for inference on parameters bounded by conditional moment inequalities in Section \ref{sec:est-inf-general-conditions}.\footnote{For example if we have
$
\theta_n^{l}\left( z\right) \leq \theta_n^{\ast }\leq \theta_n ^{u}\left(
z\right) \text{ for all }z\in \mathcal{Z}\text{,}
$
then we can equivalently write
$
\min_{z\in \mathcal{Z}}\min_{j=1,2}g_n\left( \theta_n ^{\ast },z,j\right) \geq 0
\text{,}
$
where
$g_n\left( \theta_n^{\ast },z,1\right) =\theta_n ^{u}\left(
z\right) -\theta_n ^{\ast }$ and $g_n\left( \theta_n ^{\ast },z,2\right)
=\theta_n^{\ast }-\theta_n^{l}\left( z\right)$. Then we can apply our method through use of the auxiliary function $g_n(\theta_n,z,j)$, in similar fashion as in Example C with multiple conditional moment inequalities.} Alternatively, if one wishes to perform inference on the identified set in such circumstances one can use the intersection of upper and lower one-sided intervals each based on $\tilde{p}=(1+p)/2$ as an asymptotic level-$p$ confidence set for $\Theta_I$, which is valid by Bonferroni's inequality.\footnote{In an earlier version of this paper, \cite{Chernozhukov/Lee/Rosen:09}, we provided a different method for inference on a parameter with both lower and upper bounding functions, which can also be used for valid inference on $\theta^{\ast}$.}

\subsection{Parametric and Series Estimators}\label{param-and-series-est-simulation-cv}

Let $\beta_n$ denote the bounding function parameter vector if parametric estimation is used, while $\beta_n$ denotes the coefficients of the series terms if series estimation is used, as in Section \ref{sec:NP-est-series}. $K$ denotes the dimension of $\beta_n$ and $I_{K}$ denotes the $K$-dimensional identity matrix. As in the main text let $p_n(v)=\partial\theta_n                                                                                                                                                                        \left(
    v,\widehat{\beta}_n\right)
    /\partial\beta_n$, which are simply the series terms in the case of series estimation.

\begin{algorithm}[Implementation for Parametric and Series Estimation] (1) Set $\tilde \gamma_n \equiv 1- .1/\log n$. Simulate a large number $R$ of draws denoted $Z_{1},...,Z_{R}$ from the $K$-variate standard normal distribution $\mathcal{N}\left(  0,I_{K}\right)$. (2) Compute $\widehat{\Omega}_n$, a consistent estimator for the large sample variance of $\sqrt{n}\left( \widehat{\beta}_n-\beta_{n}\right)$. (3) For each $v\in\mathcal{V}$, compute $\widehat{g}\left(  v\right)=p_n(v)'\widehat{\Omega }_n^{1/2}$ and set $s_n(v) =  \| \widehat {g}\left(  v\right) \|/\sqrt{n}$. (4) Compute
$$
k_{n,\mathcal{V}}\(\tilde{\gamma}_n\) = {\tilde\gamma}_n-\text{quantile of } \{\sup_{ v \in \mathcal{V}}
\left(\widehat{g}\left(  v\right)
^{\prime}Z_{r}/\left\Vert \widehat {g}\left(  v\right)
\right\Vert\right) , r=1,...,R \}\text{, and}
$$
$$
\widehat V_n = \{ v \in \V: \widehat \theta_n(v) \leq \min_{v \in \mathcal{V}}\( \widehat \theta_n(v) +  k_{n, \mathcal{V}}(\tilde \gamma_n)s_n(v)\) + 2   k_{n, \mathcal{V}}(\tilde \gamma_n) s_n(v) \},$$
(5) Compute
$$k_{n,\widehat V_n}\left(  p\right) =p-\text{quantile of } \left \{ \sup_{v\in
\widehat{V}_{n}}\left(\widehat{g}\left(  v\right)
^{\prime}Z_{r}/\left\Vert \widehat {g}\left(  v\right)
\right\Vert\right), r=1,...,R\right \}\text{, and set}$$
$$\widehat{\theta}_{n0}(p)=\inf_{v\in \mathcal{V}}\left [
    \widehat{\theta }_n \left(  v\right)  + k_{n,\widehat V_n}(p) \left\Vert \widehat {g}\left(  v\right)
\right\Vert/\sqrt{n}  \right ] . $$
\end{algorithm}

An important special case of the parametric setup is that where the support of $v$ is finite, as in Example 1 of Section \ref{Section:ConditionP}, so
that $\mathcal{V=}\left\{1,...,J\right\}  $. \ In this
case the algorithm
applies with $\theta_n\left(  v,\beta_n\right)  =\sum_{j=1}%
^{J}1[v=j]\beta_{nj}$, i.e. where for each $j$,
$\theta_n\left(j,\beta_n\right)  =\beta_{nj}$ and
$\widehat{g}\left(  v\right) =\left(  1\left[  v=1\right]
,...,1\left[  v=J\right]  \right) \cdot\widehat{\Omega}_n^{1/2}$. Note that this covers the case where the bounding function is a conditional mean or quantile with discrete conditioning variable, such as conditional mean estimation with discrete regressors, in which case $\beta_{nj} = E[Y|V=j]$ can be estimated by a sample mean.

\begin{remark} In the case of series estimation, if desired one can bypass simulation of
the stochastic process by instead employing the analytical critical value in step 4, $k_{n,\textsf{V}}(p)= a_n(\textsf{V})  -2\log (1-p)/a_n(\textsf{V})$ from Remark \ref{remark: feasible an series case} in Section \ref{sec:NP-est-series}.
This is convenient because it does not involve simulation, though it requires computation of $a_n(\widehat V_n) = 2\sqrt{\log  \{ \ell_n ( 1 + \ell_n L_n \text{diam}(\widehat V_n))^d \}}$.
Moreover, it could be too conservative in some applications. Thus, we recommend using simulation, unless the
computational cost is too high. 
\end{remark}

\subsection{Kernel-Type Estimators}\label{kernel-est-simulation-cv}

In this section we describe the steps for implementation of kernel-type estimators.

\begin{algorithm}[Implementation for Kernel Case] (1)  Set ${\gamma}_n \equiv 1- .1/\log n$. Simulate $R \times n$ independent draws from $N (0,1)$, denoted by $\{\eta_{ir}:i=1,\ldots,n, r=1,\ldots,R \}$, where $n$ is the sample size and $R$ is the number of simulation repetitions. (2) For each $v\in\mathcal{V}$ and $r = 1,\ldots,R$, compute
$
\Gn^{o} (\hat{g}_v; r) =  \frac{1}{\sqrt{ n}} \sum_{i=1}^n \eta_{ir}  \hat g_v (U_i,Z_i),
$
where $\hat g_v (U_i,Z_i)$ is defined in Section \ref{sec:est:local}, that is
\begin{align*}
\hat g_v (U_i,Z_i) = \frac{ e_j'\hat U_{i} }{ (h_n^d )^{1/2} \hat f_n(z) } \mathbf{K} \left( \frac{z - Z_i}{h_n} \right).
\end{align*}
Let $s_n^2(v) = \En[\hat g^2_v]/(n h_n^d)$ and $\En[\hat g^2_v] = n^{-1} \sum_{i=1}^n   \hat g_v^2 (U_i,Z_i)$. Here, $\hat U_{i}$ is the kernel-type  regression residual and $\hat f_n(z)$ is the kernel density estimator of density of $Z_i$.
(3) Compute
$
k_{n,\mathcal{V}}\({\gamma}_n\) =  {\gamma}_n-\text{quantile of }
\left \{\sup_{ v \in \mathcal{V}} \Gn^{o}  (\hat{g}_v; r)/ \sqrt{\En[\hat g^2_v]}, r=1,...,R \right \},
$
and
$
\widehat V_n = \{ v \in \V: \widehat \theta_n(v) \leq \min_{v \in \mathcal{V}}\( \widehat \theta_n(v) +  k_{n, \mathcal{V}}(\gamma_n)s_n(v)\) + 2   k_{n, \mathcal{V}}(\gamma_n) s_n(v) \}$.
(4) Compute
$k_{n,\widehat V_n}\left(  p\right) =p-\text{quantile of } \{ \sup_{v\in \widehat{V}_{n}}  \Gn^{o}  (\hat{g}_v; r)/ \sqrt{\En[\hat g^2_v]} , r=1,...,R \},$
and set    $\widehat{\theta}_{n0}(p)=\inf_{v\in \mathcal{V}} [
    \widehat{\theta }\left(  v\right)  + k_{n,\widehat V_n}(p) s_n(v) ].$
\end{algorithm}

\begin{remark}(1) The researcher also has the option of employing an analytical
approximation in place of simulation if desired. This can be done by using $k_{n,\textsf{V}}(p)= a_n(\textsf{V})  -2\log (1-p)/a_n(\textsf{V})$ from Remark \ref{remark: feasible an kernel case}, but
requires computation of $$a_n(\widehat V_n) = 2\sqrt{\log  \{\ell_n ( 1 + \ell_n (1+ h_n^{-1}) \text{diam}(\widehat V_n))^d\}  }.$$ This approximation could be too conservative in some applications, and thus we recommend using simulation, unless the
computational cost is too high. (2) In the case where the bounding function is non-separable in a parameter of interest, a confidence interval for this parameter can be constructed as described in Section \ref{param-and-series-est-simulation-cv}, where step (1) is carried out once and steps (2)-(4) are executed iteratively on a set of parameter values approximating the parameter space. However, the bandwidth, $\hat f_n(z)$, and $\mathbf{K} \left( \frac{z - Z_i}{h_n} \right)$ do not vary across iterations and thus only need to computed once. \qed\end{remark}

\section{Monte Carlo Experiments}\label{sec:monte-carlo}

In this section we present results of Monte Carlo
experiments to illustrate the finite-sample performance of
our method. We consider a Monte Carlo design with bounding function
\begin{align}\label{mp-bound-sim}
\theta(v) := L \phi(v),
\end{align}
where $L$ is a constant and $\phi(\cdot)$ is the standard normal density function. Throughout the Monte Carlo
experiments, the parameter of interest is $\theta_0=\sup_{v \in \mathcal{V}} \theta (v)$.\footnote{Previous sections focused on $\theta_0=\inf_{v \in \mathcal{V}} \theta (v)$ rather than $\theta_0=\sup_{v \in \mathcal{V}} \theta (v)$. This is not a substantive difference as for any function $\theta (\cdot)$, $\sup_{v \in \mathcal{V}} \theta (v)=-\inf_{v \in \mathcal{V}} ( - \theta (v) )$.}

\subsection{Data-Generating Processes}

We consider four Monte Carlo designs for the sake of illustration.\footnote{We consider some additional Monte Carlo designs in Section \ref{sec:mc-supplement} of the on-line supplement.}  In the first Monte Carlo design, labeled DGP1, the bounding function is completely flat so that $V_0 = \mathcal{V}$.  In the second design, DGP2, the bounding function is non-flat, but smooth in a neighborhood of its maximizer, which is unique so that $V_0$ is singleton.  In DGP3 and DGP4, the bounding function is also non-flat and smooth in a neighborhood of its (unique) maximizer, though relatively peaked. Illustrations of these bounding functions are provided in Figures \ref{figure_dgp12} and \ref{figure_dgp34} of our on-line supplement. In practice the shape of the bounding function is unknown, and the inference and estimation methods we consider do not make use of this information. \ As we describe in more detail below, we evaluate the finite sample performance of our approach in terms of coverage probability for the true point $\theta_0$ and coverage for a false parameter value $\theta$ that is close to but below $\theta_0$.  We compare the performance of our approach to that of the Cramer Von-Mises statistic proposed by AS. \ DGP1 and DGP2 in particular serve to effectively illustrate the relative advantages of both procedures as we describe below. \ Neither approach dominates.

For all DGPs we generated 1000 independent samples from the following model:
$$
V_i
\sim \text{Unif}[-2,2],
U_i = \min \{\max\{-3, \sigma \tilde{U}_i\}, 3 \} ,
\ \text{and}
\ \ Y_i = L \phi (V_i) + U_i,
$$
where
$\tilde{U}_i \sim N(0,1)$ and $L$ and $\sigma$ are constants.
We set these constants in the following way:
\begin{eqnarray*}
& & \textrm{DGP1:  } L = 0 \text{ and } \sigma = 0.1 \text{; } \ \ \ \textrm{DGP2:  } L = 1 \text{ and } \sigma = 0.1 \text{;} \\
& & \textrm{DGP3:  } L = 5 \text{ and } \sigma = 0.1 \text{; } \ \ \ \textrm{DGP4:  } L = 5 \text{ and } \sigma = 0.01 \text{.}
\end{eqnarray*}

We considered sample sizes  $n=500$ and  $n=1000$, and we implemented both
series and kernel-type estimators to estimate the
bounding function $\theta(v)$ in
\eqref{mp-bound-sim}.
We set $\mathcal{V}$ to be an interval between the $.05$ and $.95$ sample quantiles of $V_i$'s in order to avoid undue influence of
outliers at the boundary of the support of $V_i$.
For both types of estimators, we computed critical values via simulation as described in
Section \ref{sec:implementation-details}, and we implemented our
method with both the conservative but simple, non-stochastic choice $\widehat{V} = \mathcal{V}$ and the set estimate $\widehat{V} = \widehat V_n$
described in Section \ref{sec:inf-est-strategy}.

\subsection{Series Estimation}\label{sec:monte-carlo-series}  For basis functions we use polynomials and cubic B-splines
 with knots equally spaced over the sample quantiles of $V_i$.
 The number $K=K_n$ of approximating functions was obtained by the following simple rule-of-thumb:
   \begin{align}\label{K-rule}
K = \underline{\widehat{K}}, \ \ \widehat{K} :=  \widehat{K}_{cv} \times n^{-1/5} \times n^{2/7},
\end{align}
where $\underline{a}$ is defined as the largest integer that is smaller than or equal to $a$, and $\widehat{K}_{cv}$ is the minimizer of the leave-one-out least squares cross validation score. If $\theta(v)$ is twice continuously differentiable, then a cross-validated $K$ has the form $K \propto n^{1/5}$
asymptotically. Hence, the multiplicative factor $n^{-1/5} \times n^{2/7}$ in \eqref{K-rule} ensures that
the bias is asymptotically negligible from under-smoothing. \footnote{For B-splines the optimal $\widehat{K}_{cv}$ was first selected from the first $5 \times n^{1/5}$ values starting from $5$, with $n^{1/5}$ rounded up to the nearest integer. If the upper bound was selected, the cross validation (CV) score of $\widehat{K}_{cv}$ was compared to that of $\widehat{K}_{cv}+1$ iteratively, such that $\widehat{K}_{cv}$ was increased until further increments resulted in no improvement. This allows $\widehat{K}_{cv} \propto n^{1/5}$ and provides a crude check against the upper bound binding in the CV search, though in these DGPs results differed little from those searching over $\{5,6,7,8,9\}$, reported in \cite{Chernozhukov/Lee/Rosen:09}.  For polynomials the CV search was limited to the set $\{3,4,5,6\}$ due to multicollinearity issues that arose when too many terms were used.}

\subsection{Kernel-Type Estimation}\label{monte-carlo-kernel}\footnote{Appendices \ref{sec:proofs-strong-local} and \ref{sec:est:local:proofs} of the on-line supplement provide strong approximation results and proofs for kernel-type estimators, including the local linear estimator used here.}
We use local linear smoothing since
it is known to behave better at the boundaries of the support than the standard
kernel method.  We used the kernel function $K(s)=\frac{15}{16}%
(1-s^{2})^{2}1(|s|\leq 1)$ and the rule of thumb bandwidth:
\begin{align}\label{h-rule}
h = \widehat{h}_{ROT} \times \widehat{s}_v \times n^{1/5} \times n^{-2/7},
\end{align}
where $\widehat{s}_v$ is the square root of the sample variance of the $V_i$, and $\widehat{h}_{ROT}$ is the rule-of-thumb bandwidth for estimation of
$\theta(v)$ with studentized $V$, as prescribed in Section 4.2 of
\cite{Fan/Gijbels:96}.
The exact form of $\widehat{h}_{ROT}$ is
\begin{align*}
\widehat{h}_{ROT} = 2.036 \left[ \frac{ \tilde{\sigma}^2 \int w_0(v) dv }{ n^{-1} \sum_{i=1}^n \left\{
\tilde{\theta}^{(2)}(\tilde{V}_i) \right\}^2 w_0(\tilde{V}_i) } \right]^{1/5} n^{-1/5},
\end{align*}
where
$\tilde{V}_i$'s are studentized $V_i$'s,
$\tilde{\theta}^{(2)}(\cdot)$ is the second-order derivative of the global quartic parametric fit of
$\theta(v)$ with studentized $V_i$,
$\tilde{\sigma}^2$ is the simple average of squared residuals from the parametric fit,
$w_0(\cdot)$ is a uniform weight function that has value 1 for any $\tilde{V}_i$ that is between
the $.10$  and $.90$ sample quantiles of $\tilde{V}_i$.
Again, the factor $n^{1/5} \times n^{-2/7}$ is multiplied in \eqref{h-rule} to ensure that
the bias is asymptotically negligible due to under-smoothing.

\subsection{Simulation Results}

To evaluate the relative performance of our inference method, we also implemented one of the inference methods proposed by AS, specifically their Cram\'{e}r-von Mises-type (CvM) statistic with both plug-in asymptotic (PA/Asy) and asymptotic generalized moment selection (GMS/Asy) critical values.  For instrument functions we used countable hypercubes and the $S$-function of AS Section 3.2.\footnote{All three $S$-functions in AS Section 3.2 are equivalent in our design, since there is a single conditional moment inequality.} We set the weight function and tuning parameters for the CvM statistic exactly as in AS (see AS Section 9). \ These values performed well in their simulations, but our Monte Carlo design differs from theirs, and alternative choices of tuning parameters could perform more or less favorably in our design. We did not examine sensitivity to the choice of tuning parameters for the CvM statistic.

The coverage probability (CP) of confidence intervals with nominal level 95\% is evaluated for the true lower bound $\theta_0$, and false coverage probability (FCP) is reported at $\theta = \theta_0 - 0.02$. There were 1,000 replications for each experiment. \ Tables \ref{mc1}, \ref{mc2}, and \ref{mc1-cont} summarize the results. CLR and AS refer to our inference method and that of AS, respectively.

We first consider the performance of our method for DGP1.  In terms of coverage for $\theta_0$ both series estimators and the local linear estimator perform reasonably well, with the series estimators performing best. The polynomial series and local linear estimators perform somewhat better in terms of false coverage probabilities, which decrease with the sample size for all estimators. The argmax set $V_0$ is the entire set $\mathcal{V}$, and our set estimator $\widehat V_n$ detects this.  Turning to DGP2 we see that coverage for $\theta_0$ is in all cases roughly .98 to .99.  There is non-trivial power against the false parameter $\theta$ in all cases, with the series estimators giving the lowest false coverage probabilities.  For DGP3 the bounding function is relatively peaked compared to the smooth but non-flat bounding function of DGP2.  Consequently the average endpoints of the preliminary set estimator $\widehat V_n$ become more concentrated around 0, the maximizer of the bounding function. Performance in terms of coverage probabilities improves in nearly all cases, with the series estimators performing significantly better when $n=1000$ and $\widehat V_n$ is used. With DGP4 the bounding function remains as in DGP3, but now with the variance of $Y_i$ decreased by a factor of 100.  The result is that the bounding function is more accurately estimated at every point. \ Moreover, the set estimator $\widehat V_n$ is now a much smaller interval around 0. \ Coverage frequencies for $\theta_0$ do not change much relative to DGP3, but false coverage probabilities drop to 0.
Note that in DGPs 2-4, our method performs
better when $V_n$ is estimated in that it makes the coverage probability more accurate and the false coverage probability smaller. DGPs 3-4 serve to illustrate the convergence of our set estimator $\widehat{V}_n$ when the bounding function is peaked and precisely estimated, respectively.

In Table \ref{mc2} we report the results of using the CvM statistic of AS to perform inference.  For DGP1 with a flat bounding function the CvM statistic with both the PA/Asy and GMS/Asy performs well.  Coverage frequencies for $\theta_0$ were close to the nominal level, closer than our method using polynomial series or local linear regression.  The CvM statistic has a lower false coverage probability than the CLR confidence intervals in this case, although at a sample size of 1000 the difference is not large.  For DGP2 the bounding function is non-flat but smooth in a neighborhood of $V_0$ and the situation is much different. \ For both PA/Asy and GMS/Asy critical values with the CvM statistic, coverage frequencies for $\theta_0$ were 1.  Our confidence intervals also over-covered in this case, with coverage frequencies of roughly .98 to .99.  Moreover, the CvM statistic has low power against the false parameter $\theta$, with coverage 1 with PA/Asy and coverage .977 and .933 with sample size 500 and 1000, respectively using GMS/Asy critical values.  For DGP3 and DGP4 both critical values for the CvM statistic gave coverage for $\theta_0$ and the false parameter $\theta$ equal to one.  Thus under DGPs 2,3, and 4 our confidence intervals perform better by both measures. In summary, overall neither approach dominates.

 Thus, in our Monte Carlo experiments the CvM statistic exhibits better power when the bounding function is flat, while our confidence intervals exhibit better power when the bounding function is non-flat. \ AS establish that the CvM statistic has power against some $n^{-1/2}$ local alternatives under conditions that are satisfied under DGP1, but that do not hold when the bounding function has a unique minimum.\footnote{Specifically Assumptions LA3 and LA3' of AS Theorem 4 do not hold when the sequence of models has a fixed bounding function with a unique minimum. As they discuss after the statement of Assumptions LA3 and LA3', in such cases GMS and plug-in asymptotic tests have trivial power against $n^{-1/2}$ local alternatives.} We have established local asymptotic power for nonparametric estimators of polynomial order less distant than $n^{-1/2}$ that apply whether the bounding function is flat or non-flat. \ Our Monte Carlo results accord with these findings.\footnote{We did not do CP-correction in our reported results. Our conclusion will remain valid even with CP-correction as in AS, since our method performs better
in DGP2-DGP4 where we have over-coverage.}
In the on-line supplement, we present further supporting Monte Carlo evidence and local asymptotic power analysis to show why our method performs better than the AS method in non-flat cases.

In Table \ref{mc-time} we report computation times for our Monte Carlo experiments.\footnote{These computation times were obtained on a 2011 iMac desktop with a 2.7 GHz processor and 8GB RAM using our implementation. Generally speaking, performance time for both methods will depend on the efficiency of one's code, and more efficient implementation times for both methods may be possible.}  The fastest performance in terms of total simulation time was achieved with the CvM statistic of AS, which took roughly 9 minutes to execute a total of 16,000 replications. Simulations using our approach with B-spline series, polynomial series, and local linear polynomials took roughly 58, 19, and 84 minutes, respectively. Based on these times the table shows for each statistic the average time for a single test, and the relative performance of each method to that obtained using the CvM statistic.\footnote{The difference in computation time between polynomial and B-spline series implementation was almost entirely due to the simpler cross-validation search from polynomials. With a search over only 4 values of $\widehat{K}_{cv}$, simulation time for B-splines took roughly 20 minutes. Cross-validation is also in part accountable for slower performance relative to AS. We followed Section 9 of AS in choosing tuning parameters for the CvM statistic, which does not involve cross-validation. Using B-splines with a deterministic bandwidth resulted in a computation time of 12 minutes, roughly 1.4 times the total computation time for the CvM statistic. Nonetheless, we prefer cross-validation for our method in practice.}


In practice one will not perform Monte Carlo experiments but will rather be interested in computing a single confidence region for the parameter of interest. When the bounding function is separable our approach offers the advantage that the critical value does not vary with the parameter value being tested. As a result, we can compute a confidence region in the same amount of time it takes to compute a single test. On the other hand, to construct a confidence region based on the CvM statistic, one must compute the statistic and its associated critical value at a large number of points in the parameter space, where the number of points required will depend on the size of the parameter space and the degree of precision desired.  If however the bounding function is not separable in the parameter of interest, then both approaches use parameter-dependent critical values.

\section{Conclusion}\label{sec:conclusion}

In this paper we provided a novel method for inference on
intersection bounds.  Bounds of this form are common in the
recent literature, but two issues have posed difficulties for
valid asymptotic inference and bias-corrected estimation.
First, the application of the supremum and infimum operators to
boundary estimates results in finite-sample bias. Second,
unequal sampling error of estimated bounding functions
complicates inference.  We overcame these difficulties by
applying a precision-correction to the estimated bounding
functions before taking their intersection.  We employed strong
approximation to justify the magnitude of the correction in
order to achieve the correct asymptotic size.  As a by-product,
we proposed a bias-corrected estimator for intersection bounds
based on an asymptotic median adjustment.  We provided formal
conditions that justified our approach in both parametric and
nonparametric settings, the latter using either kernel or
series estimators.

At least two of our results may be of independent interest
beyond the scope of inference on intersection bounds.  First,
our result on the strong approximation of series estimators is
new.  This essentially provides a functional central limit
theorem for any series estimator that admits a linear
asymptotic expansion, and is applicable quite generally.
Second, our method for inference applies to any value that can
be defined as a linear programming problem with either finite
or infinite dimensional constraint set. Estimators of this form
can arise in a variety of contexts, including, but not limited
to intersection bounds.  We therefore anticipate that although
our motivation lay in inference on intersection bounds, our
results may have further application.

\appendix

\section{Definition of Strong Approximation}\label{sec:definitions and notation}

The following definitions are used extensively.

\begin{definition}[Strong approximation]\label{def:strong}
Suppose that for each $n$ there are random variables $Z_n$ and $Z'_n$ defined on a probability space $(A, \mathcal{A}, \P_n)$
 and taking values in the separable metric space $(S,d_S)$.  We say  that $Z_n =_d Z_n' +  o_{\P_n}(\delta_n)$,
for $\delta_n \to 0$,   if there are identically distributed copies of $Z_n$ and $Z_n'$, denoted
$\bar Z_n$ and $\bar Z_n'$, defined on $(A, \mathcal{A}, \P_n)$ (suitably enriched if needed), such that
$$
d_S(\bar Z_n, \bar Z_n') = o_{\P_n}(\delta_n).
$$
\end{definition}

\noindent  Note that copies $\bar Z_n$ and $\bar Z_n'$ can be always defined on $(A, \mathcal{A}, \P_n)$ by suitably enriching this space by taking product probability spaces.
It turns out that for the Polish spaces, this definition implies the following stronger, and much more convenient, form.

\begin{lemma}[A Convenient Implication for Polish Spaces via Dudley and Philipp]\label{lemma: DP} Suppose that $(S,d_S)$ is Polish, i.e. complete, separable metric space,  and  $(A, \mathcal{A}, \P_n)$
has been suitably enriched. Suppose that Definition \ref{def:strong} holds, then there is also an identical copy $Z^*_n$ of $Z_n'$ such that
 $Z_n =Z^*_n +  o_{\P_n}(\delta_n)$, that is,
$$
d_S(Z_n, Z^*_n) = o_{\P_n}(\delta_n)
$$

\end{lemma}

\textbf{Proof. } We start with the original probability space
$(A', \mathcal{A}', \P_{n}')$ that can carry $Z_n$ and $(\bar Z_n, \bar Z'_n)$.
In order to apply Lemma 2.11 of \cite{Dudley/Philipp:83}, we need to carry a standard uniform random variable $U \sim U(0,1)$ that is independent
of $Z_n$.  To guarantee this we can always consider $U \sim U(0,1)$ on the standard space  $([0,1], \mathcal{F}, \lambda)$, where $\mathcal{F}$
is the Borel sigma algebra on $[0,1]$ and $\lambda$ is the usual Lebesgue measure, and  then enrich the original space $(A', \mathcal{A}', \P_{n}')$  by creating
formally a new space  $(A, \mathcal{A}, \P_{n})$
as the product of  $(A', \mathcal{A}', \P_{n}')$  and $([0,1], \mathcal{F}, \lambda)$.
 Then using Polishness of $(S,d_S)$, given the joint law  of $(\bar Z_n, \bar Z_n')$, we can apply Lemma 2.11 of \cite{Dudley/Philipp:83}
to construct $Z_n^*$ such that $(Z_n, Z_n^*)$
has the same law as $(\bar Z_n, \bar Z'_n)$, so that $d_{S}(\bar Z_n,\bar Z_n') = o_{\P_n}(\delta_n)$ implies
$d_S(Z_n,Z_n^*) = o_{P_n}(\delta_n)$. \qed

Since in all of our cases  the relevant metric spaces are either the space of continuous functions
defined on a compact set equipped with the uniform
metric or finite-dimensional Euclidean spaces, which are all Polish spaces,
we can use Lemma \ref{lemma: DP} throughout the paper. Using this implication of strong approximation makes our proofs slightly simpler.

\section{Proofs for Section \ref{sec:est-inf-general-conditions}}\label{sec:main-section-proofs}

\subsection{Some Useful Facts and Lemmas}\label{sec:auxiliary-results}

A useful result in our case is the anti-concentration inequality derived in \cite{Chernozhukov/Kato:11}.
\begin{lemma}[\textbf{Anti-Concentration Inequality (\cite{Chernozhukov/Kato:11})}]\label{lemma:anti-concentration} Let $X=(X_{t})_{t \in T}$ be a separable Gaussian process
indexed by a semimetric space $T$ such that $E_P[ X_{t} ] = 0$ and $E_P[ X^{2}_{t} ] = 1$ for all $t \in T$.  Then
\begin{equation}
\sup_{x \in \mathbb{R}} P\left(\Big | \sup_{t \in T} X_{t} - x \Big | \leq \epsilon \right) \leq C \epsilon \left( E_P\left[ \sup_{t \in T} X_{t} \right] \vee 1 \right), \ \forall \epsilon > 0, \label{anti}
\end{equation}
where $C$ is an absolute constant.
\end{lemma}

An immediate consequence of this lemma is the following result:
\begin{corollary}[\textbf{Anti-concentration for} $ \sup_{v \in V_n} Z_n^*(v)$]\label{cor: anti}  Let $V_n$ be any sequence
of compact non-empty subsets in $\mathcal{V}$.  Then under condition C.2-C.3, we have that for $\delta_n \to 0$ such that $\delta_n = o(1/\bar a_n)$
$$
\sup_{x \in \mathbb{R}} \Pn\left(\Big | \sup_{v \in V_n} Z_n^*(v) - x \Big | \leq \delta_n \right) = o(1).
$$
\end{corollary}

\begin{proof} Continuity in Condition C.2 implies separability of $Z^*_n$.  Condition C.3 implies that $E_{\P_n}[\sup_{v \in V_n} Z^*_n(v)]\leq E_{\P_n}[\sup_{v \in \V} Z^*_n(v)]  \leq K \bar a_n$ for some constant $K$ that depends only on $\eta$, so that
$$
\sup_{x \in \mathbb{R}} \Pn\left(\Big | \sup_{v \in V_n} Z_n^*(v) - x \Big | \leq \delta_n \right) \leq C \delta_n [K \bar a_n \vee 1] = o(1).
$$
\end{proof}

\begin{lemma}[\textbf{Closeness in Conditional Probability Implies Closeness of Conditional Quantiles Unconditionally}]\label{lemma: quantiles are close} Let $X_n$ and $Y_n$ be random variables and $\D_n$ be a random vector. Let $F_{X_n}(x\mid\D_n)$ and $F_{Y_n}(y\mid\D_n)$
denote the conditional distribution functions, and  $F^{-1}_{X_n}(p\mid\D_n)$ and $F^{-1}_{Y_n}(p\mid\D_n)$
denote the corresponding conditional quantile functions. If $\P_n(|X_n - Y_n|> \xi_n \mid \D_n) = o_{\P_n}(\tau_n)$ for some sequence $\tau_n \searrow 0$,
then with unconditional probability $\P_n$ converging to one, for some $\varepsilon_n = o(\tau_n)$,
$$
F^{-1}_{X_n}(p\mid\D_n) \leq F^{-1}_{Y_n}(p+\varepsilon_n\mid\D_n) + \xi_n \text{ and } F^{-1}_{Y_n}(p\mid\D_n) \leq F^{-1}_{X_n}(p+\varepsilon_n\mid\D_n) + \xi_n, \forall p \in (0, 1- \varepsilon_n).
$$
\end{lemma}

\noindent \textbf{Proof.}  We have that  for some $\varepsilon_n = o(\tau_n)$, $\P_n [ \P_n\{ |X_n - Y_n| >\xi_n \mid \D_n\} \leq \varepsilon_n] \to 1$, that is,
there is a set $\Omega_n$  such that $\P_n(\Omega_n) \to 1$
such that $P_n\{ |X_n - Y_n| >\xi_n \mid \D_n\} \leq \varepsilon_n$ for all $\D_n \in \Omega_n$. So, for all $\D_n \in \Omega_n$
$$
F_{X_n}(x \mid \D_n)+ \varepsilon_n \geq F_{Y_n+ \xi_n}(x \mid \D_n) \text { and } F _{Y_n}(x\mid \D_n) + \varepsilon_n \geq F_{X_n+ \xi_n}(x\mid \D_n), \forall x \in \Bbb{R},
$$
which implies the inequality stated in the lemma, by definition of the conditional quantile function and equivariance of quantiles to location shifts. \qed

\subsection{Proof of Lemma \ref{lemma: concentrate on balls}.}(Concentration of Inference on $\Vn$.)   Step 1. Letting
 \ba
&&\A_n :=   \sup_{v\in \Vn} Z_n(v),   \ \  \B_n := \sup_{v\in \V} Z_n(v), \ \ R_n := \(\sup_{v \in \V} |Z_n(v)| + \kappa_n\) \sup_{v \in \V} \left | \frac{\sigma_n(v)}{s_n(v)} -1 \right |, \\
&& \A^*_n :=   \sup_{v\in \Vn} Z^*_n(v),   \ \  \B^*_n := \sup_{v\in \V} Z^*_n(v), \ \ R^*_n := \(\sup_{v \in \V} |Z^*_n(v)| + \kappa_n\) \sup_{v \in \V} \left | \frac{\sigma_n(v)}{s_n(v)} -1 \right |,
\ea
we obtain
 \begin{eqnarray*}
& & \sup_{v \in \mathcal{V}} \frac{\theta_{n0} - \widehat \theta_n(v)}{s_n(v)} =
\sup_{v \in \mathcal{V}} \left\{ \frac{\theta_{n0} - \theta_n(v)}{s_n(v)} + Z_n(v) \frac{\sigma_n(v)}{s_n(v)} \right\} \\
& & =  \sup_{v\in \Vn} \left\{ \frac{(\theta_{n0} - \theta_n(v))}{s_n(v)} + Z_n(v)  \frac{ \sigma_n(v)}{s_n(v)}\right\}  \vee
  \sup_{v\not \in \Vn} \left\{ \frac{(\theta_{n0} - \theta_n(v))}{s_n(v)} + Z_n(v)  \frac{ \sigma_n(v)}{s_n(v)}\right\}  \\
& & \leq_{(1)}  \sup_{v\in \Vn} \left\{  Z_n(v)  \frac{ \sigma_n(v)}{s_n(v)} \right\}  \vee
  \sup_{v\not \in \Vn} \left\{ \frac{- \kappa_n \sigma_n(v)}{s_n(v)} + Z_n(v)  \frac{ \sigma_n(v)}{s_n(v)}\right\}  \\
&  & \leq   \A_n \vee (B_n - \kappa_n) + 2R_n \leq_{(2)}    \A^*_n \vee (B^*_n - \kappa_n) + 2R^*_n +  o_{P_n}(\delta_n), \ \
 \end{eqnarray*}
where in (1) we used that $\theta_n(v) \geq \theta_{n0}$ and $\theta_{n0} - \theta_n(v) \leq - \kappa_n \sigma_n(v)$ outside $\Vn$,
and in (2) we used C.2.  Next, since we assumed in the statement of the lemma that
$\kappa_n \lesssim \bar a_n + \ell\ell_n$, and by C.4:
$
 R^*_n = O_{\Pn}(\bar a_n + \bar a_n + \ell\ell_n) o_{\Pn}(\delta_n/(\bar a_n + \ell\ell_n))= o_{\Pn}(\delta_n).
$
Therefore, there is a deterministic term $ o(\delta_n)$  such that $\Pn(2R^*_n + o_{P_n}(\delta_n) >  o(\delta_n)) = o(1)$.\footnote{
Throughout the paper we use the elementary fact: If $X_n = o_{\P_n}(\Delta_n)$, for some $\Delta_n \searrow 0$, then there is $o(\Delta_n)$ term
such that $\P_n\{ |X_n|> o(\Delta_n)  \} \to 0$.}

Hence uniformly in $ x \in [0,\infty)$
 \begin{eqnarray*}
&& \di \Pn\left( \sup_{v \in \mathcal{V}} \frac{(\theta_{n0} - \widehat \theta_n(v))}{s_n(v)}> x\right) \leq \Pn(   \A^*_n +  o(\delta_n)
 >x) + \Pn(\B^*_n - \kappa_n +  o(\delta_n) > 0)  + o(1) \\
&& \ \ \leq   \Pn(   \A^*_n
 >x) + \Pn(\B^*_n - \kappa_n > 0)  + o(1) \leq \Pn(   \A^*_n > x) + (1-\gamma_n') + o(1),
\end{eqnarray*}
where the last two inequalities follow by Corollary \ref{cor: anti} and by $\kappa_n = Q_{\gamma_n'}(\B^*_n)$.

Step 2. To complete the proof, we must show that there is $\gamma_n' \nearrow 1$
that obeys the stated condition.  Let $1-\gamma_n' \searrow
0$ such that $1- \gamma_n' \geq C/\ell_n$.  It suffices to show that
\begin{eqnarray}
 \label{ineq: a1}
\kappa_n  \leq \( \bar a_n + \frac{c(\gamma_n')}{\bar a_n}\) \leq  \( \bar a_n + \frac{\eta \ell\ell_n + \eta \log C^{-1}}{\bar a_n}\),
\end{eqnarray}
where $c(\gamma_n') = Q_{\gamma_n'}(\mathcal{E})$.
To show the first inequality in (\ref{ineq: a1}) note
 \begin{eqnarray*}
  \P_n\(  \sup_{v \in \V} Z_n^*(v)  \leq (\bar a_n + c(\gamma'_n)/\bar a_n)  \) &=_{(1)}& \P_n \( \mathcal{E}_n(\V) \leq c(\gamma'_n) \)  \geq_{(2)} \P_n \( \mathcal{E} \leq c(\gamma'_n)\)  = \gamma_n',
  \end{eqnarray*}
where (1) holds by
definition of $\mathcal{E}_n(\V)$ and (2) by C.3. To show
 the second inequality in (\ref{ineq: a1}) note that by C.3 $\Pr\( \mathcal{E} > t \) \leq \exp \(- t \eta^{-1}\)$, for some constant $\eta>0$, so that $c(\gamma_n') \leq - \eta \log (1-\gamma_n') \leq \eta \ell\ell_n + \eta \log C^{-1}$. \qed

\subsection{Proof of Theorem \ref{theorem: inference analytical} (Analytical Construction).} Part 1.(Level) Observe that
\begin{align*}
\di & \Pn \left (\theta_{n0} \leq \widehat \theta_{n0}(p) \right)= \Pn\left ( \sup_{v \in \V} \frac{\theta_{n0} - \widehat \theta_n(v)}{s_n(v)}  \leq k_{n, \widehat V_n}(p)\right ) \\
& \geq_{(1)} \Pn\( \sup_{v \in \V} \frac{\theta_{n0} - \widehat \theta_n(v)}{s_n(v)}  \leq k_{n,  \Vn}(p) \)  - \P_n\( \Vn \not \subseteq \widehat V_n \) \\
& \geq_{(2)} \Pn\( \sup_{v \in \Vn} Z_n^*(v) \leq k_{n,  \Vn}(p) \)  - o(1) \\
& = \Pn\( \mathcal{E}_n(\Vn) \leq c(p)  \)  - o(1) \geq_{(3)} \Pn\( \mathcal{E} \leq c(p) \)  - o(1) =_{(4)} p -o(1),
\end{align*}
where (1) follows by monotonicity of $\textsf{V} \mapsto k_{n, \textsf{V}}(p)= a_n(\textsf{V}) + c(p)/a_n(\textsf{V})$ for large $n$ holding by construction, (2) holds by Lemma \ref{lemma: concentrate on balls}, by
$\P_n\( \Vn \not \subseteq \widehat V_n \)=o(1)$ holding by  Lemma \ref{lemma: estimation of V}, and also by the fact
that the critical value $k_{n,  \Vn}(p) \geq 0$ is non-stochastic, and (3) and (4) hold by C.3.

Part 2.(Estimation Risk) We have that under $\P_n$
\begin{align*}
\di &\left | \widehat \theta_{n0}(p)-\theta _{n0}\right |
 = \left | \inf_{v\in\mathcal{V}} \left [\widehat{\theta}_n( v)
+ k_{n,\widehat V_n}(p)s_n( v) \right ] -\theta_{n0}\right | \\
& = \left | \sup_{v\in\mathcal{V}}
\( \left [\frac{\theta_{n0} - \widehat{\theta}_n(v)}{s_n(v)}
- k_{n,\widehat V_n}(p)\right ]\sigma_n(v) \frac{s_n(v)}{\sigma_n(v)} \) \right|   \notag \\
& \leq_{(1)}  \left ( \left| \sup_{v\in\mathcal{V}} \frac{\theta_{n0} - \widehat{\theta}_n(v)}{s_n(v)}\right|
+ k_{n,\widehat V_n} (p) \right) \bar \sigma_n \left(1 + o_{\Pn}\(\frac{\delta_n}{\bar a_n + \ell\ell_n}\)\right)\\
& \leq_{(2)}  \left ( \left| \sup_{v\in\mathcal{V}} \frac{\theta_{n0} - \widehat{\theta}_n(v)}{\sigma_n(v)}\right|
+ k_{n,\widehat V_n} (p) \right) \bar \sigma_n \left(1 + o_{\Pn}\(\frac{\delta_n}{\bar a_n + \ell\ell_n}\)\right)^2\\
& \leq_{(3)}  \(  \sup_{v\in \Vn} \left | Z^*_n(v) \right | + o_{\P_n}(\delta_n)
+  k_{n,\widehat V_n} (p)\) \bar \sigma_n \left(1 + o_{\Pn}\(\frac{\delta_n}{\bar a_n + \ell\ell_n}\)\right)^2 \text{ wp $\to 1$ } \\
& \leq_{(4)}  \(  \sup_{v\in \Vn} \left | Z^*_n(v) \right | + o_{\P_n}(\delta_n)
+ k_{n,\barVn}(p)\) \bar \sigma_n \left(1 + o_{\Pn}\(\frac{\delta_n}{\bar a_n + \ell\ell_n}\)\right)^2 \text{ wp $\to 1$ } \\
& \leq_{(5)}  3 \left |  {a}_n(\barVn) + \frac{O_{\Pn}(1)}{a_n(\barVn)} + o_{\P_n}(\delta_n) \right |\bar\sigma_n  \left(1 + o_{\Pn}\(\frac{\delta_n}{\bar a_n + \ell\ell_n}\)\right)^2 \text{ wp $\to 1$} \\
& \leq_{(6)} 4  \left |  {a}_n(\barVn) + \frac{O_{\Pn}(1)}{a_n(\barVn)}  \right |\bar\sigma_n  \text{ wp $\to 1$, }
\end{align*}
where (1) holds by C.4 and the triangle inequality; (2) holds by C.4;
(3) follows because  wp $\to 1$, for some $o(\delta_n)$
$$
\sup_{v \in V_0 } Z^*_n(v) - o(\delta_n) \leq_{(a)} \sup_{v \in V_0 } Z_n(v) \leq_{(b)} \sup_{v\in\mathcal{V}} \frac{\theta_{n0} - \widehat{\theta}_n(v)}{\sigma_n(v)} \leq_{(c)} \( \sup_{v\in \Vn}  Z^*_n(v)\) \vee 0 \ + o(\delta_n),
$$
where (a) is by C.2, (b) by definition of $Z_n$, while (c) by the proof of Lemma \ref{lemma: concentrate on balls},
so that wp $\to 1$
$$
\left| \sup_{v\in\mathcal{V}} \frac{\theta_{n0} - \widehat{\theta}_n(v)}{\sigma_n(v)} \right| \leq  \sup_{v\in \Vn} \left | Z^*_n(v) \right | + o_{\P_n}(\delta_n);
$$
(4) follows by  Lemma \ref{lemma: estimation of V}
 which implies  $\Vn \subseteq \widehat V_n \subseteq  \barVn$  wp $\to $ 1, so that
$$\di k_{n,\widehat V_n}(p) \leq k_{n,\barVn}(p) =  {a}_n(\barVn) + \frac{c(p)}{ a_n(\barVn)},$$
Condition C.3 gives (5).  Inequality (6) follows because  $a_n(\barVn) \geq 1$,
$\bar a_n \geq 1$, and $\delta_n = o(1)$;  this inequality is the claim that we needed to prove.

Part 3. We have that
\begin{align*}
\di \theta _{na}-\theta _{n0}  &\geq   4 \bar \sigma_n \left ( {a}_n(\barVn)+ \frac{\mu_n}{{a}_{n}(\barVn)}\right )  >   \widehat{\theta}_{n0}(p) -\theta_{n0} \text{ wp $\to$ 1},
\end{align*}
with the last inequality occurring by Part 2 since $\mu_n \to_{\Pn} \infty$.  \qed

\subsection{Proof of Theorem \ref{theorem: inference1} (Simulation Construction).} Part 1. (Level Consistency) Let us compare  critical values
$$
 \di k_{n,\Vn}(p) = Q_p\(\sup_{ v \in \Vn} Z_n^\star(v) \mid \D_n\)  \text{ and }
 \kappa_{n,\Vn}(p) = Q_p\(\sup_{ v \in \Vn} \bar Z^*_n(v)\).
$$
The former is data-dependent while the latter is deterministic.
Note that $k_{n,\Vn}(p) \geq 0 $ by C.2(b) for $p \geq 1/2$.  By C.2,   for some deterministic term $o(\delta_n)$,
$$
\P_n\(|\sup_{ v \in \Vn} Z^\star_n(v) - \sup_{v \in \Vn} \bar Z^*_n(v)| > o(\delta_n) \mid \mathcal{D}_n\) = o_{\P_n}(1),
$$
which implies by Lemma \ref{lemma: quantiles are close} that for some $\varepsilon_n \searrow 0$, wp $\to 1$
\begin{equation}\label{eq: quantile comparison}
k_{n,\Vn}(p) \geq   (\kappa_{n,\Vn}(p-\varepsilon_n) - o(\delta_n))_+    \ \ \text{ for all } \ \ p\in [1/2, 1- \varepsilon_n).
\end{equation}
 The result follows analogously to the proof in Part 1 of Theorem \ref{theorem: inference analytical}, namely:
\begin{align*}
\di & \Pn \left (\theta_{n0} \leq \widehat \theta_{n0}(p)\right)= \Pn\left ( \sup_{v \in \V} \frac{\theta_{n0} - \widehat \theta_n(v)}{s_n(v)}  \leq k_{n, \widehat V_n}(p)\right ) \\
& \geq_{(1)}  \Pn\left ( \sup_{v \in \V} \frac{\theta_{n0} - \widehat \theta_n(v)}{s_n(v)}  \leq k_{n,  \Vn}(p) \right ) - o(1)\\
& \geq_{(2)} \Pn\( \sup_{v \in \V} \frac{\theta_{n0} - \widehat \theta_n(v)}{s_n(v)}  \leq  (\kappa_{n,  \Vn}(p-\varepsilon_n)- o(\delta_n))_+  \)  -o(1) \\
& \geq_{(3)} \Pn\( \sup_{v \in \Vn} Z_n^*(v) \leq (\kappa_{n,  \Vn}(p-\varepsilon_n)  - o(\delta_n))_+  \)  - o(1)  \\
& \geq \Pn\( \sup_{v \in \Vn} Z_n^*(v) \leq \kappa_{n,  \Vn}(p-\varepsilon_n)  - o(\delta_n)  \)  - o(1) \geq_{(4)}  p-\varepsilon_n -o(1) = p- o(1),
\end{align*}
where (1) follows by monotonicity of $\textsf{V} \mapsto k_{n, \textsf{V}}(p)$ and by
$\P_n\( \Vn \not \subseteq \widehat V_n \)=o(1)$ shown in  Lemma \ref{lemma: estimation of V},
(2) holds by the comparison of quantiles in equation (\ref{eq: quantile comparison}), (3) by  Lemma \ref{lemma: concentrate on balls}.
(4) holds by anti-concentration Corollary \ref{cor: anti}.

Parts 2 \& 3.(Estimation Risk and Power) By Lemma \ref{lemma: estimation of V} wp $\to $ 1, $ \widehat V_n \subseteq  \barVn$, so that
$
 k_{n,\widehat V_n}(p) \leq  k_{n,\barVn}(p).
$
By C.2 for some deterministic term $o(\delta_n)$,
\begin{equation}\label{eq: close stars}
\P_n\(|\sup_{ v \in \barVn} Z^\star_n(v) - \sup_{v \in \barVn} \bar Z^*_n(v)| > o(\delta_n) | \mathcal{D}_n\)  = o_{\P_n}(1/\ell_n),
 \end{equation}
which implies by Lemma \ref{lemma: quantiles are close} that for some $\varepsilon_n \searrow 0$, wp $\to 1$, for all $p\in(\varepsilon_n, 1- \varepsilon_n)$
\begin{equation}\label{eq: quantile comparison-re}
 k_{n,\barVn}(p) \leq \kappa_{n,\barVn}(p+\varepsilon_n)  + o(\delta_n)
\end{equation}
where the terms $o(\delta_n)$ are different in different places.  By C.3, for any fixed $p \in (0,1)$,
$$    \kappa_{\barVn}(p+\varepsilon_n) \leq {a}_n(\barVn) + c(p+\varepsilon_n)/ a_n(\barVn) =  {a}_n(\barVn) + O(1)/ a_n(\barVn).$$
Thus, combining inequalities above and $o(\delta_n)   = o(\bar a_n^{-1})= o(a^{-1}_n(\barVn))$ by C.2,  wp $\to 1$,
$$    k_{n,\widehat V_n}(p)  \leq {a}_n(\barVn) + O(1)/ a_n(\barVn).$$
Now  Parts 2 and 3 follow as in the Proof of Parts 2 and 3 of  Theorem \ref{theorem: inference analytical} using
this bound on the simulated critical value instead of the bound on the analytical critical value. \qed

\subsection{Proof of Lemma \ref{lemma: concentrate on it} (Concentration on $V_0$).}
By Conditions S and V. $\text{wp} \rightarrow 1$,
\begin{align} \label{eq: sharp 1}
|\sup_{ v \in \Vn} Z^*_n(v)
 - \sup_{v \in V_{0}} Z_n^*(v)| \leq \sup_{\|v - v'\| \leq r_n} |Z_n^*(v) - Z_n^*(v')| =   o_{\P_n}(\bar a_n^{-1}).
\end{align}
Conclude similarly to the proof of Lemma \ref{lemma: concentrate on balls}, using anti-concentration Corollary \ref{cor: anti}, that
\begin{align*}
\P_n \( \sup_{v \in \V} \frac{\theta_{n0} - \widehat \theta_n(v)}{s_n(v)} \leq x \)
    &\geq \P_n \(  \sup_{v \in V_{0}} Z_n^*(v) + o(\bar a_n^{-1}) \leq x  \) - o(1)
     \geq \P_n \(  \sup_{v \in V_{0}} Z_n^*(v)  \leq x  \)- o(1)
    \end{align*}
This gives a lower bound.  Similarly, using C.3 and C.4 and  anti-concentration Corollary \ref{cor: anti}
\begin{align*}
& \P_n \( \sup_{v \in \V} \frac{\theta_{n0} - \widehat \theta_n(v)}{s_n(v)} \leq x \)
    \leq \P_n \(  \sup_{v \in V_{0}} Z_n (v) \frac{\sigma_n(v)} {s_n(v)}\leq x  \) \\
    &\leq \P_n \(  \sup_{v \in V_{0}} Z_n^*(v) - o(\delta_n) \leq x  \) + o(1) \leq \P_n \(  \sup_{v \in V_{0}} Z_n^*(v) \leq x  \) + o(1)
    \end{align*}
   where $o(\cdot)$ terms above are different in different places, and
   the first inequality follows from
  $$
  \sup_{v \in \V} \frac{\theta_{n0} - \widehat \theta_n(v)}{s_n(v)} \geq   \sup_{v \in V_0} \frac{\theta_{n0} - \widehat \theta_n(v)}{s_n(v)}  =  \sup_{v \in V_{0}} Z_n (v) \frac{\sigma_n(v)} {s_n(v)}.
  $$
    This gives the upper bound.  \qed

\subsection{Proof of Theorem \ref{theorem: sharp inference} (When Simulation Inference Becomes Sharp)} Part 1.
 (Size)
 By Lemma \ref{lemma: estimation of V} wp $\to $ 1, $ \widehat V_n \subseteq  \barVn$, so that
$ k_{n,\widehat V_n}(p) \leq  k_{n,\barVn}(p)$ wp $\to $ 1.
So let us compare  critical values
$$
 \di k_{n,\barVn}(p) = Q_p\(\sup_{ v \in \barVn} Z_n^\star(v) \mid \D_n\)  \text{ and }
 \kappa_{n,V_{0}}(p) = Q_p\(\sup_{ v \in V_{0}} \bar Z^*_n(v)\).
$$
The former is data-dependent while the latter is deterministic.
Recall that by C.2 wp $\to 1$ we have (\ref{eq: close stars}).
By Condition V $d_H(\barVn, V_0) \leq r_n$, and so by S, we have for some  $o(\bar a_n^{-1})$,
$$
\P_n\( | \sup_{ v \in \barVn} \bar Z^*_n(v)  - \sup_{v \in V_0} \bar Z^*_n(v) | >  o(\bar a_n^{-1}) \mid \D_n \) = o_{\P_n}(1).
$$
Combining (\ref{eq: close stars}) and this relation, we obtain that for some $o(\bar a_n^{-1})$ term,
$$
\P_n\( | \sup_{ v \in \barVn} Z^\star_n(v)  - \sup_{v \in V_0} \bar Z^*_n(v) | >  o(\bar a_n^{-1}) \mid \D_n \) = o_{\P_n}(1).
$$
 This  implies by Lemma \ref{lemma: quantiles are close} that for some $\varepsilon_n \searrow 0$, and any $p \in (\varepsilon_n,1-\varepsilon_n)$, wp $\to 1$,
\begin{equation}\label{eq: quantile comparison 4}
 k_{n,\widehat V_n}(p)  \leq k_{n,\barVn}(p) \leq   \kappa_{n,V_0}(p+\varepsilon_n) + o(\bar a_n^{-1}) .
\end{equation}
Hence, for any fixed $p$,
\begin{align*}
\di & \Pn \left (\theta_{n0} \leq \widehat \theta_{n0}(p)\right)= \Pn\left ( \sup_{v \in \V} \frac{\theta_{n0} - \widehat \theta_n(v)}{s_n(v)}  \leq k_{n, \widehat V_n}(p)\right ) \\
& \leq_{(1)}  \Pn\left ( \sup_{v \in \V} \frac{\theta_{n0} - \widehat \theta_n(v)}{s_n(v)}  \leq \kappa_{n, V_0}(p+ \varepsilon_n) + o(\bar a_n^{-1}) \right ) + o(1) \\
& \leq_{(2)} \Pn\( \sup_{v \in V_{0}} Z_n^*(v) \leq \kappa_{n, V_0}(p+ \varepsilon_n) + o(\bar a_n^{-1})  \)  +o(1)  \leq_{(3)}  p+\varepsilon_n +o(1) = p + o(1),
\end{align*}
where (1) is by the quantile comparison (\ref{eq: quantile comparison 4}), (2) is by Lemma \ref{lemma: concentrate on it},
and (3) is by anti-concentration Corollary \ref{cor: anti}. Combining this with the lower bound of Theorem \ref{theorem: inference1}, we have the result.

Parts 2 \& 3.(Estimation Risk and Power)
We have that  by C.3
$$  \kappa_{n,V_{0}}(p+\varepsilon_n) \leq a_n(V_{0}) + c(p+\varepsilon_n)/ a_n(V_{0}) =  {a}_n(V_{0}) + O(1)/ a_n(V_{0}).$$
Hence combining this with equation (\ref{eq: quantile comparison 4}) we have wp $\to 1$
$$
 k_{n,\widehat V_n}(p) \leq {a}_n(V_{0}) + O(1)/ a_n(V_{0}) + o(\bar a_n^{-1}) = {a}_n(V_{0}) + O(1)/ a_n(V_{0}).
$$
 Then Parts 2 and 3 follow identically to the Proof of Parts 2 and 3 of  Theorem \ref{theorem: inference analytical} using this bound on the simulated critical value instead of the bound on the analytical critical value. \qed

\vspace{-.1in}
\section{Proofs for Section \ref{sec:leading-cases}}\label{sec:proofs-maximal-strong}

\subsection{Tools and Auxiliary Lemmas}\label{sec:proofs-max-strong-auxiliary}
We shall heavily rely on the Talagrand-Samorodnitsky Inequality, which was obtained by Talagrand
sharpening earlier results by Samorodnitsky. Here it is restated
from \cite{VanDerVaart/Wellner:96} Proposition A.2.7, page 442:

\noindent\textbf{Talagrand-Samorodnitsky Inequality:}  Let $X$ be a separable zero-mean Gaussian process indexed by
a set $T$. Suppose that for some $\Gamma>\sigma(X) = \sup_{t \in
T}\sigma(X_t)$, $0< \epsilon_0\leq \sigma(X)$,
$$ N(\varepsilon, T,\rho)  \leq \left( \frac{\Gamma}{\varepsilon}\right)^\nu, \ \mbox{for} \ 0 <\varepsilon < \epsilon_0,$$ where $N(\varepsilon,T,\rho)$ is the covering number of $T$ by $\varepsilon$-balls w.r.t. the standard deviation metric $\rho(t,t') =\sigma(X_t - X_{t'})$. Then there exists a universal constant $D$ such that for every $\lambda \geq \sigma^2(X)(1+\sqrt{\nu})/\epsilon_0$ we have
\begin{equation}\label{eq:TalagrandInequality}
P\left( \sup_{t \in T} X_t > \lambda \right) \leq \left(\frac{D\Gamma\lambda}{\sqrt{\nu}\sigma^2(X)} \right)^v (1- \Phi(\lambda /\sigma(X)))\text{,}
\end{equation}
where $\Phi(\cdot)$ denotes the standard normal cumulative distribution function.

The following lemma is an  application of this inequality that we use:

\begin{lemma}[\textbf{Concentration Inequality via Talagrand-Samorodnitsky}]\label{lemma: key concentration} Let $Z_n$ be a separable zero-mean Gaussian process indexed by
a set $\textsf{V}$ such that $\sup_{v \in \textsf{V}}\sigma(Z_n(v))=1$. Suppose that for some $ \Gamma_n(\textsf{V})>1$, and $d \geq 1$
$$ N(\varepsilon, \textsf{V},\rho)  \leq \left( \frac{\Gamma_n(\textsf{V})}{\varepsilon}\right)^d, \  \ \mbox{for} \ \  0 <\varepsilon < 1,$$ where $N(\varepsilon,\textsf{V},\rho)$ is the covering number of $\textsf{V}$ by $\varepsilon$-balls w.r.t. the standard deviation metric $\rho(v,v') =\sigma(Z_n(v) - Z_n(v'))$. Then for
$$
a_n(\textsf{V}) = (2\sqrt{\log L_n(\textsf{V})  }) \vee (1 +\sqrt{d}), \ \ L_n(\textsf{V}):= C'_n \(\frac{\Gamma_n(\textsf{V})}{\sqrt{d}}  \)^d,
$$
where for $D$ denoting Talagrand's constant in (\ref{eq:TalagrandInequality}), and $C'_n $ such that
$$
C'_n \geq D^d C_d  \frac{1}{\sqrt{2\pi}},  \ \  C_d := \max_{\lambda \geq 0} \lambda^{d-1} e^{- \lambda^2/4},
$$
we have for $z \geq 0$
$$ P\left( a_n(\textsf{V})\(\sup_{v \in \textsf{V}} Z_n(v) - a_n(\textsf{V})\) > z \right) \leq \exp\( - \frac{z}{2} - \frac{z^2}{4a^2_n
(\textsf{V})}\)\leq \exp(-z/2).$$
\end{lemma}

\noindent\textbf{Proof.} We apply the TS inequality by setting $t=v$, $X =Z$, $\sigma(X)=1$, $\epsilon_0 =1$, $\nu=d$, with
$ \lambda  \geq (1 + \sqrt{d})$, so that
\begin{eqnarray*}
\di P\( \sup_{v \in \textsf{V}} Z_n(v) > \lambda \) & \leq &   \(\frac{D \Gamma_n(\textsf{V}) \lambda}{\sqrt{d}}  \)^d (1 - \Phi(\lambda)) \\
 & \leq &  \(\frac{D \Gamma_n(\textsf{V}) \lambda}{\sqrt{d}}  \)^d \frac{1}{\sqrt{2\pi}} \frac{1}{\lambda} e^{-\lambda^2/2}
 \leq   L_n(\textsf{V}) e^{- \lambda^2/4}.
\end{eqnarray*}
Setting for $z \geq 0$,
$
\lambda = \frac{z}{a_n(\textsf{V})} + a_n(\textsf{V}) \geq (1 + \sqrt{d}),
$
we obtain
$$
L_n(\textsf{V}) \exp\(-\frac{\lambda^2}{4}\)  \leq  \exp\( - \frac{z}{2} - \frac{z^2}{4a^2_n(\textsf{V})}\).
$$
\qed

The following lemma is an immediate consequence of Corollary 2.2.8 of \cite{VanDerVaart/Wellner:96}.

\begin{lemma}[\textbf{Maximal Inequality for a Gaussian Process}]\label{lemma: maximal ineq gaussian} Let $X$ be a separable zero-mean Gaussian process indexed by
a set $T$.
 Then for every $\delta>0$
$$
E \sup_{\rho(s,t) \leq \delta } | X_s - X_{t}| \lesssim \int_0^{\delta}\sqrt{ \log N(\varepsilon, T, \rho) } d \varepsilon,
\ \ E \sup_{t \in T} | X_t| \lesssim \sigma(X) + \int_0^{2\sigma(X)}\sqrt{ \log N(\varepsilon, T, \rho) } d \varepsilon,
$$
where $\sigma(X) = \sup_{t \in
T}\sigma(X_t)$, and $N(\varepsilon,T, \rho)$ is the covering number of $T$ with respect
to the semi-metric $\rho(s,t) = \sigma(X_s - X_t)$.
\end{lemma}

\noindent\textbf{Proof.}  The first conclusion follows from Corollary 2.2.8 of \cite{VanDerVaart/Wellner:96} since covering and packing numbers are related by $N(\varepsilon, T, \rho) \leq D(\varepsilon, T, \rho) \leq N(\varepsilon/2, T, \rho)$.
The second conclusion follows from the special case of the first conclusion: for any $t_0 \in T$,
$E \sup_{t \in T} | X_t|
\lesssim E  | X_{t_0}| + \int_0^{\text{diam}(T)} \sqrt{ \log N(\varepsilon, T, \rho) } d \varepsilon  \leq \sigma(X) + \int_0^{2\sigma(X)}\sqrt{ \log N(\varepsilon, T, \rho) } d \varepsilon.$\qed

\subsection{Proof of Lemma \ref{lemma:series}}

Step 1. Verification of C.1.  This condition holds by inspection, in view of continuity of
$v \mapsto p_{n}(v)$ and by $\Omega_n$ and $\widehat \Omega_n$ being positive definite.

Step 2. Verification of C.2  Set $\delta_n = 1/\log n$. Condition NS.1 directly assumes C.2(a).

In order to show C.2(b), we employ the maximal inequality stated in Lemma \ref{lemma: maximal ineq gaussian}. Set $
X_t = Z^*_n (v) - Z_n^\star (v), \ t = v, \ T = \V$
and note that for some absolute constant $C$, conditional on $\D_n$,
$$
N(\varepsilon, T,\rho) \leq \left( \frac{1+ C \Upsilon_n \text{diam}(T)}{\varepsilon}\right)^d, \ \ 0 < \varepsilon< 1,
$$
since
$
\sigma(X_t - X_{t'}) \lesssim \Upsilon_n \| t - t'\|,  \ \ T \subset \Bbb{R}^d,
$
where $\Upsilon_n$ is an upper bound on the Lipschitz constant of the function
$$
v \mapsto  \frac{p_{n}(v)' \Omega^{1/2}_n}{\|p_{n}(v)' \Omega^{1/2}_n\|}
- \frac{p_{n}(v)'\widehat \Omega^{1/2}_n}{\|p_{n}(v)'\widehat \Omega^{1/2}_n\|}  ,
$$
where $\text{diam}(T)$ is the diameter of set $T$ under the Euclidian metric.  Using inequality \eqref{eq: eleimentary}
we can bound
$$
\Upsilon_n \leq 2 L_n \frac{\lambda_{\max} (\Omega_n^{1/2})}{\lambda_{\min} (\Omega_n^{1/2})}   +  2 L_n   \frac{\lambda_{\max} (\widehat \Omega_n^{1/2})}{\lambda_{\min} (\widehat \Omega_n^{1/2})} = O_{\Pn}(L_n),
$$
where $L_n$ is the constant defined in NS.1, and by assumption $\log L_n \lesssim \log n$. Here we
use the fact the eigenvalues of $\Omega_n$ and $\widehat \Omega_n$ are bounded away from zero and from above
by NS.1 and NS.2. Therefore,
$
\log N(\varepsilon, T,\rho) \lesssim \log n  + \log(1/\varepsilon).
$

 Using (\ref{eq: eleimentary}) again,
 \begin{eqnarray*}
 \sigma(X) && \lesssim \sup_{v \in \V} \left \| \frac{p_{n}(v)' \Omega^{1/2}_n}{\|p_{n}(v)' \Omega^{1/2}_n\|}
- \frac{p_{n}(v)'\widehat \Omega^{1/2}_n}{\|p_{n}(v)' \widehat \Omega^{1/2}_n\|}  \right\| \\
&&  \leq    \sup_{v \in \V} 2  \frac{\|p_{n}(v)'(\widehat \Omega_n^{1/2} - \Omega_n^{1/2})\|}{\|p_{n}(v)'\Omega_n^{1/2}\|}   \leq   \sup_{v \in \V}  2  \frac{\|p_{n}(v)'\Omega_n^{1/2}( \Omega_n^{-1/2} \widehat\Omega_n^{1/2}-I)\|}{\|p_{n}(v)'\Omega_n^{1/2}\|}   \\
& &  \leq  \| \Omega_n^{-1/2} \widehat \Omega_n^{1/2}-I\| \leq \|\Omega_n^{-1/2}\| \| \widehat \Omega_n^{1/2}-\Omega_n^{1/2}\|= O_{\P_n}(n^{-b})
\end{eqnarray*}
for some constant $b>0$, where we have used that the eigenvalues of $\Omega_n$ and $\widehat \Omega_n$ are bounded away from zero and from above
under NS.1 and NS.2, and the assumption $\|\widehat \Omega_n - \Omega_n\| = O_{\Pn}(n^{-b})$.
Hence
$$
E \left ( \sup_{t \in T} | X_t| \mid \mathcal{D}_n \right) \lesssim \sigma(X) +  \int_0^{2 \sigma(X)} \sqrt{ \log (n/\varepsilon)} d \varepsilon  = O_{\P_n}( n^{-b} \sqrt{ \log n})  .
$$
Hence for each $C>0$
$$
\Pn \( \sup_{v \in \V} | Z^*_n(v) - Z^\star_n(v)| > C\delta_n \mid \D_n \) \lesssim \frac{1}{C \delta_n} O_{\P_n}( n^{-b}\sqrt{\log n}) = o_{\P_n}(1/\ell_n),
$$
which verifies C.2(b).

Step 3. Verification of C.3.  We shall employ Lemma \ref{lemma: key concentration}, which
has the required notation in place. We only need to compute an upper bound
on the covering numbers $N(\varepsilon,\textsf{V}, \rho)$ for the process $Z_n^{\ast}$. We have
that
\begin{eqnarray*}
 \sigma(Z^*_n(v) - Z^*_n(\tilde v) )  & \leq &    \left\| \frac{p_{n}(v)'\Omega^{1/2}_n}{\|p_{n}(v)'\Omega^{1/2}_n\|}-   \frac{p_{n}(\tilde v)'\Omega^{1/2}_n}{\|p_{n}(\tilde v)'\Omega^{1/2}_n\|} \right\|  \leq   2 \left\| \frac{( p_{n}(v) - p_{n}(\tilde v) )'\Omega^{1/2}_n}{\|p_{n}(v)'\Omega^{1/2}_n\|}\right\|   \\
& \leq &   2 L_n \frac{\lambda_{\max}(\Omega_n^{1/2})}{\lambda_{\min}(\Omega_n^{1/2})} \| v - \tilde v\| \leq  C L_n \| v - \tilde v\|,
 \end{eqnarray*}
 where $C$ is some constant that does not depend on $n$, by the eigenvalues of $\Omega_n$ bounded
away from zero and from above.  Hence it follows that
$$
N(\varepsilon, \textsf{V},\rho) \leq \left( \frac{1+  C L_n  \text{diam}(\textsf{V})}{\varepsilon}\right)^d, \  0 < \varepsilon< 1,
$$
where the diameter of $\textsf{V}$ is measured by the Euclidian metric.
Condition C.3 now follows by Lemma \ref{lemma: key concentration}, with
$
a_n(\textsf{V}) = (2\sqrt{\log L_n(\textsf{V})  }) \vee (1 +\sqrt{d}),  \ \ L_n(\textsf{V}) =  C'  \( 1 + C L_n \text{diam}(\textsf{V}) \)^d.
$
where $C'$ is some positive  constant.

Step 4. Verification of C.4.  Under Condition NS, we have that
$$
a_n(\textsf{V}) \leq \bar a_n:= a_n(\mathcal{V}) \lesssim \sqrt{ \log \ell_n  + \log n} \lesssim \sqrt{ \log n},
$$
so that C.4(a) follows if $\sqrt{ \log n } \sqrt{ \zeta_n^2/n } \to 0$.

To verify C.4(b) note that uniformly in $v \in \V$,
\begin{eqnarray*}
& & \left | \frac{\|p_{n}(v)'\widehat \Omega_n^{1/2}\|}{\|p_{n}(v)'\Omega_n^{1/2}\|} -1 \right | \leq
  \left | \frac{\|p_{n}(v)'\widehat \Omega_n^{1/2}\|- \|p_{n}(v)'\Omega_n^{1/2}\|}{\|p_{n}(v)'\Omega_n^{1/2}\|}  \right |\\
& &  \leq  \frac{\|p_{n}(v)'(\widehat \Omega_n^{1/2} - \Omega_n^{1/2})\|}{\|p_{n}(v)'\Omega_n^{1/2}\|}  \leq
 \frac{\|p_{n}(v)'\Omega^{1/2}(\Omega_n^{-1/2} \widehat\Omega_n^{1/2}-I)\|}{\|p_{n}(v)'\Omega_n^{1/2}\|}   \\
& & \leq  \| \Omega_n^{-1/2} \widehat\Omega_n^{1/2}-I\| \leq \|\Omega_n^{-1/2}\| \| \widehat \Omega_n^{1/2}-\Omega_n^{1/2}\| = o_{\P_n}(\delta_n/\bar a_n),
\end{eqnarray*}
by $ \| \widehat \Omega^{1/2} - \Omega^{1/2}_n\|=O_{\P_n}(n^{-b})$ and  $\|\Omega_n^{-1/2}\|$ bounded,
both implied by the assumptions.
 \qed

\subsection{Proof of Lemma \ref{lemma: series-v-est}}
To show claim (1), we need to establish that  for
$
\varphi_n = o(1) \cdot \(\frac{1}{L_n \sqrt{\log n}}\),
$
with any $o(1)$ term, we have that
$
\sup_{\|v - \tilde v \| \leq \varphi_n} | Z^*_n(v) - Z_n^*(\tilde v) | = o_{\Pn}(1).
$

Consider the stochastic process $X =\{Z^*_n(v), v \in \V\}$. We shall use the standard maximal inequality
stated in Lemma \ref{lemma: maximal ineq gaussian}. From the proof of Lemma \ref{lemma:series} we have
$\sigma(Z^*_n(v) - Z^*_n(\tilde v) ) \leq    C L_n \| v - \tilde v\|,$
where $C$ is some constant that does not depend on $n$, and
$
\log N(\varepsilon, \textsf{V},\rho) \lesssim \log n + \log (1/\varepsilon).
$
Since
$
\|v - \tilde v \| \leq \varphi_n \implies \sigma(Z^*_n(v) - Z^*_n(\tilde v) ) \leq  C   \frac{o(1)}{\sqrt{\log n}}
$
we have
$$
E \sup_{\|v - \tilde v \| \leq \varphi_n  } | X_v - X_{\tilde v}| \lesssim \int_0^{C   \frac{o(1)}{\sqrt{\log n}} }\sqrt{ \log (n/\varepsilon) } d \varepsilon
\lesssim  \frac{o(1)}{\sqrt{\log n}} \sqrt{ \log n} = o(1).
$$
Hence the conclusion follows from Markov's Inequality.

Under Condition V by Lemma \ref{lemma: estimation of V}
$
r_n \lesssim \(\sqrt{\log n \frac{\zeta_n^2}{ n}}\)^{1/\rho_n} c_n^{-1},
$
so $r_n = o(\varphi_n)$ if
\begin{equation}
\(\sqrt{\log n \frac{\zeta_n^2}{ n}}\)^{1/\rho_n} c_n^{-1} = o\(\frac{1}{L_n \sqrt{\log n}}\).
\end{equation}
Thus, Condition S holds. The remainder of the lemma follows by direct calculation. \qed

\section{Proofs for Section \ref{sec:strong:series:main-text}}\label{sec:proofs-strong-series}

\subsection{Proof of Theorem \ref{theorem:strong series} and Corollary \ref{corollary:strong}.} The first step of our proof uses Yurinskii's (1977) coupling. \nocite{Yurinskii:77} For completeness we now state the formal result from \cite{Pollard:02}, page 244.

\noindent\textbf{Yurinskii's Coupling}:  Consider a sufficiently rich probability space $(A, \mathcal{A}, \P). $\ Let $\xi_1,...,\xi_n$ be independent ${K_n}$-vectors with
$E \xi_i = 0$ for each $i$, and $\Delta := \sum_{i} E \|\xi_i\|^3$ finite. Let
$S = \xi_1 +...+ \xi_n$. For each $\delta>0$ there exists a random vector $T$ with $N(0, \text{var }(S))$ distribution such that
$$
\P\{ \| S- T\| > 3 \delta \} \leq C_0 B \left( 1  + \frac{|\log (1/B)|}{{K_n}}   \right) \text{ where }
B:= \Delta {K_n} \delta^{-3},
$$
for some universal constant $C_0$.

The proof has two steps: in the first, we couple the estimator $\sqrt{n}(\widehat \beta_n - \beta_n)$ with the normal vector;  in the second, we establish the strong approximation.

\textsc{Step 1.} In order to apply the coupling, consider
$$
 \sum_{i=1}^n \xi_i, \ \ \xi_i = u_{i,n}/\sqrt{n} \sim (0, I_{K_n}/n),
$$
Then we have that $\sum_{i=1}^n E\|\xi_i\|^3  = \Delta_n$. Therefore, by Yurinskii's coupling,
\begin{eqnarray*}
\P_n \left \{ \left \| \sum_{i=1}^n \xi_i - \mathcal{N}_n \right\|  \geq 3 \delta_n \right\} \to 0 & \; \text{ if $K_n \Delta_n/\delta_n^3  \to 0$}.
\end{eqnarray*}
Combining this with the assumption on the linearization error $r_n$, we obtain
\begin{eqnarray*}
\| \Omega_n^{-1/2} \sqrt{n} ( \widehat \beta_n - \beta_n) - \mathcal{N}_n \| & \leq &  \| \sum_{i=1}^n \xi_i - \mathcal{N}_n \|
  +   \|  \Omega_n^{-1/2} \sqrt{n} (\widehat \beta_n - \beta_n) - \sum_{i=1}^n \xi_i\|  \\
 & =  & o_{\P_n}( \delta_n) + r_n = o_{\P_n}(\delta_n). \end{eqnarray*}

\textsc{Step 2.}   Using the result of Step 1 and that
$$
\frac{\sqrt{n} p_n(v)'(\widehat \beta_n - \beta_n)}{ \|g_n(v)\|} = \frac{\sqrt{n } g_n(v) '\Omega_n^{-1/2}(\widehat \beta_n - \beta_n)}{ \|g_n(v)\|}\text{,}
$$
we  conclude that
\begin{align}\begin{split}\label{eq:boundS}
|S_n(v)| & := \Big |\frac{\sqrt{n } g_n(v) '\Omega_n^{-1/2}(\widehat \beta_n - \beta_n)}{ \|g_n(v)\|}  -  \frac{g_n(v)' \mathcal{N}_n}{\| g_n(v)\|} \Big |  \\
& \quad \leq  \left \|  \sqrt{n} \Omega_n^{-1/2} (\widehat \beta_n - \beta_n)  - \mathcal{N}_n \right \| = o_{\P_n}(\delta_n),
\end{split}\end{align}
uniformly in $v \in \mathcal{V}$. Finally,
\begin{eqnarray*}
& &\sup_{v \in \mathcal{V}}  \Big |\frac{\sqrt{n } (\widehat \theta_n (v) - \theta_n(v))}{ \|g_n(v)\|}  -  \frac{g_n(v)' \mathcal{N}_n}{\| g_n(v)\|} \Big | \\
& & \quad \leq  \sup_{v \in \mathcal{V}} \Big |\frac{\sqrt{n } (\widehat \theta_n (v) - \theta_n(v))}{ \|g_n(v)\|}  - \frac{\sqrt{n } g_n(v) '\Omega_n^{-1/2}(\widehat \beta_n - \beta_n)}{ \|g_n(v)\|} \Big | \\
& & \quad +  \sup_{v \in \mathcal{V}}\Big |\frac{\sqrt{n } g_n(v) '\Omega_n^{-1/2}(\widehat \beta_n - \beta_n)}{ \|g_n(v)\|}  -  \frac{g_n(v)' \mathcal{N}_n}{\| g_n(v)\|} \Big |   \\
& & \quad = \sup_{v \in \mathcal{V}}| \sqrt{n} A_n(v)/ \| g_n(v)\| | + \sup_{v \in \mathcal{V}}| S_n(v)| = o(\delta_n) + o_{\P_n}(\delta_n),
\end{eqnarray*}
using the assumption on the approximation error  $A_n(v) = \theta(v) - p_n(v)'\beta_n$ and (\ref{eq:boundS}).  This proves
the theorem.

\textsc{Step 3.} To show the corollary note that
\begin{eqnarray*}
E_{\Pn}\|u_{i,n}\|^3  & \leq & \| \Omega_n^{-1/2} Q_n^{-1} \| ^{3} \cdot E_{\Pn} \| p_n(V_i) \epsilon_i \|^3
\lesssim   \tau_n^{3}K_n^{3/2} C_n,
\end{eqnarray*}
using the boundedness assumptions stated in the corollary.   \qed

\footnotesize
\bibliographystyle{econometrica}
\bibliography{ibounds_mar2013}

\clearpage
\begin{figure}
  \includegraphics[scale = 0.36]{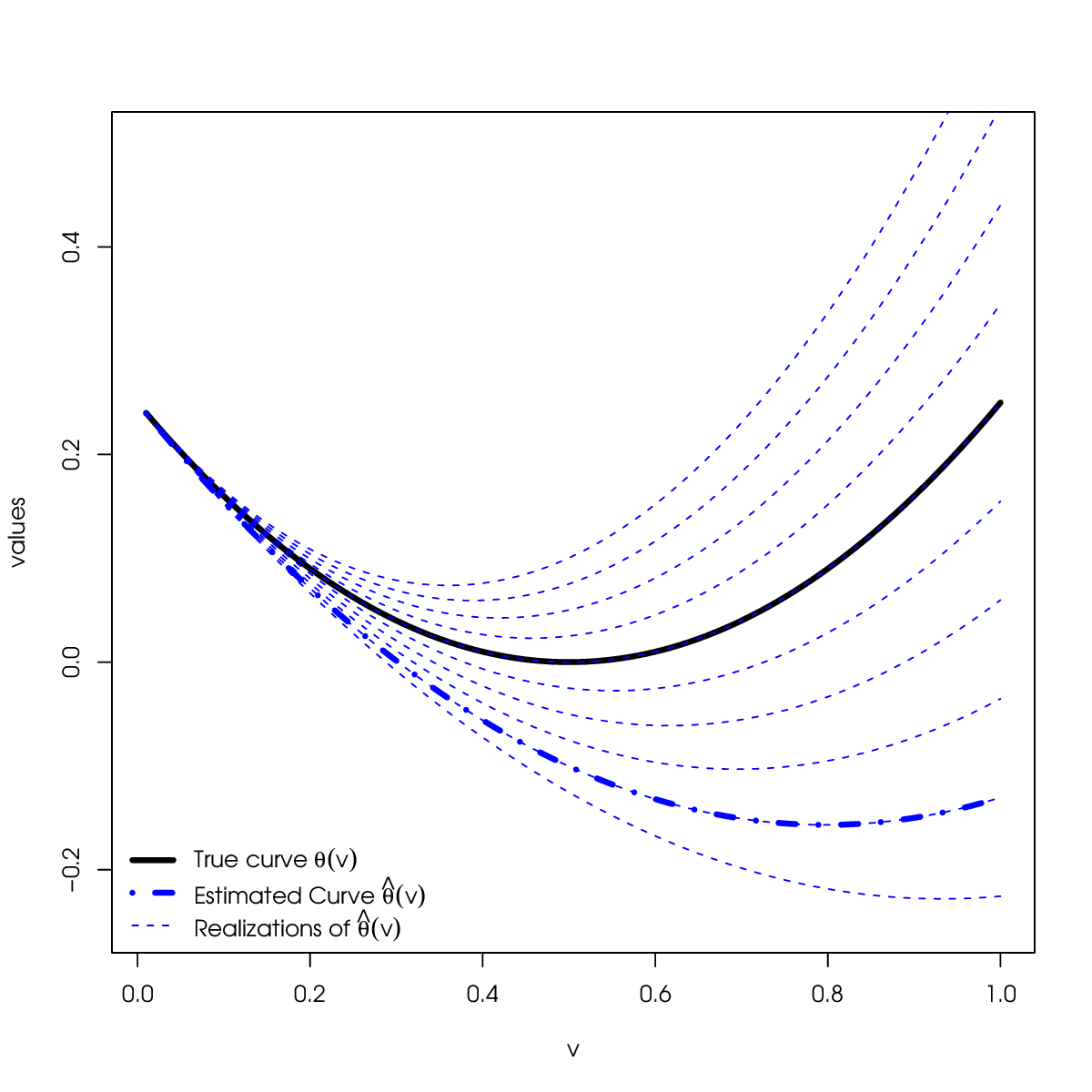}\\
  \caption{This figure illustrates how variation in the precision of the analog estimator at different points may impede inference.  The solid curve is the true bounding function $\theta(v)$, while the dash-dot curve is a single realization of its estimator, $\widehat{\theta}(v)\text{.}$  The lighter dashed curves depict eight additional representative realizations of the estimator, illustrating its precision at different values of $v$.  The minimum of the estimator $\widehat{\theta}(v)$
  is indeed quite far from the minimum of $\theta(v)$, making the empirical upper bound unduly tight. }\label{fig1}

  \includegraphics[scale = 0.36]{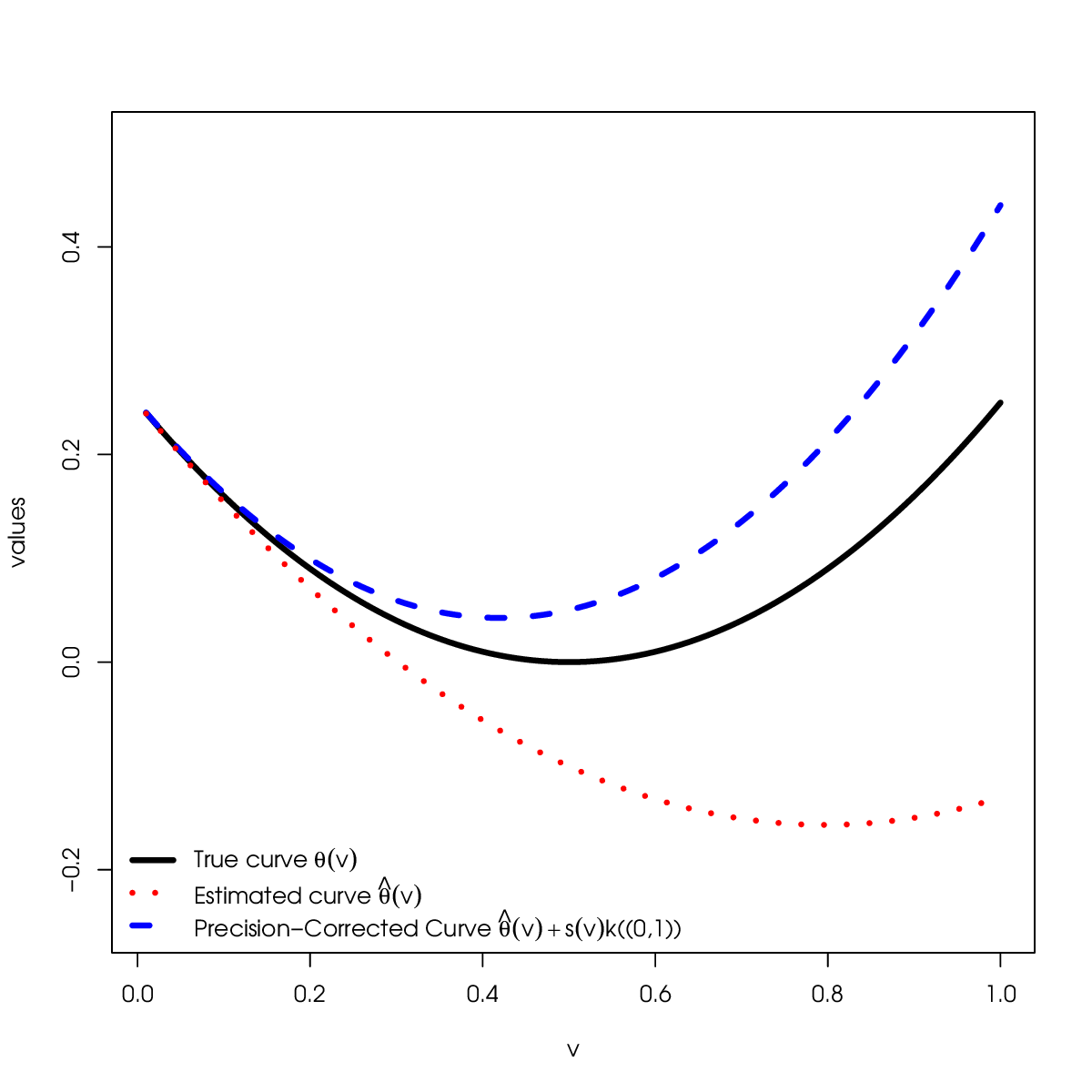}\\
  \caption{This figure depicts a precision-corrected curve (dashed curve) that adjusts the boundary estimate $\widehat{\theta}(v)$ (dotted curve) by an amount proportional to its point-wise standard error.
  The minimum of the precision-corrected curve is closer to the minimum of the true curve (solid) than the minimum of $\widehat{\theta}(v)$,
  removing the downward bias.}\label{fig2}
\end{figure}

\clearpage
\begin{figure}
  \includegraphics[scale = 0.3]{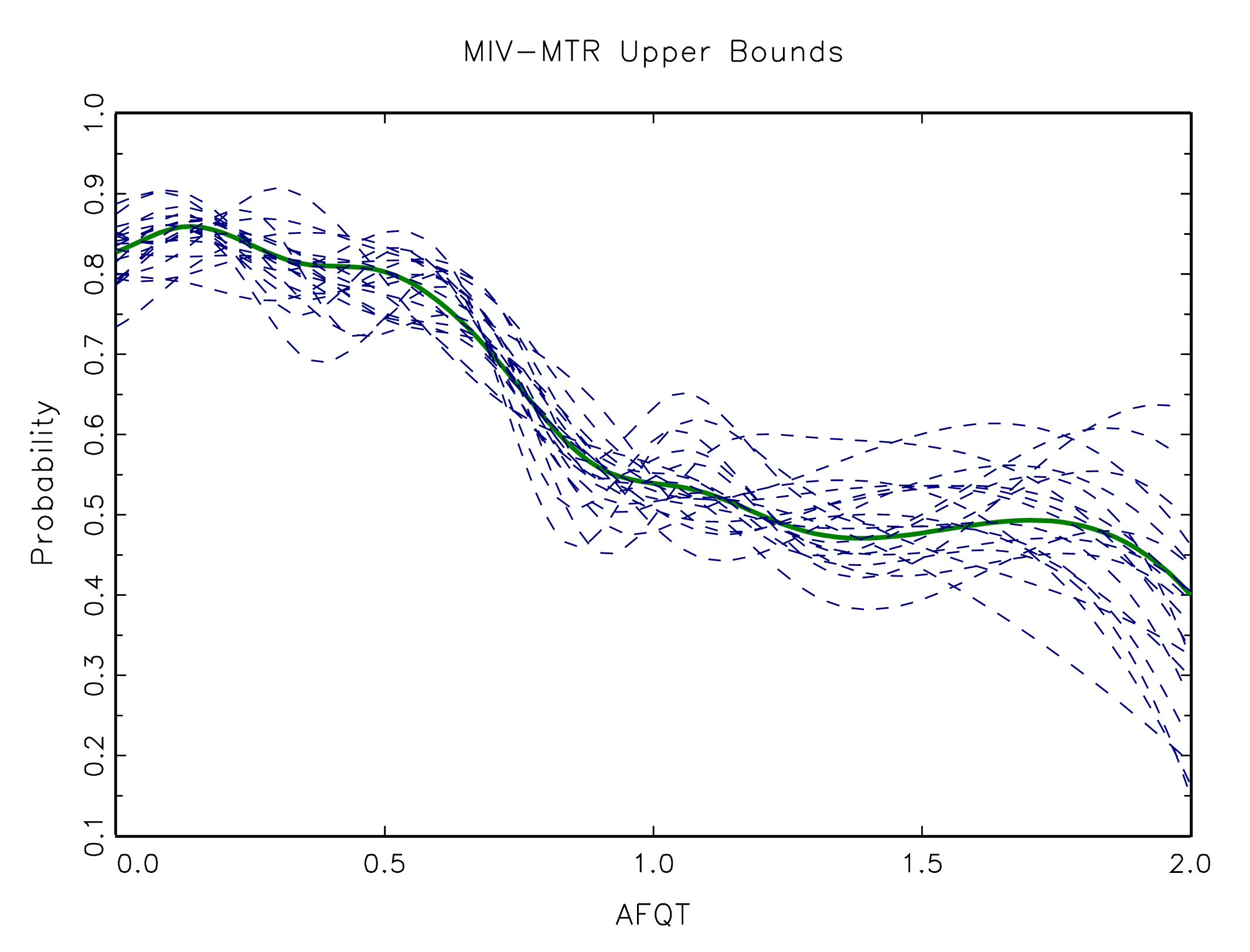}\\
  \caption{This figure, based on NLSY data and the application in Section 5 of Chernozhukov, Lee, and Rosen (2009), depicts an estimate of the bounding function (solid curve) and 20 bootstrapped estimates (dashed curves).}\label{fig2-real-b}

  \includegraphics[scale = 0.3]{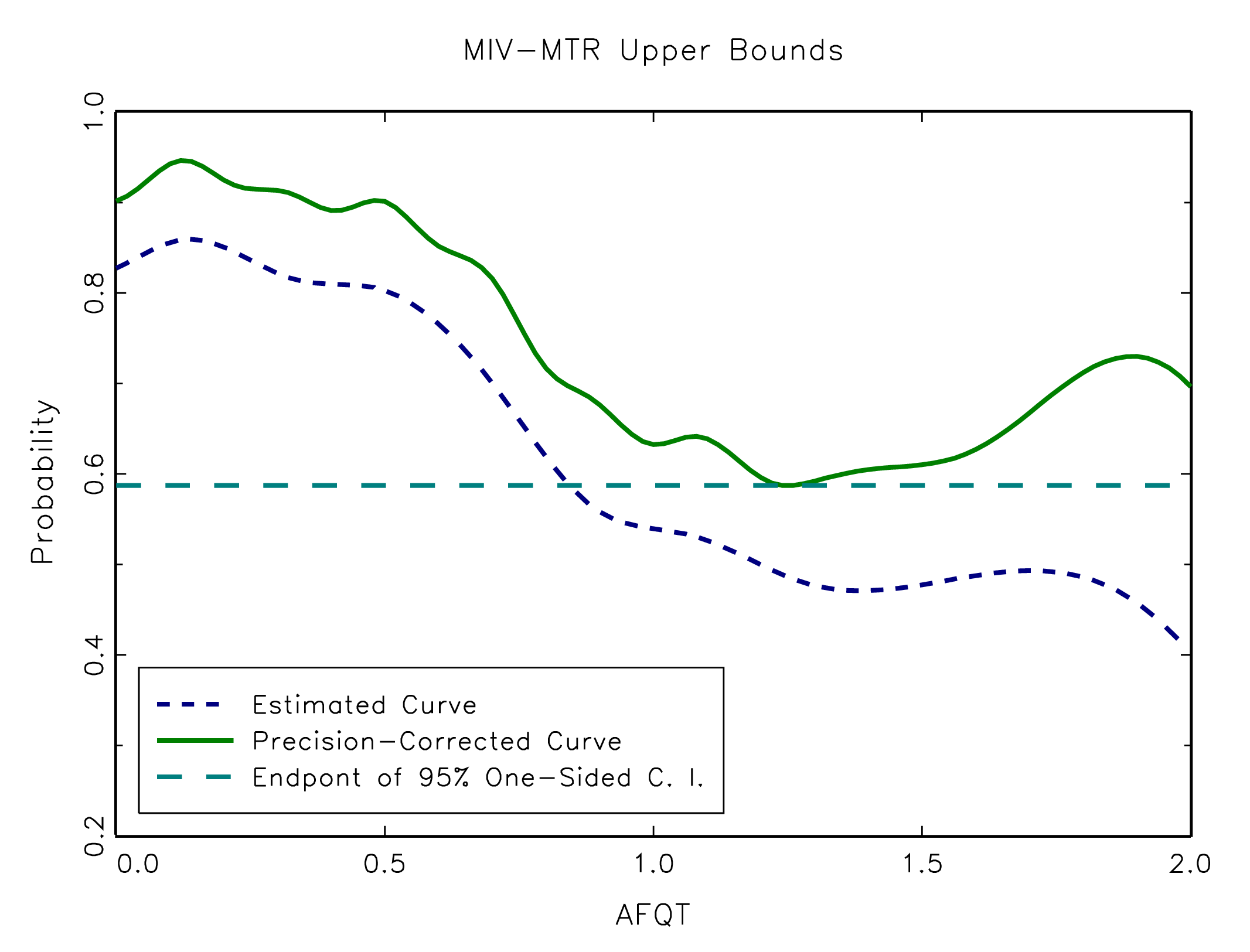}\\
  \caption{This figure, based on NLSY data and the application in Section 5 of Chernozhukov, Lee, and Rosen (2009), depicts a precision-corrected curve (solid curve) that adjusts the boundary estimate $\widehat{\theta}(v)$ (dashed curve) by an amount proportional to its point-wise standard error.
 The horizontal dashed line shows the end point of 95\% one-sided confidence interval.}\label{fig2-real}

\end{figure}

\clearpage
\newpage
\begin{table}
{\footnotesize \singlespacing
  \caption{Results for Monte Carlo Experiments (Series Estimation using B-splines) }\label{mc1}

  \begin{tabular}{rrcccccc}
    \hline\hline
    DGP & Sample  & Critical  & Ave. Smoothing & Cov.   & False Cov. & \multicolumn{2}{c}{Ave. Argmax Set} \\
        &  Size   & Value     & Parameter      & Prob.  &  Prob.     &  Min.  & Max.   \\
    \hline
  \multicolumn{8}{l}{CLR with Series Estimation using B-splines} \\
          &         & Estimating $V_n$? &    &     &    & &  \\
       1  &  500    &     No       &   9.610  &   0.944  &   0.149  &  -1.800  &   1.792  \\
       1  &  500    &     Yes      &   9.610  &   0.944  &   0.149  &  -1.800  &   1.792  \\
       1  & 1000    &     No       &  10.490  &   0.947  &   0.013  &  -1.801  &   1.797  \\  \vspace*{1ex}
       1  & 1000    &     Yes      &  10.490  &   0.947  &   0.013  &  -1.801  &   1.797  \\
       2  &  500    &     No       &   9.680  &   0.992  &   0.778  &  -1.800  &   1.792  \\
       2  &  500    &     Yes      &   9.680  &   0.982  &   0.661  &  -0.762  &   0.759  \\
       2  & 1000    &     No       &  10.578  &   0.997  &   0.619  &  -1.801  &   1.797  \\  \vspace*{1ex}
       2  & 1000    &     Yes      &  10.578  &   0.982  &   0.470  &  -0.669  &   0.670  \\
       3  &  500    &     No       &  11.584  &   0.995  &   0.903  &  -1.800  &   1.792  \\
       3  &  500    &     Yes      &  11.584  &   0.984  &   0.765  &  -0.342  &   0.344  \\
       3  & 1000    &     No       &  13.378  &   0.994  &   0.703  &  -1.801  &   1.797  \\  \vspace*{1ex}
       3  & 1000    &     Yes      &  13.378  &   0.971  &   0.483  &  -0.290  &   0.290  \\
       4  &  500    &     No       &  18.802  &   0.996  &   0.000  &  -1.800  &   1.792  \\
       4  &  500    &     Yes      &  18.802  &   0.974  &   0.000  &  -0.114  &   0.114  \\
       4  & 1000    &     No       &  20.572  &   1.000  &   0.000  &  -1.801  &   1.797  \\
       4  & 1000    &     Yes      &  20.572  &   0.977  &   0.000  &  -0.098  &   0.091  \\
  \hline
 \end{tabular}

%
}
\end{table}

\begin{table}
{\footnotesize \singlespacing
  \caption{Results for Monte Carlo Experiments (AS) }\label{mc2}

 \begin{tabular}{rrccc}
    \hline\hline
    DGP & Sample Size & Critical Value & Cov. Prob.  & False Cov. Prob. \\
    \hline
 \multicolumn{5}{l}{AS with CvM (Cram\'{e}r-von Mises-type statistic)} \\
     1  &  500    &     PA/Asy     & 0.959    &  0.007    \\
     1  &  500    &     GMS/Asy    & 0.955    &  0.007   \\
     1  & 1000    &     PA/Asy     & 0.958    &  0.000   \\  \vspace*{1ex}
     1  & 1000    &     GMS/Asy    & 0.954    &  0.000   \\
     2  &  500    &     PA/Asy     & 1.000    &  1.000  \\
     2  &  500    &     GMS/Asy    & 1.000    &  0.977   \\
     2  & 1000    &     PA/Asy     & 1.000    &  1.000   \\  \vspace*{1ex}
     2  & 1000    &     GMS/Asy    & 1.000    &  0.933  \\
     3  &  500    &     PA/Asy     & 1.000    &  1.000  \\
     3  &  500    &     GMS/Asy    & 1.000    &  1.000   \\
     3  & 1000    &     PA/Asy     & 1.000    &  1.000   \\   \vspace*{1ex}
     3  & 1000    &     GMS/Asy    & 1.000    &  1.000  \\
     4  &  500    &     PA/Asy     & 1.000    &  1.000  \\
     4  &  500    &     GMS/Asy    & 1.000    &  1.000   \\
     4  & 1000    &     PA/Asy     & 1.000    &  1.000   \\
     4  & 1000    &     GMS/Asy    & 1.000    &  1.000  \\
  \hline
  \end{tabular}
}
\end{table}

\clearpage
\newpage
\begin{table}
{\footnotesize \singlespacing
  \caption{Results for Monte Carlo Experiments (Other Estimation Methods)}\label{mc1-cont}

  \begin{tabular}{rrcccccc}
    \hline\hline
    DGP & Sample  & Critical  & Ave. Smoothing & Cov.   & False Cov. & \multicolumn{2}{c}{Ave. Argmax Set} \\
        &  Size   & Value     & Parameter      & Prob.  &  Prob.     &  Min.  & Max.   \\
    \hline
   \multicolumn{8}{l}{CLR with Series Estimation using Polynomials} \\
          &         & Estimating $V_n$? &    &     &    & &  \\
       1  &  500    &     No       &   5.524  &    0.954  &   0.086  &  -1.800  &   1.792 \\
       1  &  500    &     Yes      &   5.524  &    0.954  &   0.086  &  -1.800  &   1.792 \\
       1  & 1000    &     No       &   5.646  &    0.937  &   0.003  &  -1.801  &   1.797 \\  \vspace*{1ex}
       1  & 1000    &     Yes      &   5.646  &    0.937  &   0.003  &  -1.801  &   1.797 \\
       2  &  500    &     No       &   8.340  &    0.995  &   0.744  &  -1.800  &   1.792 \\
       2  &  500    &     Yes      &   8.340  &    0.989  &   0.602  &  -0.724  &   0.724 \\
       2  & 1000    &     No       &   9.161  &    0.996  &   0.527  &  -1.801  &   1.797 \\  \vspace*{1ex}
       2  & 1000    &     Yes      &   9.161  &    0.977  &   0.378  &  -0.619  &   0.620 \\
       3  &  500    &     No       &   8.350  &    0.998  &   0.809  &  -1.800  &   1.792 \\
       3  &  500    &     Yes      &   8.350  &    0.989  &   0.612  &  -0.300  &   0.301 \\
       3  & 1000    &     No       &   9.155  &    0.996  &   0.560  &  -1.801  &   1.797 \\  \vspace*{1ex}
       3  & 1000    &     Yes      &   9.155  &    0.959  &   0.299  &  -0.253  &   0.252 \\
       4  &  500    &     No       &   8.254  &    1.000  &   0.000  &  -1.800  &   1.792 \\
       4  &  500    &     Yes      &   8.254  &    0.999  &   0.000  &  -0.081  &   0.081 \\
       4  & 1000    &     No       &   9.167  &    0.998  &   0.000  &  -1.801  &   1.797 \\
       4  & 1000    &     Yes      &   9.167  &    0.981  &   0.000  &  -0.069  &   0.069 \\
\multicolumn{8}{l}{CLR with Local Linear Estimation} \\
        &         & Estimating $V_n$? &    &     &    & &  \\
     1  &  500    &     No       &    0.606  &    0.923  &   0.064  &  -1.799  &   1.792     \\
     1  &  500    &     Yes      &    0.606  &    0.923  &   0.064  &  -1.799  &   1.792     \\
     1  & 1000    &     No       &    0.576  &    0.936  &   0.003  &  -1.801  &   1.796     \\  \vspace*{1ex}
     1  & 1000    &     Yes      &    0.576  &    0.936  &   0.003  &  -1.801  &   1.796     \\
     2  &  500    &     No       &    0.264  &    0.995  &   0.871  &  -1.799  &   1.792     \\
     2  &  500    &     Yes      &    0.264  &    0.989  &   0.808  &  -0.890  &   0.892     \\
     2  & 1000    &     No       &    0.218  &    0.996  &   0.779  &  -1.801  &   1.796     \\  \vspace*{1ex}
     2  & 1000    &     Yes      &    0.218  &    0.990  &   0.675  &  -0.776  &   0.776     \\
     3  &  500    &     No       &    0.140  &    0.995  &   0.943  &  -1.799  &   1.792     \\
     3  &  500    &     Yes      &    0.140  &    0.986  &   0.876  &  -0.426  &   0.424     \\
     3  & 1000    &     No       &    0.116  &    0.992  &   0.907  &  -1.801  &   1.796     \\  \vspace*{1ex}
     3  & 1000    &     Yes      &    0.116  &    0.986  &   0.816  &  -0.380  &   0.377     \\
     4  &  500    &     No       &    0.078  &    0.991  &   0.000  &  -1.799  &   1.792     \\
     4  &  500    &     Yes      &    0.078  &    0.981  &   0.000  &  -0.142  &   0.142     \\
     4  & 1000    &     No       &    0.064  &    0.997  &   0.000  &  -1.801  &   1.796     \\
     4  & 1000    &     Yes      &    0.064  &    0.991  &   0.000  &  -0.127  &   0.127     \\
\hline
 \end{tabular}

%
}
\end{table}

\begin{table}
{\footnotesize \singlespacing
  \caption{Computation Times of Monte Carlo Experiments}\label{mc-time}

 \begin{tabular}{rrccc}
    \hline\hline
     & AS & Series (B-splines) & Series (Polynomials)  & Local Linear \\
    \hline
Total minutes for simulations & 8.67  &	 57.99 &	19.44 &	83.75  \\
Average seconds for each test &  0.03 	 &    0.22 &  	0.07  & 0.31    \\
         Ratio relative to AS & 1.00   & 	  6.69 &  	2.24  &	9.66   \\ 
\hline
\end{tabular}
}
\end{table}

\clearpage
\normalsize

\section{Proofs Omitted From the Main Text.}\label{sec:omitted-proof-appendix}
\subsection{Proof of Lemma \ref{lemma: estimation of V} (Estimation of $\Vn$)}
There is a single proof for both analytical and simulation methods, but it is convenient for clarity to split
the first step of the proof into separate cases. There are four steps in total.

Step 1a. (Bounds on $k_{n, \mathcal{V}}(\gamma_n)$ in Analytical Case)   We have that for some constant $\eta >0$
\begin{align*}
 & k_{n, \mathcal{V}}(\gamma_n) := \( \bar a_n + \frac{c(\gamma_n)}{\bar a_n}\), \\
 & \kappa_n := \kappa_n(\gamma'_n)  := Q_{\gamma'_n} \(  \sup_{v \in \V} Z_n^*(v)\), \ \  & \bar \kappa_n := 7  \(\bar a_n + \frac{\eta \ell\ell_n}{\bar a_n}\).
\end{align*}
The claim of this step is that given the sequence $\gamma_n$ we have for all large $n$:
\begin{eqnarray}
&& k_{n, \mathcal{V}}(\gamma_n) \geq \kappa_n(\gamma_n)\label{ineq: a}\\
&& 6 k_{n, \mathcal{V}}(\gamma_n) < \bar \kappa_n  \label{ineq: b}
\end{eqnarray}
Inequality (\ref{ineq: b}) follows from \eqref{ineq: a1} in step 2 of the proof of Lemma \ref{lemma: concentrate on balls}  (with $\gamma_n$ in place of $\gamma'_n$);
(\ref{ineq: a}) follows immediately from Condition C.3.

Step 1b. (Bounds on $k_{n, \mathcal{V}}(\gamma_n)$ in Simulation Case) We have
\begin{align*}
k_{n, \mathcal{V}}(\gamma_n) & := Q_{\gamma_n}\( \sup_{v \in \V} Z_n^\star(v)\mid  \D_n\) , \\
 \kappa_n = \kappa_n(\gamma_n') & := Q_{\gamma'_n} \(  \sup_{v \in \V} \bar Z_n^*(v) \), \ \  & \bar \kappa_n := 7\(\bar a_n  + \frac{\eta \ell \ell_n}{\bar a_n}\).
\end{align*}
The claim of this step is that given $\gamma_n$ there is $\gamma_n \geq \gamma_n' = \gamma_n - o(1)$ such that, wp $\to$ 1
\begin{eqnarray}
 && k_{n, \mathcal{V}}(\gamma_n)  \geq \kappa_n(\gamma_n') \label{ineq: aa}\\
 && 6k_{n, \mathcal{V}}(\gamma_n) < \bar \kappa_n \label{ineq: bb}
\end{eqnarray}

 To show inequality (\ref{ineq: aa}), note that by C.2 and Lemma \ref{lemma: quantiles are close} wp $\to 1$
\begin{equation}\label{eq: quantile comparison 4_again}
 \kappa_{n,\mathcal{V}}(\gamma_n + o(1/\ell_n)) + o(\delta_n) \geq k_{n,\mathcal{V}}(\gamma_n)  \geq \kappa_{n,\mathcal{V}}(\gamma_n - o(1/\ell_n)) - o(\delta_n) .
\end{equation}
Hence (\ref{ineq: aa}) follows from
 \begin{eqnarray*}
&  & \P_n \( \sup_{v \in \V} \bar Z_n^*(v)  \leq x \) \Big \vert_{x=k_{n, \mathcal{V}}(\gamma_n)}  \\
 &  & \geq_{(1)} \P_n\(  \sup_{v \in \V} \bar Z_n^*(v)  \leq \kappa_{n, \mathcal{V}}(\gamma_n - o(1/\ell_n)) - o(\delta_n)  \) - o(1) \text{ wp $\to 1$}\\
& & \geq_{(2)} \P_n \( \sup_{v \in \V} \bar Z_n^*(v) \leq \kappa_{n, \mathcal{V}}(\gamma_n - o(1/\ell_n)) \) - o(1) = \gamma_n - o(1/\ell_n) - o(1) =: \gamma_n',  \;\; (\gamma_n \geq \gamma_n')
  \end{eqnarray*}
where (1) holds by (\ref{eq: quantile comparison 4_again}) and (2) holds by anti-concentration Corollary \ref{cor: anti}.

 To show inequality (\ref{ineq: bb}) note that by C.3, we have
\begin{eqnarray*}
 \kappa_{n,\mathcal{V}}(\gamma_n + o(1/\ell_n)) + o(\delta_n) & \leq & \bar a_n + \frac{c(\gamma_n + o(1/\ell_n))}{\bar a_n} + o(\delta_n) \leq \bar a_n +  \frac{\eta \ell\ell_n + \eta \log 10}{\bar a_n}
 + o(\delta_n),
\end{eqnarray*}
where the last inequality relies on
$$
c(\gamma_n + o(1/\ell_n)) \leq -\eta \log ((1-\gamma_n - o(1/\ell_n)) \leq \eta o(\ell \ell_n) + \eta \log 10,
$$
holding for large $n$ by C.3. From this we deduce (\ref{ineq: bb}).

Step 2. (Lower Containment)   We have that for all $v \in \Vn$,
\ba
A_n(v) & := & \widehat \theta_n(v)  - \inf_{ v \in \V} \( \hat \theta_n(v) + k_{n, \mathcal{V}}(\gamma_n) s_n(v)\) \\
 & \leq & - Z_n(v) \sigma_n(v) + \kappa_n \sigma_n(v) + \sup_{v \in \V} \{ \theta_{n0} - \hat \theta_n(v) - k_{n,\V}(\gamma_n) s_n(v) \}:=  B_n(v)
\ea
since $\theta_n(v) \leq \theta_{n0}  + \kappa_n \sigma_n(v), \forall v \in V_n$ and $\widehat \theta_n(v) - \theta_n(v) = - Z_n(v) \sigma_n(v)$. Therefore,
\ba
\P_n\{ V_n \subseteq \widehat V_n\} &=&  \P_n \{ A_n(v) \leq 2 k_{n, \V}(\gamma_n) s_n(v), \forall v \in V_n \} \\
 &\geq &  \P_n \{ B_n(v) \leq 2  k_{n, \V}(\gamma_n) s_n(v), \forall v \in V_n \} \\
 & \geq & \P_n \{ - Z_n(v) \sigma_n(v)  \leq 2 k_{n,\V}(\gamma_n) s_n(v)  - \kappa_n \sigma_n(v), \forall v \in V_n\}\\
 & - &  \P_n \{ \sup_{ v \in \V} \frac{\theta_{n0}- \widehat \theta_n(v)}{s_n(v)} \geq k_{n, \V} (\gamma_n)\} \\
 &  := &  a - b = \gamma_n' - o(1) = 1 - o(1),
\ea
where $b=o(1)$ follows similarly to the proof of Theorems \ref{theorem: inference analytical} (analytical case) and Theorem \ref{theorem: inference1} (simulation case),
using  that $k_{n, \V}(\gamma_n) \geq k_{n, V_n}(\gamma_n)$ for sufficiently large $n$, and $a = 1-o(1)$ follows from the
following argument:
\ba
a  & \geq_{(1)} & \P_n \(  \sup_{v \in \mathcal{V}} - Z_n(v) \leq 2 k_{n, \V} (\gamma_n) [1 - o_{\P_n}(\delta_n/( \bar a_n + \ell \ell_n ))] - \kappa_n \) \\
& \geq_{(2)} & \P_n \(  \sup_{v \in \mathcal{V}} - Z^*_n(v) \leq 2 k_{n, \V} (\gamma_n) - \kappa_n - o_{\P_n}(\delta_n)\) - o(1)\\
&  \geq_{(3)} & \P_n \( \sup_{v \in \mathcal{V}} - Z^*_n(v) \leq  \kappa_n - o_{\P_n}(\delta_n)\) - o(1) \\
&  \geq_{(4)} & \gamma_n' - o(1) = 1 - o(1),
 \ea
where terms $o(\delta_n)$ are different in different places; where (1) follows by C.4, (2) is by C.2 and by Step 1, namely by $k_{n, \V}(\gamma_n') \leq \bar \kappa_n \lesssim \bar a_n + \ell\ell_n$ wp $\to 1$,
(3) follows by Step 1, and (4) follows by
the anti-concentration Corollary \ref{cor: anti} and the definition of $\kappa_n$.

\textit{Step} 3. -- Upper Containment.  We have that for all $v \not \in \barVn$,
\ba
A_n(v) & := & \widehat \theta_n(v) -\theta_{n0}- \inf_{ v \in \V} \( \hat \theta_n(v) - \theta_{n0} + k_{n, \mathcal{V}}(\gamma_n) s_n(v)\) \\
 & > & - Z_n(v) \sigma_n(v) + \bar \kappa_n \bar \sigma_n  \\
 & + &  \sup_{v \in \V} \{ \theta_{n0} - \hat \theta_n(v) - k_{n,\V}(\gamma_n) s_n(v) \} :=   C_n(v),
\ea
since $\theta_n(v) > \theta_{n0} + \bar \kappa_n \bar \sigma_n, \forall v \not \in \barVn$, and $\hat \theta_n(v) - \theta_n(v) = - Z_n(v) \sigma_n(v).$ Hence
\ba
 \P_n\(  \widehat V_n \not \subseteq \barVn\) & = &  \P_n \{  A_n(v) \leq 2 k_{n, \V} (\gamma_n) s_n(v), \exists v \not \in \barVn \}\\
 & \leq & \P_n\{  C_n(v) < 2 k_{n, \V} (\gamma_n) s_n(v), \exists v \not \in \barVn \} \\
 & \leq & \P_n\{ - \sup_{v \in \V} 2 |Z_n(v)| \bar \sigma_n < 3 k_{n, \V} (\gamma_n) \bar s_n - \bar \kappa_n \bar \sigma_n \} := a = o(1),
\ea
where we used elementary inequalities to arrive at the last conclusion.  Then $a=o(1)$, since
\ba
a & \leq_{(1)} & \P_n \( - 2|Z_n(v)| < 3 k_{n,\V}(\gamma_n) [1 + o_{\P_n}(\delta_n/(\bar a_n + \ell \ell_n))] - \bar \kappa_n , \exists v \in \V\)  \\
& \leq_{(2)} & \P_n \( - |Z^*_n(v)| < (3 k_{n,\V}(\gamma_n) - \bar \kappa_n)/2  + o(\delta_n), \exists v  \in \V\) + o(1)  \\
& \leq_{(3)} & \P_n \( -|Z^*_n(v)| < - k_{n,\V}(\gamma_n)  + o(\delta_n), \exists v  \in \V\) + o(1)  \\
& \leq_{(4)}  & 2 \P_n \( \sup_{v \in \V} Z^*_n(v) >  k_{n,\V}(\gamma_n)  - o(\delta_n) \) + o(1)  \\
& \leq_{(5)}  & 2 \P_n \( \sup_{v \in \V} Z^*_n(v) >  \kappa_{n}- o(\delta_n) \) + o(1)  \leq_{(6)}  2(1- \gamma_n') + o(1),
\ea
where inequality (1) follows by Condition C.4, inequality (2) follows by Condition C.2 and Step 1, namely by $k_{n, \V}(\gamma_n') \leq \bar \kappa_n \lesssim \bar a_n + \ell\ell_n$ wp $\to 1$, inequality (3) follows by Step 1 and the union bound, inequality (4) holds by the union bound and symmetry,   inequality (5) holds by Step 1,  and inequality (6) holds by the definition of $\kappa_n$ and the anti-concentration Corollary \ref{cor: anti}.

Step 4. (Rate).   We have that wp $\to 1$
\begin{align*}
& d_H(\widehat V_n, V_{0}) \leq_{(1)} d_H(\widehat V_n, \Vn) +  d_H(\Vn, V_{0}) \leq_{(2)} 2 d_H( \barVn, V_0) \leq_{(3)} 2( \bar \sigma_n \bar \kappa_n)^{1/\rho_n}/c_n
\end{align*}
where (1) holds by the triangle inequality,  (2) follows by the containment  $V_0 \subseteq \Vn \subseteq
\widehat V_n \subseteq \bar V_n$ holding wp $\to 1$, and (3) follows from $  \bar \kappa_n \bar \sigma_n \to 0$
holding by assumption, and from the following relation holding by Condition V:
\begin{eqnarray*}
d_H(\bar V_n, V_0) & = &  \sup_{v \in \bar V_n} d(v,  V_{0})  \leq   \sup \{ d(v, V_{0}): \theta_n(v) - \theta_{n0} \leq \bar \kappa_n \bar \sigma_n \} \\
& \leq  & \sup \{ d(v, V_{0}):   (c_n d(v, V_{0}))^{\rho_n}\wedge \delta \leq \bar \kappa_n \bar \sigma_n \} \\
& \leq & \sup \{ t:   (c_n t )^{\rho_n}\wedge \delta  \leq  \bar \kappa_n \bar \sigma_n \}
\leq  c_n^{-1}  (\bar \kappa_n \bar \sigma_n)^{1/\rho_n} \text{ for all }  0 \leq   \bar \kappa_n \bar \sigma_n \leq \delta.
\end{eqnarray*}
\qed

\subsection{Proof of Lemma \ref{lemma:parametric}}

Step 1. Verification of C.1.  This condition holds by inspection in view of
continuity of $v \mapsto p_{n}(v,\beta_n)$ and $v \mapsto p_{n}(v,\widehat \beta)$
implied by Condition P(ii) and by $\Omega_n$ and $\widehat \Omega_n$ being positive definite.

Step 2. Verification of C.2. Part (a). By Condition P, uniformly in $v \in \mathcal{V}$, for $\beta_n^\ast(v)$ denoting an intermediate value between $\widehat \beta_n$ and $\beta_n$,
\begin{eqnarray*}
 Z_n(v) &=& \frac{p_n(v, \beta^*_n(v))'}{ \| p_n(v, \beta_n)'\Omega_n^{1/2}\|} \sqrt{n}(\hat \beta_n - \beta_n)  \\
 & = &   \frac{p_n(v, \beta_n)'}{ \| p_n(v, \beta_n)'\Omega_n^{1/2}\|} \sqrt{n}(\hat \beta_n - \beta_n) +  \frac{L_n \sqrt{n} \| \hat \beta_n - \beta_n \|^2}{\min_{v \in \mathcal{V}}\| p_n(v,\beta_n)\|}  \frac{1}{\lambda_{\min}(\Omega_n^{1/2})} \\
 & = &   \frac{p_n(v, \beta_n)' \Omega_n^{1/2}}{ \| p_n(v, \beta_n)'\Omega_n^{1/2}\|} \mathcal{N}_k   + o_{\Pn}(\delta'_n)
 +  O_{\Pn} (n^{-1/2})\text{.} \\
 \end{eqnarray*}

 Part (b).  First note using the inequality
\begin{equation}\label{eq: eleimentary}
\left \| \frac{a}{\|a\|}- \frac{b}{\|b\|}\right\| \leq \left( 2\frac{\|a - b\|}{\|a\|}\right) \wedge \(2\frac{\|a - b\|}{\|b\|}\)\text{,}
\end{equation}
we have
\begin{eqnarray*}
M_n & = &  \left  \| \frac{p_n(v, \beta_n)' \Omega_n^{1/2}}{ \| p_n(v, \beta_n)'\Omega_n^{1/2}\|} - \frac{p_n(v, \hat \beta_n)' \hat \Omega_n^{1/2}}{ \| p_n(v, \hat \beta_n)'\hat \Omega_n^{1/2}\|} \right \|
   \leq  2\frac{\| p_n(v, \beta_n)' \Omega_n^{1/2} -  p_n(v, \hat \beta_n)' \hat \Omega_n^{1/2}\|}{\| p_n(v, \beta_n)'\Omega_n^{1/2}\|} \\
 & \leq &  2 \frac{\| p_n(v, \beta_n)' \Omega_n^{1/2} (I - \Omega_n^{-1/2} \hat \Omega_n^{1/2} ) \|}{\| p_n(v, \beta_n)'\Omega_n^{1/2}\|}
 + 2 \frac{L_n \|\widehat \beta_n - \beta_n \|}{\min_{v \in \mathcal{V}}\| p_n(v,\beta_n)\|}  \frac{\lambda_{\max}(\widehat \Omega_n^{1/2})}{\lambda_{\min}(\Omega_n^{1/2})} \\
 & \leq & 2\|\Omega_n^{-1/2}\| \| \widehat \Omega_n^{1/2}-\Omega_n^{1/2}\| + O_{\Pn}(n^{-1/2}) \leq O_{\Pn}(n^{-b}) + O_{\Pn}(n^{-1/2}) =  O_{\Pn}(n^{-b})\text{,}
 \end{eqnarray*}
for some $b>0$. We have that
$$
E_{\Pn} \( \sup_{v \in \V} |Z^*_n(v) - Z_n^{\star}(v)|  \mid \D_n \) \leq M_n E_{\Pn} \| \mathcal{N}_k\| \lesssim M_n \sqrt{k}.
$$
Hence for any $\delta_n'' \propto  n^{-b'}$ with a constant $0 < b' < b$, we have by Markov's Inequality that
$$
\Pn\(   \sup_{v \in \V} |Z^*_n(v) - Z_n^{\star}(v)| > \delta_n \ell_n  \mid \D_n \) \leq \frac{  O_{\Pn}(n^{-b})}{\delta''_n \ell_n} = o_{\Pn}(1/\ell_n).
$$

\noindent Now select $\delta_n = \delta'_n \vee \delta_n''$.

Step 3. Verification of C.3.  We shall employ Lemma \ref{lemma: key concentration}, which
has the required notation in place. We only need to compute an upper bound
on the covering numbers $N(\varepsilon,\textsf{V}, \rho)$ for the process $Z_n$. We have
that
\begin{eqnarray*}
 &\sigma(Z^*_n(v) - Z^*_n(\tilde v) )  \leq   \left\| \frac{p_n(v,\beta_n)'\Omega^{1/2}_n}{\|p_n(v,\beta_n)'\Omega^{1/2}_n\|}-   \frac{p_n(\tilde v,\beta_n)'\Omega^{1/2}_n}{\|p_n(\tilde v,\beta_n)'\Omega^{1/2}_n\|} \right\|  \\
&\leq   2 \left\| \frac{( p_n(v,\beta_n) -p_n(\tilde v,\beta_n) )'\Omega^{1/2}_n}{\|p_n(v,\beta_n)'\Omega^{1/2}_n\|}\right\|  \leq  2 \frac{L_n}{\min_{v \in \V} \|p_n(v,\beta_n)\|} \frac{\lambda_{\max}(\Omega_n^{1/2})}{\lambda_{\min}(\Omega_n^{1/2})}\| v - \tilde v\|  \leq  C L \| v - \tilde v\|,
 \end{eqnarray*}
where $C$ is some constant that does not depend on $n$, by the eigenvalues of $\Omega_n$ bounded
away from zero and from above.  Hence by the standard volumetric argument
$$
N(\varepsilon, \textsf{V},\rho) \leq \left(  \frac{1+  C L  \text{diam}(\textsf{V})}{\varepsilon}\right)^d, \ \  0 < \varepsilon< 1,
$$
where the diameter of $\textsf{V}$ is measured by the Euclidian metric.
Condition C.3 now follows by Lemma \ref{lemma: key concentration}, with
$
a_n(\textsf{V}) = (2\sqrt{\log L_n(\textsf{V})  }) \vee (1 +\sqrt{d}),  \ \ L_n(\textsf{V}) =  C'  \( 1 + C L \text{diam}(\textsf{V}) \)^d,
$
where $C'$ is some positive constant.

Step 4. Verification of C.4.  Under Condition P, we have that
$
1 \leq a_n(\textsf{V}) \leq \bar a_n:= a_n(\mathcal{V}) \lesssim 1,
$
so that C.4(a) follows since by Condition P
$$
\bar \sigma_n = \sqrt{ \max_{ v \in \V} \| p_n(v, \beta_n) \Omega_n^{1/2}\|/n } \leq   \sqrt{ \max_{ v \in \V} \| p_n(v, \beta_n)\| \|\Omega_n^{1/2}\|/n } \lesssim \sqrt{1/n}
$$
To verify C.4(b) note that uniformly in $v \in \V$,
\begin{eqnarray*}
& & \left | \frac{\|p_n(v,\beta_n)'\widehat \Omega_n^{1/2}\|}{\|p_n(v,\beta_n)'\Omega_n^{1/2}\|} -1 \right | \leq
  \left | \frac{\|p_n(v,\beta_n)'\widehat \Omega_n^{1/2}\|- \|p_n(v,\beta_n)'\Omega_n^{1/2}\|}{\|p_n(v,\beta_n)'\Omega_n^{1/2}\|}  \right |\\
& &  \leq  \frac{\|p_n(v,\beta_n)'(\widehat \Omega_n^{1/2} - \Omega_n^{1/2})\|}{\|p_n(v,\beta_n)'\Omega_n^{1/2}\|}   \leq
\frac{\|p_n(v,\beta_n)'\Omega^{1/2}(\Omega_n^{-1/2}\widehat \Omega_n^{1/2}-I)\|}{\|p_n(v,\beta_n)'\Omega_n^{1/2}\|}   \\
& & \leq  \| \Omega_n^{-1/2}\widehat \Omega_n^{1/2}-I\| \leq \|\Omega_n^{-1/2}\| \| \widehat \Omega_n^{1/2}-\Omega_n^{1/2}\| = o_{\P_n}(\delta_n),
\end{eqnarray*}
since $ \| \widehat \Omega_n^{1/2}- \Omega_n^{1/2}\| = O_{\P_n}(n^{-b})$ for some $b>0$,
and since  $\|\Omega_n^{-1/2}\|$ is uniformly bounded, both implied  by the assumptions.

Step 5. Verification of S. Then, since under Condition V with for large enough $n$,
$
r_n \lesssim c_n^{-1} (1/\sqrt{n})^{1/\rho_n}= o(1),
$
we have that $r_n \leq \varphi_n$ for large $n$ for some $\varphi_n = o(1)$. Condition S then follows by noting that for any positive $o(1)$ term,
$
\sup_{ \| v - \tilde v\| \leq o(1)} | Z_n(v) - Z_n(\tilde v)| \lesssim o(1) \| \mathcal{N}_k\| = o_{\Pn}(1).
$\qed

\section{Kernel-Type Estimation of Bounding Function from Conditional Moment Inequalities}\label{sec:est:local:CMIexample}

In this section we provide primitive conditions that justify application of kernel-type estimation methods covered in Section \ref{sec:est:local} for models characterized by conditional moment inequalities.

\begin{example}[\textbf{Bounding Function from Conditional Moment Inequalities}]
Suppose that  we have an i.i.d. sample
of $(X_i, Z_i), i=1,..., n$ defined on the probability space $(A, \mathcal{A}, \P)$,
where we take $\P$ fixed in this example.  Suppose that $\text{support}(Z_i)= \mathcal{Z} \subseteq [0,1]^d$,
and $$\theta_{n0} = \min_{v \in \V} \theta_{n}(v), $$
for $\theta_n(v) = E_{\P}[ m(X_i,\mu,j)|Z_i=z]$,  $v = (z,j)$, where
$\V \subseteq \mathcal{Z} \times \{1,...,J\}$ be the set of interest.  Suppose the first $J_0$  functions correspond
to equalities treated as inequalities, so that $m(X_i,\mu,j) = - m(X_i,\mu,j+1)$, for $j \in \mathcal{J}_0= \{1,3,...,J_0-1\}$.
Hence $\theta_n(z,j) = - \theta_n(z,j+1)$ for $j \in \mathcal{J}_0$, and
we only need to estimate functions $\theta_n(z,j)$ with the index $j \in \mathcal{J} := \mathcal{J}_0 \cup \{J_0+1, J_0 +2,...,J\}$.
Suppose we use the local polynomial approach to approximating and estimating $\theta_n(z,j)$.  For $u \equiv (u_1,\ldots,u_d)$, a $d$-dimensional vector of nonnegative integers, let $[u] = u_1 + \cdots + u_d$.
Let $A_p$ be the set of all $d$-dimensional vectors  $u$ such that $[u] \leq p$ for some integer $p \geq 0$ and let $|A_p|$ denote the number of elements in $A_p$.
For $z \in \mathbb{R}^d$ with $u \in A_p$, let $z^u = \prod_{i=1}^d z_i^{u_i}$.
Now define
\begin{align}\label{porder}
\mathbf{p}(b,z) = \sum_{u \in A_p} b_u z^u,
\end{align}
where $b = (b_u)_{u \in A_p}$ is a vector of dimension $|A_p|$.
For each $v = (z,j)$ and $Y_i(j):=m(X_i,\mu,j)$, define
\begin{align*}
S_n (b) := \sum_{i=1}^n \left[ Y_i(j) - \mathbf{p}\left(b, \frac{Z_i - z}{h_n} \right) \right]^2 K_{h_n} (Z_i - z),
\end{align*}
where $K_h (u) := K(u/h)$, $K(\cdot)$ is a $d$-dimensional kernel function and $h_n$ is a sequence of bandwidths.
The local polynomial estimator $\widehat \theta_{n}(v)$ of the regression function is the first element of
$\widehat b(z,j):= \arg\min_{b \in \Bbb{R}^{|A_p|}}S_n(b)$.

We impose the following conditions:
\begin{quote}
(i) for each $j \in \mathcal{J}$, $\theta(z,j)$ is $(p+1)$ times continuously differentiable with respect to $z \in \mathcal{Z}$,
where $\mathcal{Z}$ is convex. (ii) the probability density function $f$ of $Z_i$ is bounded above and bounded below from zero with continuous derivatives on $\mathcal{Z}$; (iii) for $Y_i(j):=m(X_i,\mu,j)$,  $Y_i :=( Y_i(j), j \in \mathcal{J})'$, and $U_i := Y_i - E_{\P}[Y_i|Z_i]$;
 and $U_i$ is a bounded random vector;
(iv) for each $j$, the conditional on $Z_i$ density of $U_i$  exists and is  uniformly bounded from above
and below,
or, more generally, condition R stated in Appendix \ref{sec:proofs-strong-local} holds;
(v) $K(\cdot)$ has support on $[-1,1]^d$, is twice continuously differentiable, $\int u K(u)du=0$, and $\int K(u)du=1$;
(vi) $h_n \rightarrow 0$, $n h_n^{d+|\mathcal{J}|+1} \rightarrow \infty$, $n h_n^{d +  2(p+1)} \rightarrow 0$,
 $\sqrt{ n^{-1} h^{-2d} } \rightarrow 0$ at polynomial rates in $n$.
\end{quote}

These conditions are imposed to verify Assumptions A1-A7 in \cite{Kong/Linton/Xia:10}. Details of verification are given in Supplementary Appendix \ref{sec:prac:lle}. Note that $p > |\mathcal{J}|/2 - 1$ is necessary to satisfy bandwidth conditions in (vi).
 The assumption that $U_i$ is bounded is technical and is made to simplify exposition and proofs.

Let $\delta_n = 1/\log n$.
 Then it follows
from Corollary 1 and Lemmas 8 and 10 of  \cite{Kong/Linton/Xia:10}
 that
\begin{align}\label{bahadur-lemma0}
\widehat \theta_n(z,j) - \theta(z,j) = \frac{1}{n h_n^d f(z)} \mathbf{e}_1'S_p^{-1} \sum_{i=1}^n (e_j'U_i) K_h( Z_i - z) \mathbf{u}_p \left( \frac{Z_i-z}{h_n} \right) + B_n(z,j) + R_n(z,j),
\end{align}
where $\mathbf{e}_1$ is an $|A_p| \times 1$ vector whose  first element is one and all others are zeros,
$S_p$ is an $|A_p| \times |A_p| $ matrix such that $S_p = \{ \int z^u (z^v)' du: u \in A_p, v \in A_p \}$,
$\mathbf{u}_p(z)$ is an  $|A_p| \times 1$ vector such that $\mathbf{u}_p(z) = \{  z^u : u \in A_p \}$,
$$
B_n(z,j) = O ( h_n^{p+1} ) \text{ and } R_n(z,j) = o_{\P} \left( \frac{ \delta_n }{ (nh_n^d)^{1/2} } \right),
$$
uniformly in  $(z,j) \in \mathcal{Z} \times \{1,...,J\}$. The exact form of $B_n(z,j)$ is given in equation (12) of \cite{Kong/Linton/Xia:10}. The result that $B_n(z,j) = O ( h_n^{p+1} )$ uniformly in $(z,j)$  follows
from the standard argument based on Taylor expansion given in \cite{Fan/Gijbels:96}, \cite{Kong/Linton/Xia:10}, or \cite{Masry:06}.
The condition that $n h_n^{d + 2(p+1)} \rightarrow 0$ at a polynomial rate in $n$ corresponds to the undersmoothing condition.

Now set $\mathbf{K}(z/h) \equiv \mathbf{e}_1'S_p^{-1} K_h( z ) \mathbf{u}_p( z/h )$, which is a kernel of order $(p+1)$ (See Section 3.2.2 of \cite{Fan/Gijbels:96}).
Let $$
g_v(U, Z):=
\frac{e_j' U }{ (h_n^d )^{1/2} f(z) } \mathbf{K} \left( \frac{Z - z}{h_n} \right).
$$
Then it follows from Lemma 15 in Appendix J that uniformly in $v \in \mathcal{V}$
\begin{align*}
(nh_n^d)^{1/2}(\widehat \theta_n(z,j) - \theta_n(z,j)) =  \mathbb{G}_n (g_v)  + o_{\mathrm{P}} (\delta_n).
\end{align*}
Application of Theorems \ref{theorem:strong-approx-kernel} and \ref{theorem:simulate-kernel} in Appendix \ref{sec:proofs-strong-local}, based on the Rio-Massart coupling, verifies condition NK.1 (a) and NK.1 (b). Finally,
Condition NK.2 holds if we take $\hat{f}_n(z)$ to be the standard kernel density estimator with kernel $K$ and let $e_j'\hat{U}_i = Y_i(j) - \widehat \theta_n(z,j)$. \qed \end{example}

\section{Strong Approximation for Kernel-Type Methods}\label{sec:proofs-strong-local}

To establish our strong approximation for kernel-type estimators we use Theorem 1.1 in \cite{Rio:94}, stated below, which builds on the earlier results of \cite{Massart:89}. After the statement of the Rio-Massart coupling we provide our strong approximation result, which generalizes the previous results to kernel-type estimators for regression models with multivariate outcomes. We then provide a novel multiplier method to approximate the distribution of such estimators. The proofs for these results are provided in Appendix \ref{sec:est:local:proofs}.

\subsection{Rio-Massart Coupling}\label{sec:proofs-RM-coupling}
Consider a sufficiently rich probability space $(A, \mathcal{A}, \P)$.  Indeed, we can always enrich an original
 space by taking the product with  $[0,1]$ equipped with the uniform measure
over Borel sets of $[0,1]$.  Consider a suitably measurable, namely image admissible Suslin, function class $\mathcal{F}$ containing functions
$f: I^d \to I$ for $I = (-1,1)$. A function class $\mathcal{F}$ is of uniformly bounded variation of
at most $K(\mathcal{F})$ if
$$
TV(\mathcal{F}) := \sup_{f \in \mathcal{F}} \sup_{ g \in \mathcal{D}_c(I^d) } \(  \int_{\Bbb{R}^d}  f(x) \text{div} g(x)/\|g\|_{\infty} dx \) \leq K(\mathcal{F}),
$$
where $\mathcal{D}_c(I^d)$ is the space of $C^{\infty}$ functions taking values in $\mathbb{R}^d$ with compact support included in $I^d$, and where div$g(x)$ is the divergence of $g(x)$.
Suppose the function class $\mathcal{F}$  obeys the following uniform $L_1$ covering condition
$$
\sup_{Q} N(\epsilon, \mathcal{F}, L_1(Q)) \leq C(\mathcal{F}) \epsilon^{d(\mathcal{F})},
$$
where $\sup$ is taken over probability measures with finite support, and  $N(\epsilon, \mathcal{F}, L_1(Q))$ is the covering number under  the $L_1(Q)$ norm on $\mathcal{F}$.  Let $X_1,...,X_n$ be an i.i.d. sample
on the probability space $(A, \mathcal{A},P)$
from density $f_X$ with support on $I^d$, bounded from above and away from zero. Let $P_X$ be
the measure induced by $f_X$. Then there exists a $P_X$-Brownian Bridge $\mathbb{B}_n$
with a.s. continuous paths with respect to the $L^1(P_X)$ metric such that for any positive $t \geq C\log{n}$,
$$
\P\(  \sqrt{n} \sup_{f \in \mathcal{F}} | \mathbb{G}_n(f) - \mathbb{B}_n(f)| \geq C \sqrt{ t n^{\frac{d-1}{d}} K(\mathcal{F})} + C t \sqrt{\log n}  \) \leq e^{-t},
$$
where constant $C$ depends only on $d$, $C(\mathcal{F})$, and $d(\mathcal{F})$.

\subsection{Strong Approximation for Kernel-Type Estimators}\label{sec:proofs-strong-kernel}

We shall use the following technical condition in what follows.

\noindent\textbf{Condition R.}  \textit{The random $(J + d)$-vector $(U_i, Z_i)$ obeys $U_i = (U_{i,1},...,U_{i,J})=\varphi_n(X_{i,1})$, and $Z_i = \tilde \varphi_n(X_{2i})$,  where $X_i = (X_{1i}', X_{2i}')' $ is a $(d_1+d)$-vector with $1 \leq d_1 \leq J$, which has density bounded away from zero by $\underline{f}$ and above by $\bar f$ on the support $I^{d_1+d}$, where $\varphi_n: I^{d_1} \mapsto I^{J}$ and $\sum_{l=1}^{d_1}\int_{I^{d_1}} | D_{x_{1l}} \varphi_n(x_{1}) | dx_{1} \leq B,$  where
$D_{x_{1l}} \varphi_n(x_1)$ denotes the weak derivative with respect to the $l$-th component of $x_1$,  and
$\tilde \varphi_n: I^{d} \mapsto I^d$ is continuously differentiable such that $\max_{k \leq d}\sup_{x_2} |\partial \tilde \varphi_n(x_{2})/ \partial x_{2k} | \leq B$ and $|\det \partial \tilde\varphi_n (x_2)/\partial x_2|\geq c>0$, where $\partial \tilde \varphi_n(x_{2})/ \partial x_{2k}$ denotes the partial derivative with respect to the $k$-th component of $x_2$.
The constants $J,B, \underline{f}$, $\bar f$, $c$ and vector dimensions do not depend on $n$. ($|\cdot|$ denotes $\ell_1$ norm.)}

A simple example of $(U_i,Z_i)$ satisfying this condition is given in Corollary \ref{corollary:strong:kernel} below.

\begin{theorem}[Strong Approximation for Kernel-Type Estimators] \label{theorem:strong-approx-kernel}Consider a suitably enriched probability space
$(A, \mathcal{A}, \mathrm{P}_n)$ for each $n$. Let $n \rightarrow \infty$.
Assume the following conditions hold for each $n$: (a) There are $n$ i.i.d. $(J+d)$-dimensional random vectors of the form
$(U_i, Z_i)$ that obey Condition R, and  the density $f_n$ of $Z$ is bounded from above and away from zero on the set $\mathcal{Z}$, uniformly in $n$.  (b) Let $v = (z,j)$ and $\mathcal{V} = \mathcal{Z} \times \{1,...,J\}$, where $\mathcal{Z} \subseteq I^d$. The kernel estimator $v \mapsto \widehat \theta_{n}(v)$ of some target function $v \mapsto
\theta_n(v)$ has an asymptotic linear expansion uniformly in $ v \in \mathcal{V}$
\begin{align*}
(nh_n^d)^{1/2}(\widehat \theta_n(v) - \theta_n(v)) =  \mathbb{G}_n (g_v)  + o_{\mathrm{P}_n} (\delta_n),  \ \ g_v(U_i, Z_i):=
\frac{1}{ (h_n^d )^{1/2} f_n(z) } e_j' U_{i} \mathbf{K} \left( \frac{z - Z_i}{h_n} \right),
\end{align*}
where $e_j' U_{i} \equiv U_{ij}$,  $\mathbf{K}$ is twice continuously differentiable product kernel function with support
on $I^d$, $\int \mathbf{K}(u)du=1$, and $h_{n}$ is a
sequence of bandwidths that converges to zero, (c) for a given $\delta_n \searrow 0$, the bandwidth sequence obeys:
$
\left(n^{-1/(d+d_1)} h_{n}^{-1} \log n\right) ^{1/2} + (nh_n^d)^{-1/2}\log^{3/2} n = o(\delta_n).
$
Then there exists a sequence of centered $\mathrm{P}_n$-Gaussian Bridges $\mathbb{B}_n$  such that
\begin{align*}
 \sup_{v\in \mathcal{V}} |(nh_n^d)^{1/2}(\widehat \theta_n(v) - \theta_n(v))  - \mathbb{B}_n(g_v)| =o_{\P_n}
  (\delta_n).
\end{align*}
Moreover, the paths of $v \mapsto \mathbb{B}_n(g_v)$ can be chosen to be continuous a.s.
\end{theorem}

\begin{remark}
Conditions (a) and (b) cover standard conditions in the literature, imposing
a uniform Bahadur expansion for kernel-type estimators,
which have been shown in
\cite{Masry:06} and \cite{Kong/Linton/Xia:10}
for kernel mean regression estimators and also local polynomial
estimators under fairly general conditions. Implicit in the expansion above is that
the asymptotic bias is negligible, which can be achieved by
undersmoothing, i.e. choosing the bandwidth to be smaller than
the rate-optimal bandwidth. \end{remark}

\begin{corollary}[\textbf{A Simple Leading Case for Moment Inequalities
Application}]\label{corollary:strong:kernel}
Suppose that $(U_i, Z_i)$ has bounded support, which we then take to be a subset of  $I^{J+d}$ without loss of generality. Suppose that
$U_i = (U_{ij}, j=1,...,J)$ where for the first $J_0/2$ pairs of terms, we have
$U_{ij} = - U_{i{j+1}}, j =1, 3, ...,J_0-1$.  Let $\mathcal{J} =
\{ 1, 3, ...,J_0-1, J_0+1, J_0+ 2,... \}$.
Suppose that $( U_{ij},Z_i, j \in \mathcal{J})$
have joint density bounded from above and below by some constants $\bar f$ and $\underline f$.
Suppose these constants and $d$, $J$, and $d_1 = | \mathcal{J}|$ do not depend on $n$. Then
Condition R holds, and the conclusions of Theorem \ref{theorem:strong-approx-kernel} then hold under
the additional conditions imposed in the theorem.
\end{corollary}

Note that Condition R allows for much more general error terms and regressors. For example,
it allows error terms $U_i$ not to have a density at all, and $Z_i$ only to have density
bounded from above.

The next theorem shows that the Brownian bridge $\mathbb{B}_n(g_v)$ can be approximately simulated
via the Gaussian multiplier method.  That is, consider the following symmetrized process
 \begin{equation}\label{useful}
\mathbb{G}^o_n(g_v)=  \frac{1}{\sqrt{n}} \sum_{i=1}^n \xi_i g_v(U_i,Z_i)  = \mathbb{G}_n( \xi g_v),
 \end{equation}
where $\xi_1,..., \xi_n$  are i.i.d $N(0,1)$, independent of the data $\mathcal{D}_n$ and of $\{(U_i,Z_i)\}_{i=1}^n $, which
are i.i.d. copies of $(U,Z).$  Conditional on the data this is a Gaussian process
with a covariance function which is a consistent estimate of the covariance function of $v \mapsto
\mathbb{B}_n(g_v)$. The theorem below shows that the uniform distance between a copy of $\mathbb{B}_n(g_v)$ and $\mathbb{G}^o_n(g_v)$
is small with an explicit probability bound.  Note that if the function class $\{g_v, v \in \V\}$ were
Donsker, then such a result would
follow from the multiplier functional central limit theorem.  In our case,
this function class is not Donsker, so we require a different argument.

For the following theorem, consider now a sufficiently rich probability space $(A, \mathcal{A}, \Pn)$.   Note that we can always enrich the original space if needed by taking the product with $[0,1]$ equipped with the uniform measure over Borel sets of $[0,1]$.

\begin{theorem}[\textbf{Multiplier Method for Kernels}]\label{theorem:simulate-kernel}
Let $v = (z,j)$ and $\mathcal{V} \subseteq \mathcal{Z} \times \{1,...,J\}$, where $\mathcal{Z}$ is a compact
convex set that does not depend on $n$. The estimator $v \mapsto \widehat \theta_n(v)$ and the function $v \mapsto \theta_n(v)$ are continuous in $v$.  In what follows, let $e_j$ denote the $J$- vector with $j$th element one and all other elements zero. Suppose
 that $(U, Z)$ is a $(J+d)$-dimensional random vector, where
 $U$ is a generalized residual
 such that $E[U|Z]= 0$ a.s. and $Z$ is a covariate;
 the density $f_n$ of $Z$ is continuous and bounded away from zero and from above on $\mathcal{Z}$, uniformly in $n$;
  and the support of $U$ is bounded uniformly in $n$. $\mathbf{K}$ is a twice continuously differentiable, possibly higher-order, product kernel function with support on $[-1,1]^{d}$, $\int \mathbf{K}(u)du=1$;  and $h_{n}$ is a sequence of bandwidths such that $h_n \rightarrow 0$ and $n h^d \to \infty$ such that $\sqrt{ n^{-1} h^{-2d} } = o( (\delta_n/[\ell_n \sqrt{ \log n}])^{d+1} )$.
  Let $\{(U_i,Z_i)\}_{i=1}^n $ be i.i.d. copies of $(U,Z),$ where $\{Z_i\}_{i=1}^n$ are a part of the data $\D_n$ and $\{ U_i \}$ are a measurable transformation of data. Let $\mathbb{B}_n$ denote the $\P_n$-Brownian bridge, and
$$
g_v(U, Z):=
\frac{e_j' U }{ (h_n^d )^{1/2} f_n(z) } \mathbf{K} \left( \frac{z - Z}{h_n} \right).
$$

Then there exists an independent from data $\D_n$, identically distributed copy $v \mapsto \bar{\mathbb{B}}_n (g_v)$ of the process $v \mapsto \mathbb{B}_n (g_v)$, such that for some $o(\delta_n)$ and $o (1/ \ell_n)$ sequences,
\begin{equation}\label{eq: kernel multiplier approximation}
\Pn\(  \sup_{v \in \V} \left| \Gn^o( g_v) - \barBn(g_v) \right| > o( \delta_n) \Big| \D_n\)=o_{\Pn}(1/\ell_n).
\end{equation}
\end{theorem}

\section{Proofs for Nonparametric Estimation of $\theta(v)$ via Kernel-Type methods}\label{sec:est:local:proofs}

\subsection{Proof of Lemma \ref{lemma:kernels}} There are six steps, with the first four
verifying conditions C.1-C.4, and the last two providing auxiliary calculations.
Let $U_{ij} \equiv e_j'U_i$.

Step 1. Verification of C.1.  Condition C.1 holds by inspection, in view of continuity of
$v \mapsto \hat \theta_n(v)$, $v \mapsto \theta_n(v)$, $v \mapsto \sigma_n(v)$,
 and $v \mapsto s_n(v)$.

Step 2. Verification of C.3.  Note that
$$
\frac{g_v(U_i,Z_i)}{\sigma_n(v) \sqrt{n h_n^d}} = \frac{ \KU}{ \normKU }
$$
We shall employ Lemma \ref{lemma: key concentration}, which
has the required notation in place. We only need to compute an upper bound
on the covering numbers $N(\varepsilon,\textsf{V}, \rho)$ of $\textsf{V}$ under the metric
$\rho(v, \bar v) = \sigma(Z^*_n(v) -Z^*_n(\bar v))$. We have
that for $v = (z,j)$ and $\bar v  = (\bar z, j)$
\begin{eqnarray*}
& & \sigma(Z^*_n(v) - Z^*_n(\bar v) )   \leq  \Upsilon_n \| z - \bar z\|, \\
&& \Upsilon_n:= \sup_{v \in \mathcal{V}, 1 \leq k \leq d} \left \| \nabla_{z_k} \frac{ \KU}{ \normKU }  \right \|_{\Pn,2}.
 \end{eqnarray*}
We have that
\begin{eqnarray*}
\Upsilon_n \leq \sup_{v \in \mathcal{V}, 1 \leq k \leq d}   \frac{ \left \| \nabla_{z_k} \KU \right \|_{\Pn,2} }{  \normKU } +  \frac{ \left |\nabla_{z_k} \normKU \right|}{ \normKU}\text{,}
\end{eqnarray*}
which is bounded by $C (1 + h_n^{-1})$ for large $n$ by Step 6.  Since $J$ is finite, it follows that  for all large $n > n_0$ for all non-empty subsets of $\textsf{V} \subseteq \V$,
$$
N(\varepsilon, \textsf{V},\rho) \leq \left( \frac{J^{1/d} (1+  C (1 + h_n^{-1})  \text{diam}(\textsf{V}))}{\varepsilon}\right)^d, \ \ \ 0 < \varepsilon< 1. $$
Condition C.3 now follows for all $n> n_0$ by Lemma \ref{lemma: key concentration}, with
$$
a_n(\textsf{V}) = (2\sqrt{\log L_n(\textsf{V})  }) \vee (1 +\sqrt{d}),  \ \ L_n(\textsf{V}) =  C'  \( 1 + C (1+ h_n^{-1})\text{diam}(\textsf{V}) \)^d,
$$
where $C'$ is some positive constant.

Step 3. Verification of C.4.  Under Condition NK, we have that
$$
a_n(\textsf{V}) \leq \bar a_n:= a_n(\mathcal{V}) \lesssim \sqrt{ \log \ell_n  + \log n} \lesssim \sqrt{ \log n},
$$
so that C.4(a) follows if $\sqrt{ \log n/(n h_n^d) } \to 0$.

To verify C.4(b) note that
$$
\left |\frac{s_n(v)}{ \sigma_n(v)} -1 \right| =  \Bigg |  \underbrace{\(\frac{f_n(z)}{\hat f_n(z)}\)}_{a} \Bigg ( \underbrace{\frac{\enorm{\KK \hat U_{ij}}}{\norm{\KU}}}_{b/c} \Bigg )  -1 \Bigg |.
$$
Since  $ | a (b/c) -1| \leq  2 |a -1| + |(b -c)/c|$ when $|(b -c)/c| \leq 1$, the result follows from
$|a-1| = O_{\Pn}(n^{-b}) = o_{\Pn}(\delta_n/(\bar a_n + \ell_n))$  holding by NK.2 for some $b>0$ and from
\begin{eqnarray*}
|(b -c)/c| & \leq &  \max_{1 \leq i \leq n} \|\hat U_i - U_i\| \frac{\enorm{\KK}}{\normKU} +  \bigg| \frac{\enormKU}{\normKU}-1 \bigg|,\\
& \leq &  O_{\Pn}(n^{-b}) O_{\Pn}(1) + O_{\Pn}\(\sqrt{\frac{\log n}{ nh^d}}\) = O_{\Pn}(n^{-b}) = o_{\Pn}(\delta_n/(\bar a_n + \ell_n))
\end{eqnarray*}
for some $b>0$ where we used NK.2, the results of Step 6, and the condition that $nh_n^d \to \infty$ at a polynomial rate.

Step 4.  Verification of C.2.  By NK.1
and $1 \lesssim E_{\Pn}[g^2_v] \lesssim 1$ uniformly in $v \in \V$  holding by Step 6 give
$$
\sup_{v \in \V} \left | \frac{\Gn(g_v)}{ \sqrt{E_{\Pn}[g^2_v]}} - \frac{\Bn(g_v)}{ \sqrt{E_{\Pn}[g^2_v]}}\right| = O_{\Pn}(\delta_n),
$$
where $v \mapsto \Bn(g_v)$ is zero-mean $\Pn$-Brownian bridge, with a.s. continuous sample paths. This
and the condition on the remainder term in NK.1 in turn imply C.2(a).

To show C.2(b) we need to show that for any $C>0$
$$
\Pn\( \sup_{v \in \V} \left | \frac{\Gn^o(\hat g_v)}{ \sqrt{\En[\hat g^2_v]}} - \frac{\barBn(g_v)}{ \sqrt{E_{\Pn}[g^2_v]}} \right| > C \delta_n \mid \D_n  \) = o_{\Pn}(1/\ell_n),
$$
where $\barBn$ is a copy of $\Bn$, which is independent of the data. First,
Condition NK.1 with the fact that $1 \lesssim E_{\Pn}[g^2_v] \lesssim 1$ uniformly in $v \in \V$ implies that
$$
\Pn\(  \sup_{v \in \V} \left | \frac{\Gn^o( g_v)}{ \sqrt{E_{\Pn}[g^2_v]}} - \frac{\barBn(g_v)}{ \sqrt{E_{\Pn}[g^2_v]}} \right| > C \delta_n \mid \D_n\)=o_{\Pn}(1/\ell_n).
$$
Therefore, in view of the triangle inequality and the union bound, it remains to show that
\begin{eqnarray}\label{eq: difference small}
\Pn\(  \sup_{v \in \V} \left | \frac{\Gn^o(\hat g_v)}{ \sqrt{\En[\hat g^2_v]}} - \frac{\Gn^o(g_v)}{ \sqrt{E_{\Pn}[g^2_v]}}  \right| > C \delta_n \mid \D_n\)=o_{\Pn}(1/\ell_n).
\end{eqnarray}
We have that
$$
 \sup_{v \in \V} \left| \frac{\Gn^o(\hat g_v)}{ \sqrt{\En[\hat g^2_v]}} - \frac{\Gn^o(g_v)}{ \sqrt{E_{\Pn}[g^2_v]}} \right| \leq
 \sup_{v \in \V} \left| \frac{\Gn^o(\hat g_v - g_v)}{ \sqrt{E_{\Pn}[g^2_v]}}\right| +  \sup_{v \in \V} \left| \frac{\Gn^o(g_v)}{ \sqrt{E_{\Pn}[g^2_v]}}\right| \sup_{v \in \V} \left| \frac{\sigma_n(v)}{s_n(v)} -1 \right|.
$$
We observe that
\begin{eqnarray*}
&& E_{\Pn} \(\sup_{v \in \V} \left| \frac{\Gn^o(g_v)}{ \sqrt{E_{\Pn}[g^2_v]}}\right| \sup_{v \in \V} \left| \frac{\sigma_n(v)}{s_n(v)} -1 \right| \mid \D_n \)  \\
 &&= E_{\Pn} \(\sup_{v \in \V} \left| \frac{\Gn^o(g_v)}{ \sqrt{E_{\Pn}[g^2_v]}}\right|\mid \D_n \) \sup_{v \in \V} \left| \frac{\sigma_n(v)}{s_n(v)} -1 \right|= O_{\Pn}\(\sqrt{\log n} n^{-b}\) = O_{\Pn}(\delta_n/\ell_n),
\end{eqnarray*}
where the last equality follows from Steps 5 and 3.  Also we note that
\begin{eqnarray*}
&& E_{\Pn} \(\sup_{v \in \V} \left| \frac{\Gn^o(\hat g_v - g_v)}{ \sqrt{E_{\Pn}[g^2_v]}}\right|\mid \D_n \)  \\
&&  \leq_{(1)}   O_{\Pn}( \sqrt{\log n})   \sup_{v \in \V} \frac{ \enorm{ \(\frac{U_{ij}}{f_n(z)} - \frac{\hat U_{ij}}{\hat f_n(z)}\)\KK}}{
 \frac{1}{f_n(z)} \normKU}  \\
 && \lesssim_{(2)}   O_{\Pn}( \sqrt{\log n}) \sup_{v \in \V} \frac{\enorm{\|\KK (1 + | U_{ij}|)}}{\norm{\KU}} \( \left | \frac{f_n(z)}{\hat f_n(z)} -1 \right| \vee \max_{1 \leq i \leq n} \| \hat U_{i} - U_{i} \| \) \\
 && \leq_{(3)}   O_{\Pn}( \sqrt{\log n}) O_{\Pn}(1) O_{\Pn}(n^{-b}) = o_{\Pn}(\delta_n/\ell_n),
\end{eqnarray*}
where (1) follows from Step 5, (2) by elementary inequalities, and (3) by Step 6 and NK.2.
It follows that (\ref{eq: difference small}) holds by Markov's Inequality.

Step 5. This step shows that
 \begin{eqnarray}
 E_{\Pn} \(\sup_{v \in \V} \left| \frac{\Gn^o(g_v)}{ \sqrt{E_{\Pn}[g^2_v]}}\right|\mid \D_n \) & = &  O_{\Pn}\(\sqrt{\log n}\) \label{eq: MI1}  \\
E_{\Pn} \(\sup_{v \in \V} \left| \frac{\Gn^o(\hat g_v - g_v)}{ \sqrt{E_{\Pn}[g^2_v]}}\right|\mid \D_n \) &\leq &   O_{\Pn}( \sqrt{\log n}) \times  \nonumber\\
& \times &   \sup_{v \in \V} \frac{ \enorm{\(\frac{U_{ij}}{f_n(z)} - \frac{\hat U_{ij}}{\hat f_n(z)}\)\KK}}{
 \frac{1}{f_n(z)} \normKU}.   \label{eq: MI2}
 \end{eqnarray}

To show (\ref{eq: MI1})  we use Lemma \ref{lemma: maximal ineq gaussian} applied to
$X_v = \frac{\Gn^o(g_v)}{ \sqrt{E_{\Pn}[g^2_v]}}$ conditional on $\D_n$.  First, we compute
$$
\sigma(X) = \sup_{v \in \V} \left( E_{\P_n} (X_v^2 |\D_n) \right)^{1/2} = \sup_{v \in \V} \frac{ \enorm{\KK U_{ij}}}{
  \normKU} =1 + o_{\Pn}(1),
$$
where the last equality holds by Step 6.
Second, we observe that for $v =(z,j)$ and $\bar v = (\bar z, j)$
\begin{eqnarray*}
& & \sigma(X_v - X_{\bar v} )   \leq  \Upsilon_n \| z - \bar z\|, \ \ \   \Upsilon_n:= \sup_{v \in \mathcal{V}, 1 \leq k \leq d} \left \| \nabla_{z_k} \frac{ \KU}{ \normKU }  \right \|_{\EPn,2}.
 \end{eqnarray*}
We have that
\begin{eqnarray*}
\Upsilon_n && \leq \sup_{v \in \mathcal{V}, 1 \leq k \leq d}   \frac{ \left \| \nabla_{z_k} \KK  U_{ij} \right \|_{\EPn,2} }{  \normKU }
\\ && + \frac{  \enorm{\KK  U_{ij}} }{ \normKU} \cdot  \frac{ \left |\nabla_{z_k} \normKU  \right|}{ \normKU}.
\end{eqnarray*}
which is bounded with probability converging to one by $C (h_n^{-1} +1)$ for large $n$ by Step 6 and NK.2. Since $J$ is finite, it follows that  for all large $n > n_0$, the covering number
for $\V$ under $\rho(v, \bar v) = \sigma(X_v - X_{\bar v})$ obeys with probability converging to 1,
$$
N(\varepsilon, \V,\rho) \leq \left(\frac{J^{1/d}(1+  C (1 + h_n^{-1})  \text{diam}(\V))}{\varepsilon}\right)^d, \ \  0 < \varepsilon< \sigma(X),
$$
Hence
$
\log N(\varepsilon, \V,\rho)  \lesssim \log n + \log(1/\varepsilon).
$
Hence  by Lemma \ref{lemma: maximal ineq gaussian}, we have that
$$
E_{\Pn}\( \sup_{v \in \V} |X_v| \mid \D_n\)  \leq \sigma(X) +  \int_0^{2 \sigma(X)} \sqrt{ \log (n/\varepsilon)} d \varepsilon  = O_{\P_n}(\sqrt{ \log n}).
$$

To show (\ref{eq: MI2}) we use Lemma \ref{lemma: maximal ineq gaussian} applied to
$X_v = \frac{\Gn^o(\hat g_v - g_v)}{ \sqrt{E_{\Pn}[g^2_v]}}$ conditional on $\D_n$.  First, we compute
$$
\sigma(X) = \sup_{v \in \V} \left( E_{\P_n} (X_v^2 |\D_n) \right)^{1/2} = \sup_{v \in \V} \frac{ \enorm{\(\frac{U_{ij}}{f_n(z)} - \frac{\hat U_{ij}}{\hat f_n(z)}\)\KK}}{
 \frac{1}{f_n(z)} \normKU}.
$$
Second, we observe that for $v =(z,j)$ and $\bar v = (\bar z, j)$
$$
 \sigma(X_v - X_{\bar v} )   \leq  (\Upsilon_n + \hat \Upsilon_n) \| z - \bar z\|,
$$
where
$$
 \hat \Upsilon_n:= \sup_{v \in \mathcal{V}, 1 \leq k \leq d} \left \| \nabla_{z_k} \frac{ \KK \hat U_{ij}}{ \normKU }  \right \|_{\EPn,2},
$$
and $\Upsilon_n$ is the same as defined above.

We have that
\begin{eqnarray*}
\hat \Upsilon_n && \leq \sup_{v \in \mathcal{V}, 1  \leq k \leq d}   \frac{ \left \| \nabla_{z_k} \KK \hat U_{ij} \right \|_{\EPn,2} }{  \normKU } \\
&& + \frac{  \enorm{\KK \hat U_{ij}} }{ \normKU} \cdot  \frac{ \left |\nabla_{z_k} \normKU  \right|}{ \normKU}
\end{eqnarray*}
The first term is bounded by
$$
\sup_{v \in \mathcal{V}, 1 \leq k \leq d}   \frac{ \left \| \nabla_{z_k} \KK U_{ij} \right \|_{\EPn,2} }{  \normKU } + \max_{1 \leq i \leq n} \|\hat U_{i} - U_{i}\|  \sup_{v \in \mathcal{V}, 1 \leq k \leq d}   \frac{ \left \| \nabla_{z_k} \KK  \right \|_{\EPn,2} }{  \normKU }
$$
which is bounded by $C (1 + h_n^{-1}) + O_{\Pn}(n^{-b}) O_{\Pn}(1)$ for large $n$ by Step 6 and NK.2.
In the second term, the left term of the product is bounded by
$$
\sup_{v \in \mathcal{V}, 1 \leq k \leq d}   \frac{ \left \| \KK U_{ij} \right \|_{\EPn,2} }{  \normKU } + \max_{1 \leq i \leq n} \|\hat U_{i} - U_{i}\|  \sup_{v \in \mathcal{V}, 1 \leq k \leq d}   \frac{ \left \|  \KK  \right \|_{\EPn,2} }{  \normKU }
$$
which is bounded by $C (1 + o_{\Pn}(1)) + O_{\Pn}(n^{-b}) O_{\Pn}(1)$ for large $n$ by Step 6 and NK.2; the right
term of the product is bounded by $C(1 + h_n^{-1} + o_{\Pn}(1))$ by Step 6.  Conclude that
$\hat \Upsilon_n \leq C (1 + h_n^{-1})$ for some constant $C>0$ with probability converging to one.

Since $J$ is finite, it follows that  for all large $n > n_0$, the covering number
for $\V$ under $\rho(v, \bar v) = \sigma(X_v - X_{\bar v})$ obeys with probability converging to 1,
$$
N(\varepsilon, \V,\rho) \leq \left( \frac{J^{1/d} (1+  C (1 + h_n^{-1})  \text{diam}(\V))}{\varepsilon}\right)^d, 0 < \varepsilon< \sigma(X),
$$
Hence
$$
\log N(\varepsilon, \V,\rho)  \lesssim \log n + \log(1/\varepsilon).
$$
Hence  by Lemma \ref{lemma: maximal ineq gaussian}, we have that
$$
E_{\Pn}\( \sup_{v \in \V} |X_v| \mid \D_n\)  \lesssim \sigma(X) +  \int_0^{2 \sigma(X)} \sqrt{ \log (n/\varepsilon)} d \varepsilon  = O_{\P_n}(\sqrt{ \log n}) \sigma(X).
$$

Step 6.  The claim of this step are the following relations:
uniformly in $ v \in \mathcal{V}, 1 \leq k \leq d$
\begin{eqnarray}
& & 1 \ \ \  \lesssim {\norm{\KK U_{ij}}} \lesssim 1 \label{rel: K1} \\
&& 1 \ \ \ \lesssim {\norm{\KK }} \lesssim 1 \\
&& h_n^{-1} \lesssim {\norm{\nabla_{z_k} \KK U_{ij}}}  \lesssim h_n^{-1} \\
&& h_n^{-1} \lesssim {\norm{ \nabla_{z_k} \KK }}  \lesssim h_n^{-1} \\
&& h_n^{-1} \lesssim \left |\nabla_{z_k} \normKU  \right | \lesssim h_n^{-1}  \label{rel: K5}
\end{eqnarray}
and
\begin{eqnarray}
 \frac{\enorm{\KK U_{ij}}}{\norm{\KK U_{ij}}} & = & 1+ O_{\Pn}\(\sqrt{\frac{\log n}{nh_n^d}} \)  \label{rel: K6} \\
 \frac{\enorm{\KK }}{\norm{\KK }} & = &  1+ O_{\Pn}\(\sqrt{\frac{\log n}{nh_n^d}} \) \\
  \frac{\enorm{\nabla_{z_k} \KK U_{ij}}}{\norm{\nabla_{z_k} \KK U_{ij}}} & = & 1+ O_{\Pn}\(\sqrt{\frac{\log n}{nh_n^{d}}} \)  \\
  \frac{\enorm{ \nabla_{z_k} \KK }}{\norm{ \nabla_{z_k} \KK }} & = &  1+ O_{\Pn}\(\sqrt{\frac{\log n}{nh_n^{d}}} \)   \label{rel: K9}
\end{eqnarray}

The proofs of (\ref{rel: K1})-(\ref{rel: K5}) are all similar to one another, as are those of (\ref{rel: K6})-(\ref{rel: K9}), and are standard in the kernel estimator literature.
We therefore prove only (\ref{rel: K1}) and (\ref{rel: K6}) to demonstrate the argument.
To establish (\ref{rel: K1}) we have
\begin{eqnarray*}
{\norm{\KK U_{ij}}}^2 & = &  h_n^{-d}\int \mathbf{K}^2( (z- \bar z)/h_n) E[U_{ij}^2|\bar z] f_n(\bar z) d \bar z \\
 & \leq_{(1)} &  h_n^{-d} \int\mathbf{ K}^2( (z- \bar z)/h_n) C d \bar z  \leq_{(2)}  \int \mathbf{K}^2(u) C du,
\end{eqnarray*}
for some constant $ 0< C< \infty$, where in (1) we use the assumption that $E[U_{ij}^2|z]$ and $f_n(z)$ are bounded uniformly
from above and in (2), change of variables. On the other hand,
\begin{eqnarray*}
{\norm{\KK U_{ij}}}^2 & = & h_n^{-d} \int \mathbf{K}^2( (z- \bar z)/h_n) E[U_{ij}^2|\bar z]  f_n(\bar z) d \bar z \\
 & \geq_{(1)} &  h_n^{-d} \int\mathbf{ K}^2( (z- \bar z)/h_n) C d \bar z
  \geq_{(2)} \int \mathbf{K}^2(u) C du,
\end{eqnarray*}
for some constant $ 0< C< \infty$, where in (1) we use the assumption that $E[U_{ij}^2|z]$ and $f_n(z)$ are bounded away from
zero uniformly in $n$, and in (2), change of variables.

Moving to (\ref{rel: K6}), it suffices to show that uniformly in $v \in \mathcal{V}$,
$$
\En\(\(\KK\)^2 U^2_{ij}\)- E_{\Pn}\(\(\KK\)^2 U^2_{ij}\) = O_{\Pn}\(\sqrt{\frac{\log n}{nh_n^d}}\),
$$
or equivalently
\begin{align}\label{hk-lemma}
\En\( \mathbf{K}^2 \( \frac{z - Z_i}{h_n}\) U^2_{ij}\)- E_{\Pn}\( \mathbf{K}^2 \( \frac{z - Z_i}{h_n}\) U^2_{ij}\) = O_{\Pn}\(\sqrt{\frac{h_n^d \log n}{n}}\).
\end{align}
Given the boundedness of $U_{ij}$ imposed by Condition R, this is in fact a standard result on local
 empirical processes, using Pollard's empirical process methods. Specifically, (\ref{hk-lemma}) follows by the application of Theorem 37 in chapter II of \cite{Pollard:84}.
 \qed

\subsection{Proof of Lemma \ref{lemma: kernels-v-est}}
To show claim (1), we need to establish that for
$$
\varphi_n = o(1) \cdot \(\frac{h_n}{ \sqrt{\log n}}\),
$$
for any $o(1)$ term, we have that
$$
\sup_{\|v - v' \| \leq \varphi_n} | Z^*_n(v) - Z_n^*(v') | = o_{\Pn}(1)\text{.}
$$

Consider the stochastic process $X =\{Z_n(v), v \in \V\}$. We shall use the standard maximal inequality
stated in Lemma \ref{lemma: maximal ineq gaussian}. From the proof of Lemma \ref{lemma:kernels} we have
that for $v = (z,j)$ and $v' = (z',j)$, $
\sigma(Z^*_n(v) - Z^*_n(v') ) \leq    C (1 + h_n^{-1}) \| z - z'\|,$
where $C$ is some constant that does not depend on $n$, and
$
\log N(\varepsilon, \textsf{V},\rho) \lesssim \log n + \log (1/\varepsilon).
$
Since
$$
\|v - v' \| \leq \varphi_n \implies \sigma(Z^*_n(v) - Z^*_n(v') ) \leq  C   \frac{o(1)}{\sqrt{\log n}}\text{,}
$$
we have
$$
E \sup_{\|v - v' \| \leq \varphi_n  } | X_v - X_{v'}| \lesssim \int_0^{C   \frac{o(1)}{\sqrt{\log n}} }\sqrt{ \log (n/\varepsilon) } d \varepsilon
\lesssim  \frac{o(1)}{\sqrt{\log n}} \sqrt{ \log n} = o(1).
$$
Hence the conclusion follows from Markov's Inequality.

Under Condition V by lemma \ref{lemma: estimation of V}
$$
r_n \lesssim \( \sqrt{\frac{\log n}{ nh_n^d}\log n} \)^{1/\rho_n} c_n^{-1},
$$
so $r_n = o(\varphi_n)$ if
\begin{equation*}
\( \sqrt{\frac{\log n}{ nh_n^d}\log n} \)^{1/\rho_n}c_n^{-1} = o\( \frac{h_n}{ \sqrt{\log n} } \).
\end{equation*}
Thus, Condition S is satisfied.
\qed

\subsection{Proof of Theorem \ref{theorem:strong-approx-kernel}.}
To prove this theorem, we use the Rio-Massart coupling. First we note that
$$
\mathcal{M}=\{ h^{d/2}_n f_n(z) g_v(U_i,Z_i) = e_j'U_i\mathbf{K}((z-Z_i)/h_n),  z \in \mathcal{Z}, j \in\{1,...,J\} \}
$$
is the product of $\{ e_j'U_i, j \in {1,...,J}\}$ with covering number trivially bounded above by $J$ and
$\mathcal{K}:=\{\mathbf{K}((z-Z_i)/h_n), z \in \mathcal{Z}\}$ obeys $\sup_{Q} N(\epsilon, \mathcal{K}, L_1(Q)) \lesssim \epsilon^{-\nu}$
for some finite constant $\nu$; see Lemma 4.1 of \cite{Rio:94}. Therefore, by Lemma A.1 in \cite{Ghosal/Sen/vanderVaart:00}, we have that
\begin{equation}\label{eq: entropy bound}
\sup_{Q} N(\epsilon, \mathcal{M}, L_1(Q)) \lesssim  J (\epsilon/2)^{-\nu} \lesssim  \epsilon^{-\nu}.
\end{equation}
Next we bound, for $  \mathbf{K}_{l}(u) = \partial \mathbf{K}(u)/\partial u_l$
\begin{eqnarray*}
  TV(\mathcal{M})   & &   \leq \sup_{f \in \mathcal{M}} \int | D_{(x'_1,x'_2)'} f(x_1,x_2) |  dx_1 dx_2 \\
&& \leq \sup_{v \in \mathcal{V}} \int_{I^{d}} \int_{I^{d_1}}   \Bigg (
\sum_{l=1}^{d_1}|e_j'D_{x_{1l}}\varphi_n(x_{1})  \mathbf{K}((z- \tilde \varphi_n(x_2))/h_n)|   \\
&  & +
\sum_{l=1}^{d}|  e_j'\varphi_n(x_1) \mathbf{K}_l ((z-\tilde \varphi_n(x_2))/h_n) h_n^{-1} \partial \tilde \varphi(x_2)/\partial x_{2k} | \Bigg) dx_1 dx_2  \\
 & &  \leq C \max_{1 \leq l \leq n} \sup_{v \in \mathcal{V}}  \int_{I^{d}} \Bigg ( |\mathbf{K}((z- \tilde \varphi_n(x_2))/h_n)| +
 h_n^{-1} |   \mathbf{K}_{l} ((z- \tilde \varphi_n(x_2))/h_n) | B  \Bigg) d x_2  \\
 && \leq C h_n^d + C  h_n^{-1} h_n^{d} \leq C h_n^{d-1} =: K(\mathcal{M})
\end{eqnarray*}
where $C$ is a generic constant, possibly different in different places, and where we rely on
$$
\int_{I^{d_1}} |D_{x_{1l}}\varphi_n(x_{1})| d x_1 \leq B, \ \ \sup_{x_1} |e_j'\varphi_n(x_1)| \leq B, \ \ \sup_{x_2} |\partial \tilde \varphi(x_2)/\partial x_{2k}|  \leq B
$$
as well as on
$$
\int_{I^{d}}  |\mathbf{K}((z- \tilde \varphi_n(x_2))/h_n)| d x_2 \leq C h^d, \ \  \int_{I^{d}} |   \mathbf{K}_{l} ((z- \tilde \varphi_n(x_2))/h_n) | d x_2   \leq C h^d.
$$
To see how the latter relationships holds, note that $Y=\tilde \varphi_n(v)$ when $ v\sim U(I^d)$ has a density bounded uniformly from above:
$
f_Y(y) \lesssim 1/|  \det \partial \tilde \varphi_n(v)/\partial v |   \lesssim 1/c.$
Moreover, the functions $|\mathbf{K}((z- y)/h_n)|$ and  $|\mathbf{K}_{l} ((z- y)/h_n)|$
are bounded above by some constant $\bar K$ and are non-zero only over a $y$ belonging to cube centered at $z$ of volume $(2h)^d$.  Hence
$$
\int_{I^{d}}  |\mathbf{K}((z- \tilde \varphi_n(x_2))/h_n)| d x_2 \leq \int_{I^{d}} |\mathbf{K}((z- y)/h_n)| f_Y(y) dy \leq
\bar K (2h)^d (1/c) \leq C h^d,
$$
and similarly for the second term.

By the Rio-Massart coupling we have that for some constant $C$ and  $t \geq C\log n$:
$$
\mathrm{P}_n\(  \sqrt{n} \sup_{f \in \mathcal{M}} | \mathbb{G}_n(f) -  \mathbb{B}_n(f)| \geq C \sqrt{ t n^{\frac{d+d_1-1}{d+d_1}} K(\mathcal{M})} + C t \sqrt{\log n}  \) \leq e^{-t},
$$
which implies that
$$
\mathrm{P}_n\( \sup_{v \in \mathcal{V}} | \mathbb{G}_n(g_v) -  \mathbb{B}_n(g_v)| \geq n^{-1/2}C \sqrt{ t n^{\frac{d+d_1-1}{d+d_1}} h_n^{d-1}} h_n^{-d/2} +  n^{-1/2} h_n^{-d/2} C t \sqrt{\log n}  \) \leq e^{-t},
$$
which  upon inserting $t= C\log n$ gives
\begin{align*}
 \P_n\( \sup_{v\in \mathcal{V}} |\mathbb{G}_n(g_v) - \mathbb{B}_n(g_v)| \geq C
  \left[ n^{-1/2(d+d_1)}\left(h_{n}^{-1} \log n\right) ^{1/2} + (nh_n^d)^{-1/2}\log^{3/2} n\right] \) \lesssim 1/n.
\end{align*}
This implies the required conclusion. Note that  $g_v \mapsto \mathbb{B}_n(g_v)$ is continuous under the $L_1(f_X)$ metric by the Rio-Massart coupling, which implies continuity of $v \mapsto  \mathbb{B}_n(g_v)$, since $v - v' \to 0$ implies $g_v - g_{v'} \to 0$ in the $L_1(f_X)$ metric. \qed

\subsection{Proof of Theorem \ref{theorem:simulate-kernel}}
Step 1. First we note that  (\ref{eq: kernel multiplier approximation}),
is implied by
\begin{equation}\label{eq: main multiplier}
E_{\Pn}\(  \sup_{v \in \V} \left| \Gn^o( g_v) - \barBn(g_v) \right| \Big| \D_n\)=o_{\Pn}(\delta_n/\ell_n)
\end{equation}
in view of the Markov inequality.  Using calculations similar to those in Step 5 in the proof of Lemma \ref{lemma:kernels}, we can conclude that for $X_v:= \Gn^o( g_v)$, $v=(z,j)$, and $\bar{v}=(\bar{z},j)$,
$$
\sigma(X_v - X_{\bar v}) \leq \|v - \bar v\| O_{\Pn}(1 + h_n^{-1}) \text{,}
$$
where $\sigma^2(X_v - X_{\bar v}) := E_{P_n}((X_v - X_{\bar v})^2| \D_n)$. Application of the Gaussian maximal inequality quoted in Lemma \ref{lemma: maximal ineq gaussian}, similarly Step 6 in the proof of Lemma \ref{lemma:kernels}, then gives:
\begin{equation}\label{eq: discr1}
E_{\Pn} \( \sup_{ \|v - \bar v\| \leq \varepsilon } |X_v - X_{\bar v}|\mid \D_n \) = \varepsilon O_{\Pn} \( (1 + h_n^{-1}) \sqrt{\log n}\),
\end{equation}
where
$$
\varepsilon \propto o \( \frac{\delta_n/\ell_n}{ h_n^{-1} \sqrt{\log n}}  \),
$$
whence
\begin{equation}\label{eq: discr2}
E_{\Pn} \( \sup_{ \|v - \bar v\| \leq \varepsilon } |X_v - X_{\bar v}|\mid \D_n \) = o_{\Pn} ( \delta_n/\ell_n).
\end{equation}
Next we setup a regular  mesh  $\V_0 \subset \mathcal{V}$ with mesh width
$ \varepsilon$.   The cardinality
of the mesh is given by
$$
K_n \propto (1/\varepsilon)^d \propto h_n^{-d} \lambda^d_n, \ \ \lambda_n= \frac{\sqrt{\log n}} {o(\delta_n/\ell_n)}.
$$
With such mesh selection, we have that
\begin{equation}\label{eq: discretize multiplier}
E_{\Pn} \( \sup_{ v \in \V}  |X_v - X_{\pi(v)}| \mid \D_n \) \leq   E_{\Pn} \( \sup_{ \|v - \bar v\| \leq \varepsilon } |X_v - X_{\bar v}|\mid \D_n \)  \leq  o_{\Pn} ( \delta_n/\ell_n),
\end{equation}
where $\pi(v)$ denotes a point in $\V_0$ that is closest to $v$.

The steps given below will show that there is a Gaussian process $\{Z_v, v \in \V\}$,
which is independent of $\D_n$, having the same law as $\{\mathbb{B}_n(g_v), v \in \V\}$,
and having the following two key properties:
\begin{eqnarray}\label{eq: approximate discreitized multiplier}
& & E_{\Pn} \( \sup_{ v \in \V_0}  |X_v - Z_{v}| \mid \D_n \) = o_{\Pn} ( \delta_n/\ell_n),\\
& & E_{\Pn} \( \sup_{ v \in \V}  |Z_v - Z_{\pi(v)}| \) = o ( \delta_n/\ell_n).
\label{eq: discretize gaussian}
\end{eqnarray}

The claim of the lemma then follows by setting
$\{\bar{\mathbb{B}}_n(g_v), v \in \V\} = \{Z_v, v \in \V\}$, and then
noting that
$$
E_{\Pn} \( \sup_{ v \in \V}  |X_v - Z_{v}| \mid \D_n \)= o_{\Pn} ( \delta_n/\ell_n)
$$
holds by the triangle inequality for the sup norm
and (\ref{eq: discretize multiplier})-(\ref{eq: discretize gaussian}). Note that
the last display is equivalent to  (\ref{eq: main multiplier}). We now prove these assertions in the followings steps.

Step 2.  In this step we construct the process $Z_v$ on points $v \in \V_0$, and
show that (\ref{eq: approximate discreitized multiplier}) holds.  In what follows,
we use the notation $(X_v)_{v \in \V_0}$ to denote a $K_n$ vector collecting $X_v$ with indices
$v \in \V_0$. We have that conditional on the data $\D_n$,
$$
(X_v)_{v \in \V_0} = \hat \Omega^{1/2}_n \mathcal{N},  \ \ \mathcal{N} \sim N(0,I),
$$
where $\mathcal{N}$ is independent of $\D_n$, and
$$
\hat \Omega = \En[ p_i p_i'] \ \textrm{ and }  \  p_i = ( g_v(U_i, Z_i) )_{v \in \mathcal{V}_0}.
$$
We then set $(Z_v)_{v \in \V_0} = \Omega^{1/2}_n \mathcal{N}$ for $\Omega = E_{\Pn}[ p_i p_i']$ and
the same $\mathcal{N}$ as defined above.

Before proceeding further, we note that by construction the process $\{Z_v, v \in \V_0\}$
is independent of the data $\mathcal{D}_n$.  This is facilitated by suitably enlarging the
the probability space as needed.\footnote{Given the space $(A', \mathcal{A}', \Pn')$ that carries $\mathcal{D}_n$ and given
a different space $(A'', \mathcal{A}'', \Pn'')$ that carries $\{Z_v, v \in \V_0\}$ as
well as its complete version $\{Z_v, v \in \V\}$, we can take $(A, \mathcal{A}, \Pn)$ as the
product of the two spaces, thereby maintaining independence between the data and the constructed process.
Since $\{Z_v, v \in \V\}$ constructed below takes values in a separable metric space, it suffices
to take $(A'', \mathcal{A}'', \Pn'')$ as the canonical probability space, as noted in Appendix A.}

Since the support of $\mathbf{K}$ is compact and points of the grid $\V_0$ are equally spaced, we have that
$$N_i := |\{ v \in \mathcal{V}_0: g_v(U_i,Z_i) \neq 0\}| \lesssim (h_n/\varepsilon)^d \lesssim \lambda^d_n.$$
Using the boundedness assumptions of the lemma, we have that
$$\| p_i\|  \leq  (\bar U/\underline{f}) \sqrt{N_i}/h^{d/2}_n \lesssim (\lambda_n/h_n)^{d/2},$$
where $\bar U$ is the upper bound on $U$ and $\underline{f}$ is the lower bound on the density $f_n$,
both of which do not depend on $n$.

The application of Rudelson's Law of Large Numbers for operators, \cite{Rudelson99}, yields
$$
E_{\Pn} \|\hat \Omega_n - \Omega_n\| \lesssim \sqrt{\log n/(n (h_n/\lambda_n)^d)}.
$$
The application of the Gaussian maximal inequality quoted in Lemma \ref{lemma: maximal ineq gaussian} gives:
$$
E_{\Pn} \( \sup_{ v \in \V_0}  |X_v - Z_{v}| \mid \D_n \) \lesssim \sqrt{ \log K_n} \max_{v \in \V_0} \sigma(X_v -Z_v).
$$
Since $(X_v)_{v \in \V_0} - (Z_v)_{v \in \V_0} =  (\hat \Omega^{1/2}_n - \Omega_n^{1/2})' \mathcal{N}$, we have that:
$$
\max_{v \in \V_0} \sigma(X_v -Z_v)^2 \leq \| (\hat \Omega^{1/2}_n - \Omega_n^{1/2})^2 \| \leq \|  \hat \Omega^{1/2}_n - \Omega_n^{1/2}\|^2
\leq K_n \| \hat \Omega_n- \Omega_n\|^2,
$$
where the last inequality follows by a useful matrix inequality derived in \cite{Chetverikov:11}.
Putting bounds together and using $\log K_n \lesssim \log n$
\begin{eqnarray}
E_{\Pn} \( \sup_{ v \in \V_0}  |X_v - Z_{v}| \mid \D_n \) & = &  O_{\Pn} \( \sqrt{ \log n} \sqrt{K_n} \sqrt{\log n/(n (h_n/\lambda_n)^d)}\)
\\
& = & O_{\Pn}( \log n/\sqrt{n (h_n/\lambda_n)^{2d}}) = o_{\Pn}(\delta_n/\ell_n),
\end{eqnarray}
where the last condition holds by the conditions on the bandwidth.

Step 3. In this step we complete the construction of the process $\{Z_v, v \in \V\}$.
We have defined the process $Z_v$ for all $v \in \V_0$.    We want to embed these random
variables into a path of a Gaussian process $\{Z_v, v \in \V\}$, whose covariance function
is given by  $(v, \bar v) \mapsto E_{Pn}[ g_v g_{\bar v}]$. We want to maintain
the independence of the process from $\D_n$.  The construction follows by Lemma 11 in
\cite{Belloni/Chernozhukov/Fernandez-Val:10}.  This lemma requires
that a version of $\{Z_v, v \in \V\}$ has a.s. uniformly continuous sample paths, which follows from the Gaussian maximal inequalities and entropy calculations similar
to those given in Step 3 in the proof of Lemma \ref{lemma:kernels}. Indeed,
we can conclude that
$$
\sigma(Z_v - Z_{\bar v}) \leq \|v - \bar v\| C(1 + h_n^{-1}),
$$
which establishes total boundedness of $\mathcal{V}$ under the standard deviation pseudometric.
Moreover, application of the Gaussian maximal inequality Lemma \ref{lemma: maximal ineq gaussian}
 to $Z_v$ gives:
$$
E_{\Pn} \( \sup_{ \|v - \bar v\| \leq \varepsilon } |Z_v - Z_{\bar v}| \) \leq C \varepsilon \( (1 + h_n^{-1}) \sqrt{\log n}\).
$$
By a standard argument,  e.g. \cite{VanDerVaart/Wellner:96}, these facts imply that the paths of $Z_v$ are a.s. uniformly continuous.\footnote{Note however that the process depends on $n$, and the statement here is a non-asymptotic statement, holding for any $n$.  This property should not be confused with asymptotic equicontinuity, which does not hold here.}

The last claim (\ref{eq: discretize gaussian}) follows from the preceding display, the choice of meshwidth $\varepsilon$, and the inequality:
$$
E_{\Pn} \( \sup_{ v \in \mathcal{V}} |Z_v - Z_{\pi(v)}| \) \leq E_{\Pn} \( \sup_{ \|v - \bar v\| \leq \varepsilon } |Z_v - Z_{\bar v}| \) = o(\delta_n/\ell_n).
$$

\qed

\section{Asymptotic Linear Representation for Series Estimator of a Conditional Mean}\label{sec:prac:roymodel}

In this section we use the primitive conditions set out in Example 5 of the main text to verify the required asymptotically linear representation for $\sqrt{n}(\widehat \beta_n - \beta_n)$ using \cite{Newey:97}. This representation is also Condition (b) of Theorem \ref{theorem:strong series}. We now reproduce the imposed conditions from the example for clarity. We note that it is also possible to develop similar conditions for nonlinear estimators, see for example Theorem 1(d) of \cite{Horowitz/Mammen:04}.

We have that $\theta_n(v) = E_{\Pn}[Y_i | V_i =v]$,
assumed to be a continuous function.  There is an i.i.d. sample $(Y_i, V_i), i=1,...,n$, with $\V  \subseteq \text{support}(V_i) \subseteq [0,1]^d$
for each $n$. Here $d$ does not depend on $n$, but all
other parameters, unless stated otherwise, can depend on $n$.  Then
we  have  $\theta_n(v) = p_{n}(v)'\beta_{n} + A_{n}(v)$, for $p_{n}: \text{support}(V_i) \mapsto \Bbb{R}^{K_n}$ representing
the series functions;  $\beta_{n}$
is the coefficient of the best least squares approximation to $\theta_n(v)$ in the population,
and $A_{n}(v)$ is the approximation error.  The number of series terms $K_n$ depends on $n$.

Recall that we have imposed the following technical conditions in the main text:
\begin{quote}
Uniformly in $n$, (i)  $p_{n}$ are either B-splines of a fixed order or trigonometric series terms or any
other terms $p_n = (p_{n1},\ldots,p_{nK_n})'$ such that $\|p_n(v)\| \lesssim \zeta_n = \sqrt{K_n}$ for all $v \in \text{support}(V_i)$,
$\|p_n(v)\| \gtrsim   \zeta_n' \geq 1$ for all $v \in \mathcal{V}$, and $\log \textrm{lip}(p_{n}) \lesssim \log K_n$, (ii) the mapping $v \mapsto \theta_n(v)$
  is sufficiently smooth, namely $\sup_{v \in \mathcal{V}}|A_{n}(v)| \lesssim K_n^{-s}$, for some $s>0$,  (iii) $\lim_{n \to \infty }(\log n)^c\sqrt{n}K_n^{-s} = 0$ for each $c>0$,\footnote{This condition,
which is based on \cite{Newey:97} can be relaxed to $(\log n)^c K_n^{-s + 1} \to 0$ and $(\log n)^c \sqrt{n} K^{-s}_n/\zeta_n' \to 0$, using the
recent results of \cite{Belloni/Chen/Chernozhukov:11} for least squares series estimators.} (iv) for  $\epsilon_i = Y_i - E_{\Pn}[Y_i|V_i]$,
  $E_{\Pn}[\epsilon_i^2|V_i=v]$ is bounded away from zero  uniformly in $v \in \text{support}(V_i)$, and (v)
  eigenvalues of $Q_{n} = E_{\Pn}[p_{n}(V_i) p_{n}(V_i)']$ are bounded away from zero and from above, and (vi) $E_{\Pn}[|\epsilon_i|^4 |V_i=v]$ is bounded from above uniformly in $v \in \text{support}(V_i)$,  (vii) $\lim_{n \to \infty}(\log n)^c K_n^5/n =0$ for each $c>0$.  \end{quote}

We impose Condition (i) directly through the choice of basis functions.  Condition (ii) is a standard condition on the error of the series approximation, and is the same as Assumption A3 of \cite{Newey:97}, also used by \cite{Chen:07}.
Condition (v) is Assumption 2(i) of \cite{Newey:97}.
The constant $s$ will depend on the choice of basis functions. For example, if splines or power series are used, then $s = \alpha/d$, where $\alpha$ is the number of continuous derivatives of
$\theta_n\left( v \right) $ and $d$ is the dimension of $v$.
Restrictions on $K_n$ in conditions (iii) and (vii) require that  $\alpha > 5d/2$.
Conditions (i), (vi), and (vii) and Theorem 7, namely Corollary 1, ensure that the constraint on the growth rate for the number of series terms is satisfied.

Define $S_n \equiv E[ \epsilon_i^2 p_{n}(V_i) p_{n}(V_i)']$
and $\Omega_n \equiv  Q_n^{-1} S_n Q_n^{-1}$. Arguments based on  \cite{Newey:97} give the following lemma,
which verifies the linear expansion required in condition (b) of Theorem \ref{theorem:strong series} with $\delta_n = 1/\log n$.

\begin{lemma}[\textbf{Asymptotically Linear Representation of Series Estimator}]\label{lemma: asym_lin_seris} Suppose
conditions (i)-(vii) hold. Then we have the following asymptotically linear representation:
\begin{align*}
\Omega_n^{-1/2} \sqrt{n}(\widehat \beta_n  - \beta_n ) =  \Omega_n^{-1/2} Q_{n}^{-1} \frac{1}{\sqrt{n}} \sum_{i=1}^n  p_{n}(Z_i) \epsilon_i + o_{\Pn}(1/\log n).
\end{align*}
\end{lemma}

\begin{proof}[Proof of Lemma \ref{lemma: asym_lin_seris}]
As in \cite{Newey:97}, we have the following representation: with probability approaching one,
\begin{align}\label{series-bahadur1}
\widehat \beta_n - \beta_n &=  n^{-1}  \widehat Q_n^{-1} \sum_{i=1}^n p_{n}(V_i) \epsilon_i + \nu_n,
\end{align}
where $ \widehat Q_n \equiv \En [p_{n}(V_i) p_{n}(V_i)']$, $\epsilon_i \equiv Y_i - E_{\Pn}[Y|V=V_i]$,
 $\nu_n \equiv n^{-1} \widehat Q_n^{-1} \sum_{i=1}^n p_{n}(V_i)A_n(V_i)$, where $A_n(v) : =\theta_n(v) - p_{n}(v)'\beta_n\text{.}$
As shown in the proof of Theorem 1 of  \cite{Newey:97}, we have $ \| \nu_n \| = O_{\Pn}( K_n^{-s})$.
In addition, write
\begin{align*}
\bar{R}_{n} := \left[ \widehat Q_n^{-1} - Q_n^{-1} \right] n^{-1} \sum_{i=1}^n p_{n}(V_i) \epsilon_i
&=    Q_n^{-1} \left[  Q_n - \widehat Q_n \right] n^{-1} \widehat Q_n^{-1} \sum_{i=1}^n p_{n}(V_i) \epsilon_i.
\end{align*}
Then it follows from the proof of Theorem 1 of  \cite{Newey:97} that
\begin{align*}
\|\bar{ R}_n \| &= O_{\Pn} \left( \zeta_n K_n/n \right),
\end{align*}
where $\zeta_n  = \sqrt{K_n}$ by condition (i).
Combining the results above gives
\begin{align}\label{series-bahadur2}
\widehat \beta_n - \beta_n &=  n^{-1}   Q_n^{-1} \sum_{i=1}^n p_{n}(V_i) \epsilon_i + R_n,
\end{align}
where the remainder term $R_n$ satisfies
$$
\| R_n \|
= O_{\Pn} \left( \frac{K_n^{3/2}}{n} + K_n^{-s} \right).
$$
 Note that by condition (iv), eigenvalues of $S_n^{-1}$
are bounded above. In other words, using the notation used in Corollary \ref{corollary:strong} in the main text, we have that
$\tau_n \lesssim 1$.
Then
\begin{align}\label{series-bahadur3}
\Omega_n^{-1/2} \sqrt{n} (\widehat \beta_n - \beta_n) &=  n^{-1/2} \sum_{i=1}^n u_{i,n} + r_n,
\end{align}
where
\begin{align}
u_{i,n} := \Omega_n^{-1/2}  Q_n^{-1} p_{n}(V_i) \epsilon_i\text{,}
\end{align}
and the new remainder term $r_n$ satisfies
$$
\| r_n \|
= O_{\Pn} \left[  n^{1/2} \left( K_n^{3/2}/n + K_n^{-s} \right) \right].
$$
Therefore, $r_n = o_{\Pn}(1/\log n)$ if
\begin{align}\label{lin-cond}
(\log n)  n^{1/2} \left( K_n^{3/2}/n + K_n^{-s} \right) \rightarrow 0,
\end{align}
which is  satisfied under conditions (iii) and (vii), and we have proved the lemma.
\end{proof}

\section{Asymptotic Linear Representation for Local Polynomial Estimator of a Conditional Mean}\label{sec:prac:lle}

In this section we provide details of Example 7 that are omitted in Appendix F. Results obtained in \cite{Kong/Linton/Xia:10} give the following lemma,
which verifies the linear expansion required in condition (b) of Theorem \ref{theorem:strong-approx-kernel} with $\delta_n = 1/\log n$.

\begin{lemma}[\textbf{Asymptotically Linear Representation of Local Polynomial Estimator}]\label{lemma: asym_lin_kernel} Suppose
conditions (i)-(vi) hold. Then we have the following asymptotically linear representation:
uniformly in $v = (z,j) \in \mathcal{V} \subseteq \mathcal{Z} \times \mathcal{J}$,
\begin{align*}
(nh_n^d)^{1/2}(\widehat \theta_n(v) - \theta_n(v)) =  \mathbb{G}_n (g_v)  + o_{\mathrm{P}} (1/\log n).
\end{align*}
\end{lemma}

\begin{proof}[Proof of Lemma \ref{lemma: asym_lin_kernel}]
We first verify Assumptions A1-A7 in \cite{Kong/Linton/Xia:10} (KLX hereafter).
In our example, $\rho(y;\theta) = \frac{1}{2}(y-\theta)^2$ using the notation in KLX. Then
$\varphi(y;\theta)$ in Assumptions A1 and A2 in KLX is $\varphi(y;\theta) = \varphi(y - \theta) = -(y - \theta)$.
Then Assumption A1 is satisfied since the pdf of $U_i$ is bounded and   $U_i$ is a bounded random vector.
Assumption A2 is trivially satisfied since $\varphi(u) = -u$.
Assumption A3 follows since $K(\cdot)$ has compact support and is twice continuous differentiable.
Assumption A4 holds by condition (ii) since $X_i$ and $X_j$ are independent in our example $(i \neq j)$.
Assumption A5 is implied directly by Condition (i).
Since we have i.i.d. data, mixing coefficients ($\gamma[k]$ using the notation of KLX) are identically zeros for any $k \geq 1$. The regression error $U_i$ is assumed to be bounded, so that $\nu_1$ in KLX can be arbitrary large. Hence, to verify Assumption A6 of KLX, it suffices to check that
for some $\nu_2 > 2$, $h_n \rightarrow 0$, $nh_n^d/\log n \rightarrow \infty$, $h_n^{d+2(p+1)}/\log n < \infty$, and $n^{-1} (n h_n^d/\log n)^{\nu_2/8} d_n \log n /  M_n^{(2)} \rightarrow \infty$, where $d_n = (n h_n^d/\log n)^{-1/2}$
and $M_n^{(2)} = M^{1/4} (n h_n^d/\log n)^{-1/2}$ for some $M > 2$,
by choosing $\lambda_2 = 1/2$ and $\lambda_1 = 3/4$ on page 1540 in KLX.
By choosing a sufficiently large $\nu_2$ (at least greater than  8),  the following holds: $n^{-1} (n h_n^d)^{\nu_2/8}\rightarrow \infty$.
Then condition (vi) implies Assumption A6. Finally, condition (iv) implies Assumption A7 since we have i.i.d. data.
Thus, we have verified all the conditions in KLX.

Let $\delta_n = 1/\log n$.
 Then it follows
from Corollary 1 and Lemmas 8 and 10 of KLX that
\begin{align}\label{bahadur-lemma}
\widehat \theta_n(z,j) - \theta_n(z,j) = \frac{1}{n h_n^d f(z)} \mathbf{e}_1'S_p^{-1} \sum_{i=1}^n (e_j'U_i) K_h( Z_i - z) \mathbf{u}_p \left( \frac{Z_i-z}{h_n} \right) + B_n(z,j) + R_n(z,j),
\end{align}
where $\mathbf{e}_1$ is a $|A_p| \times 1$ vector whose  first element is one and all others are zeros,
$S_p$ is a $|A_p| \times |A_p| $ matrix such that $S_p = \{ \int z^u (z^v)' du: u \in A_p, v \in A_p \}$,
$\mathbf{u}_p(z)$ is a  $|A_p| \times 1$ vector such that $\mathbf{u}_p(z) = \{  z^u : u \in A_p \}$,
$$
B_n(z,j) = O ( h_n^{p+1} ) \text{ and } R_n(z,j) = o_{P} \left( \frac{ \delta_n }{ (nh_n^d)^{1/2} } \right),
$$
uniformly in  $(z,j) \in \mathcal{V}$. The exact form of $B_n(z,j)$ is given in equation (12) of KLX. The result that $B_n(z,j) = O ( h_n^{p+1} )$ uniformly in $(z,j)$  follows
from the standard argument based on Taylor expansion given in \cite{Fan/Gijbels:96}, KLX, or \cite{Masry:06}.
The condition that $n h_n^{d + 2(p+1)} \rightarrow 0$ at a polynomial rate in $n$ corresponds to the undersmoothing condition.
Now the lemma follows from \eqref{bahadur-lemma} immediately since $\mathbf{K}(z/h) \equiv \mathbf{e}_1'S_p^{-1} K_h( z ) \mathbf{u}_p( z/h )$ is a kernel of order $(p+1)$ (See Section 3.2.2 of \cite{Fan/Gijbels:96}).
\end{proof}

\section{Local Asymptotic Power Comparisons}\label{sec:local-asy-power}

We have shown in the main text that the test of $\mbox{H}_0: \theta_{na} \leq \theta_{n0}$ of the form
$$
\text{Reject } \mbox{H}_0  \text{ if }  \theta_{na} > \widehat \theta_{n0}(p),
$$
can reject all local alternatives $\theta_{na}$ that are more distant than $\bar\sigma_n \bar a_n$.
We now provide a couple of examples of local alternatives against which our test has non-trivial power, but for which the CvM statistic of \cite {Andrews/Shi:08}, henceforth AS,
does not. See also \cite{Armstrong:11} for a comprehensive analysis of power properties
of KS statistic of \cite{Andrews/Shi:08}.  It is evident from the results of AS on local asymptotic power that there are also models for which their CvM statistic will have power against some $n^{-1/2}$ alternatives, where our approach will not.\footnote{For the formal results, see AS Section 7, Theorem 4. In the examples that follow their Assumption LA3' is violated, as is also the case in the example covered in their Section 13.5.} We conclude that neither approach dominates.

We consider two examples in which
\begin{equation*}
Y_{i}=\theta_n \left( V_{i}\right) +U_{i}\text{,}
\end{equation*}
where $U_{i}$ are iid with $E\left[ U_{i}|V_{i}\right] =0$ and $V_{i}$ are iid random variables uniformly distributed on $\left[ -1,1\right] $. Suppose that for all $v\in \left[ -1,1\right] $ we have
\begin{equation*}
\theta^\ast \leq E\left[ Y_{i}|V_{i}=v\right] \text{,}
\end{equation*}
equivalently
\begin{equation*}
\theta^\ast \leq \theta _{0}=\min_{v\in \left[ -1,1\right] }\theta_n(v)\text{.}
\end{equation*}%
In the examples below we consider two specifications of the bounding function $\theta_n(v)$, each with
$$
\min_{v\in \left[ -1,1\right] }\theta_n(v)=0\text{,}
$$
and we analyze asymptotic power against a local alternative $\theta _{na}>\theta _{0}$.

Following AS, consider the CvM test statistic
\begin{equation}
T_{n}\left( \theta \right) :=\int \left[ n^{1/2}\frac{\overline{m}_{n}\left(
g;\theta \right)}{\widehat \sigma_n(g;\theta)\vee \varepsilon}  \right]_{-}^2 dQ\left(
g\right) \text{,}  \label{CvM statistic}
\end{equation}%
for some $\varepsilon>0$, where $[u]_{-} := -u 1(u < 0)$ and $\theta$ is the parameter value being tested. \
In the present context we have
\begin{equation*}
\overline{m}_{n}\left( g;\theta \right) :=\frac{1}{n}\sum\limits_{i=1}^{n}%
\left( Y_{i}-\theta \right) g\left( V_{i}\right) \text{,}
\end{equation*}
where $g\in \mathcal{G}$ are instrument functions used to transform the
conditional moment inequality $E\left[ Y-\theta |V=v\right] $ a.e. $v\in
\mathcal{V}$ to unconditional inequalities, and $Q\left( \cdot \right) $ is
a measure on the space $\mathcal{G}$ of instrument functions as described in
AS Section 3.4. $\widehat\sigma_n(g;\theta)$ is a uniformly consistent estimator for $\sigma_n(g;\theta)$, the standard deviation of $n^{1/2} \overline{m}_{n}\left( g;\theta \right)$.

We can show that $T_{n}(\theta)=\tilde{T}_{n}(\theta) +o_{p}\left( 1\right),
$ where
\begin{align*}
\tilde{T}_{n}(\theta)  &:=\int \left[ \beta_{n}\left( \theta, g\right)/\(\sigma_n(g;\theta) \vee \varepsilon \) +w\left( \theta, g\right) \right]_{-}^2 dQ\left(g\right),
\end{align*}
where $w\left( \theta,g\right) $ is a mean zero Gaussian process, and $
\beta _{n}\left( \theta,g\right) $ is a deterministic function of the form
\begin{equation*}
\beta _{n}\left( \theta,g\right) \equiv \sqrt{n}E\left\{ \left[ \theta_n \left(
V_{i}\right) -\theta  \right] g(V_i) \right\} \text{.}
\end{equation*}
For any $\theta$, the testing procedure based on the CvM statistic rejects $H_0:\theta \leq \theta_{n0}$ if
$$T_n(\theta) > c(\theta,1-\alpha)\text{,}$$
where $c(\theta,1-\alpha)$ is a generalized moment selection (GMS) critical value that satisfies
$$c(\theta,1-\alpha) = (1-\alpha) \text{-quantile of}  \(\int \left[ \varphi_{n}\left( \theta, g\right)/\(\sigma_n(g;\theta) \vee \varepsilon \) +w\left( \theta, g\right) \right]_{-}^2 dQ\left(g\right)\) + o_p(1)\text{.}$$
$\varphi_{n}\left( \theta, g\right)$ is a GMS function that satisfies $0 \leq \varphi_{n}\left( \theta, g\right) \leq \beta_{n}\left( \theta, g\right)$ with probability approaching 1 whenever $\beta_{n}\left( \theta, g\right) \geq 0$, see AS Section 4.4 for further details. Relative to $\tilde T_n(\theta)$, in the integrand of the expression above $\varphi_{n}\left( \theta, g\right)$ is replaced with $\beta_{n}\left( \theta, g\right)$. Hence if
$$
\sup_{g \in \mathcal{G}} \left[ \beta _{n}\left( \theta_{na},g\right) \right] _{-} \to 0,
$$
for the sequence of local alternatives $\theta_{na}$, then
$$
\liminf_{n\rightarrow\infty} \Pr\(T_{n}(\theta_{na}) > c(\theta_{na},1-\alpha)\) \leq \alpha \text{,}
$$
since asymptotically $c\(\theta_{na},1-\alpha)\)$ exceeds the $1-\alpha$ quantile of $\tilde T_n(\theta)$.
It follows that the CvM test has only trivial power against such a sequence of alternatives. \ The same conclusion holds using plug-in asymptotic critical values, since these are no smaller than GMS critical values.

In the following two examples we now verify that $\sup_{g \in \mathcal{G}} \left[ \beta _{n}\left( \theta_{na},g\right) \right] _{-} \to 0$. We assume that instrument functions are $g$ are either indicators of
boxes or cubes, defined in AS Section 3.3, and hence bounded between zero and one.

\subsection{Example K.1 (Unique, well-defined optimum)}

Let the function $\theta \left( \cdot \right) $ be specified as
\begin{equation*}
\theta_n( v ) =\left\vert v\right\vert ^{a}\text{,}
\end{equation*}%
for some $a\geq 1$.

Let us now proceed to bound, using that $0 \leq g \leq 1$,
\begin{eqnarray*}
\left[ \beta _{n}\left( \theta_{na},g\right) \right] _{-}
&=&\sqrt{n}\left[ E\left\{ \left[ \theta_n( V_{i}) -\theta_{na}\right] g(V_i) \right\} \right] _{-} \\
&\leq &\sqrt{n}E\left\{ \left[ \theta_n( V_{i}) -\theta_{na}\right]
_{-}\right\}  \\
&=&\sqrt{n}\int\limits_{-1}^{1}\left( \theta_{na}-|v|^{a}\right) 1\left\{
|v|^{a}\leq \theta_{na}\right\} dv \\
&=&2\sqrt{n}\int\limits_{0}^{1}\left( \theta_{na}-v^{a}\right) 1\left\{
v\leq \theta_{na}^{1/a}\right\} dv \\
&=&\frac{2a}{a+1}\sqrt{n}\theta_{na}^{\left( a+1\right) /a} \\
&\equiv &\overline{\beta }_{n}\text{.}
\end{eqnarray*}
Note that
$$
\theta_{na}=o\left( n^{-a/[2(a+1)]} \right) \Rightarrow \overline{\beta }%
_{n}\rightarrow 0.
$$
Thus, in this case the asymptotic rejection probability of the CvM test for the local alternative $\theta_{na}$ is bounded above by $\alpha$. On the other hand, by Theorems \ref{theorem: inference analytical} and \ref{theorem: inference1} of the main text, our test rejects all local alternatives $\theta_{na}$ that are more distant than $\bar\sigma_n \bar a_n$ with probably at least $\alpha$ asymptotically. It suffices to find a sequence of local alternatives  $\theta_{na}$ such that
$\theta_{na}=o\left( n^{-a/[2(a+1)]} \right)$ but $\theta_{na} \gg \bar\sigma_n \bar a_n$.

For instance, consider the case where $a=2$.\ Then
\begin{equation*}
\sqrt{n}\theta _{na}^{3/2}\rightarrow 0\Rightarrow \overline{\beta }
_{n}\rightarrow 0\text{,}
\end{equation*}
i.e.  $\theta _{na}=o\left( n^{-1/3}\right) $ $\Rightarrow \overline{\beta }
_{n}\rightarrow 0$, so the CvM test has trivial asymptotic power against $\theta_{na}.$ In contrast, since this is a very smooth case, our approach can achieve
$\bar\sigma_n \bar a_n = O(n^{-\delta})
$ for  some $\delta$ that can be close to $1/2$, for instance by using a series
estimator with a slowly growing number of terms, or a higher-order kernel
or local polynomial estimator. Our test would then be able to reject any $\theta_{na}$ that converges to zero
faster than $n^{-1/3}$ but more slowly than $n^{-\delta}$.

\subsection{Example K.2 (Deviation with Small Support)}

Now suppose that the form of the conditional mean function, $\theta_n(v) \equiv E\left[ Y_{i}|V_{i}=v\right]$, is given by
\begin{equation*}
\theta _{n}\left( v\right) := \bar{\theta} \left(
v\right) -\tau _{n}^{a}\left( \phi \left( v/\tau _{n}\right) -\phi \left(
0\right) \right) \text{,}
\end{equation*}%
where $\tau _{n}$ is a sequence of positive constants converging to zero and $\phi(\cdot)$ is the standard normal density function. \
Let $\bar{\theta} \left( v\right) $ be minimized at zero so that
\begin{equation*}
\theta _{0}=\min_{v\in \left[ -1,1\right] }\theta _{n}\left( v\right)
=\min_{v\in \left[ -1,1\right] }\bar{\theta} \left( v\right) =0\text{.}
\end{equation*}%
Let the alternative by $\tilde{\theta}_{na}\equiv \tau _{n}^{a}\phi \left(
0\right) $. \ Again, the behavior of the AS statistic is driven by $\left[\beta _{n}\left( \tilde{\theta}_{na},g\right) \right] _{-}$, which we bound from above as
\begin{eqnarray*}
\left[ \beta _{n}\left( \tilde{\theta}_{na},g\right) \right] _{-} &=&\sqrt{n}\left[ E\left\{ %
\left[ \theta _{n}\left( V_{i}\right) -\tilde{\theta}_{na}\right]
g \left( V_{i}  \right) \right\} \right] _{-} \\
&\leq &\sqrt{n}E\left\{ \tau _{n}^{a}\phi \left( V_i/\tau _{n}\right) \right\}
\\
&=&\frac{\sqrt{n}}{2} \int\limits_{-1}^{1}\tau _{n}^{a}\phi \left( v/\tau _{n}\right)
dv \\
&\leq&\frac{\sqrt{n}}{2} \tau _{n}^{a+1}\equiv \overline{\beta }_{n}\text{.}
\end{eqnarray*}

Consider the case $ a=2 $.  \ If $\tau _{n}=o\left(
n^{-1/6}\right) $ then $\overline{\beta }_{n}\rightarrow 0$, so that again the CvM test has only trivial asymptotic power. \ If $\tau
_{n}=n^{-1/6-c/2}$ for some small positive constant $c$,  then $\tilde{\theta}_{na}\equiv n^{-1/3-c}\phi \left( 0\right)$.
Note that
\begin{equation*}
f\left( v\right): =\tau _{n}^{2}\phi \left( v/\tau _{n}\right) \Rightarrow
f^{\prime \prime }\left( v\right) =\phi ^{\prime \prime }\left( v/\tau
_{n}\right) \leq \overline{\phi ^{\prime \prime }}<\infty\text{,}
\end{equation*}
for some constant $\overline{\phi ^{\prime \prime }}$. Hence, if $\bar{\theta}(v)$ is twice continuously differentiable,
we can use a series or kernel estimator to estimate $\theta_n(v)$ uniformly at the rate of $(\log n)^d n^{-2/5}$ for some $d>0$, leading to non-trivial power against alternatives $\tilde{\theta}_{na}$ for sufficiently small $c$.

\section{Results of Additional Monte Carlo Experiments}\label{sec:mc-supplement}

In this section we present the results of some additional Monte Carlo
experiments to further illustrate the finite-sample performance of
our method. We consider two types of additional data-generating processes (DGPs). The first set of DGPs, DGPs 5-8, are motivated by \cite{Manski/Pepper:08}, discussed briefly in Example B of the main text. The second set, DGPs 9-12, are from Section 10.3 of AS.

\subsection{Monte Carlo Designs}

In DGPs 5-8 we consider the lower bound on $\theta^* = E[Y_i(t)|V_i=v]$ under the monotone
instrumental variable (MIV) assumption, where $t$ is a
treatment, $Y_i(t)$ is the corresponding potential outcome, and
$V_i$ is a monotone instrumental variable. The lower bound  on
$E[Y_i(t)|V_i=v]$ can be written as
\begin{align}\label{mp-bound-sim-supp}
\max_{u \leq v} E \left[ Y_i \cdot 1\{Z_i = t\} + y_0 \cdot 1\{Z_i \neq t\} | V_i = u \right],
\end{align}
where $Y_i$ is the observed outcome, $Z_i$ is a realized treatment,
and $y_0$ is the left end-point of the support of $Y_i$, see
\cite{Manski/Pepper:08}. The parameter of interest is $\theta^\ast =
E[Y_i(1)|V_i=1.5]$.

In DGP5, $V_0 = \mathcal{V}$ and the MIV assumption has no
identifying power.  In other words, the bound-generating function is flat on $\mathcal{V}$,
in which case the bias of the analog estimator is most acute, see \cite{Manski/Pepper:08}.  In DGP6, the MIV assumption has identifying power, and $V_0$ is a strict subset of $\mathcal{V}$.
In DGPs 7-8, we set $V_0$ to be a singleton set.

Specifically, for DGPs 5-8 we generated 1000 independent samples as follows:
$$
V_i
\sim \text{Unif}[-2,2], Z_i = 1\{ \varphi_0(V_i) + \varepsilon_i > 0 \}, \ \text{and}
\ \ Y_i = \min \{\max\{-0.5, \sigma_0(V_i) U_i\}, 0.5 \} ,
$$
where $\varepsilon_i \sim N(0,1)$,
$U_i \sim N(0,1)$, $\sigma_0(V_i) = 0.1 \times |V_i|$,
 and $(V_i,U_i)$ are
statistically independent $(i=1,\ldots,n)$.
The
bounding function has the form
\begin{align*}
\theta(v) &:= E \left[ Y_i \cdot 1\{Z_i
= 1\} + y_0 \cdot 1\{Z_i \neq 1\} | V_i = v \right] \\
&= -0.5 \Phi[ -\varphi_0(v) ],
\end{align*}
where $\Phi(\cdot)$ is the standard normal cumulative
distribution function.
For DGP5, we set $\varphi_0(v) \equiv 0$. In this case, the
bounding function
is completely
flat ($\theta_l(v) = -0.25$ for each $v \in \mathcal{V} =
[-2,1.5]$). For DGP6, an alternative specification is
considered:
$$
\varphi_0(v) = v 1( v \leq 1) + 1( v > 1).
$$
In this case, $v \mapsto \theta(v)$ is strictly increasing on
$[-2,1]$ and is flat on $[1,2]$, and  $V_0 = [1,1.5]$ is a
strict subset of $\mathcal{V} = [-2,1.5]$.
For DGP7, we consider
$$
\varphi_0(v) = -2 v^2.
$$
In this case, $v \mapsto \theta_l(v)$ has a unique maximum at $v=0$, and thus,
  $V_0 = \{ 0 \}$ is singleton.
For DGP8, we consider
$$
\varphi_0(v) = -10 v^2.
$$
In this case, $v \mapsto \theta(v)$  has a  unique maximum at $v=0$ and is more peaked than that of DGP7. Figures \ref{figure_dgp56} and \ref{figure_dgp78} show data realizations and bounding functions for these DGPs.

DGPs 9-12 use the bounding functions in Section 10.3 of AS. DGP9 and DGP10 feature a roughly plateau-shaped bounding function given by
\begin{align}
\label{mp-bound-AS1} \theta(v) := L \phi(v^{10})\text{,}
\end{align}
instead of $\theta(v) := L \phi(v)$ as in Section \ref{sec:monte-carlo} of the main text. DGP11 and DGP12 use the roughly double-plateau-shaped bounding function
\begin{align}
\label{mp-bound-AS2} \theta(v) := L \cdot \max{\{\phi((v-1.5)^{10}),\phi((v+1.5)^{10})\}} \text{.}
\end{align}
Specifically, we generated 1000 independent samples from the model:
$$
V_i
\sim \text{Unif}[-2,2],
U_i = \min \{\max\{-3, \sigma \tilde{U}_i\}, 3 \} ,
\ \text{and}
\ \ Y_i = \theta(V_i) + U_i,
$$
where
$\tilde{U}_i \sim N(0,1)$, with $L$ and $\sigma$ as follows:
\begin{eqnarray*}
& & \textrm{DGP9 and DGP11:  } L = 1 \text{ and } \sigma = 0.1 \text{; } \ \ \ \textrm{DGP10 and DGP12:  } L = 5 \text{ and } \sigma = 0.1 \text{.}
\end{eqnarray*}
Figure \ref{figure_dgp12_AS} illustrates the bounding function and data realizations for DGPs 9 and 10; figure \ref{figure_dgp34_AS} provides the same for DGPs 11 and 12. Interest again lies in inference on $\theta_0 = \sup_{v \in \mathcal{V}} \theta(v)$, which in these DGPs is $\theta_0 = L \phi(0)$.

\subsection{Simulation Results}

To evaluate the relative performance of our inference method in DGPs 5-8, we used our method with cubic B-splines with knots equally spaced over the sample quantiles of $V_i$, and we also implemented one of the inference methods proposed by AS, specifically their Cram\'{e}r-von Mises-type (CvM) statistic with PA/Asy and GMS/Asy critical values. Implementation details for our method with B-splines are the same as in Section \ref{sec:monte-carlo-series} of the main text. Tuning parameters for CvM were chosen exactly as in AS (see Section 9).\footnote{In DGPs 5-8 our Monte Carlo design differs from that of AS, and alternative choices of tuning parameters could perform more or less favorably in our design. We did not examine sensitivity to the choice of tuning parameters for their method.} We considered
sample sizes $n=250$, $n=500$, and  $n=1000$.

The coverage probability is evaluated at the true lower bound $\theta_0$ (with the nominal level of
95\%), and the false coverage probability (FCP) is evaluated at a $\theta$ value outside the identified set. For DGP5,
we set $\theta = \theta_0 - 0.03$; for DGP6-DGP7,  $\theta = \theta_0 - 0.05$; and for DGP8,   $\theta = \theta_0 - 0.07$.
These points are chosen differently across different DGPs to ensure that the FCPs have similar values.
This type of FCP was reported in AS, along with a so-called ``CP-correction'' (similar to size correction in
testing).  We did not do CP-correction in our reported results.
There were 1,000 replications for each experiment.
Table \ref{mc1-sup} summarizes the results of Monte
Carlo experiments. CLR and AS refer to our inference method and that of AS, respectively.

First, we consider Monte Carlo results for  DGP5.
The discrepancies between nominal and actual coverage probabilities
are not large  across all methods, implying that all of them perform well in finite samples. For DGP5,
since the true argmax
set $V_0$ is equal to $\mathcal{V}$, an estimated $V_0$
should be the entire set $\mathcal{V}$. Thus the
simulation results are the same whether or not estimating $V_0$ since for most of simulation
draws, $\widehat V_n = \mathcal{V}$. Similar conclusions
hold for AS with  CvM between PA/Asy and GMS/Asy critical values.
In terms of false coverage probability, CvM with either critical value performs better than our method.

We now move to DGPs 6-8. In DGP6, the true argmax set $V_0$ is
$[1,1.5]$ and in DGP7 and DGP8, $V_0$ is a singleton set. In these cases the true argmax set $V_0$ is a strict subset of
$\mathcal{V}$. Hence, we expect that it is important to estimate
$V_0$. On average, for DGP6, the estimated sets were
$[-0.793, 1.5]$ when $n=250$,
$[-0.359,1.5]$ when $n=500$, and $[-0.074,1.5]$ when $n=1,000$;
for DGP7, the estimated sets were
$[-0.951,   0.943]$ when $n=250$,
$[-0.797,   0.798]$ when $n=500$, and $[-0.684,   0.680]$ when $n=1,000$;
for DGP8, the estimated sets were
$[-1.197,   0.871]$ when $n=250$,
$[-0.662,   0.645]$ when $n=500$, and $[-0.403,   0.402]$ when $n=1,000$.

Hence, an average estimated set is larger than $V_0$; however, it is still a strict subset of $\mathcal{V}$ and gets smaller as $n$ gets large.
For all the methods, the Monte Carlo results are consistent with asymptotic theory. Unlike in DGP5, the CLR method performs better than the AS method in terms of false coverage probability.\footnote{As in Section \ref{sec:monte-carlo}, this conclusion will remain valid even with CP-correction as in AS, since our method performs better in DGPs 6-8 where we have over-coverage.} As can be seen from the
table, the CLR method performs  better when $V_0$ is estimated in
terms of making the coverage probability  less
conservative and also of making the false coverage probability smaller.
Similar gains are obtained for the CvM with GMS/Asy critical values, relative to that with PA/Asy critical values.

We now turn to DGPs 9-12. AS, Section 10.3, report results for their approach using their CvM and KS statistics, and we refer the reader to their paper for results using their method. They also include results for our approach using B-splines and local-linear estimation of the bounding function. Here we provide further investigation of  the performance of our method in additional Monte Carlo simulations.

From Figures \ref{figure_dgp12_AS} and \ref{figure_dgp34_AS} we see that the bounding function is very nearly flat in some regions, including areas close to $V_0$, and also has very large derivatives a bit further away from $V_0$. The functions are smooth, but the steep derivatives mimic discontinuity points, and are challenging for nonparametric estimation methods. The AS approach does not rely on smoothness of the bounding function, and performs better in most - though not all - of the comparisons of AS. The performance of our approach improves with the sample size, as expected.

Our Monte Carlo experiments for DGPs 9-12 further examine the performance of our method in such a setup. In all experiments we report coverage probabilities (CPs) at $\theta_0 = L \phi(0)$ and FCPs at $\theta_0 - 0.02$ as in AS. We provide results for sample sizes $n=500$ and $n=1000$, both with and without estimation of the argmin set. In Table \ref{mc2-supp} we report the results of series estimation via B-splines. We used cubic B-splines and our implementation was identical to that described in Section \ref{sec:monte-carlo} of the main text. Compared to the results in Table \ref{mc1} of the main text for DGPs 1-4, we see that the average number of series terms is much higher. This is due to a higher number of terms selected during cross-validation, presumably because of the regions with very large derivatives. Our coverage probabilities are below the nominal level, but they improve with the sample size, as in AS, ranging from .885 to .937 across DGPs at $n=1000$. Moreover, we see that our method using $\textsf{V}=\mathcal{V}$ rather than $\textsf{V}=\widehat V_n$ actually performs better in this setup.

To further investigate the challenge of nonparametrically estimating a bounding function with steep derivatives, we implemented our method with a locally-constant Nadaraya-Watson kernel estimator. The functions are in fact nearly locally constant at most points, with the exception of the relatively narrow regions with steep derivatives. The top half of Table \ref{mc3-supp} presents the results with a bandwidth selected the same way as for the local-linear estimator in Section \ref{monte-carlo-kernel}, equation \eqref{h-rule}. When $V_n$ is estimated coverage probabilities in these DGPs range from $.903$ to $.923$ when $n=500$ and $.926$ to $.945$ when $n=1000$, becoming closer to the nominal level. The procedure exhibits good power in all cases, with FCPs decreasing with the sample size. These results are qualitatively similar to those reported in AS for the local linear estimator. We also include results when $ \mathcal{V}$ is used instead of estimating the argmin set. This results in higher coverage probabilities for $\theta_0$, in most cases closer to the nominal level, but also somewhat higher FCPs. Overall performance remains reasonably good.

The bottom half of Table \ref{mc3-supp} gives results for locally-constant kernel estimation using an identical rule of thumb for bandwidth selection, $h=\widehat{h}_{ROT} \times \widehat{s}_v$, but without applying an undersmoothing factor. The proofs of our asymptotic results use undersmoothing, but with a locally-constant kernel estimator this does not appear essential. Our exploratory Monte Carlo results are very good, offering support to that view. Compared to the results with undersmoothing, all coverage probabilities increase, and all FCPs decrease. This suggests that future research on the possibility to abandon undersmoothing may be warranted.

The overall results of this section support the conclusions reached in Section \ref{sec:monte-carlo} of the main text regarding comparisons to AS. In completely flat cases, the AS method outperforms our method, whereas in non-flat cases, our method outperforms the AS method. In this section we also considered some intermediate cases. In DGP7, where the bounding function is partly-flat, our method performed favorably. More generally there is a wide range of intermediate cases that could be considered, and we would expect the approach of AS to perform favorably in some cases too. Indeed, in DGPs 9-12 from AS, the bounding function exhibits areas with extremely steep derivatives. Their results indicate that in these DGPs their approach performs better at smaller sample sizes ($n=100,250$) than does our approach, which is based on nonparametric estimation methods. However, at larger sample sizes ($n=500,1000$) even with the very steep derivatives of DGPs 9-12 our approach performs well, and in a handful of cases, e.g. DGP10 with kernel estimation, can even perform favorably. The main conclusions we draw from the full range of Monte Carlo experiments are that our inference method generally performs well both in coverage probabilities and false coverage probabilities and that in terms of a comparison between our approach and that of AS, each has their relative advantages and neither approach dominates.

\clearpage
\newpage
\begin{table}
{\footnotesize \singlespacing
  \caption{Results for Monte Carlo Experiments }\label{mc1-sup}

  \begin{tabular}{rrccc}
    \hline\hline
    DGP & Sample Size & Critical Value & Cov. Prob.  & False Cov. Prob. \\
    \hline
  \multicolumn{5}{l}{CLR with Series Estimation using B-splines} \\
          &         & Estimating $V_n$? &    &           \\
       5  &  250    &     No     &   0.914  &   0.709   \\
       5  &  250    &     Yes    &   0.914  &   0.709   \\
       5  &  500    &     No     &   0.935  &   0.622   \\
       5  &  500    &     Yes    &   0.935  &   0.622   \\
       5  & 1000    &     No     &   0.947  &   0.418   \\  \vspace*{1ex}
       5  & 1000    &     Yes    &   0.947  &   0.418   \\
       6  &  250    &     No     &   0.953  &   0.681   \\
       6  &  250    &     Yes    &   0.942  &   0.633   \\
       6  &  500    &     No     &   0.967  &   0.548   \\
       6  &  500    &     Yes    &   0.941  &   0.470   \\
       6  & 1000    &     No     &   0.973  &   0.298   \\  \vspace*{1ex}
       6  & 1000    &     Yes    &   0.957  &   0.210   \\
       7  &  250    &     No     &   0.991  &   0.899   \\
       7  &  250    &     Yes    &   0.980  &   0.841   \\
       7  &  500    &     No     &   0.996  &   0.821   \\
       7  &  500    &     Yes    &   0.994  &   0.697   \\
       7  & 1000    &     No     &   0.987  &   0.490   \\  \vspace*{1ex}
       7  & 1000    &     Yes    &   0.965  &   0.369   \\
       8  &  250    &     No     &   0.999  &   0.981   \\
       8  &  250    &     Yes    &   0.996  &   0.966   \\
       8  &  500    &     No     &   1.000  &   0.984   \\
       8  &  500    &     Yes    &   0.999  &   0.951   \\
       8  & 1000    &     No     &   0.998  &   0.909   \\
       8  & 1000    &     Yes    &   0.995  &   0.787   \\
  \hline
\multicolumn{5}{l}{AS with CvM (Cram\'{e}r-von Mises-type statistic)} \\
     5  &  250    &     PA/Asy     &  0.951   & 0.544   \\
     5  &  250    &     GMS/Asy    &  0.945   & 0.537   \\
     5  &  500    &     PA/Asy     &  0.949   & 0.306    \\
     5  &  500    &     GMS/Asy    &  0.945   & 0.305   \\
     5  & 1000    &     PA/Asy     &  0.962   & 0.068   \\  \vspace*{1ex}
     5  & 1000    &     GMS/Asy    &  0.956   & 0.068   \\
     6  &  250    &     PA/Asy     &  1.000   & 0.941  \\
     6  &  250    &     GMS/Asy    &  0.990   & 0.802   \\
     6  &  500    &     PA/Asy     &  1.000   & 0.908  \\
     6  &  500    &     GMS/Asy    &  0.980   & 0.674   \\
     6  & 1000    &     PA/Asy     &  1.000   & 0.744   \\  \vspace*{1ex}
     6  & 1000    &     GMS/Asy    &  0.980   & 0.341  \\
     7  &  250    &     PA/Asy     &  1.000   & 1.000  \\
     7  &  250    &     GMS/Asy    &  0.997   & 0.948   \\
     7  &  500    &     PA/Asy     &  1.000   & 0.997  \\
     7  &  500    &     GMS/Asy    &  0.997   & 0.916   \\
     7  & 1000    &     PA/Asy     &  1.000   & 0.993   \\   \vspace*{1ex}
     7  & 1000    &     GMS/Asy    &  0.997   & 0.823  \\
     8  &  250    &     PA/Asy     &  1.000   & 1.000  \\
     8  &  250    &     GMS/Asy    &  1.000   & 0.988   \\
     8  &  500    &     PA/Asy     &  1.000   & 1.000  \\
     8  &  500    &     GMS/Asy    &  0.999   & 0.972   \\
     8  & 1000    &     PA/Asy     &  1.000   & 1.000   \\
     8  & 1000    &     GMS/Asy    &  1.000   & 0.942  \\
  \hline
  \end{tabular}

\parbox{5in}{
Notes: CLR and AS refer to our inference methods and those of \cite{Andrews/Shi:08}, respectively. There were 1000 replications per experiment.}

}
\end{table}

\begin{table}[hbtc]
{\footnotesize \singlespacing
  \caption{Results for Monte Carlo Experiments (Series Estimation Using B-splines) }\label{mc2-supp}

  \begin{tabular}{rrcccccc}
    \hline\hline
    DGP & Sample  & Critical  & Ave. Smoothing & Cov.   & False Cov. & \multicolumn{2}{c}{Ave. Argmax Set} \\
        &  Size   & Value     & Parameter      & Prob.  &  Prob.     &  Min.  & Max.   \\
    \hline
  \multicolumn{8}{l}{CLR with Series Estimation Using B-splines} \\
          &         & Estimating $V_n$? &    &     &    & &  \\
       9  &  500    &     No       &  35.680  &   0.920  &   0.562  &  -1.799  &   1.792  \\
       9  &  500    &     Yes      &  35.680  &   0.870  &   0.475  &  -1.001  &   1.001  \\
       9  & 1000    &     No       &  39.662  &   0.937  &   0.487  &  -1.801  &   1.797  \\  \vspace*{1ex}
       9  & 1000    &     Yes      &  39.662  &   0.913  &   0.380  &  -0.977  &   0.977  \\
       10  &  500    &     No       &  39.090  &   0.887  &   0.534  &  -1.799  &   1.792  \\
       10  &  500    &     Yes      &  39.090  &   0.825  &   0.428  &  -0.912  &   0.912  \\
       10  & 1000    &     No       &  41.228  &   0.920  &   0.477  &  -1.801  &   1.797  \\  \vspace*{1ex}
       10 & 1000    &     Yes      &  41.228  &   0.891  &   0.351  &  -0.902  &   0.903  \\
       11  &  500    &     No       &  35.810  &   0.880  &   0.462  &  -1.799  &   1.792  \\
       11  &  500    &     Yes      &  35.810  &   0.853  &   0.399  &  -1.799  &   1.792  \\
       11  & 1000    &     No       &  40.793  &   0.937  &   0.374  &  -1.801  &   1.797  \\  \vspace*{1ex}
       11  & 1000    &     Yes      &  40.793  &   0.912  &   0.299  &  -1.801  &   1.797  \\
       12  &  500    &     No       &  39.474  &   0.836  &   0.459  &  -1.799  &   1.792  \\
       12  &  500    &     Yes      &  39.474  &   0.811  &   0.386  &  -1.799  &   1.792  \\
       12  & 1000    &     No       &  42.224  &   0.917  &   0.367  &  -1.801  &   1.797  \\
       12  & 1000    &     Yes      &  42.224  &   0.885  &   0.294  &  -1.801  &   1.797  \\
  \hline
 \end{tabular}

\parbox{5in}{
Notes: DGPs 9-12 correspond to DGPs 1-4 in \cite{Andrews/Shi:08} Section 10.3. The last two columns report the average values of the minimum and maximum of the argmax set. The estimated set is allowed to be disconnected and so the interval between the minimum and maximum of the argmax set is just an outer set for the estimated argmax set.}
}
\end{table}

\begin{table}[hbtc]
{\footnotesize \singlespacing
  \caption{Results for Monte Carlo Experiments (Nadaraya-Watson Kernel Estimation) }\label{mc3-supp}

  \begin{tabular}{rrcccccc}
    \hline\hline
    DGP & Sample  & Critical  & Ave. Smoothing & Cov.   & False Cov. & \multicolumn{2}{c}{Ave. Argmax Set} \\
        &  Size   & Value     & Parameter      & Prob.  &  Prob.     &  Min.  & Max.   \\
    \hline
  \multicolumn{8}{l}{Using Bandwidth $h=\widehat{h}_{ROT} \times \widehat{s}_v \times n^{1/5} \times n^{-2/7}$} \\
          &         & Estimating $V_n$? &    &     &    & &  \\
       9  &  500    &     No      &  0.206  &   0.944   &   0.496   &  -1.799   &   1.792  \\
       9  &  500    &     Yes      &  0.206  &   0.916   &   0.391   &  -0.955   &   0.954  \\
       9  & 1000    &     No      &  0.169  &   0.966   &   0.326   &  -1.801   &   1.796  \\ \vspace*{1ex}
       9  & 1000    &     Yes      &  0.169  &   0.945   &   0.219   &  -0.941   &   0.940  \\
       10  &  500    &     No      &  0.166  &   0.945   &   0.523   &  -1.799   &   1.792  \\
       10  &  500    &     Yes      &  0.166  &   0.903   &   0.411   &  -0.879   &   0.880  \\
       10  &  1000    &     No      &  0.136  &   0.963   &   0.387   &  -1.801   &   1.796  \\ \vspace*{1ex}
       10  &  1000    &     Yes      &  0.136  &   0.926   &   0.266   &  -0.868   &   0.868  \\
       11  &  500    &     No      &  0.195  &   0.938   &   0.403   &  -1.799   &   1.792  \\
       11  &  500    &     Yes      &  0.195  &   0.923   &   0.345   &  -1.799   &   1.792  \\
       11  & 1000    &     No      &  0.160  &   0.951   &   0.201   &  -1.801   &   1.796  \\ \vspace*{1ex}
       11  & 1000    &     Yes      &  0.160  &   0.926   &   0.152   &  -1.801   &   1.796  \\
       12  &  500    &     No      &  0.169  &   0.937   &   0.439   &  -1.799   &   1.792  \\
       12  &  500    &     Yes      &  0.169  &   0.917   &   0.365   &  -1.799   &   1.792  \\
       12  & 1000    &     No      &  0.138  &   0.947   &   0.235   &  -1.801   &   1.796  \\
       12  & 1000    &     Yes      &  0.138  &   0.933   &   0.176   &  -1.801   &   1.796  \\
  \hline
 \multicolumn{8}{l}{Using Bandwidth $h=\widehat{h}_{ROT} \times \widehat{s}_v$ (no undersmoothing factor)} \\
         &         & Estimating $V_n$? &    &     &    & &  \\
       9  &  500    &     No      &  0.351  &   0.968   &   0.470   &  -1.799   &   1.792  \\
       9  &  500    &     Yes      &  0.351  &   0.944   &   0.360   &  -0.908   &   0.907  \\
       9  & 1000    &     No      &  0.306  &   0.977   &   0.211   &  -1.801   &   1.796  \\ \vspace*{1ex}
       9  & 1000    &     Yes      &  0.306  &   0.956   &   0.138   &  -0.885   &   0.884  \\
       10  &  500    &     No      &  0.283  &   0.959   &   0.520   &  -1.799   &   1.792  \\
       10  &  500    &     Yes      &  0.283  &   0.931   &   0.409   &  -0.839   &   0.839  \\
       10  &  1000    &     No      &  0.247  &   0.977   &   0.285   &  -1.801   &   1.796  \\ \vspace*{1ex}
       10  &  1000    &     Yes      &  0.247  &   0.959   &   0.188   &  -0.828   &   0.828  \\
       11  &  500    &     No      &  0.333  &   0.955   &   0.316   &  -1.799   &   1.792  \\
       11  &  500    &     Yes      &  0.333  &   0.939   &   0.261   &  -1.799   &   1.792  \\
       11  & 1000    &     No      &  0.290  &   0.960   &   0.118   &  -1.801   &   1.796  \\ \vspace*{1ex}
       11  & 1000    &     Yes      &  0.290  &   0.943   &   0.079   &  -1.801   &   1.796  \\
       12  &  500    &     No      &  0.287  &   0.956   &   0.376   &  -1.799   &   1.792  \\
       12  &  500    &     Yes      &  0.287  &   0.944   &   0.295   &  -1.799   &   1.792  \\
       12  & 1000    &     No      &  0.250  &   0.960   &   0.154   &  -1.801   &   1.796  \\
       12  & 1000    &     Yes      &  0.250  &   0.948   &   0.111   &  -1.801   &   1.796  \\
 \hline
\end{tabular}

\parbox{5in}{
Notes: DGPs 9-12 correspond to DGPs 1-4 in \cite{Andrews/Shi:08} Section 10.3. The last two columns report the average values of the minimum and maximum of the argmax set. The estimated set is allowed to be disconnected and so the interval between the minimum and maximum of the argmax set is just an outer set for the estimated argmax set.
}

}
\end{table}

\clearpage
\begin{figure}[htbp]
\caption{Simulated Data and Bounding Functions: DGP1 and DGP2}
\label{figure_dgp12}
\begin{center}
\makebox{
\includegraphics[origin=bl,scale=.5,angle=90]{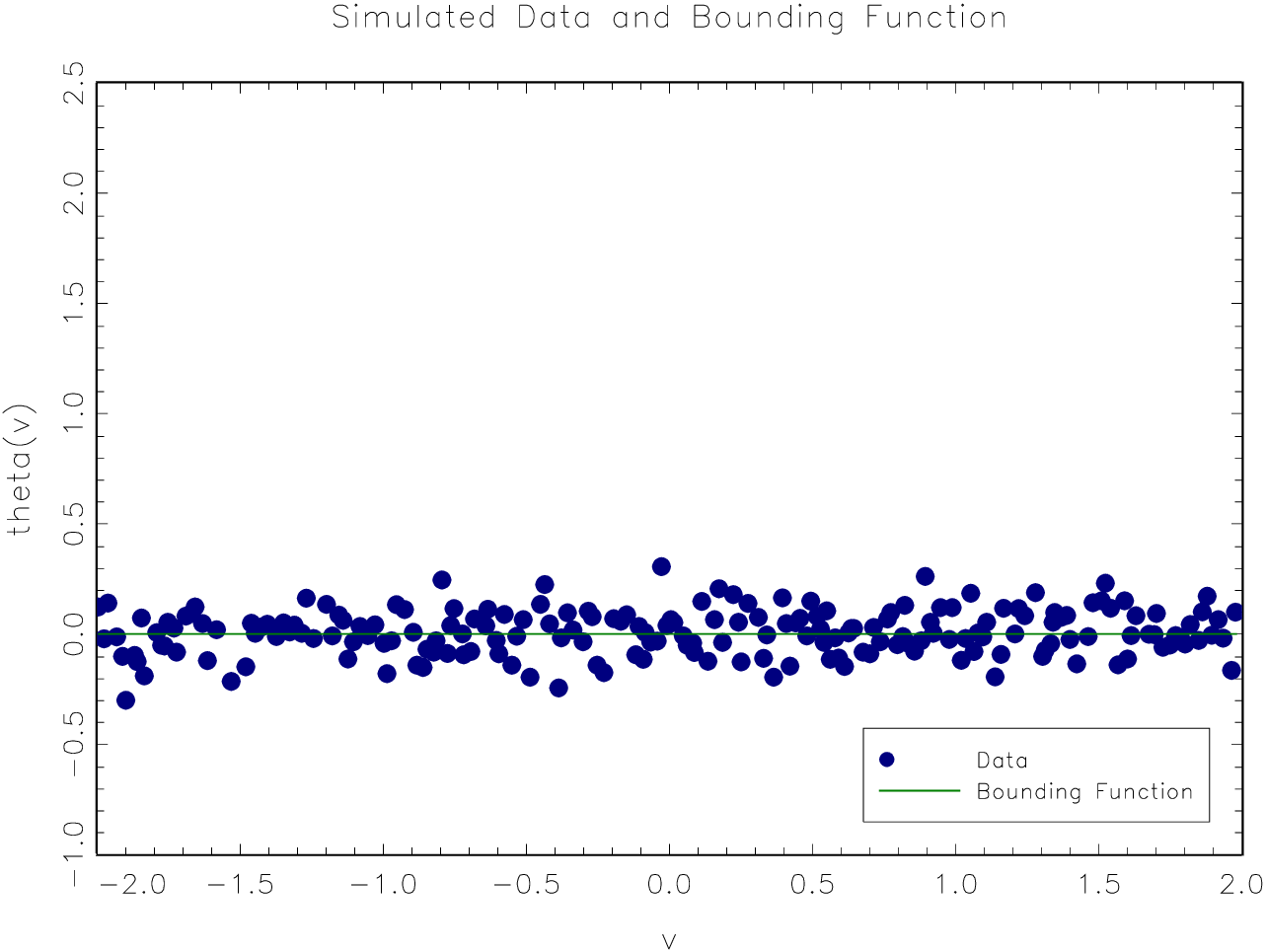}
}
\makebox{
\includegraphics[origin=bl,scale=.5,angle=90]{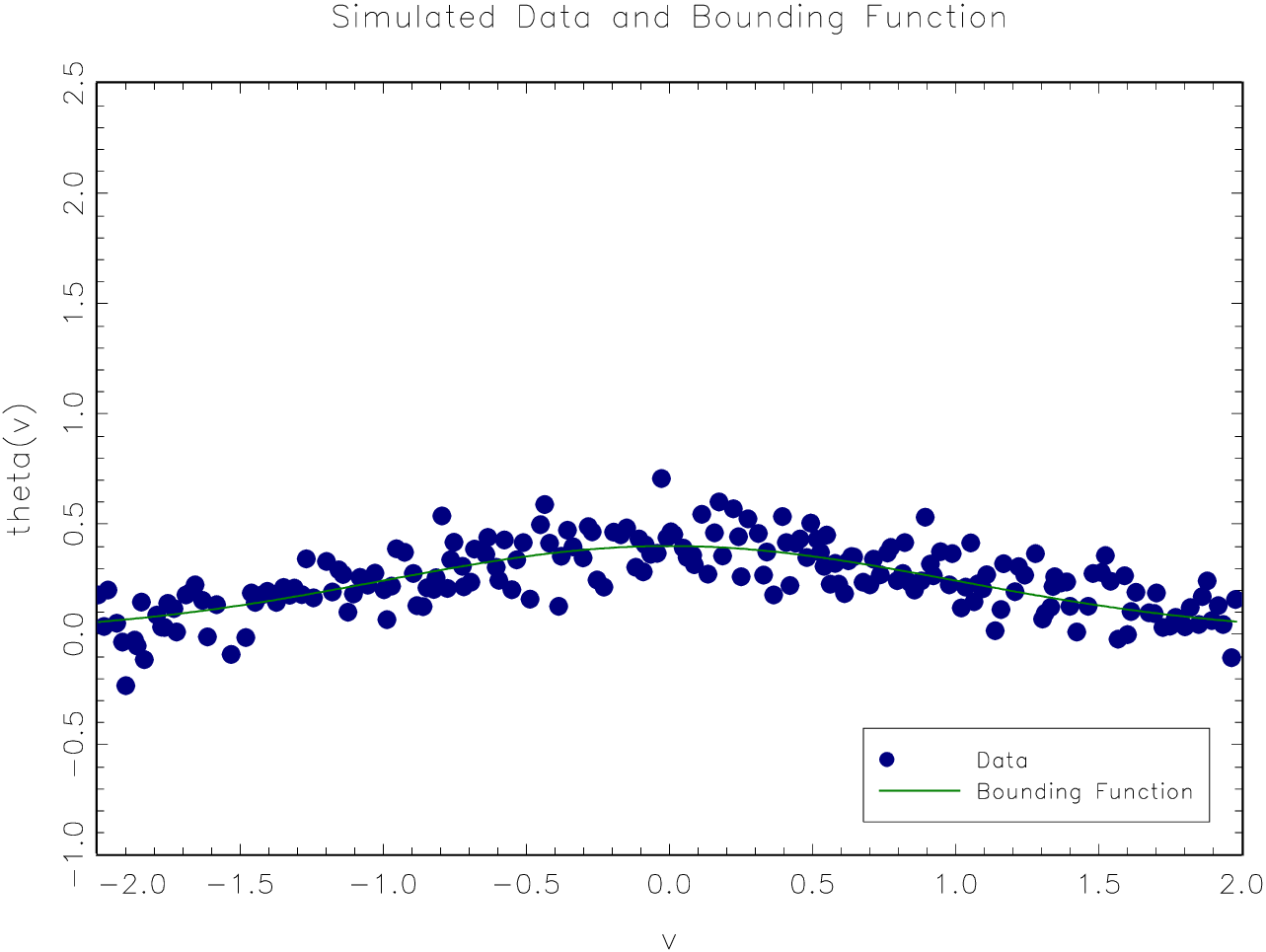}
}
\end{center}
\end{figure}

\clearpage
\begin{figure}[htbp]
\caption{Simulated Data and Bounding Functions: DGP3 and DGP4}
\label{figure_dgp34}
\begin{center}
\makebox{
\includegraphics[origin=bl,scale=.5,angle=90]{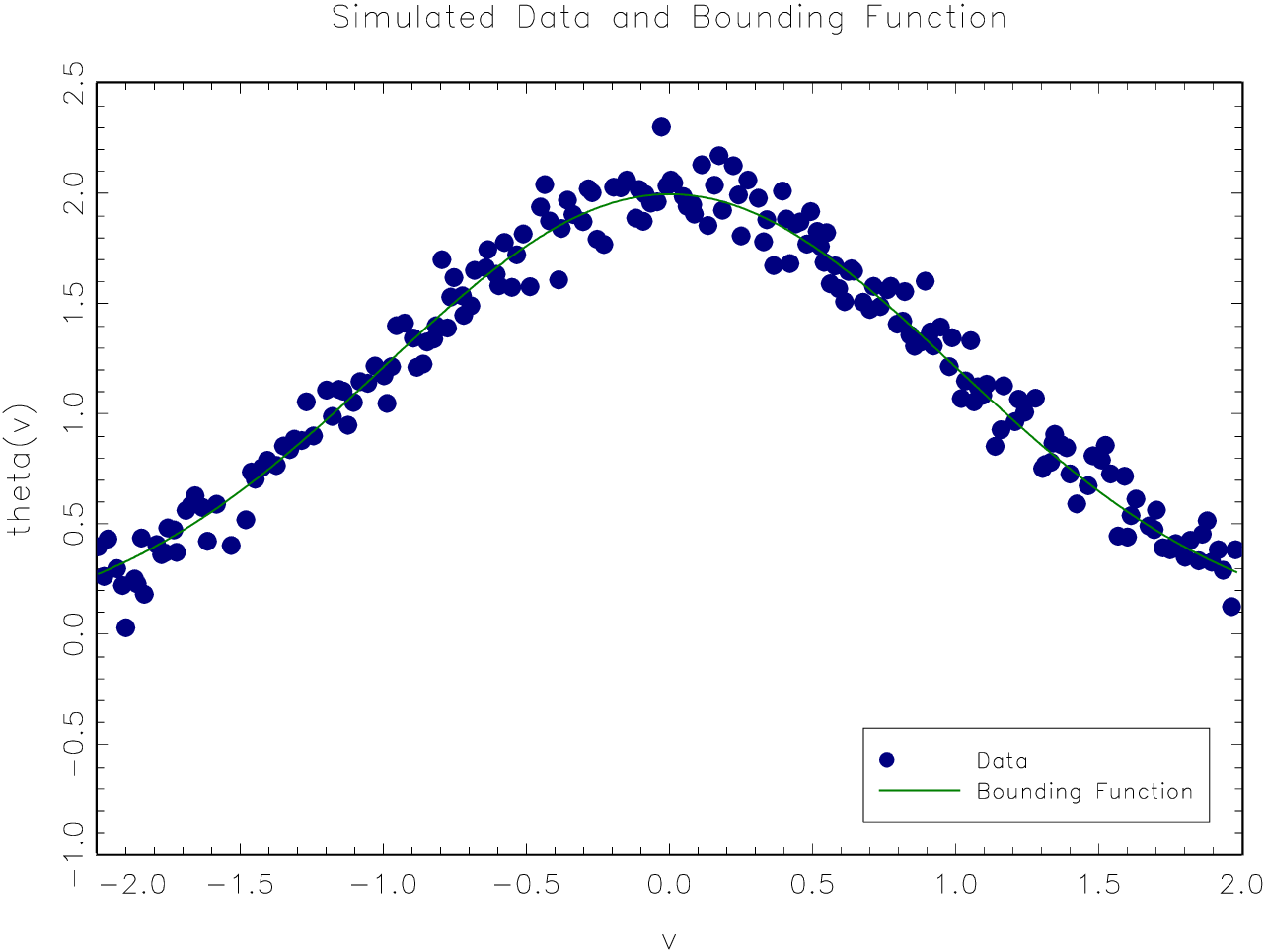}
}
\makebox{
\includegraphics[origin=bl,scale=.5,angle=90]{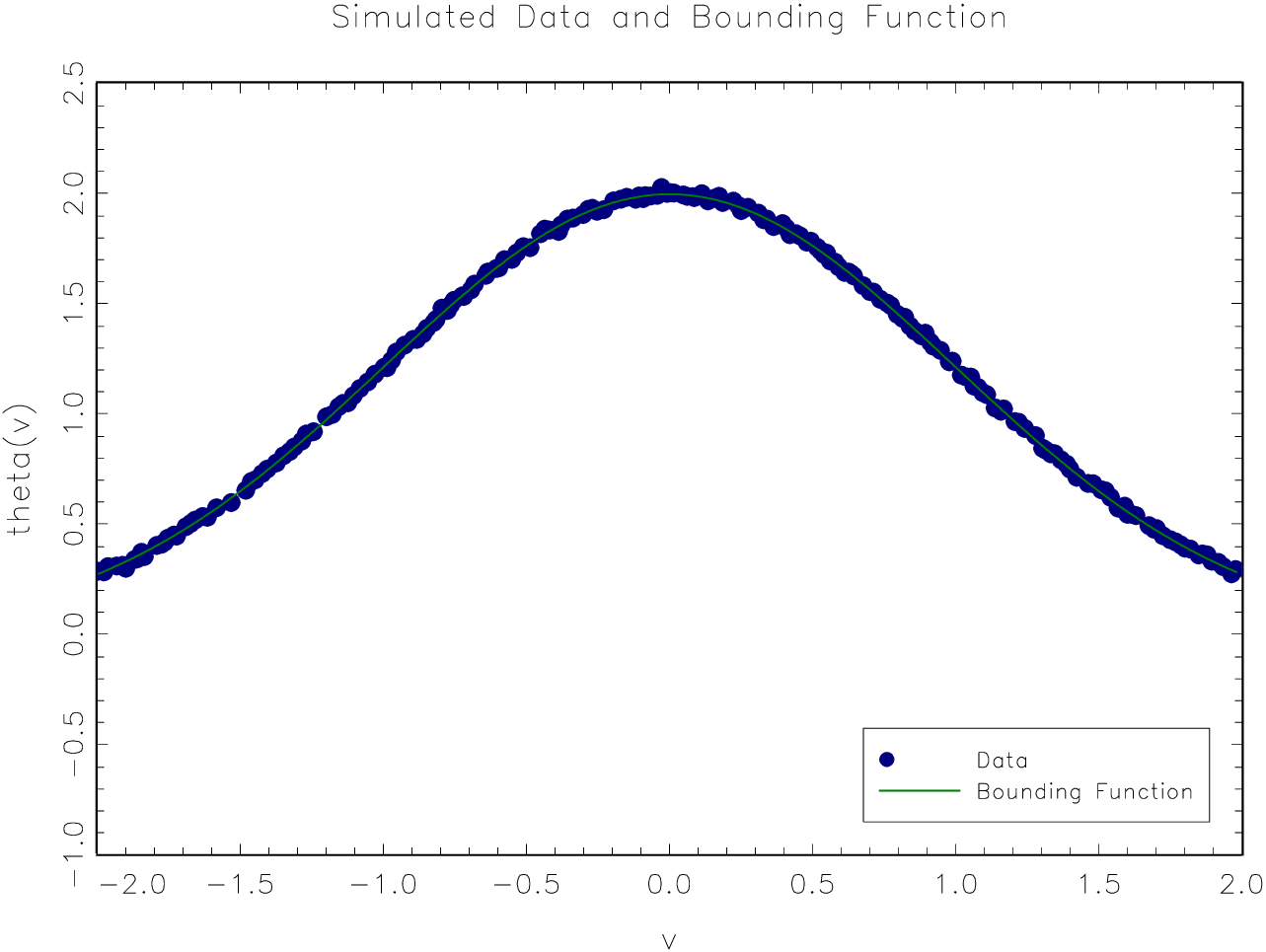}
}
\end{center}
\end{figure}

\clearpage
\begin{figure}[htbp]
\caption{Simulated Data and Bounding Functions: DGP5 and DGP6}
\label{figure_dgp56}
\begin{center}
\makebox{
\includegraphics[origin=bl,scale=.5,angle=90]{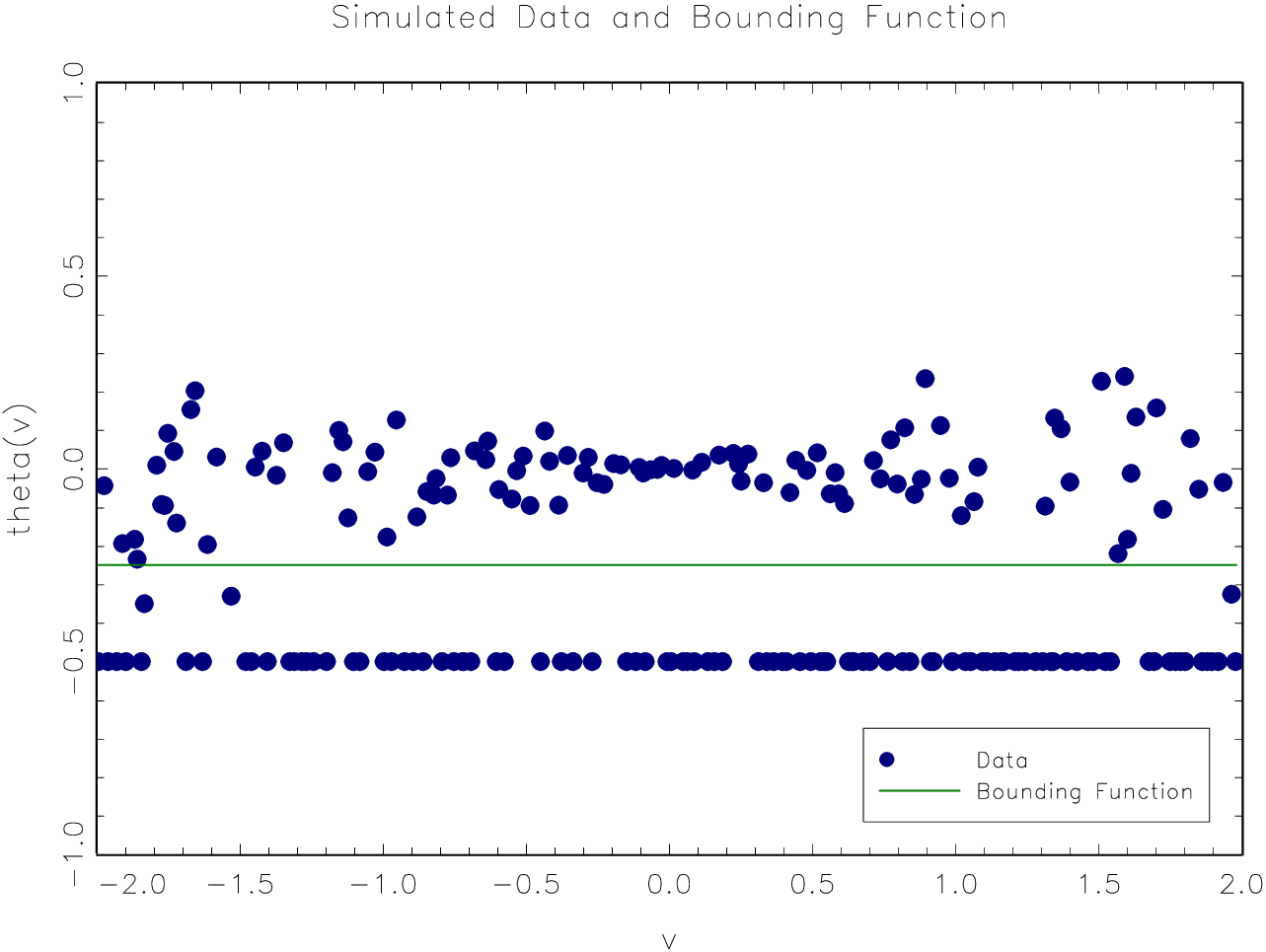}
}
\makebox{
\includegraphics[origin=bl,scale=.5,angle=90]{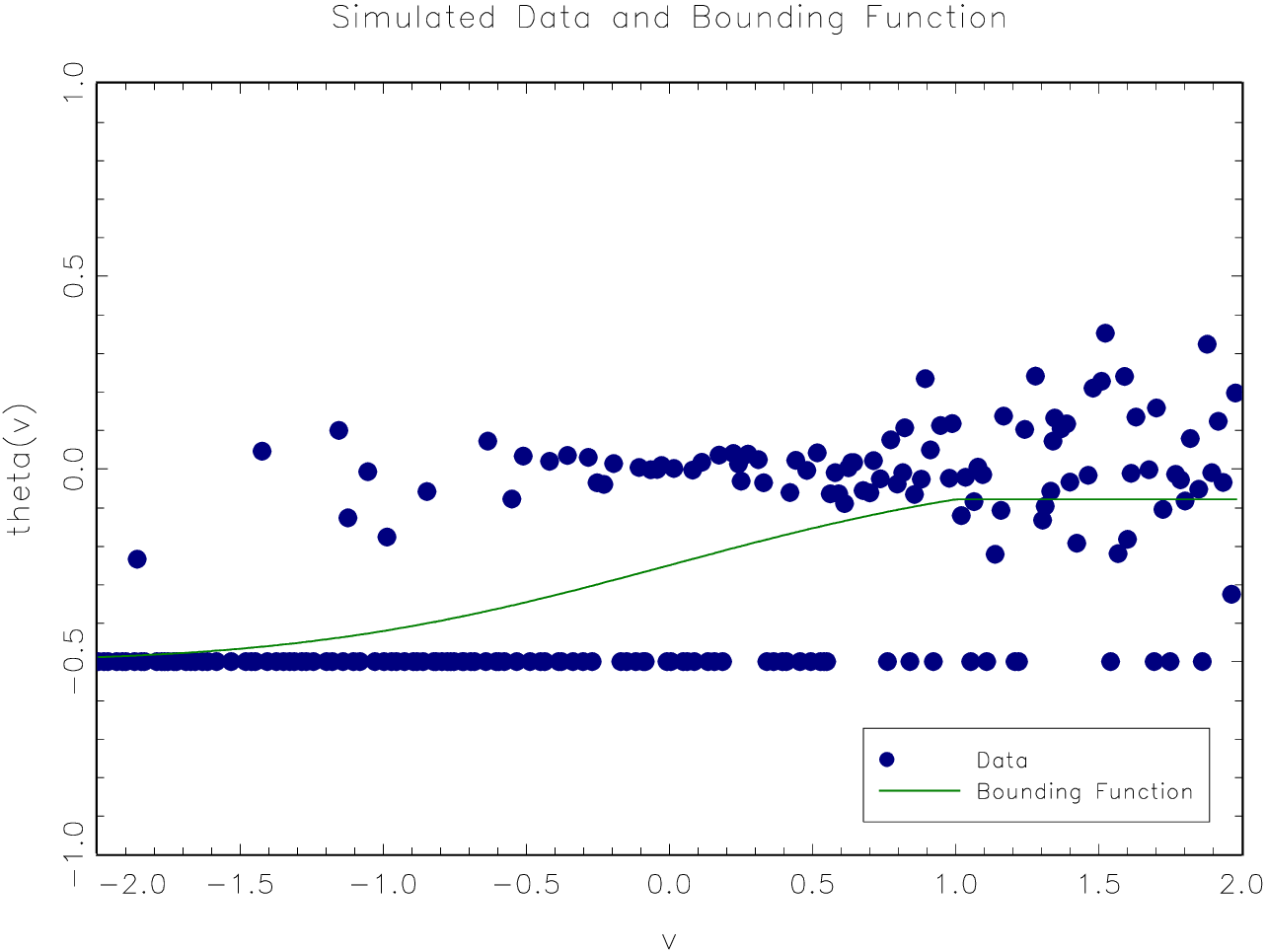}
}
\end{center}
\end{figure}

\clearpage
\begin{figure}[htbp]
\caption{Simulated Data and Bounding Functions: DGP7 and DGP8}
\label{figure_dgp78}
\begin{center}
\makebox{
\includegraphics[origin=bl,scale=.5,angle=90]{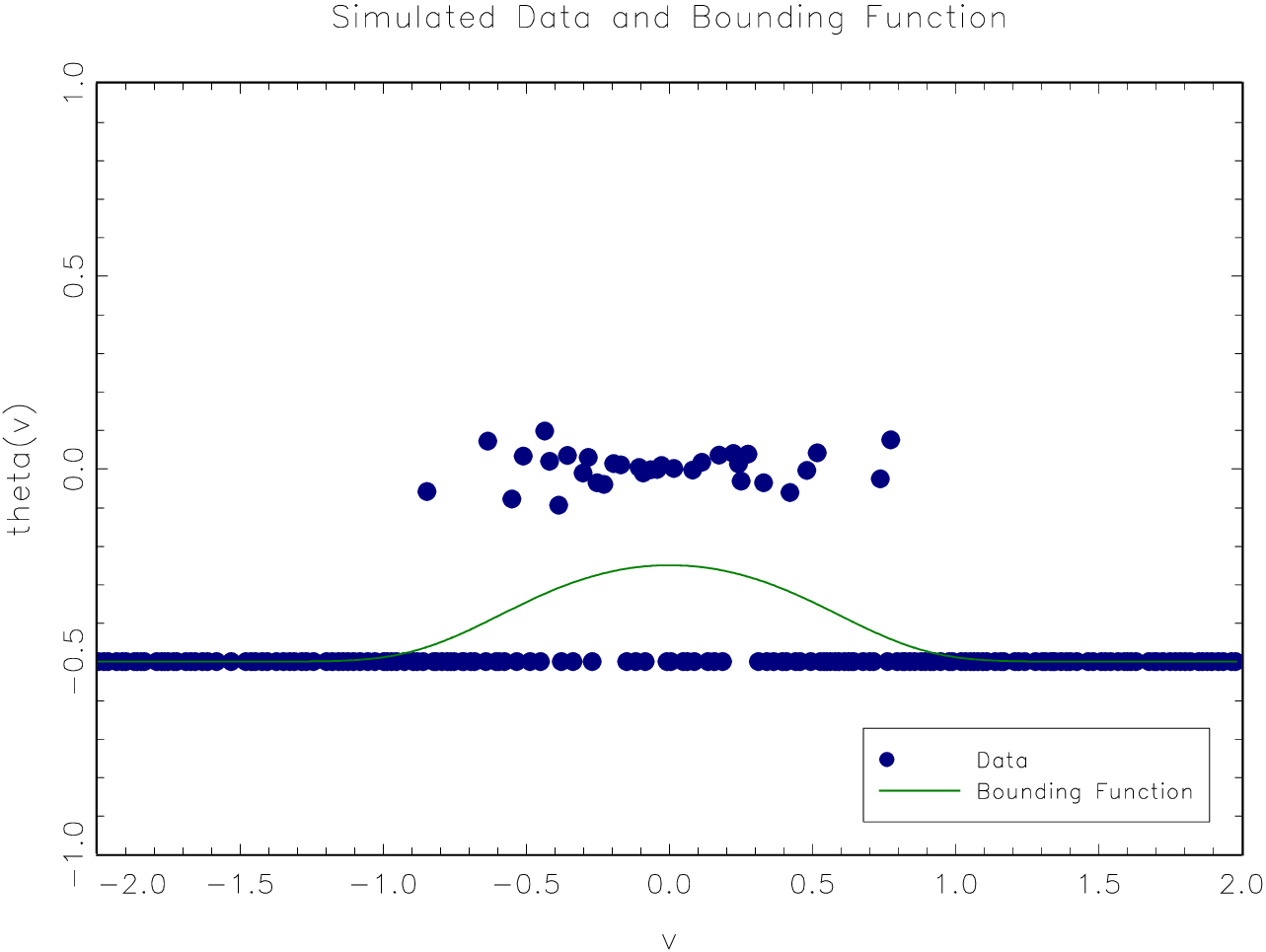}
}
\makebox{
\includegraphics[origin=bl,scale=.5,angle=90]{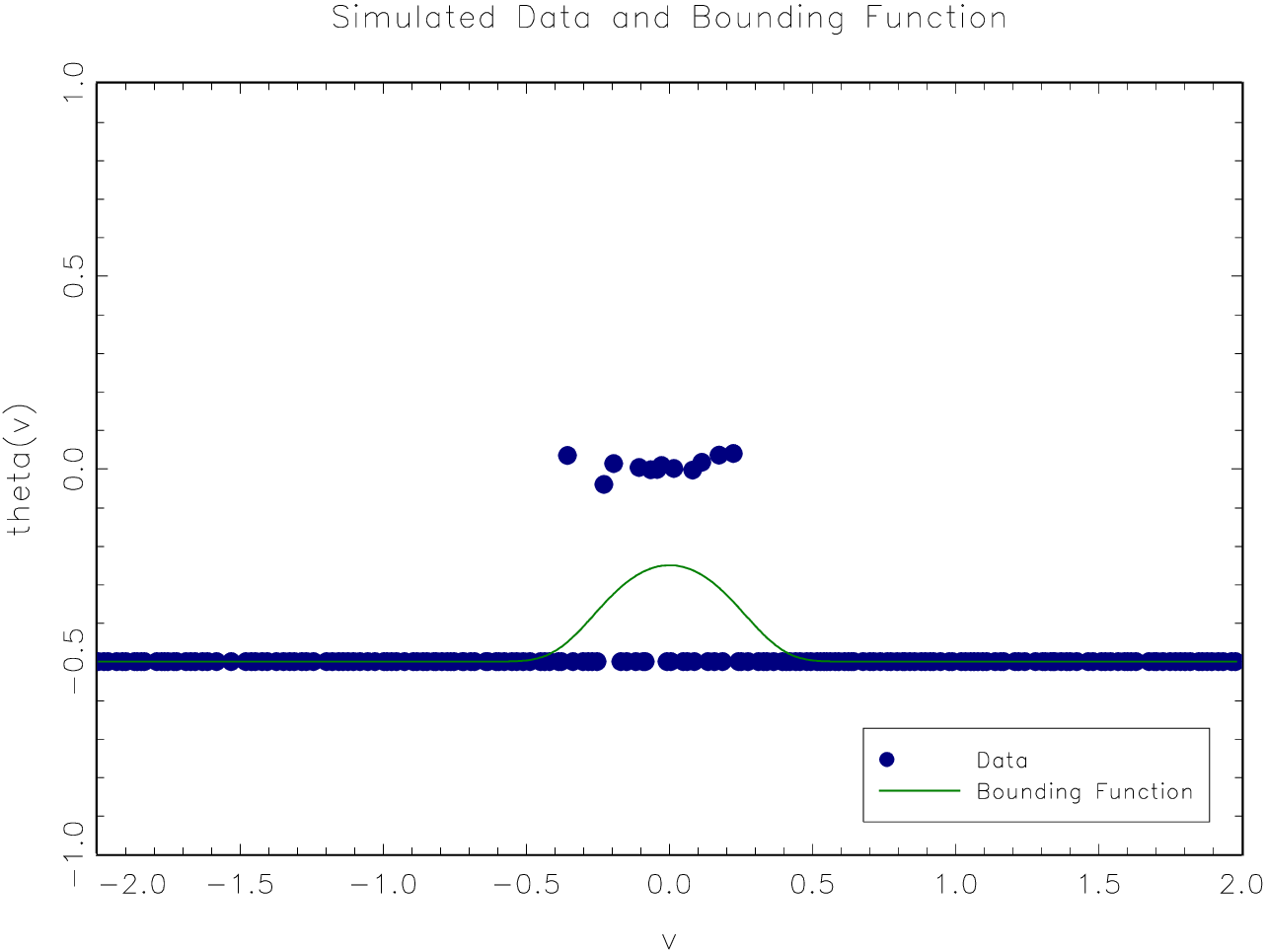}
}
\end{center}
\end{figure}

\clearpage
\begin{figure}[htbp]
\caption{Simulated Data and Bounding Functions: DGP9 and DGP10}
\label{figure_dgp12_AS}
\begin{center}
\makebox{
\includegraphics[origin=bl,scale=.5,angle=90]{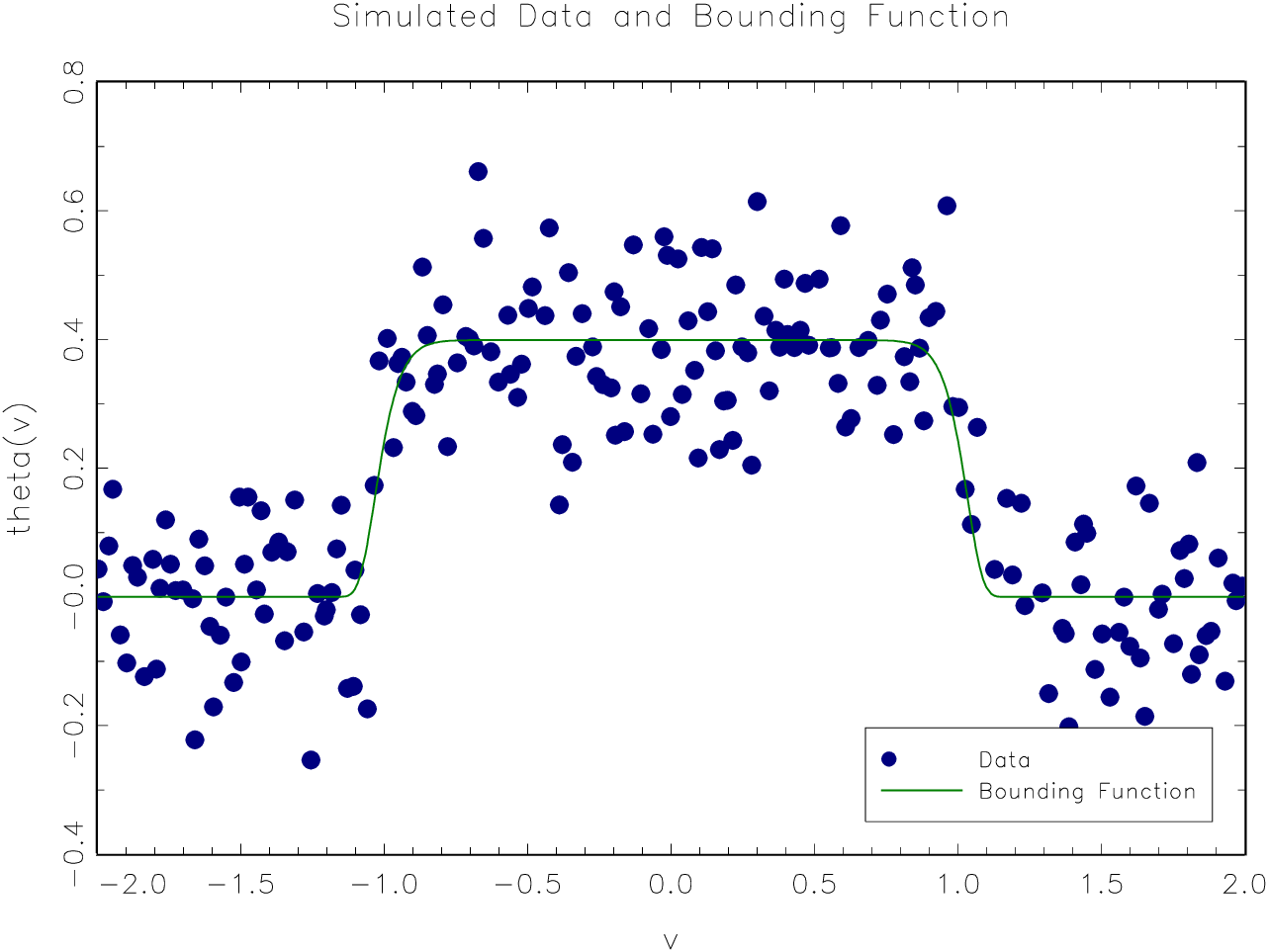}
}
\makebox{
\includegraphics[origin=bl,scale=.5,angle=90]{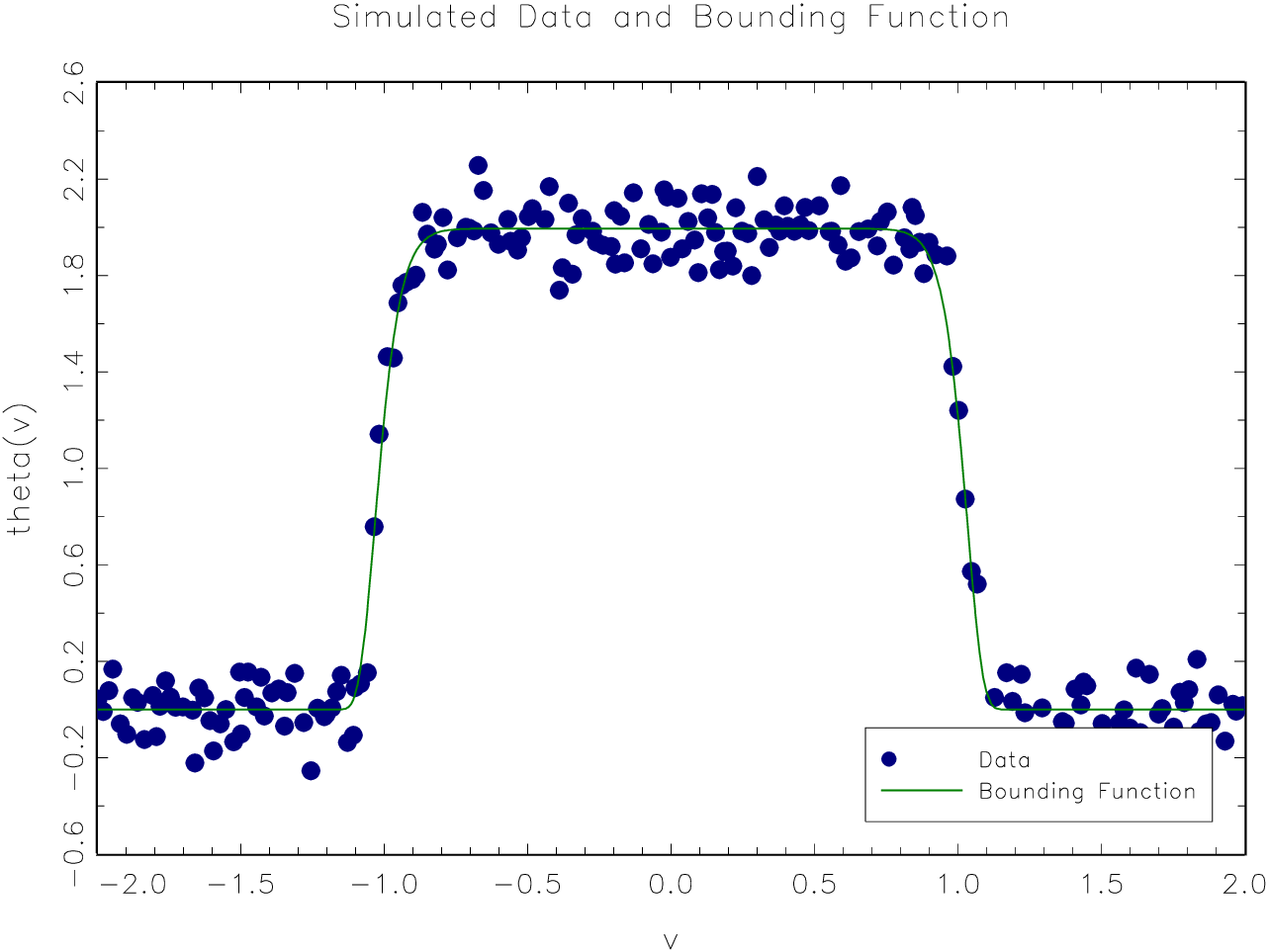}
}
\end{center}
\end{figure}

\clearpage
\begin{figure}[htbp]
\caption{Simulated Data and Bounding Functions: DGP11 and DGP12}
\label{figure_dgp34_AS}
\begin{center}
\makebox{
\includegraphics[origin=bl,scale=.5,angle=90]{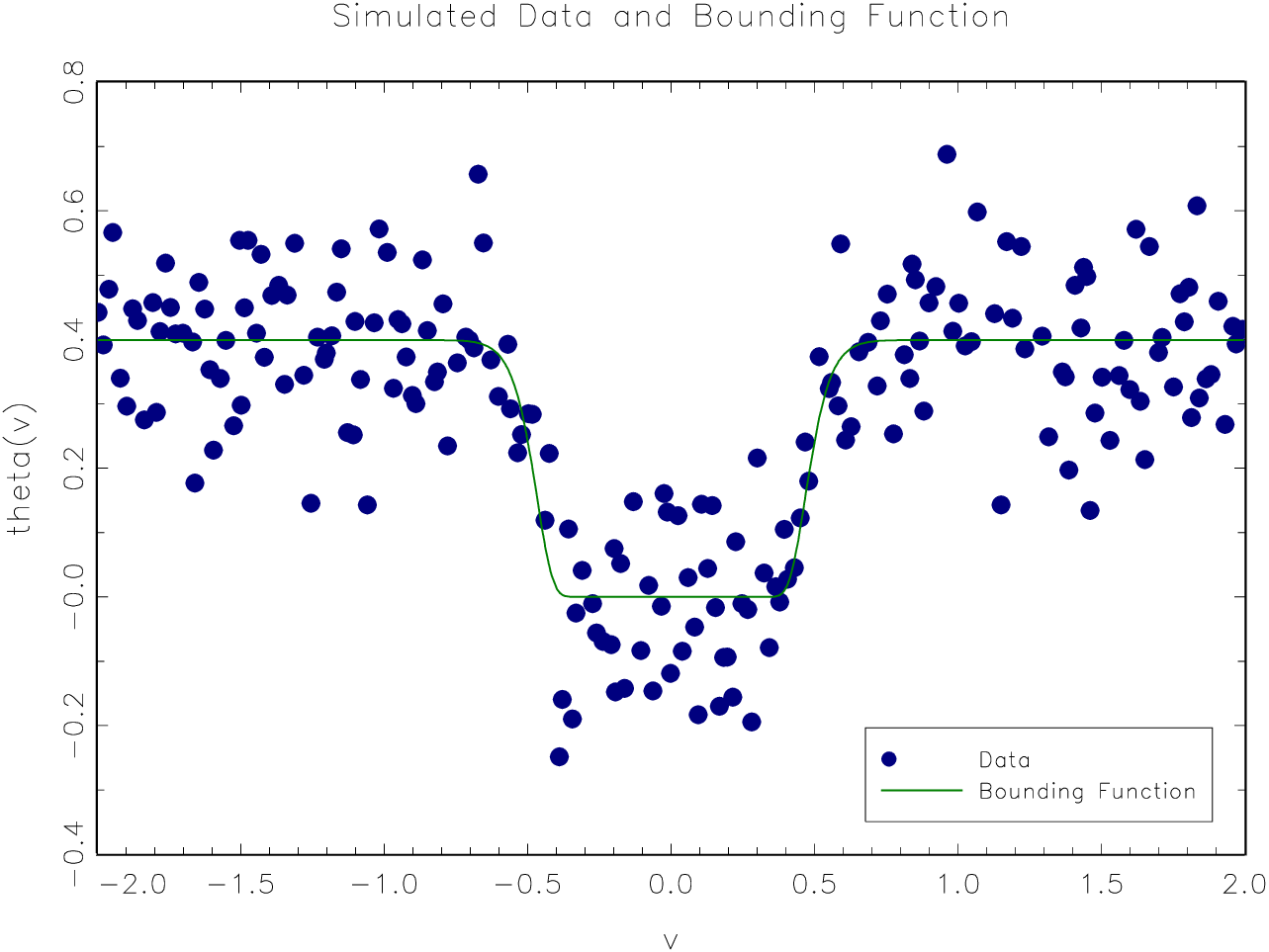}
}
\makebox{
\includegraphics[origin=bl,scale=.5,angle=90]{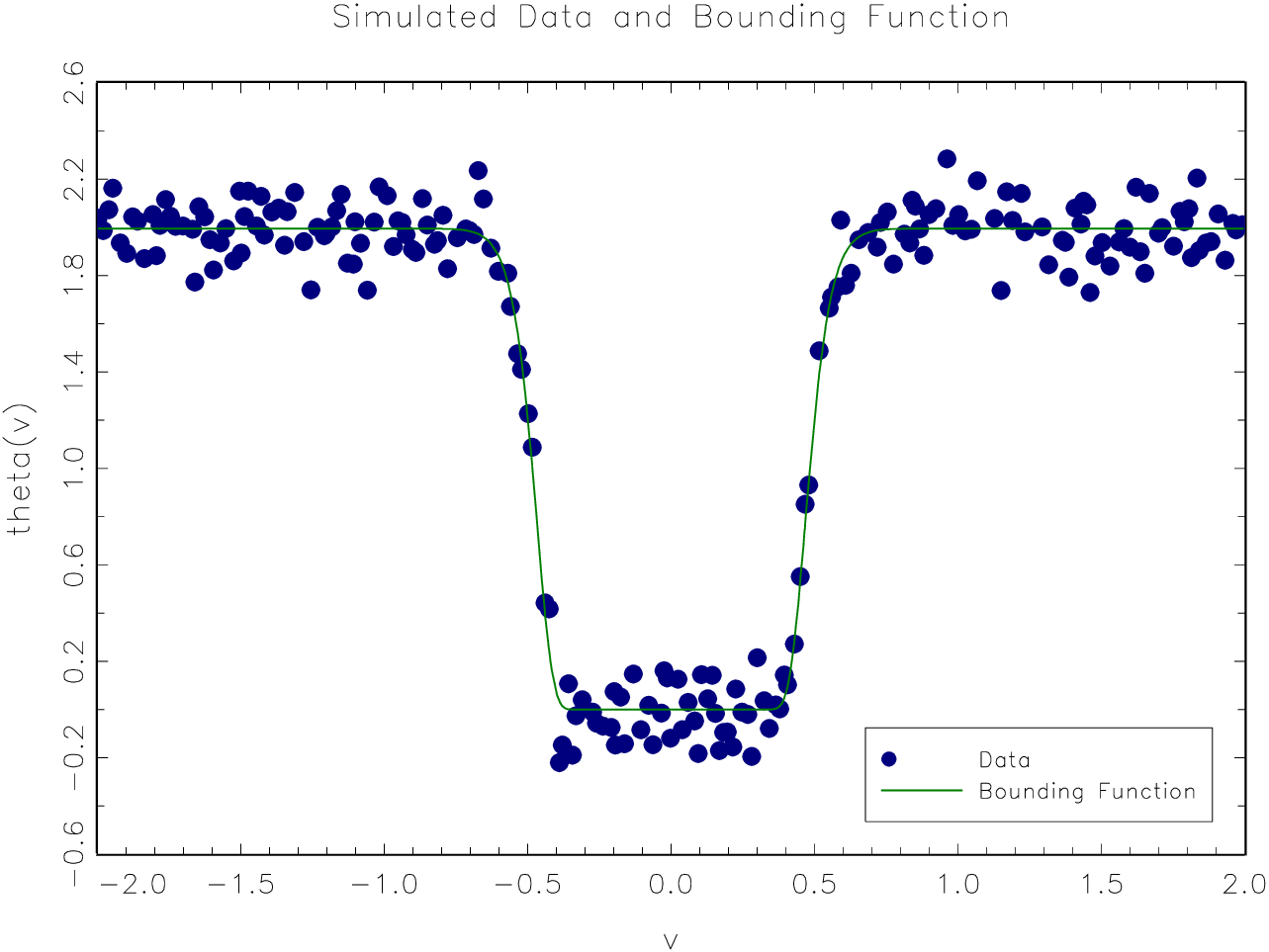}
}
\end{center}
\end{figure}

\end{document}